\documentclass[article,onefignum,onetabnum]{siamart190516}

\usepackage{lipsum}
\usepackage{amsfonts}
\usepackage{graphicx}
\usepackage{epstopdf}
\ifpdf
  \DeclareGraphicsExtensions{.eps,.pdf,.png,.jpg}
\else
  \DeclareGraphicsExtensions{.eps}
\fi

\newsiamthm{assumption}{Assumption}
\crefname{assumption}{Assumption}{Assumptions}
\Crefname{assumption}{Assumption}{Assumptions}

\newsiamthm{remark}{Remark}

\crefname{section}{Section}{Sections}
\crefname{subsection}{Section}{Sections}
\Crefname{subsection}{Section}{Sections}

\headers{PinT for scalar conservation laws}{O. A. Krzysik, H. De Sterck, R. D. Falgout, J. B. Schroder}

\title{Parallel-in-time solution of scalar nonlinear conservation laws\thanks{A published version of this preprint is available at \url{https://doi.org/10.1137/24M1630268}.
\funding{Los Alamos Laboratory report number LA-UR-25-24409. 
This work was performed under the auspices of the U.S. Department of Energy by Lawrence Livermore National Laboratory under Contract DE-AC52-07NA27344 (LLNL-JRNL-858106).
This work was supported in part by the U.S. Department of Energy, Office of Science, Office of Advanced Scientific Computing Research, Applied Mathematics program, and by NSERC of Canada.
}}}

\author{O. A. Krzysik\thanks{Department of Applied Mathematics, University of Waterloo, Waterloo, Ontario, Canada. \textit{Present address:} Theoretical Division, Los Alamos National Laboratory, Los Alamos, New Mexico, USA
  (\email{okrzysik@lanl.gov}, \url{https://orcid.org/0000-0001-7880-6512}).}
\and
H. De Sterck\thanks{Department of Applied Mathematics, University of Waterloo, Waterloo, Ontario, Canada} 
  (\email{hans.desterck@uwaterloo.ca}, \url{https://orcid.org/0000-0002-1641-932X}).
\and R. D. Falgout\thanks{Center for Applied Scientific Computing, Lawrence Livermore National Laboratory, Livermore, California, USA 
  (\email{falgout2@llnl.gov}, \url{https://orcid.org/0000-0003-4884-0087}).}
  \and J. B. Schroder\thanks{Department of Mathematics and Statistics, University of New Mexico, Albuquerque, New Mexico, USA
  (\email{jbschroder@unm.edu}, \url{https://orcid.org/0000-0002-1076-9206}).}
}

\usepackage{amsopn}
\DeclareMathOperator{\diag}{diag} 

\usepackage{xcolor}
\usepackage{bm} 

\usepackage{enumitem} 

\usepackage{comment} 

\usepackage{mathtools} 
\mathtoolsset{centercolon} 

\usepackage{algorithm,algorithmicx}
\usepackage[noend]{algpseudocode}

\usepackage{amssymb} 
\usepackage{makecell} 
\usepackage{multirow} 
\usepackage{makecell} 
\usepackage{pifont}      

\usepackage{boldline}  

\usepackage{grffile} 

\Crefname{subsection}{Section}{Sections}

\usepackage{tikz}
\usetikzlibrary{calc}
\usepackage{graphbox}

\renewcommand{\d}[0]{\ensuremath{\operatorname{d}\!}} 

\newcommand{\wh}[1]{\widehat{#1}} 								  

\newcommand{\wt}[1]{\widetilde{#1}}

\newcommand{\bb}[1]{\bar{\bm{#1}}}

\makeatletter
\renewcommand\tableofcontents{%
\@starttoc{toc}%
}
\makeatother

\usepackage[us,12hr]{datetime} 

\usepackage{xcolor}
\usepackage{marginnote}
\makeatletter
\@mparswitchfalse%
\normalmarginpar
\makeatother

\begin{document}

\allowdisplaybreaks

%
%

%
%
%



\maketitle

\begin{abstract}
	We consider the parallel-in-time solution of scalar nonlinear conservation laws in one spatial dimension. 
	The equations are discretized in space with a conservative finite-volume method using weighted essentially non-oscillatory (WENO) reconstructions, and in time with high-order explicit Runge-Kutta methods.
	The solution of the global, discretized space-time problem is sought via a nonlinear iteration that uses a novel linearization strategy in cases of non-differentiable equations.
	Under certain choices of discretization and algorithmic parameters, the nonlinear iteration coincides with Newton's method, although, more generally, it is a preconditioned residual correction scheme.
	At each nonlinear iteration, the linearized problem takes the form of a certain discretization of a linear conservation law over the space-time domain in question. 
	An approximate parallel-in-time solution of the linearized problem is computed with a single multigrid reduction-in-time (MGRIT) iteration, however, any other effective parallel-in-time method could be used in its place.
	The MGRIT iteration employs a novel coarse-grid operator that is a modified conservative semi-Lagrangian discretization and generalizes those we have developed previously for non-conservative scalar linear hyperbolic problems.
	Numerical tests are performed for the inviscid Burgers and Buckley--Leverett equations.
	For many test problems, the solver converges in just a handful of iterations with convergence rate independent of mesh resolution, including problems with (interacting) shocks and rarefactions.
%
\end{abstract}



\section{Introduction}
\label{sec:introduction}

Over the last six decades a wide variety of parallel-in-time methods for ordinary and partial differential equations (ODEs and PDEs) have been developed \cite{Gander2015,Ong_Schroder_2020}.
A key challenge for parallel-in-time methods is their apparent lack of robustness for hyperbolic PDEs, and for advection-dominated PDEs more broadly.
Many of these solvers are documented to perform inadequately on advective problems, where they are typically either non-convergent, or their convergence is slow and not robust with respect to solver and/or problem parameters \cite{Gander_Vandewalle_2007,
Dai_Maday_2013,
Steiner_etal_2015,
Dobrev_etal_2017,
Schmitt_etal_2018,
Schroder_2018,
Ruprecht_2018,
Howse_etal_2019,
DeSterck_etal_2021,
Southworth_etal_2021,
KrzysikThesis2021}. 
In particular this is true for the iterative multilevel-in-time solver, multigrid reduction-in-time (MGRIT) \cite{Falgout_etal_2014}, and for the iterative two-level-in-time solver Parareal \cite{Lions_etal_2001}. In contrast, these solvers typically converge rapidly for diffusion-dominated PDEs.

Recently in \cite{DeSterck_etal_2023_MOL,DeSterck_etal_2023_SL,DeSterck_etal_2025_LFA} we developed novel MGRIT/Parareal solvers for scalar, \textit{linear} hyperbolic PDEs that converge substantially more robustly than existing solvers. 
Most existing approaches use a coarse-grid problem that is based on discretizing directly the underlying PDE, without regard to the fine-grid discretization.
In contrast, we use a semi-Lagrangian discretization of the PDE to ensure that coarse-level information propagates correctly along characteristic curves, but, crucially, the discretization is modified with a correction term to approximately account for the truncation error of the fine-grid problem.
The end result is a coarse-grid discretization which more faithfully matches the fine-grid discretization, leading to fast and robust convergence of the solver.
This coarse-level correction technique is based on that devised in \cite{Yavneh_1998} to improve performance of classical geometric multigrid on steady state advection-dominated PDEs. In fact, the issue leading to non-robustness in the steady case, an inadequate coarse-grid correction of smooth \textit{characteristic components} \cite{Brandt_1981,Yavneh_1998}, is also responsible, at least in part, for convergence issues of MGRIT/Parareal on advective problems \cite{DeSterck_etal_2025_LFA}.

In this paper, we develop parallel-in-time solvers for \textit{nonlinear} hyperbolic PDEs.
Specifically, we consider solving discretizations of time-dependent, one-dimensional, scalar, \textit{nonlinear conservation laws} of the form
\begin{align} \label{eq:cons-law} \tag{cons}
\frac{\partial u}{\partial t} + \frac{\partial f(u)}{\partial x}  = 0, 
\quad (x, t) \in (-1, 1) \times (0, T],
\quad u(x, 0) = u_0(x),
\end{align}
with solution $u = u(x,t)$ and \textit{flux function} $f = f(u)$. For simplicity we consider periodic spatial boundary conditions $u(-1, t) = u(1, t)$, although other boundary conditions could be used.
Our numerical tests consider the Burgers and Buckley--Leverett equations with solutions containing interacting shock and rarefaction waves.

Some previous work has considered parallel-in-time methods for PDEs of this form. 
For example, \cite{Howse_etal_2019} applied MGRIT to the inviscid Burgers equation, with a coarse-grid problem based on directly discretizing the PDE;
while the solver converged for 1st-order accurate discretizations, it did not lead to speed-up in parallel tests, and it diverged when applied to higher-order discretizations. 
In \cite{Danieli_MacLachlan_2023}, MGRIT was applied to high-order accurate discretizations of nonlinear hyperbolic PDEs; however, the approach appears to require a fine-grid time-step size so small that the CFL limit is not violated on the coarsest grid, so that the method is not practical.
In \cite{Nielsen_etal_2018}, Parareal was applied to a high-order discretization of a hyperbolic system, where the ``coarse'' problem was a cheaper discretization instead of one coarsened in time. While this did result in some modest speed-up in parallel, the robustness of the algorithm with respect to parameter or problem choice is unclear.
In \cite{Liu_etal_2023}, discretized PDEs of the form \eqref{eq:cons-law} were solved with non-smooth optimization algorithms by reformulating the underlying PDE as a constrained optimization problem. 
While the solver in \cite{Liu_etal_2023} is parallel-in-time, the convergence rate is slow relative to the work per iteration such that it is unlikely to be competitive with sequential time-stepping.

In this work we focus on multilevel-in-time methods, namely MGRIT. We note, however, that some alternative approaches have been developed recently which do not seem to suffer for some hyperbolic problems, at least to the same extent, as MGRIT and Parareal do (when used with naive direct coarse-grid discretizations) \cite{Gander_Guttel_2013,
Gander_etal_2018_paraexp-nonlin,
Goddard-Wathen-2019-all-at-once-wave,
Gander_etal_2019_direct_wave,
Gander_Wu_2020,
Liu_Wu_2020,
Danieli_Wathen_2021,
Liu_etal_2022,
Liu-Wu-2022-sinc-nystrom,
Hon-SC-2023-block-precond-wave}. 
With the exception of \cite{Gander_Guttel_2013,Gander_etal_2018_paraexp-nonlin}, these approaches target temporal parallelism by diagonalizing in time, often using discrete-Fourier-transform-like techniques.
While some of these methods have been applied to certain quasilinear hyperbolic-like PDEs
(wave equations with nonlinear reaction terms \cite{Gander_etal_2018_paraexp-nonlin,Gander_Wu_2020,Liu_Wu_2020}; viscous Burgers equation \cite{Wu_Zhou_2021}),
we are unaware of their application to PDEs of the form \eqref{eq:cons-law}.
Thus there is a need for effective parallel-in-time methods applied to nonlinear conservation laws, which we address in this paper.

The methodology used here for solving \eqref{eq:cons-law} employs a global linearization of the discretized problem and uses multigrid (i.e., MGRIT) to approximately solve the linearized problem at each nonlinear iteration. 
Specifically, the linearized problems correspond to certain non-standard discretizations of the \textit{linear conservation law}
\begin{align}
\label{eq:cons-lin} \tag{cons-lin}
\frac{\partial e}{\partial t} + \frac{\partial }{\partial x} \big( \alpha(x, t) e \big)  &= 0, 
\quad (x, t) \in (-1, 1) \times (0, T],
\quad e(x, 0) = e_0(x),
\end{align}
with solution $e(x,t)$, flux function $f(e, x, t) = \alpha(x, t) e$, and periodic spatial boundary conditions $e(-1, t) = e(1, t)$.
To solve these linearized problems with MGRIT, we adapt our existing methodology for (non-conservative) linear hyperbolic PDEs described in  \cite{DeSterck_etal_2023_MOL,DeSterck_etal_2023_SL}.
We emphasize, however, that any other parallel-in-time method effective on linearized problems of the form \eqref{eq:cons-lin} could be used in place of MGRIT.

Global linearization paired with multigrid as an inner solver is widely used in the context of (steady) elliptic PDEs, as in Newton--multigrid, for example \cite{Briggs_etal_2000,Trottenberg_etal_2001}, and has been considered in the parallel-in-time literature \cite{Benedusi_etal_2016,Duennebacke_etal_2021}.
The approach contrasts with fully nonlinear multigrid, known as the full approximation scheme (FAS) \cite{Brandt1977}, which, in essence, is what was used in \cite{Nielsen_etal_2018,Howse_etal_2019,Danieli_MacLachlan_2023}.
In principle, our existing MGRIT methodology for linear hyperbolic problems  \cite{DeSterck_etal_2023_MOL,DeSterck_etal_2023_SL,DeSterck_etal_2025_LFA} could be extended in a FAS-type approach to solve \eqref{eq:cons-law} by using a modified semi-Lagrangian discretization of \eqref{eq:cons-law} on the coarse grid.
However, it is not obvious how to develop a coarse-grid semi-Lagrangian discretization for \eqref{eq:cons-law} when the PDE solution contains shocks, because the coarse grid requires large time-step sizes. 
That is, as far as we are aware, known semi-Lagrangian methods capable of solving \eqref{eq:cons-law} have an Eulerian-style CFL limit $\delta t \lesssim h$ when the solution contains shocks  \cite{Huang_etal_2016,Huang_Arbogast_2016,Cai_etal_2021}.
In contrast, semi-Lagrangian methods for linear problems of type \eqref{eq:cons-lin} do not have a CFL limit.
Having said that, we are aware of discretizations developed by LeVeque \cite{Leveque1982,LeVeque1985} that can solve \eqref{eq:cons-law} with large time-steps, even though the cost per time step is relatively large.
Application of these discretizations in a FAS setting could be investigated in future work.

The remainder of this paper is organized as follows.
\Cref{sec:discretization} details the PDE discretizations we use. 
\Cref{sec:nonlin-scheme} presents the nonlinear iteration scheme to solve the discretized problems, and \cref{sec:linearization} details the linearization procedure. 
\Cref{sec:mgrit} develops an MGRIT iteration for approximately solving the linearized problems, and numerical results are then given in \cref{sec:num-res}.
Conclusions are drawn in \Cref{sec:conclusion}.
Supplementary materials are also included which, primarily, describe some additional details of the MGRIT solver.
The MATLAB code used to generate the results in this manuscript can be found at \url{https://github.com/okrzysik/pit-nonlinear-hyperbolic} (v1.1.0).

\section{PDE discretization}
\label{sec:discretization}

This section gives an overview of the discretizations used for \eqref{eq:cons-law} and \eqref{eq:cons-lin}: Explicit Runge-Kutta (ERK) time integration combined with the finite-volume (FV) method in space.
For more detailed discussion on these discretizations see, e.g., \cite{Shu_1998,Shu2009,Hesthaven_2017}.

Since we are concerned with both the nonlinear \eqref{eq:cons-law} and linear \eqref{eq:cons-lin} PDEs, we consider discretizing a conservation law with flux (possibly) depending explicitly on $u$, $x$ and $t$, $f = f(u, x, t)$.
However, to improve readability, we omit for the most part any explicit dependence on $x$ and $t$ and write $f = f(u)$.

Begin by discretizing the spatial domain $x \in [-1, 1]$ into $n_x$ FV cells of width $h$. The $i$th cell is ${\cal I}_i = [x_{i - 1/2}, x_{i + 1/2}]$, with $x_{i \pm 1/2} = x_i \pm h/2$.
Integrating the PDE over ${\cal I}_i$ gives the local conservation relation,
\begin{align} 
\label{eq:cons-integral-form}
\frac{\d \bar{u}_i}{\d t} 
=
-
\frac{
   f \big( u(x_{i+1/2}) \big) 
- f \big( u(x_{i-1/2}) \big)
}{h},
\quad
\bar{u}_i(t) := \frac{1}{h} \int_{{\cal I}_i} u(x, t) \d x,
\end{align}
in which $\bar{u}_i$ is the cell-average of the exact solution over the $i$th cell.
The vector of cell averages is written as $\bb{u} = (\bar{u}_1, \ldots, \bar{u}_{n_x})^\top \in \mathbb{R}^{n_x}$.
Next, \eqref{eq:cons-integral-form} is approximated by replacing the physical flux $f$ with the numerical flux: 
$
f \big( u(x_{i+1/2} ) \big) 
\approx 
\wh{f} 
\big( 
u_{i+1/2}^-, u_{i+1/2}^+
\big)
=: \wh{f}_{i+1/2}.
$
Specific details about our choice of numerical flux follow in \cref{sec:LFF}.
The numerical flux takes as inputs $u_{i+1/2}^- = u_{i+1/2}^-(\bar{\bm{u}})$ and $u_{i+1/2}^+ = u_{i+1/2}^+(\bar{\bm{u}})$ which are reconstructions of the solution at $x = x_{i+1/2}$ based on cell averages of $u$ in neighboring cells.
The spatial accuracy of the scheme is determined by the accuracy of these reconstructions, with further details on this procedure given in \cref{sec:reconstruction}.
%

Plugging the numerical flux into \eqref{eq:cons-integral-form} gives the semi-discretized scheme
\begin{align} \label{eq:FV-MOL-ODEs}
\frac{\d \bar{u}_i}{\d t} 
\approx
-
\frac{\wh{f}_{i+1/2} - \wh{f}_{i-1/2}}{h}
=: \big( L(\bar{\bm{u}}) \big)_i, 
\quad i = 1, \ldots, n_x.
\end{align}
The operator $L \colon \mathbb{R}^{n_x} \to \mathbb{R}^{n_x}$ represents the spatial discretization.
%
%
The ODE system \eqref{eq:FV-MOL-ODEs} is then approximately advanced forward in time using an ERK method.
The simplest time integration is the forward Euler method:
\begin{align} \label{eq:ERK1}
\bar{\bm{u}}^{n+1} = \bar{\bm{u}}^{n} + \delta t L (\bar{\bm{u}}^{n}) =: F( \bar{\bm{u}}^{n}).
\end{align}
Here, $\bb{u}^{n} \approx \bb{u}(t_n)$ is the numerical approximation.
The most commonly used time integration method for high-order discretizations of hyperbolic PDEs is the so-called optimal 3rd-order strong-stability preserving Runge-Kutta method:
\begin{subequations} \label{eq:ERK3}
\begin{align}
\label{eq:ERK3-a}
\bar{\bm{u}}^{n, 1} &= F( \bar{\bm{u}}^{n} ),
\\
\label{eq:ERK3-b}
\bar{\bm{u}}^{n, 2} &= \frac{3}{4} \bar{\bm{u}}^{n} + \frac{1}{4} F \big( \bar{\bm{u}}^{n,1}),
\\
\label{eq:ERK3-c}
\bar{\bm{u}}^{n+1} &= \frac{1}{3}\bar{\bm{u}}^{n} + \frac{2}{3} F \big( \bar{\bm{u}}^{n,2} \big).
\end{align}
\end{subequations}
In either case of \eqref{eq:ERK1} or \eqref{eq:ERK3}, the application of an ERK method to the ODE system \eqref{eq:FV-MOL-ODEs} results in a fully discrete system of equations of the form
\begin{align} \label{eq:one-step-disc}
\bar{\bm{u}}^{n+1} 
= 
\Phi(\bar{\bm{u}}^n),
\quad n = 0, 1, \ldots, n_t-2,
\end{align}
with $\Phi \colon \mathbb{R}^{n_x} \to \mathbb{R}^{n_x}$ the \textit{time-stepping operator} carrying out some ERK scheme.
In either case, the time discretization is ``one step,'' meaning that the solution at $t_{n+1}$ is computed only from the solution at the previous time step, $t_n$.
In this work, the time domain $t \in [0, T]$ of \eqref{eq:cons-law} and \eqref{eq:cons-lin} is assumed to be discretized with $n_t$ points equispaced by a distance $\delta t$.

When 1st- or 3rd-order FV spatial discretizations are used, they are paired with \eqref{eq:ERK1} or \eqref{eq:ERK3}, respectively.
We do not present results for discretizations with order of accuracy greater than three; note that our methodology is extensible to higher-order discretizations but we have not rigorously tested it on such problems.
In all cases, the constant time-step size $\delta t$ is chosen such that $c_{\max} := \max_{u \in [\min(u_0(x)), \max(u_0(x))]} |f'(u)| \frac{\delta t}{h}$ is slightly smaller than one.\footnote{When sequential time-stepping, for efficiency reasons, one would typically adapt $\delta t$ at each step so as to step close to the CFL limit rather than fix $\delta t$ as a constant for all steps as we do here. 
We expect that such adaptive time-stepping could be incorporated into our approach by utilising nested iteration methodology, as has been done with MGRIT for diffusion problems \cite{Falgout_etal_2019_BDF,Falgout-etal-2021-richardson}.}
%

\subsection{Numerical flux}
\label{sec:LFF}

In this work the numerical flux in \eqref{eq:FV-MOL-ODEs} is chosen as the Lax--Friedrichs (LF) flux \cite[Section 12.5]{LeVeque_2004}:
\begin{align} \label{eq:LFF}
\wh{f}_{i+1/2}
&= 
\frac{1}{2}
\Big[
\Big( 
f \big( u_{i+1/2}^{-} \big) 
+ 
f \big(u_{i+1/2}^{+} \big) \Big) 
+ \nu_{i+1/2} 
\big( 
u_{i+1/2}^{-}, u_{i+1/2}^{+} 
\big)
\big( 
u_{i+1/2}^{-} - u_{i+1/2}^{+} 
\big)
\Big].
\end{align}
In \eqref{eq:LFF}, $\nu_{i+1/2} = \nu_{i+1/2} 
\big( 
u_{i+1/2}^{-}, u_{i+1/2}^{+} 
\big)$ controls the strength of the numerical dissipation in the spatial discretization.
We consider both \textit{global} Lax--Friedrichs (GLF) and \textit{local} Lax--Friedrichs (LLF) fluxes, corresponding to $\nu_{i+1/2}$ being chosen globally over the spatial domain and locally for each cell interface, respectively.
The LLF flux is less dissipative than the GLF flux, resulting in sharper approximations at discontinuities, especially for low-order reconstructions.
%

In cases of a GLF flux, we write \eqref{eq:LFF} as $\wh{f}^{\rm (global)}_{i+1/2}$, and we take
\begin{align} \label{eq:LFF-global}
\nu_{i+1/2} = \nu^{(\rm global)} := 
\max 
\limits_{
\substack{w \in [\min(u_0(x)), \max(u_0(x))] \\ x \in [-1,1]}
} 
|f'(w)|.
\end{align}
That is, $\nu_{i+1/2}$ is constant over the entire space-time domain, and is therefore independent of the local reconstructions $u_{i+1/2}^{\pm}$.
In cases of a LLF flux we write \eqref{eq:LFF} as $\wh{f}^{\rm (local)}_{i+1/2}$, and choose the dissipation parameter according to
\begin{align} \label{eq:LFF-local}
\nu_{i+1/2} 
= 
\nu_{i+1/2}^{\rm (local)}
:= \max \limits_{w \in 
[
\min (u_{i+1/2}^{-}, u_{i+1/2}^{+} ), 
\max (u_{i+1/2}^{-}, u_{i+1/2}^{+} )
]} 
|f'(w)|.
\end{align}
If $f$ is convex (i.e., $f''$ does not change sign), as in the case of the Burgers equation (see \eqref{eq:burgers}), for example, then \eqref{eq:LFF-local} simplifies to
$
\max 
\big( 
	\big| f' \big(u_{i+1/2}^{-} \big) \big|, 
	\big| f' \big(u_{i+1/2}^{+} \big) \big| 
\big)
$ \cite[p. 233]{LeVeque_2004}.
For non-convex $f$, as in the case of the Buckley--Leverett equation (see \eqref{eq:buck}), \eqref{eq:LFF-local} is less straightforward to compute.

Finally, let us consider the LF flux \eqref{eq:LFF} in the special case of the linear conservation law \eqref{eq:cons-lin} where the flux is $f(e, x, t) = \alpha(x,t) e$.
For this problem, we denote the flux as (omitting $t$ dependence for readability)
\begin{align} 
\wh{f}\big( e_{i+1/2}^{-}, e_{i+1/2}^{+}) 
&=
\frac{1}{2}
\left[
 \big( \alpha_{i+1/2}^{-} + \nu_{i+1/2} \big) e_{i+1/2}^{-}
+
\big( \alpha_{i+1/2}^{+} - \nu_{i+1/2} \big) e_{i+1/2}^{+} 
\right],
\\ 
\label{eq:LFF-lin}
&=: 
\wh{f}^{\textrm{lin}}\big( \alpha_{i+1/2}^{-}, e_{i+1/2}^{-}, \alpha_{i+1/2}^{+}, e_{i+1/2}^{+}, \nu_{i+1/2} \big).
\end{align}
In these expressions, $\alpha_{i+1/2}^{\pm} \approx \alpha(x_{i+1/2})$ are reconstructions of the interfacial wave-speed; direct evaluations of $\alpha(x_{i+1/2})$ are not used in this flux because later on we encounter situations where $\alpha$ is not known explicitly as a function of $x$, but is available only through reconstruction.
Furthermore, we specify later in \cref{sec:linearization-II} how the dissipation coefficient $\nu_{i+1/2}$ in \eqref{eq:LFF-lin} is chosen.
%

\subsection{WENO reconstructions}
\label{sec:reconstruction}

We now outline how the interface reconstructions $u_{i+1/2}^{\pm}$ that enter into the numerical flux \eqref{eq:LFF} are computed.
We begin by introducing polynomial reconstruction, for which further details can be found in Supplementary Materials Section \ref{SMsec:reconstruction-overview}.

Consider the reconstruction of a smooth function $u(x)$ based on its cell averages.
At a given $x \in {\cal I}_i$, let $q_{i}^{\ell}(x) \approx u(x)$ be a reconstruction polynomial, of degree at most $k-1$, depending on $k \geq 1$ cell averages of $u$ over the cells $S^{\ell}_i := \{  {\cal I}_{i - \ell}, \ldots, {\cal I}_{ i - \ell + (k - 1)} \}$, with \textit{left-shift} $\ell \in \{0, ..., k - 1\}$. 
For each cell ${\cal I}_i$, there are $k$ such polynomials since there are $k$ possible left-shifted stencils.
Polynomial reconstructions at the left and right interfaces of cell ${\cal I}_i$ using a reconstruction stencil with left-shift $\ell$ are then denoted by
\begin{align} \label{eq:reconstuct-lin}
u_{i,-1/2}^{\ell}
:=
q_{i}^{\ell}\big( x_{i-1/2} \big) 
\equiv
\big( \wt{R}^{\ell} \bb{u} \big)_i,
\quad
u_{i,+1/2}^{\ell}
:=
q_{i}^{\ell}\big( x_{i+1/2} \big) 
\equiv 
\big( R^{\ell} \bb{u} \big)_i.
\end{align}
That is, we define $\wt{R}^{\ell}, {R}^{\ell} \in \mathbb{R}^{n_x \times n_x}$ as the linear operators that reconstruct at left- and right-hand interfaces, respectively, of all $n_x$ cells.
To obtain high-order reconstructions of order $2k - 1$, one then takes \textit{weighted combinations} of the $k$ different $k$th-order reconstructions with different left-shifts $\ell$.
Specifically, we write
\begin{align} \label{eq:reconstruct-weighted}
u_{i,-1/2}  ( \bb{u} )
:=
\sum \limits_{\ell =0}^{k-1} 
\wt{b}^{\ell}_{i} ( \bb{u} )
\big( \wt{R}^{\ell} \bb{u} \big)_i, 
\quad
u_{i,+1/2}  ( \bb{u} )
:=
\sum \limits_{\ell = 0}^{k - 1}
b^{\ell}_{i} ( \bb{u} )
\big( R^{\ell} \bb{u} \big)_i,
\end{align}
using weights $\big\{ \wt{b}_i^{\ell}  ( \bb{u} ) \big\}$, $\big\{ {b}_i^{\ell}  ( \bb{u} ) \big\}$.
Then, the inputs to the flux \eqref{eq:LFF} are $u_{i+1/2}^- = u_{i,+1/2}$ and $u_{i+1/2}^+ = u_{i+1,-1/2}$. 
%
There are several choices for $\big\{ \wt{b}_i^{\ell}  ( \bb{u} ) \big\}$, $\big\{ {b}_i^{\ell}  ( \bb{u} ) \big\}$ in \eqref{eq:reconstruct-weighted}.
In the simplest case, they are taken as the so-called \textit{optimal linear weights}, $\big\{ \wt{d}^{\ell} \big\}$, $\big\{ {d}^{\ell} \big\}$, which are independent of $\bb{u}$, and of $i$, and are such that the reconstructions \eqref{eq:reconstruct-weighted} are fully $(2k-1)$-order accurate.
For $k > 1$, however, $\big\{ \wt{d}^{\ell} \big\}$, $\big\{ {d}^{\ell} \big\}$ are not suitable when $u(x)$ lacks sufficient regularity over the large stencil $\cup_{\ell = 0}^{k-1} S^{\ell}_i$ because they lead to spurious oscillations that do not reduce in size as the mesh is refined. 
Instead, the typical choice for hyperbolic conservation laws is to use so-called \textit{weighted essentially non-oscillatory} (WENO) weights, $\big\{ \wt{w}_i^{\ell} (\bb{u}) \big\}$, $\big\{ {w}_i^{\ell} (\bb{u}) \big\}$.
WENO weights are designed such that over stencils where $u$ is smooth they reduce to, as $h \to 0$, the optimal linear weights, and otherwise they adapt to eliminate contributions to the reconstructions from the specific cells where $u(x)$ is non-smooth.
%

\section{Nonlinear iteration scheme}
\label{sec:nonlin-scheme}

In this section, we describe our method for solving the system of equations \eqref{eq:one-step-disc}.
To this end, let us write the system of equations \eqref{eq:one-step-disc} in the form of a residual equation:
\begin{align} \label{eq:non-lin-sys}
\bm{r}(\bb{u}) 
:= 
\bm{b} - {\cal A}(\bb{u}) 
\equiv
\begin{bmatrix}
\bb{u}^0 \\
\bm{0} \\
\vdots \\
\bm{0}
\end{bmatrix}
-
\begin{bmatrix}
& I & \\
&- \Phi( \, \cdot \, ) & I \\
& & \ddots & \ddots  \\
& & & - \Phi( \, \cdot \, ) & I
\end{bmatrix}
\begin{bmatrix}
\bb{u}^{0} \\
\bb{u}^{1} \\
\vdots \\
\bb{u}^{n_t-1}
\end{bmatrix}
= \bm{0}.
\end{align}
Here $\bb{u} = \big( \bb{u}^{0}, \bb{u}^{1}, \ldots, \bb{u}^{n_t-1} \big)^\top \in \mathbb{R}^{n_x n_t}$ is created from concatenating the cell-averages  across the time domain, ${\cal A} \colon \mathbb{R}^{n_x n_t} \to \mathbb{R}^{n_x n_t}$ is the nonlinear space-time discretization operator we seek to invert, and $\bm{b} \in \mathbb{R}^{n_x n_t}$ contains the initial data for the problem.
For some approximate solution $\bb{u}_k \approx \bb{u}$, with iteration index $k$, the residual \eqref{eq:non-lin-sys} is denoted as $\bm{r}_k := \bm{r}(\bb{u}_k)$.

In this paper, we iteratively solve \eqref{eq:non-lin-sys} using the linearly preconditioned residual correction scheme outlined in \cref{alg:richardson}.
The remainder of this section discusses in greater detail the development of \cref{alg:richardson}; however, we first use it to contextualize the key contributions of this paper.
In \cref{alg:richardson}, $P_k$ is a certain linearization of the space-time discretization ${\cal A}$ in \eqref{eq:non-lin-sys}, and the development of this linearization in \cref{sec:linearization} is the first key contribution of this work; specifically, $P_k$ is realized as a certain non-standard discretization of the linear conservation law \eqref{eq:cons-lin}.
In \cref{alg:richardson} linear systems of the form $P_k \bb{e}_k^{\textrm{lin}} = \bm{r}_k$ are solved either directly with sequential time-stepping or approximately with MGRIT.
We again emphasize that the use of MGRIT here is optional, and that any parallel-in-time solver which is effective on problems of the form \eqref{eq:cons-lin} could be used in its place.
In this work we focus on the MGRIT solution option in \cref{alg:richardson} since we are concerned with parallel-in-time solvers. 
The development of an effective MGRIT solver for $P_k$ is the second key contribution of this work, and is the subject of \cref{sec:mgrit}.
Our MGRIT solver is based on a novel, modified FV semi-Lagrangian coarse-grid discretization of \eqref{eq:cons-lin} which is an extension of our modified finite-difference (FD) semi-Lagrangian discretizations in \cite{DeSterck_etal_2023_SL,DeSterck_etal_2023_MOL,DeSterck_etal_2025_LFA} for non-conservative linear hyperbolic PDEs.
The efficacy of our solution methodology is demonstrated in \cref{sec:num-res} by way of numerical tests on challenging nonlinear problems containing shock and rarefaction waves.

\begin{algorithm}[t!]
  \caption{Preconditioned residual correction scheme for \eqref{eq:non-lin-sys}. Input: $\bb{u}_0 \approx \bb{u}$.
  \label{alg:richardson}}
  \begin{algorithmic}[1]
  \While{$\Vert \bm{r}_k \Vert > \textrm{tol}$}
  	\State{${\bb{u}}_{k} \gets \textrm{relax on }{\cal A}(\bb{u}_k) \approx \bm{b}$}\Comment{Nonlinear F-relaxation}\label{ln:nonlin-relax}
  		\State{${\bb{r}}_{k} \gets \bm{b} - {\cal A} ({\bb{u}}_{k})$}\Comment{Compute residual}
  	\If{exact linear solve} \Comment{Solve for linearized error}
	  	\State{${\bb{e}}_k^{\textrm{lin}} \gets P_k^{-1} {\bm{r}}_k$} \Comment{Direct solve (sequential in time)}\label{ln:direct-solve}
  	\ElsIf{approx. linear solve}
	\State{${\bb{e}}_k^{\textrm{lin}} \gets \textrm{MGRIT} (P_k, {\bm{r}}_k )$} \Comment{Approx. solve (parallel in time)}\label{ln:approx-solve} 
	\EndIf
	\State{$\bb{u}_{k+1} \gets {\bb{u}}_k + {\bb{e}}_k^{\textrm{lin}}$}\Comment{Compute new iterate}
	\EndWhile 
  \end{algorithmic}
\end{algorithm}

The remainder of this section now motivates and discusses finer details of the scheme outlined in \cref{alg:richardson}. We seek an iteration of the form $\bb{u}_{k+1} \approx \bb{u}_k + \bb{e}_k$, where $\bb{e}_k := \bb{u} - \bb{u}_k$ is the \textit{nonlinear} algebraic error. 
Assuming that the residual \eqref{eq:non-lin-sys} is sufficiently smooth at $\bb{u}_k$ we can expand it about this point to get
$\bm{r}( \bb{u}_k + \bb{e}_k ) 
= 
\bm{r}(\bb{u}_k) + \nabla_{\bb{u}_k} \bm{r}( \bb{u}_k ) \, \bb{e}_k + {\cal O}( \Vert \bb{e}_k \Vert^2 )
=
\bm{r}_k - \nabla_{\bb{u}_k} {\cal A}( \bb{u}_k ) \, \bb{e}_k + {\cal O}( \Vert \bb{e}_k \Vert^2 )
$.
Now let $P_k \approx \nabla_{\bm{u}_k} {\cal A}( \bb{u}_k ) \in \mathbb{R}^{n_x n_t \times n_x n_t}$ be an approximation to the Jacobian of ${\cal A}$ in \eqref{eq:non-lin-sys} at the point $\bb{u}_k \approx \bb{u}$. 
Dropping squared error terms and setting the linearized residual equal to zero results in the linearly preconditioned residual correction scheme:
\begin{align} \label{eq:lpnri}
\bb{u}_{k+1} = \bb{u}_k + P^{-1}_k \bm{r}_k, \quad k = 0, 1, \ldots
\end{align}
We can define $\bb{e}_k^{\textrm{lin}} := P^{-1}_k \bb{r} \approx \bb{e}_k$ as the \textit{linearized} algebraic error.
The choice $P_k = \nabla_{\bm{u}_k} {\cal A}( \bb{u}_k )$ in \eqref{eq:lpnri} recovers Newton's method, while in the more general case where $P_k \approx \nabla_{\bm{u}_k} {\cal A}( \bb{u}_k )$ we may call \eqref{eq:lpnri} a linearly preconditioned residual iteration.

Since ${\cal A}$ in \eqref{eq:non-lin-sys}, and consequently $\nabla_{\bb{u}} {\cal A}(\bb{u})$, corresponds to a space-time discretization utilizing a one-step method in time, we impose that $P_k$ has the corresponding sparsity structure:
\begin{align} \label{eq:Pk-def}
P_k
=
\begin{bmatrix}
& I \\
& - \Phi^{\textrm{lin}}( \bb{u}_k^{0}) & I \\
& & \ddots & \ddots  \\
& & & - \Phi^{\textrm{lin}}( \bb{u}_k^{n_t - 2}) & I
\end{bmatrix}
\in
\mathbb{R}^{n_x n_t \times n_x n_t}.
\end{align}
Here, the linear time-stepping operator $\Phi^{\textrm{lin}}( \bb{u}_k^n ) \in \mathbb{R}^{n_x \times n_x}$ arises from the affine approximation $\Phi( \bb{u}_k^n + \bb{e}_k^n ) \approx \Phi( \bb{u}_k^n ) + \Phi^{\textrm{lin}} ( \bb{u}_k^n )  \bb{e}_k^n $ that occurs during the linearization of the residual $\bm{r}(  \bb{u}_k^n + \bb{e}_k^n )$.\footnote{While $\Phi^{\textrm{lin}}( \bb{u}_k^n )$ can be interpreted as a matrix, it is never actually formed as one. Rather, where required, its action is computed by other means as described in the coming sections.}
If $\Phi$ is differentiable, choosing this linear time-stepping operator as its Jacobian, $\Phi^{\textrm{lin}}( \bb{u}_k^n ) = \nabla_{\bb{u}_k^n} \Phi( \bb{u}_k^n )$, results in an affine approximation of $\Phi( \bb{u} )$ that is tangent to $\Phi( \bb{u} )$ at $\bb{u}_k^n$. Moreover, \eqref{eq:lpnri} coincides with Newton's method in this case.
However, regardless of whether $\Phi^{\textrm{lin}}( \bb{u}_k^n )$ is this Jacobian, it still plays the role of the linear operator in an affine approximation of $\Phi( \bb{u} )$ about $\bb{u}_k^n$, so we refer to it as the linearized time-stepping operator.

We now discuss two reasons for introducing the approximation $\Phi^{\textrm{lin}}( \bb{u}^n_k ) \approx \nabla_{\bb{u}^n_k} \Phi( \bb{u}^n_k )$.
The first is that the true Jacobian $\nabla_{\bb{u}^n_k} \Phi( \bb{u}^n_k )$ may be unnecessarily expensive, from the perspective that a cheaper approximation can be used while still maintaining satisfactory convergence speed of the iteration \eqref{eq:lpnri}.
The second is that for some of the discretizations we consider, $\Phi$ is not differentiable. Specifically, in our tests this arises when using the LLF numerical flux \eqref{eq:LFF-local}, and also in certain circumstances we limit the solution (see Remark \ref{SMrem:BL-limiting} in the Supplementary Materials), leading to further non-smoothness.

In either of the above situations our strategy for defining $\Phi^{\textrm{lin}}( \bb{u}_n )$ can be written using the following formalism.
We introduce the nonlinear time-stepping operator $\wt{\Phi}( \bm{\alpha}, \bm{\beta} ) \approx \Phi( \bm{\alpha} )$ which is chosen so that it is differentiable in its first argument, and so that it is consistent, $\wt{\Phi}( \bm{\alpha}, \bm{\alpha} ) = \Phi( \bm{\alpha} )$.
Specifically, our approximations $\wt{\Phi}( \bm{\alpha}, \bm{\beta} )$ are based on factoring certain terms in $\Phi$ into those that we want to differentiate (and these are written as functions of $\bm{\alpha}$), and those which we do not (which are written as functions of $\bm{\beta}$).
Then, $\Phi^{\textrm{lin}}( \bb{u}^n_k )$ is defined as taking the derivative of $\wt{\Phi}$ with respect to its first argument and then evaluating the result at the point in question:
\begin{align} \label{eq:Phi-wt-def}
\Phi^{\textrm{lin}}( \bb{u}^n_k )
=
\nabla_{\bm{\alpha}} \wt{\Phi}( \bm{\alpha}, \bm{\beta} ) \big|_{ (\bm{\alpha}, \bm{\beta}) = (\bb{u}^n_k, \bb{u}^n_k) }.
\end{align}
In the coming sections we detail specifically how $\wt{\Phi}$ is chosen (see e.g., \eqref{eq:fhat-approx}).
Moving forward, to maintain a uniform and simple notation, we continue to refer to the linearized time-stepping operator $\Phi^{\textrm{lin}}( \bb{u}^n_k )$ in \eqref{eq:Phi-wt-def} as the ``Jacobian of $\Phi$'' and use the corresponding notation ``$\nabla \Phi$,'' also when the Jacobian does not exist.
In cases where it does not exist, ``Jacobian of $\Phi$'' and ``$\nabla \Phi$'' are to be understood in the sense described above and as written in \eqref{eq:Phi-wt-def}.

Finally, our numerical tests indicate that the linearized iteration \eqref{eq:lpnri} may benefit from the addition of nonlinear block relaxations of the type that would be done on \eqref{eq:non-lin-sys} if the system were to be solved with FAS MGRIT. Hence we add Line \ref{ln:nonlin-relax} in \cref{alg:richardson}. 
To this end, let $m \in \mathbb{N} \setminus \{ 1 \}$ induce a coarse-fine splitting (CF-splitting) of the time points such that every $m$th time point is a C-point, and all other time points are F-points. Then, nonlinear C- and F-relaxations update the current approximation $\bb{u}_k$ such that the nonlinear residual is zero at C- and F-points, respectively \cite{Howse_etal_2019}. 
A C-relaxation requires time-stepping to each C-point from its preceding F-point, and an F-relaxation requires time-stepping from each C-point to its succeeding $m-1$ F-points. 
In practice we only do an F-relaxation in Line \ref{ln:nonlin-relax}, but note that any stronger relaxation ending in F is also possible.
%
%

\section{Linearization}
\label{sec:linearization}

This section develops the linearization $P_k$ that is used in the nonlinear iteration scheme \cref{alg:richardson}, and is a key novelty of this work.
\Cref{sec:linearization-I,sec:linearization-II,sec:linearization-III} discuss linearizing the time integration method, the numerical flux function, and the reconstruction procedure, respectively. 
\Cref{sec:linearization-IV} provides a summary.

\subsection{Linearization I: Time integration}
\label{sec:linearization-I}

Recall from \eqref{eq:one-step-disc} that at time-step $n$ the solution is updated as $\bb{u}^{n+1} = \Phi(\bb{u}^n)$.
We now work through linearizing $\Phi(\bb{u}^n_k)$ about the point $\bb{u}^n_k$ to create $\Phi^{\textrm{lin}}(\bb{u}^n_k)$ used in \eqref{eq:Pk-def}.
To simplify notation, we drop the $k$ subscripts from $\bb{u}^n_k$.
Recall that we are concerned with the action of $\Phi^{\textrm{lin}}(\bb{u}^n)$ on some error vector $\bb{e}^n$, where we omit the $k$ and ``lin'' sub and superscripts from $\bb{e}^{\textrm{(lin)},n}_k$ to further simplify notation.
%

%
%

Considering first the forward Euler step \eqref{eq:ERK1}, the Jacobian applied to $\bb{e}^n$ is
\begin{align} \label{eq:ERK1-lin}
\big[ \nabla_{\bb{u}^n} \Phi(\bb{u}^n) \big] \bb{e}^n
=
\big[ \nabla_{\bb{u}^n} F(\bb{u}^n) \big] \bb{e}^n 
= 
\underbrace{
\big[ I + \delta t  \nabla_{\bb{u}^n} L(\bb{u}^n) \big]
}
_{=: F^{\textrm{lin}}(\bb{u}^n)} 
\bb{e}^n.
\end{align}
%
%
Note the matrix $F^{\textrm{lin}}(\bb{u}^n)$ represents the Jacobian of $F$ at the point $\bb{u}^n$ (which will be approximated in further steps, unless we consider Newton's method).

Now we consider the 3rd-order ERK method \eqref{eq:ERK3}.
Recall that the application of \eqref{eq:ERK3} gives rise to the two auxiliary vectors $\bb{u}^{n,1}, \bb{u}^{n,2}$.
To linearize $\Phi$ defined by \eqref{eq:ERK3}, we work through the scheme in reverse order beginning with \eqref{eq:ERK3-c}:
\begin{align} 
\label{eq:ERK3-lin-a}
\nabla_{\bb{u}^{n}} \Phi( \bb{u}^{n} ) \bb{e}^n
=
\frac{1}{3} \bb{e}^n 
+ 
\frac{2}{3} \nabla_{\bb{u}^{n}} F \big( \bb{u}^{n,2} \big) \bb{e}^n
=
\frac{1}{3} \bb{e}^n 
+ 
\frac{2}{3} F^{\textrm{lin}} \big( \bb{u}^{n,2} \big) 
\underbrace{
\nabla_{\bb{u}^{n}} \bb{u}^{n, 2} \bb{e}^n}_{=: \bb{e}^{n, 2}}.
\end{align}
Here the chain rule for Jacobians has been used to re-express $\nabla_{\bb{u}^{n}} F ( \bb{u}^{n,2} )$ due to the fact that $\bb{u}^{n,2}$ is a function of $\bb{u}^n$.
Considering \eqref{eq:ERK3-b} to evaluate $\bb{e}^{n, 2}$ and then similarly \eqref{eq:ERK3-a} we find
\begin{align}
\label{eq:ERK3-lin-b}
\bb{e}^{n, 2}
=
\frac{3}{4} \bb{e}^n 
+ 
\frac{1}{4} F^{\textrm{lin}} \big( \bb{u}^{n,1} \big) 
\underbrace{
\nabla_{\bar{\bm{u}}^{n}} \bb{u}^{n, 1} \bb{e}^n
}_{=: \bb{e}^{n, 1}},
\quad
\textrm{and}
\quad
\bb{e}^{n, 1}
=
F^{\textrm{lin}} ( \bb{u}^{n} ) \bb{e}^n.
\end{align}
%

\subsection{Linearization II: Numerical flux}
\label{sec:linearization-II}

The previous section discussed linearization of the ERK methods \eqref{eq:ERK1-lin} and \eqref{eq:ERK3-lin-a}/\eqref{eq:ERK3-lin-b} using the Jacobian of the spatial discretization $L$.
This section considers the first step in computing this Jacobian, as it relates to the numerical flux function; the Jacobian of reconstructions is considered next in \cref{sec:linearization-III}.
To simplify notation further, we drop temporal superscripts.
In the following, the row vector $\nabla_{\bar{\bm{u}}} a_i (\bar{\bm{u}}) \in \mathbb{R}^{n_x}$ is the gradient of the function $a_i \colon \mathbb{R}^{n_x} \to \mathbb{R}$ with respect to $\bar{\bm{u}} \in \mathbb{R}^{n_x}$.

From \eqref{eq:FV-MOL-ODEs}, $\big( L(\bar{\bm{u}}) \big)_i := - \frac{\wh{f}_{i+1/2} - \wh{f}_{i-1/2}}{h}$, with $\wh{f}_{i \pm 1/2}$ the two-point numerical flux function described in \cref{sec:LFF}.
Thus, linearizing $L$ requires linearizing $\wh{f}_{i \pm 1/2}$, and we have, in terms of the error,
\begin{align}
- \nabla_{\bb{u}} \big( L(\bar{\bm{u}}) \big)_i \bb{e} 
= 
\frac{\nabla_{\bar{\bm{u}}} \wh{f}_{i+1/2} (\bar{\bm{u}}) \bb{e} 
-  
\nabla_{\bar{\bm{u}}} \wh{f}_{i-1/2} (\bar{\bm{u}}) \bb{e}}
{h}.
\end{align}
Recall that the two-point LF flux \eqref{eq:LFF} is defined as $\wh{f}_{i+1/2} = \wh{f}_{i+1/2}(u_{i+1/2}^{-},u_{i+1/2}^{+})$, where $u_{i+1/2}^{\pm}$ are reconstructions of $u(x_{i+1/2})$ computed from $\bb{u} \in \mathbb{R}^{n_x}$.
Generally, we can express the gradient of $\wh{f}_{i+1/2}$ with respect to $\bb{u}$ in terms of its action on the column vector $\bar{\bm{e}} \in \mathbb{R}^{n_x}$.

First we consider the GLF flux $\wh{f}_{i+1/2}^{\rm (global)}$ corresponding to the LF flux \eqref{eq:LFF} with global dissipation parameter $\nu_{i+1/2} = \nu^{\rm (global)}$ chosen independently of the solution state $\bb{u}$, as in \eqref{eq:LFF-global}.
Applying the chain rule to \eqref{eq:LFF} in this case yields
\begin{align}
&\nabla_{\bar{\bm{u}}} \wh{f}_{i+1/2}^{\rm (global)} (\bar{\bm{u}}) \bb{e}
=
\nabla_{\bar{\bm{u}}} \wh{f}^{\rm (global)} \big( u_{i+1/2}^{-} (\bb{u}), u_{i+1/2}^{+} (\bb{u}) \big) 
\bb{e},
\\
&=
\left(
\frac{\partial \wh{f}^{\rm (global)}}{\partial u_{i+1/2}^{-}} 
\nabla_{\bar{\bm{u}}} u_{i+1/2}^{-} (\bb{u})
+
\frac{\partial \wh{f}^{\rm (global)}}{\partial u_{i+1/2}^{+}} 
\nabla_{\bar{\bm{u}}} u_{i+1/2}^{+} (\bb{u}) 
\right) \bb{e},
\\
\begin{split}
&=
\frac{1}{2} 
\left(
f'\big( u_{i+1/2}^{-} \big)
+ 
\nu^{\rm (global)}
\right) 
\Big( 
\nabla_{\bar{\bm{u}}} u_{i+1/2}^{-} (\bb{u}) \bb{e}
\Big)
\\
&\hspace{8ex} +
\frac{1}{2} 
\left( 
f'\big( u_{i+1/2}^{+} \big)
-
\nu^{\rm (global)}
\right)
\Big(
\nabla_{\bar{\bm{u}}} u_{i+1/2}^{+} (\bb{u})  \bb{e}
\Big)
\end{split},
\\
\label{eq:gLLF-linearization}
\begin{split}
&=
\wh{f}^{\textrm{lin}} 
\Big( 
f'\big( u_{i+1/2}^{-} \big), 
\nabla_{\bar{\bm{u}}} u_{i+1/2}^{-} (\bb{u}) \bb{e}, 
f'\big( u_{i+1/2}^{+} \big), 
\nabla_{\bar{\bm{u}}} u_{i+1/2}^{+} (\bb{u}) \bb{e}, 
\nu^{\rm (global)}
\Big).
\end{split}
\end{align}
In \eqref{eq:gLLF-linearization} $\wh{f}^{\textrm{lin}}$ is the LF flux defined in \eqref{eq:LFF-lin} for the linear conservation law \eqref{eq:cons-lin}.
Thus, the action of linearized GLF flux \eqref{eq:LFF} with global dissipation \eqref{eq:LFF-global} on $\bb{e}$ in fact corresponds to a LF flux for the linear conservation law \eqref{eq:cons-lin} evaluated at $\bb{e}$ where:
\begin{enumerate}

\item The wave-speed $\alpha_{i+1/2}^{\pm}$ at cell boundaries in the linear conservation law is the wave-speed of the nonlinear conservation law \eqref{eq:cons-law} at the linearization point, $\alpha_{i+1/2}^{\pm} = f'(u_{i+1/2}^{\pm})$

\item Reconstructions of $e(x_{i+1/2})$ are computed using linearizations of the reconstruction used in the nonlinear problem, $e^{\pm}_{i+1/2} = \nabla_{\bar{\bm{u}}} u_{i+1/2}^{\pm} (\bb{u}) \bar{\bm{e}}$ 

\item The dissipation coefficient $\nu^{\rm (global)}$ used in the linear problem is the same as that in the nonlinear problem.

\end{enumerate}

Second we consider the LLF flux $\wh{f}_{i+1/2}^{\rm (local)} (\bb{u})$, corresponding to the LF flux \eqref{eq:LFF} with dissipation coefficient $\nu_{i+1/2} = \nu_{i+1/2}^{\rm (local)}(\bb{u})$ chosen locally at every cell interface according to \eqref{eq:LFF-local}.
The issue here is that $\wh{f}_{i+1/2}^{\rm (local)} (\bb{u})$ is not differentiable for some $\bb{u} \in \mathbb{R}^{n_x}$, because $\nu_{i+1/2}^{\rm (local)} (\bb{u})$ is not. 
Therefore we pursue an alternative linearization strategy based on the discussion in \cref{sec:nonlin-scheme}.
Notice in \eqref{eq:LFF} that the non-smooth term, $\nu_{i+1/2} (\bb{u}) \big( u_{i+1/2}^{-} (\bb{u}) - u_{i+1/2}^{+} (\bb{u}) \big)$, is at least a factor ${\cal O}(h)$ smaller than the smooth terms in the flux, $f( u_{i+1/2}^{-} (\bb{u}) )$ and $f(u_{i+1/2}^{+} (\bb{u}) )$, when the solution is sufficiently smooth.
We therefore anticipate that an effective linearization will not need to accurately capture this term. 
Accordingly, we linearize the non-differentiable component here with a Picard-style linearization using the framework presented in \cref{sec:nonlin-scheme}.
Using the linearization formalism introduced in \eqref{eq:Phi-wt-def}, we define the following approximation to $\wh{f}_{i+1/2}^{\rm (local)} (\bm{\alpha})$:
\begin{align} \label{eq:fhat-approx}
\begin{split}
\wt{f}_{i+1/2}^{\rm (local)} (\bm{\alpha}, \bm{\beta}) 
&:= 
\frac{1}{2} \Big[ f( u_{i+1/2}^{-} (\bm{\alpha}) ) + f(u_{i+1/2}^{+} (\bm{\alpha})) \Big] \\
&\hspace{8ex}
+ \frac{1}{2} \nu_{i+1/2} (\bm{\beta}) 
\Big[ u_{i+1/2}^{-} (\bm{\alpha}) - u_{i+1/2}^{+} (\bm{\alpha}) \Big]
\end{split}.
\end{align}
Then the action of the linearized flux is given by
\begin{align}
\label{eq:lLLF-linearization}
\begin{split}
&\nabla_{\bar{\bm{u}}} \wh{f}_{i+1/2}^{\rm (local)} (\bar{\bm{u}})
\bb{e}
=
\nabla_{\bm{\alpha}} \wt{f}_{i+1/2}^{\rm (local)} (\bm{\alpha}, \bm{\beta})|_{ (\bm{\alpha}, \bm{\beta}) = (\bb{u}, \bb{u}) } 
\bb{e}
=
\\
&\wh{f}^{\textrm{lin}} 
\Big( 
f'\big( u_{i+1/2}^{-} \big), 
\nabla_{\bar{\bm{u}}} u_{i+1/2}^{-} (\bb{u})  \bb{e}, 
f'\big( u_{i+1/2}^{+} \big), 
\nabla_{\bar{\bm{u}}} u_{i+1/2}^{+} (\bb{u}) \bb{e}, 
\nu_{i+1/2}^{\rm (local)}(\bb{u})
\Big).
\end{split}
\end{align}
Here, the linearized function $\wh{f}^{\textrm{lin}}$ is defined in \eqref{eq:LFF-lin}, and is the numerical flux function for the linear problem \eqref{eq:cons-lin}, just as for the GLF case in \eqref{eq:gLLF-linearization}. 
Moreover, the dissipation in the numerical flux for this linearized problem is also the one from the nonlinear problem, just as in \eqref{eq:gLLF-linearization}, although unlike in \eqref{eq:gLLF-linearization} the dissipation now varies per cell interface.

A common approach when applying Newton-like methods to non-smooth equations is to additively split them into smooth and non-smooth components, and then base the approximate Jacobian on that of the smooth components \cite{Chen_Yamamoto_1989,Heinkenschloss_etal_1992,Coffey_etal_2003}.
Such an approach is not directly applicable in the above LLF case because the resulting linearized problem would be numerically unstable due its lack of dissipation. 
%

\subsection{Linearization III: Reconstructions}
\label{sec:linearization-III}

In this section, we describe the linearization of the reconstructions used in the linearized numerical flux operators \eqref{eq:gLLF-linearization} and \eqref{eq:lLLF-linearization}.
Recall that \eqref{eq:gLLF-linearization} and \eqref{eq:lLLF-linearization} evaluate the linear LF flux \eqref{eq:LFF-lin} with reconstructions 
\begin{align} \label{eq:e-pm-linearized-def}
e_{i+1/2}^{\pm} 
= 
\nabla_{\bar{\bm{u}}} u_{i+1/2}^{\pm} (\bb{u}) \bb{e},
\end{align}
where $u_{i+1/2}^{\pm}$ are reconstructions of $u(x)$ at $x_{i+1/2}$ based on $\bar{\bm{u}}$.
Now consider the gradients of these reconstructions based on \eqref{eq:reconstruct-weighted}, recalling that $u_{i+1/2}^- = u_{i,+1/2}$, and $u_{i+1/2}^+ = u_{i+1,-1/2}$.
Applying the chain and product rules to \eqref{eq:reconstruct-weighted} gives
\begingroup 
\begin{align} \label{eq:weighted-reconstruct-grad-neg}
\nabla_{\bb{u}}  u^-_{i+1/2} ( \bb{u} ) \bb{e}
&=
\underbrace{
\sum \limits_{\ell =0}^{k-1} 
b_i^{\ell} ( \bb{u} ) 
\big( R^{\ell} \bb{e} \big)_i
}_{=: R_i^{1}}
+
\underbrace{
\sum \limits_{\ell =0}^{k-1} 
\big( R^{\ell} \bb{u} \big)_i
\big(
\nabla_{\bb{u}} \, b_i^{\ell} ( \bb{u} ) \bb{e}
\big)_{i}
}_{=: R_i^{2}},
\\
\label{eq:weighted-reconstruct-grad-pos}
\nabla_{\bb{u}} u^+_{i+1/2} ( \bb{u} ) \bb{e}
&=
\underbrace{
\sum \limits_{\ell =0}^{k-1} 
\wt{b}_{i+1}^{\ell} ( \bb{u} ) 
\big( \wt{R}^{\ell} \bb{e} \big)_{i+1}
}_{=: \wt{R}_{i+1}^{1}}
+
\underbrace{
\sum \limits_{\ell =0}^{k-1} 
\big( \wt{R}^{\ell} \bb{u} \big)_{i+1}
\big(
\nabla_{\bb{u}} \, \wt{b}_{i+1}^{\ell} ( \bb{u} ) \bb{e}
\big)_{i+1}
}_{=: \wt{R}_{i+1}^{2}}.
\end{align}
\endgroup
Recall that $\big\{ b_i^{\ell} \big\}_{\ell = 0}^{k-1}$ and $\big\{ \wt{b}_{i+1}^{\ell} \big\}_{\ell = 0}^{k-1}$ are the (potentially) $\bb{u}$-dependent weights used to combine the $k$ linear reconstructions $\big\{ ({R}^{\ell} \bb{u} )_i \big\}_{\ell = 0}^{k-1}$ and $ \big\{ (\wt{R}^{\ell} \bb{u} )_{i+1} \big\}_{\ell = 0}^{k-1}$, respectively, as defined in \eqref{eq:reconstuct-lin}.
The expressions \eqref{eq:weighted-reconstruct-grad-neg}, \eqref{eq:weighted-reconstruct-grad-pos} are the sum of two terms which we now discuss.

\underline{$R_i^{1}, \wt{R}_{i+1}^{1}$:} From \eqref{eq:reconstuct-lin}, $\big( R^{\ell} \bb{e} \big)_i$ and $\big( \wt{R}^{\ell} \bb{e} \big)_{i+1}$ are standard $k$th-order \textit{linear} reconstructions of $e$ on stencils with left shift $\ell$.
Thus, $R_i^{1}$ and $\wt{R}_{i+1}^{1}$ are \textit{weighted} reconstructions of $e$, with weights $b, \wt{b}$ being the same as those used for $u_{i+1/2}^{-}$ and $u_{i+1/2}^{+}$, respectively.

\underline{$R_i^{2}, \wt{R}_{i+1}^{2}$:} The values $\big( \nabla_{\bb{u}} \, {b}_{i}^{\ell} ( \bb{u} )  \bb{e}
\big)_{i}$ and $\big(
\nabla_{\bb{u}} \, \wt{b}_{i+1}^{\ell} ( \bb{u} ) \bb{e}
\big)_{i+1}$ are reconstructions of $e$, with weights based on gradients of $ b_i, \wt{b}_{i+1}$ at the point $\bb{u}$.
Then $R_i^{2}$ and $\wt{R}_{i+1}^{2}$ take weighted combinations of these reconstructions, where the combination weights are the reconstructions $\big( R^{\ell} \bb{u} \big)_i$ and $\big( \wt{R}^{\ell} \bb{u} \big)_{i+1}$ from \eqref{eq:reconstuct-lin}.

Now consider two different choices for the weights $b, \wt{b}$ in \eqref{eq:weighted-reconstruct-grad-neg} and \eqref{eq:weighted-reconstruct-grad-pos}: Optimal linear weights, and WENO weights.

\textbf{\underline{Optimal linear weights:}} That is, ${b}_{i+1}^{\ell} = {d}^{\ell}$, $\wt{b}_{i+1}^{\ell} = \wt{d}^{\ell}$.
We have \cite{Shu_1998}, 
\begin{align}
&k = 1: \quad d^{0} = 1, 
\quad
\wt{d}^{0} = 1, \\
&k = 2: \quad \big( d^{0}, d^{1} \big) =  \Big( \frac{2}{3}, \frac{1}{3} \Big), 
\quad
\big( \wt{d}^{0}, \wt{d}^{1} \big) = \Big( \frac{1}{3}, \frac{2}{3} \Big).
\end{align}
In this case $R_i^{1}, \wt{R}_{i+1}^{1}$ are simply linear, $(2k-1)$-order reconstructions of $e$.
Moreover, the second terms in \eqref{eq:weighted-reconstruct-grad-neg} and \eqref{eq:weighted-reconstruct-grad-pos} vanish, $R_i^{2} = \wt{R}_{i+1}^{2} = 0$, because the weights are independent of $\bb{u}$.
Recall that these optimal linear weights are used in a 1st-order accurate FV method. However, for higher-order discretizations, these weights are not suitable when the solution lacks regularity because they result in spurious oscillations. 

\textbf{\underline{WENO weights:}} 
Here, $b_i^{\ell} = w_i^{\ell}$, $\wt{b}^{\ell}_{i+1} = \wt{w}^{\ell}_{i+1}$.
We consider the so-called \textit{classical} WENO weights  \cite{Jiang_Shu_1996}, which are given by
\begin{align} 
\label{eq:WENO-weight-neg}
w^{\ell}_{i} (\bar{\bm{u}}) &= \frac{\alpha_{i}^{\ell} (\bar{\bm{u}})}{\sum_{\ell = 0}^{k-1} \alpha_{i}^{\ell} (\bar{\bm{u}})}, 
\quad
\alpha_{i}^{\ell} (\bar{\bm{u}}) = \frac{d^{\ell}}{ \big(\epsilon + \beta_{i}^{\ell} (\bar{\bm{u}}) \big)^2},
\\
\label{eq:WENO-weight-pos}
\wt{w}^{\ell}_{i+1} (\bar{\bm{u}}) &= \frac{\wt{\alpha}_{i+1}^{\ell} (\bar{\bm{u}}) }{\sum_{\ell = 0}^{k-1} \wt{\alpha}_{i+1}^{\ell} (\bar{\bm{u}})}, 
\quad
\wt{\alpha}_{i+1}^{\ell} (\bar{\bm{u}}) = \frac{\wt{d}^{\ell}}{ \big(\epsilon + \beta_{i+1}^{\ell} (\bar{\bm{u}}) \big)^2}, 
\end{align}
with $\epsilon = 10^{-6}$.
The quantity $\beta^{\ell}$ reflects the smoothness of $u(x)$ over the $\ell$th stencil. 
For $k = 1$, $\beta = 0$.
For $k = 2$ we have \cite{Shu_1998},
\begin{align}
\beta^{0}_i(\bb{u}) = ( \bar{u}_{i+1} - \bar{u}_{i} )^2, 
\quad 
\beta^{1}_i(\bb{u}) = (  \bar{u}_{i} - \bar{u}_{i-1} )^2.
\end{align}
Clearly the WENO weights \eqref{eq:WENO-weight-neg}, \eqref{eq:WENO-weight-pos} are differentiable; 
note that the choice of WENO weights is not unique, and other choices exist which may not be differentiable \cite{Castro-etal-2011-WENO-Z}.

In summary, if optimal linear weights are used to discretize \eqref{eq:cons-law}, then \eqref{eq:weighted-reconstruct-grad-neg} and \eqref{eq:weighted-reconstruct-grad-pos} are straightforward to compute.
Otherwise, if WENO weights are used there are several options for computing or approximating \eqref{eq:weighted-reconstruct-grad-neg} and \eqref{eq:weighted-reconstruct-grad-pos}. 
We could consider computing the gradients in \eqref{eq:weighted-reconstruct-grad-neg} and \eqref{eq:weighted-reconstruct-grad-pos} exactly; however, this is costly both in terms of FLOPs and memory requirements because the gradients of \eqref{eq:WENO-weight-pos} and \eqref{eq:WENO-weight-neg} are complicated.
Thus, in our numerical tests, we instead approximately compute \eqref{eq:weighted-reconstruct-grad-neg} and \eqref{eq:weighted-reconstruct-grad-pos} via the finite difference
\begin{align} \label{eq:weighted-reconstruct-grad-FD}
\nabla_{\bb{u}}  u^\pm_{i+1/2} ( \bb{u} ) \bb{e}
\approx
\frac{u^\pm_{i+1/2} ( \bb{u} + \mu \bb{e} ) - u^\pm_{i+1/2} ( \bb{u} ) }{\mu},
\end{align}
where $\mu \in \mathbb{R}$ is a small parameter. If $\mu$ is chosen appropriately, we find that the convergence of the nonlinear solver is indistinguishable from when \eqref{eq:weighted-reconstruct-grad-neg} and \eqref{eq:weighted-reconstruct-grad-pos} are computed exactly.
We find $\mu = 0.1$ is sufficient, since taking $\mu$ much smaller appears to lead to significant round-off error and results in the solver stalling.
We refer to \eqref{eq:weighted-reconstruct-grad-FD} as an approximate Newton linearization of the reconstructions.

Another approach for approximating \eqref{eq:weighted-reconstruct-grad-neg} and \eqref{eq:weighted-reconstruct-grad-pos} is to take just the first of the two terms:
\begin{align} \label{eq:weighted-reconstruct-grad-zero}
\nabla_{\bb{u}}  u^-_{i+1/2} ( \bb{u} ) \bb{e}
\approx
R_i^{1},
\quad
\textrm{and}
\quad
\nabla_{\bb{u}} u^+_{i+1/2} ( \bb{u} ) \bb{e}
\approx
\wt{R}_{i+1}^{1}.
\end{align}
We refer to this as a Picard linearization of the reconstructions, and it is much less expensive than computing the full directional derivative.\footnote{It is possible to write this linearization using the formalism introduced in \eqref{eq:Phi-wt-def}. 
E.g., consider the following approximation to the reconstruction in
\eqref{eq:reconstruct-weighted}: 
$u_{i,-1/2}  ( \bm{\alpha}, \bm{\beta} )
:=
\sum_{\ell =0}^{k-1} 
\wt{b}^{\ell}_{i} ( \bm{\beta} )
\big( \wt{R}^{\ell} \bm{\alpha} \big)_i.
$
Then $\nabla_{\bm{\alpha}} u_{i,-1/2}  ( \bm{\alpha}, \bm{\beta} ) \big|_{ (\bm{\alpha}, \bm{\beta}) = (\bb{u}, \bb{u}) } = R_i^{1}$
}
This linearization is motivated by the fact that where $u$ is sufficiently smooth, the WENO weights approach, as $h \to 0$, the optimal linear weights, and, thus, they become only weakly dependent on $\bb{u}$. More specifically, where $u$ is sufficiently smooth \cite{Shu_1998},
$
w^{\ell}_i = d^{\ell} + {\cal O}(h^{k-1}), 
\wt{w}_{i+1}^{\ell} = \wt{d}^{\ell} + {\cal O}(h^{k-1})
$.
Thus, we might expect that the terms neglected in the linearization \eqref{eq:weighted-reconstruct-grad-zero}, i.e., those containing gradients of  WENO weights, may be small across much of the space-time domain except for isolated regions of non-smoothness, e.g., at shocks, and especially so for higher-order methods (i.e., larger $k$).
Our numerical results tend to show that the linearization \eqref{eq:weighted-reconstruct-grad-zero} is less robust than \eqref{eq:weighted-reconstruct-grad-FD}; however, acceptable convergence can often still be attained with the less expensive \eqref{eq:weighted-reconstruct-grad-zero}.
%

\subsection{Linearization IV: Summary}
\label{sec:linearization-IV}

We now summarize the results of \cref{sec:linearization-I,sec:linearization-II,sec:linearization-III}, and explain how to invert the linearized space-time operator $P_k$ serial in time.
Recall from \cref{alg:richardson} that at iteration $k$, given the approximation $\bb{u}_k \approx \bb{u}$ we evaluate the associated residual $\bm{r}_k$ according to \eqref{eq:non-lin-sys}.
At a high level, this residual corresponds to evaluating a quantity of the form
\begin{align}
\bm{r}_k = - \textrm{discretization} \bigg( \frac{\partial \, \cdot \,}{\partial t} + \frac{\partial f( \, \cdot \, )}{\partial x} \bigg) \bb{u}_k ,
\end{align}
by which we mean applying the ERK+FV discretization from \cref{sec:discretization} to the current iterate $\bb{u}_k$.
Then, we compute the updated solution $\bb{u}_{k+1} = \bb{u}_k + \bb{e}^{\textrm{lin}}_k$, where $\bb{e}^{\textrm{lin}}_k := P_k^{-1} \bm{r}_k$ is the linearized error. 
This linearized error is given by solving the system $P_k \bb{e}^{\textrm{lin}}_k = \bm{r}_k$, which, at a high level, corresponds to
\begin{align} \label{eq:con-linearized-problem}
P_k
\bb{e}^{\textrm{lin}}_k
=
\textrm{discretization} \bigg( 
\frac{\partial  \, \cdot \,}{\partial t} + \frac{\partial}{\partial x} \Big( f'(\bb{u}_k)  \, \cdot \, \Big) \bigg)
\bb{e}^{\textrm{lin}}_k 
 = 
 \bm{r}_k.
\end{align}
The linearized PDE in \eqref{eq:con-linearized-problem} has the same structure as the linear conservation law \eqref{eq:cons-lin} with linear wave-speed $\alpha(x, t) = f'(u)$.
However, the discretization in \eqref{eq:con-linearized-problem} is non-standard in the sense that it is not a direct application of the ERK+FV discretization from \cref{sec:discretization} to \eqref{eq:cons-lin}.
Rather, \eqref{eq:con-linearized-problem} is a discretization of \eqref{eq:cons-lin} arising from linearizing the ERK+FV discretization of \eqref{eq:cons-law} about the current iterate $\bb{u}_k$.

We now summarize how to invert $P_k$ serial in time.
Using the structure of $P_k$ in \eqref{eq:Pk-def}, it follows that, at time $t_{n+1}$, the solution of \eqref{eq:con-linearized-problem} is obtained by $\bb{e}^{n+1}_k = \Phi^{\textrm{lin}}(\bb{u}_k^n) \bb{e}^{n}_k + \bm{r}_k^{n+1}$, where we have dropped the superscript ``lin'' from the vectors $\bb{e}^{n}_k, \bb{e}^{n+1}_k$ for readability. The action of $\Phi^{\textrm{lin}}(\bb{u}_k^n)$ on $\bb{e}^{n}_k$ is as follows:
\begin{itemize}

\item The time discretization of \eqref{eq:cons-lin} is a linearized version of the same time integration method used for the nonlinear PDE \eqref{eq:cons-law}, be that \eqref{eq:ERK1-lin} or \eqref{eq:ERK3-lin-a}/\eqref{eq:ERK3-lin-b}. 
If \eqref{eq:ERK3-lin-a}/\eqref{eq:ERK3-lin-b} is used, the below steps are to be repeated for each stage of the method. 

\item The spatial discretization of \eqref{eq:cons-lin} is based on the linear LF flux \eqref{eq:LFF-lin}.
Specifically, according to \eqref{eq:gLLF-linearization} and \eqref{eq:lLLF-linearization}, the linear flux \eqref{eq:LFF-lin} is evaluated using interfacial wave-speeds given by $\alpha_{i+1/2}^{\pm} = f'(u_{i+1/2}^{\pm})$, where  $u_{i+1/2}^{\pm}$ are reconstructions of $u(x_{i+1/2})$ computed from the current iterate $\bb{u}^n_k$.

\item The dissipation coefficient $\nu_{i+1/2}$ in the linear flux \eqref{eq:LFF-lin} is equal to the dissipation coefficient computed from the current iterate $\bb{u}^n_k$, be it global as in \eqref{eq:LFF-global} or local as in \eqref{eq:LFF-local}, see \eqref{eq:gLLF-linearization} or \eqref{eq:lLLF-linearization}, respectively.

\item The reconstructions $e_{i+1/2}^{\pm}$, defined in \eqref{eq:e-pm-linearized-def}, and used in the linear flux \eqref{eq:LFF-lin}, are computed from the cell averages in $\bb{e}^{n}_k$. Specifically, reconstruction rules are applied to $\bb{e}^{n}_k$ based on linearizations of those used to compute $u_{i+1/2}^{\pm}$. 
For 1st-order reconstructions these rules are independent of $\bb{u}^n_k$, so that the standard 1st-order reconstruction can be used. 
For WENO reconstructions, the linearized reconstruction can be based on the Newton linearization with numerical Jacobian \eqref{eq:weighted-reconstruct-grad-FD}, or on the Picard linearization \eqref{eq:weighted-reconstruct-grad-zero}.

\end{itemize}

\begin{remark}[Newton vs. Picard iteration]
The linearization strategy outlined in this section is rather complicated, and it is natural to ask whether it needs to be this complicated. For example, one could instead conceive of a simpler Picard iteration based on the quasi-linear form of the PDE \eqref{eq:cons-law}. 
Recall that wherever the solution of \eqref{eq:cons-law} is differentiable, the PDE is equivalent to $\frac{\partial u}{\partial t} + f'(u) \frac{\partial u}{\partial x}  = 0$.
Given some approximation $u_k \approx u$, a residual correction scheme can then be formed by solving a linearized version of $\frac{\partial} {\partial t} (u_k + e_k ) + f'(u_k + e_k ) \frac{\partial}{\partial x}(u_k + e_k ) = 0$ arising from replacing the wave-speed with the approximation: $f'(u_k + e_k) \approx f'(u_k)$.
At the discrete level, this leads to a linearized problem of the form
\begin{align} \label{eq:picard-linearized-problem}
\mathrm{discretization} \bigg( 
\frac{\partial  \, \cdot \,}{\partial t} + f'(\bb{u}_k) \frac{\partial \, \cdot \,}{\partial x}  \bigg)
\bb{e}^{\mathrm{lin}}_k 
 = 
 \bm{r}_k.
\end{align}
This is similar to the linearized problem \eqref{eq:con-linearized-problem} that we actually solve, with the difference being that the PDE on the left-hand side here is in non-conservative form.
We have also tested this non-conservative linearization approach  \eqref{eq:picard-linearized-problem} instead of \eqref{eq:con-linearized-problem}. Those tests showed some potential for problems with smooth solutions, but were less successful for problems with shocks. 
\end{remark}

\section{Approximate MGRIT solution of linearized problems}
\label{sec:mgrit}

In this section, we discuss the approximate MGRIT solution of the linearized problems that arise at every iteration of \cref{alg:richardson}.
At iteration $k$ of \cref{alg:richardson}, recall that the linearized system $P_k \bb{e}^{\textrm{lin}}_k = \bm{r}_k$ is solved to determine the linearized error $\bb{e}^{\textrm{lin}}_k$. 
The iteration index $k$ is not important throughout this section, so we drop it for readability.

As detailed in \cref{sec:linearization}, the blocks in the system matrix $P$ take the form of a certain, non-standard discretization of the linear conservation law \eqref{eq:cons-lin}.
To simplify notation, in this section we write the linearized time-stepping operator $\Phi^{\textrm{lin}}( \bb{u}^n )$ as $\Phi^{n}$, where we recall from \cref{sec:nonlin-scheme} that $\bb{u}^n$ is the linearization point,
and the linearized error $\bb{e}^{\textrm{lin}}$ is written as $\bb{e}$.\footnote{Note, $\Phi^{n}$ is not to be confused with the nonlinear time-stepping operator $\Phi$ defined in \cref{sec:discretization}, and the linearized error here $\bb{e}$ is not to be confused with the true ``nonlinear'' error $\bb{e}$ described in \cref{sec:nonlin-scheme} even though they now use the same notation.}
Considering the structure of $P$ in \eqref{eq:Pk-def}, the $(n+1)$st block row of the system $P \bb{e} = \bm{r}$ is $\bb{e}^{n+1} - \Phi^{n} \bb{e}^{n} = \bm{r}^{n+1}$.

The remainder of this section is structured as follows.
A brief overview of MGRIT is given in \cref{sec:mgrit-summary} and our proposed MGRIT coarse-grid operator is described in \cref{sec:Psi-overview}.
\Cref{sec:SL} discusses the coarse-grid semi-Lagrangian method. 
Many of the finer details underlying the coarse-grid operator are described in the supplementary materials and these will be referenced as required.
%

\subsection{MGRIT overview}
\label{sec:mgrit-summary}

We now give a brief overview of a two-level MGRIT algorithm as it applies to the linear problem at hand; see \cite{Falgout_etal_2014} for more details on MGRIT.
Consider the fine-grid linear system $P \bb{e} = \bm{r}$ consisting of $n_t$ time points. 
For a given $\bb{e}_j \approx \bb{e}$, with $j$ the iteration index of the MGRIT iterative process, the associated fine-grid residual is $\bm{\delta}_j := \bm{r} - P \bb{e}_j$.
Given $m \in \mathbb{N}$, define a CF-splitting of the fine time grid as already explained in the context of \cref{alg:richardson}. 

Define a coarse-grid space-time matrix $P^{\Delta} \in \mathbb{R}^{\frac{n_x n_t}{m} \times \frac{n_x n_t}{m}}$ with the same block bidiagonal structure as $P$ in \eqref{eq:Pk-def}, such that the $(n+1)$st block row of its action on a coarse-grid vector $\bm{z} \in \mathbb{R}^{\frac{n_x n_t}{m}}$ is $\big[P^{\Delta} \bm{z}\big]^{n+1} = \bm{z}^{n+1} - \Psi^{n} \bm{z}^n$.
Here, $\Psi^{n}$ is the (linear) \textit{coarse-grid time-stepping operator} evolving vectors from time $t_n \to t_n + m \delta t$, and should be such that $\Psi^{n} \approx \Psi_{\rm ideal}^{n} :=
 \prod_{k = 0}^{m - 1} \Phi^{n+k}$, with $\Psi_{\rm ideal}^{n}$ being the \textit{ideal} coarse-grid operator.
Fast MGRIT convergence hinges on $\Psi^{n}$ being a sufficiently accurate approximation to $\Psi_{\rm ideal}^{n}$ \cite{Dobrev_etal_2017,Southworth_2019,DeSterck_etal_2025_LFA}, while obtaining speed-up in a parallel environment also necessitates that its action is significantly cheaper than that of $\Psi_{\rm ideal}^{n}$.
As described in \cref{sec:introduction}, the design of effective coarse-grid operators in the context of hyperbolic PDEs is a long-standing issue for MGRIT and Parareal.

A two-level MGRIT iteration is given by pre-relaxation on the fine grid followed by coarse-grid correction with $P^{\Delta}$. 

\subsection{Coarse-grid operator}
\label{sec:Psi-overview}

Since $\Phi^{n}$ is a one-step discretization of a linear hyperbolic PDE, to design a suitable coarse-grid operator we leverage our recent work on this topic  \cite{DeSterck_etal_2023_SL,DeSterck_etal_2023_MOL,DeSterck_etal_2025_LFA}.
In these works, we propose coarse-grid operators which are \textit{modified semi-Lagrangian discretizations} of the underlying hyperbolic PDE.
A key distinction with \cite{DeSterck_etal_2023_SL,DeSterck_etal_2023_MOL,DeSterck_etal_2025_LFA} is that they considered \textit{FD} discretizations of \textit{non-conservative} linear hyperbolic PDEs, while in this work we consider \textit{FV} discretizations of \textit{conservative} linear hyperbolic PDEs to solve nonlinear problems.
Based on \cite{DeSterck_etal_2023_SL,DeSterck_etal_2023_MOL,DeSterck_etal_2025_LFA},  we propose a coarse-grid operator of the form
\begin{align} \label{eq:Psi}
\Psi^{n}
=
\Big(
I 
+ 
{\cal T}_{\rm ideal}^{n}
-
{\cal T}_{\rm direct}^{n}
\Big)^{-1}
 {\cal S}^{n, m\delta t}_{p, 1_*}
 \approx
  \Psi_{\rm ideal}^{n}.
\end{align}
Here, ${\cal S}^{n, m\delta t}_{p, 1_*}$ is a conservative, FV-based discretization of \eqref{eq:cons-lin} evolving solutions from time $t_n \to t_n + m \delta t$, and has, in a certain sense to be clarified below, orders of accuracy $p$ in space and one in time.
We take $p$ as the formal order of the FV discretization described in \cref{sec:discretization}, i.e., $p = 2k-1$.
Furthermore, ${\cal T}_{\rm direct}^{n}$ is an approximation of the operator associated with the leading-order term in the truncation error of ${\cal S}^{n, m\delta t}_{p, 1_*}$. Further details on ${\cal S}^{n, m\delta t}_{p, 1_*}$ and ${\cal T}_{\rm direct}^{n}$ are given in \cref{sec:SL}.
Similarly, ${\cal T}_{\rm ideal}^{n}$ in \eqref{eq:Psi} is an approximation of the operator associated with the leading-order term in the truncation error of $\Psi^{n}_{\rm ideal}$, and is discussed further below.
The coarse-grid operator \eqref{eq:Psi} has the same structure as the coarse-grid operators proposed in \cite{DeSterck_etal_2023_SL,DeSterck_etal_2023_MOL} in the sense of being a semi-Lagrangian discretization modified so that its truncation error better approximates that of the ideal coarse-grid operator.

The main complicating factor in this work, differentiating it from that in \cite{DeSterck_etal_2023_SL,DeSterck_etal_2023_MOL,DeSterck_etal_2025_LFA}, is that the fine-grid operator is not associated with a standard discretization of \eqref{eq:cons-lin}, and this makes estimating its truncation error difficult. 
That is, recall from \cref{sec:linearization-IV} that $\Phi^n$ is a discretization of \eqref{eq:cons-lin} arising from the linearization of an ERK+FV discretization of the underlying nonlinear PDE \eqref{eq:cons-law}.

Our approach for estimating ${\cal T}_{\rm ideal}^{n}$ is lengthy and details can be found in the supplementary materials.
In essence, we first develop an approximate truncation error analysis in Supplementary Materials Section \ref{SMsec:ideal-err-est} for a linear ERK+FV-type discretization of the PDE \eqref{eq:cons-lin}. 
Then in Section \ref{SMsec:MGRIT-linear-standard} this analysis is applied to estimate the truncation error of a standard linear ERK+FV discretization of \eqref{eq:cons-lin}. Numerical results are presented there to show that the approach results in a fast MGRIT solver despite the analysis being largely heuristic.
This work generalizes \cite{DeSterck_etal_2023_MOL}, which considered (fine-grid) method-of-lines discretizations only for constant wave-speed problems.
Finally, in Section \ref{SMsec:ideal-err-est-linearized} the coarse-grid operator developed in Section \ref{SMsec:MGRIT-linear-standard} is heuristically extended to the case at hand where the discretization of \eqref{eq:cons-lin} corresponds to a linearization of an ERK+FV discretization of the nonlinear PDE \eqref{eq:cons-law}.

Ultimately, the heuristic analysis results in the following choice of ${\cal T}_{\rm ideal}^{n}$ in \eqref{eq:Psi}:
\begin{align} \label{eq:T-ideal}
{\cal T}_{\rm ideal}^{n}
=
\begin{cases}
{\cal D}_1 \diag 
\bigg( 
\sum \limits_{j = 0}^{m-1} 
\bm{\beta}^{n+j} 
\bigg) 
{\cal D}_1^\top, 
\quad k = 1,
\\[3ex]
{\cal D}_1
\bigg(
\diag
\bigg( 
\sum \limits_{j = 0}^{m-1} \bm{\beta}^{1,n+j} 
\bigg) 
{\cal D}_2^{1}
-
\diag
\bigg( 
\sum \limits_{j = 0}^{m-1} \bm{\beta}^{0,n+j} 
\bigg) 
{\cal D}_2^{0}
\bigg),
\quad k = 2.
\end{cases}
\end{align} 
Here, $({\cal D}_1 \bb{e})_i = (\bar{e}_i - \bar{e}_{i-1})/h$, $({\cal D}_2^{0} \bb{e})_{i} = (\bar{e}_i - 2\bar{e}_{i+1} + \bar{e}_{i+2})/h^2$, and $({\cal D}_2^{1} \bb{e})_{i} = (\bar{e}_{i-1} - 2 \bar{e}_{i} + \bar{e}_{i+1})/h^2$.
Furthermore, the vectors $\bm{\beta}^{n}, \bm{\beta}^{0, n}, \bm{\beta}^{1, n} \in \mathbb{R}^{n_x}$ are defined element-wise by
\begin{subequations}
\label{eq:beta-k}
\begin{align}
\beta_{i}^{n}
&:=
\left[
- \frac{1}{2} \Big( \delta t \, \alpha^{n}_{i+1/2} \Big)^2
+
\frac{h \delta t}{2} \nu^{n}_{i+1/2}
\right],
\\
\label{eq:beta-k=2}
\beta_{i}^{*, n}
&:=
\frac{3}{4} \left[
s_{\rm RK}
\,
\frac{1}{24 h} \Big( \delta t \, \alpha^{n}_{i+1/2} \Big)^4
+
s_{\rm FV}
\,
\frac{h^2 \delta t }{12} \nu^{n}_{i+1/2} 
\right]
\gamma_i^{*, n}.
\end{align}
\end{subequations}
In \eqref{eq:beta-k}, $\alpha^{n}_{i+1/2}$ is an approximation to wave-speed $\alpha(x,t)$ of the linear conservation law \eqref{eq:cons-lin} at $x = x_{i+1/2}$. 
%
Recalling from \cref{sec:linearization-IV} that we are discretizing \eqref{eq:cons-lin} with $\alpha(x, t) = f'(u)$, 
we choose $\alpha^{n}_{i+1/2} = \frac{f'(u_{i+1/2}^-) + f'(u_{i+1/2}^+)}{2}$ in which $u_{i+1/2}^{\pm}$ are reconstructions of $u(x_{i+1/2})$ based on the either $\bb{u}^n$ or $\bb{u}^{n+1}$.
Likewise, $\nu_{i+1/2}^{n}$ in \eqref{eq:beta-k} is the LFF dissipation coefficient in \eqref{eq:LFF} arising from either $\bb{u}^n$ or $\bb{u}^{n+1}$.

Further used in \eqref{eq:beta-k=2} is the short hand notation
\begin{align} \label{eq:gamma-def}
\gamma_i^{0, n} := 2 \wt{b}_{i+1}^{0, n} 
+
b_{i}^{0, n},
\quad
\gamma_i^{1, n} := \wt{b}_{i+1}^{1, n}
+
2 b_{i}^{1, n},
\end{align}
where, recalling from \eqref{eq:reconstruct-weighted} the reconstruction coefficients $\big\{ b_i^{\ell}( \bb{u} ) \big\}_{\ell = 1}^{2k - 1}$, $\big\{ \wt{b}_i^{\ell}(\bb{u}) \big\}_{\ell = 1}^{2k - 1}$, here, $b_i^{\ell, n}$ and $\wt{b}_{i}^{\ell, n}$ correspond to reconstruction coefficients $b_i^{\ell}$ and $\wt{b}_i^{\ell}$, respectively, arising from either $\bb{u}_k^n$ or $\bb{u}_k^{n+1}$.
Recall that for a 1st-order reconstruction, $\big\{ b_i^{\ell} \big\}, \big\{\wt{b}_i^{\ell} \big\}$ are always independent of $\bb{u}$, while for higher-order discretizations they depend on $\bb{u}$ only if WENO reconstructions are used.

In our numerical tests, we do not typically find a significant difference whether the time-dependent quantities in \eqref{eq:beta-k} are based on $\bb{u}^n$ or $\bb{u}^{n+1}$. 
An exception, however, is in our Buckley--Leverett test problems for order three (when $k = 2$), where we find significantly better MGRIT convergence if the WENO weights, i.e., the $\{ b_i^{\ell, n} \}$ and $\{ \wt{b}_i^{\ell, n} \}$ in \eqref{eq:gamma-def}, are based on $\bb{u}^{n+1}$ rather than $\bb{u}^{n}$.
As such, $k = 2$ numerical tests in \cref{sec:num-res} (including for Burgers), compute all time-dependent quantities in \eqref{eq:beta-k} based on $\bb{u}^{n+1}$; $k=1$ tests compute time-dependent quantities in \eqref{eq:beta-k} based on $\bb{u}^{n}$.

The parameters $s_{\rm RK}, s_{\rm FV} \in \mathbb{R}$ in \eqref{eq:beta-k=2} are not motivated by our heuristic truncation error analysis, i.e., the analysis predicts $s_{\rm RK} = s_{\rm FV} = 1$, but are added because we observe empirically that choosing them different from $1$ may improve solver performance for the non-convex Buckley--Leverett problem.
Specifically, we find these parameters are crucial for obtaining a scalable solver in our Buckley--Leverett tests.
In all $k = 2$ numerical tests in \cref{sec:num-res} we set $s_{\rm RK} = \frac{16}{3}$ and $s_{\rm FV} = \frac{4}{3}$.
For Burgers problems, the solver does not seem overly sensitive to these parameters, e.g., $s_{\rm RK} = \frac{16}{3}$ and $s_{\rm FV} = \frac{4}{3}$ do not result in significantly better convergence than $s_{\rm RK} = s_{\rm FV} = 1$.

Finally, note that the truncation operator \eqref{eq:T-ideal} is designed such that if optimal linear reconstruction weights are used in the discretization of \eqref{eq:cons-law} are used rather than WENO weights then it reduces to the coarse-grid operator developed in Supplementary Materials Section \ref{SMsec:MGRIT-linear-standard} for standard linear ERK+FV discretizations of \eqref{eq:cons-lin}.
%

\subsection{Semi-Lagrangian discretization and its truncation error}
\label{sec:SL}

Here we discuss the semi-Lagrangian method ${\cal S}^{n, m\delta t}_{p, 1_*}$ used in \eqref{eq:Psi} to discretize \eqref{eq:cons-lin}, and then we describe the choice of its truncation error approximation ${\cal T}_{\rm direct}^{n}$ from \eqref{eq:Psi}.
Recall that ${\cal S}^{n, m\delta t}_{p, 1_*}$ differs from the semi-Lagrangian discretizations we used in  \cite{DeSterck_etal_2023_SL,DeSterck_etal_2023_MOL,DeSterck_etal_2025_LFA} because it is an FV discretization of a conservative hyperbolic PDE rather than an FD discretization of a non-conservative hyperbolic PDE.

Applying the discretization to \eqref{eq:cons-lin} on the time interval $[t_n, t_{n} + \delta t]$ results in the following approximate update for the cell average of $e(x)$ in cell ${\cal I}_i$:
\begin{align} \label{eq:SL-FV-scheme}
\bar{e}_i^{n+1} 
\approx
\bar{e}_i^{n} - \frac{\wt{f}_{i+1/2} - \wt{f}_{i-1/2}}{h} 
= 
\big( {\cal S}_{p,r}^{n, \delta t} \bar{\bm{e}}^n \big)_i.
\end{align}
Here $\wt{f}_{i+1/2}$ is the numerical flux function that defines the semi-Lagrangian scheme---note the difference in notation with the numerical flux $\wh{f}_{i+1/2}$ from \cref{sec:discretization}. 
The time-stepping operator ${\cal S}_{p,r}^{n, \delta t}$ indicates the method has accuracies of order $p$ and $r$ for the two numerical procedures that it approximates (translating to order $p$ in ``space,'' and order $r$ in ``time'').

The numerical flux $\wt{f}_{i+1/2}$ in \eqref{eq:SL-FV-scheme} is chosen so that the scheme in \eqref{eq:SL-FV-scheme} approximates the following local  conservation relationship satisfied by the exact solution of \eqref{eq:cons-lin} (see also \cref{fig:SL-FV-conservation}):
\begin{align} \label{eq:cons-lin-SL-cons-local}
\bar{e}_i^{n+1} 
=
\frac{1}{h}
\int_{\wt{\xi}_{i - 1/2}}^{\wt{\xi}_{i+1/2}} e(x, t_n) \d x
=
\bar{e}_i^{n}
-
\frac{1}{h}
\left(
\int_{\wt{\xi}_{i + 1/2}}^{x_{i+1/2}} e(x, t_n) \d x
-
\int_{\wt{\xi}_{i - 1/2}}^{x_{i-1/2}} e(x, t_n) \d x
\right).
\end{align}
Here, $x = \wt{\xi}_{i \pm 1/2}$ are \textit{departure points} at time $t_n$ of the space-time FV cell defined on $t \in [t_n, t_{n+1}]$ that at time $t_{n+1}$ \textit{arrive} at points $x = x_{i \pm 1/2}$, and otherwise at intermediate times has lateral boundaries with slopes $\frac{\d x}{\d t} = \alpha(x, t)$. 
This relation follows from an application of the divergence theorem to \eqref{eq:cons-lin}, and serves as the basis for the semi-Lagrangian discretizations designed in \cite{Huang_etal_2012,Huang_Arbogast_2016}. 
%
The departure point $\wt{\xi}_{i+1/2} := \xi_{i+1/2}(t_n)$ can be computed by solving the following final-value problem for the characteristic curve that is the right-hand boundary of the deformed space-time FV cell:
\begin{align} \label{eq:depart-SL-FV}
\frac{\d \xi_{i+1/2}}{\d t} = \alpha(\xi_{i+1/2}, t), 
\quad 
t \in [t_n, t_n + \delta t], 
\quad 
\xi_{i+1/2}(t_n + \delta t) = x_{i+1/2}.
\end{align}

\begin{figure}[t!]
\centerline{
\fbox{
\includegraphics[scale=0.8]{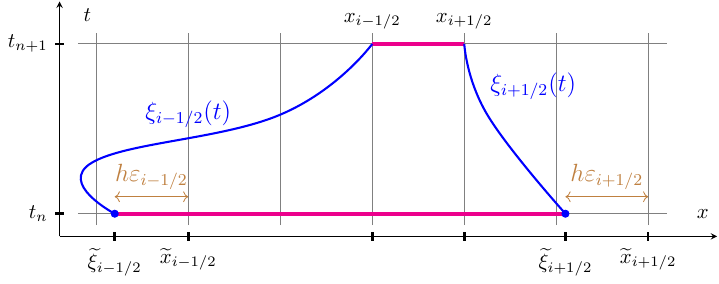}
}
}
\caption{Schematic illustration of the conservation property \eqref{eq:cons-lin-SL-cons-local} that underlies the FV semi-Lagrangian scheme \eqref{eq:SL-FV-scheme}: The integrals of the solution of \eqref{eq:cons-lin} over the bold magenta lines are equal.
The east-bounding characteristic curve satisfies the final-value problem \eqref{eq:depart-SL-FV}, and the west-bounding characteristic curve satisfies the same problem with the arrival point shifted from $x_{i+1/2}$ to $x_{i-1/2}$. 
At the departure time $t_n$, the departure points $\wt{\xi}_{i\pm 1/2} := \xi_{i \pm 1/2}(t_n)$ are decomposed into the sum of their east neighbouring cell interfaces $\wt{x}_{i \pm 1/2}$ and their mesh-normalized distances $\varepsilon_{i \pm 1/2}$ from these interfaces, $\wt{\xi}_{i\pm 1/2} = \wt{x}_{i \pm 1/2} - h \varepsilon_{i \pm 1/2}$.
\label{fig:SL-FV-conservation}
}
\end{figure}

Based on \eqref{eq:cons-lin-SL-cons-local}, the numerical flux $\wt{f}_{i+1/2}$ in \eqref{eq:SL-FV-scheme} is then defined by an approximation of the integral in \eqref{eq:cons-lin-SL-cons-local}.
Specifically, 
\begin{align} \label{eq:SL-FV-num-flux}
\int_{\wt{\xi}_{i + 1/2}}^{x_{i+1/2}} e(x, t_n) \d x
\approx
\int_{\wt{\xi}_{i + 1/2}}^{x_{i+1/2}} R_{i + 1/2}[ \bar{\bm{e}}^n ](x) \d x
=:
\wt{f}_{i+1/2}.
\end{align}
Here, $R_{i + 1/2}[\bar{\bm{e}}^n](x) \approx e(x, t_n)$ is a cell-average-preserving \textit{piecewise polynomial reconstruction} of $e(x, t_n)$; see Supplementary Materials Section \ref{SMsec:SL-flux} for further discussion.

Since we apply the above described semi-Lagrangian method on the coarse grid, we need to approximately solve \eqref{eq:depart-SL-FV} on the coarse grid to estimate coarse-grid departure points. To do so, we first estimate departure points on the fine grid and then combine these using the linear interpolation and backtracking strategy we describe in \cite[Section 3.3.1]{DeSterck_etal_2023_SL}.
To estimate departure points on the fine grid at time $t = t_n$, we approximately solve \eqref{eq:depart-SL-FV} with a single forward Euler step, which requires the wave-speed function $\alpha(x,t)$ at the arrival point $(x,t) = (x_{i+1/2}, t_{n} + \delta t)$.
Recall (see \cref{sec:linearization}) that the wave-speed in the linearized problem is given by that of the nonlinear problem \eqref{eq:cons-law} frozen at the current nonlinear iterate in \cref{alg:richardson}. Thus, we take the linear wave-speed at the arrival point as the average of the two reconstructions that we have available at this point: $\alpha(x_{i+1/2},t_n + \delta t) \approx \big[ f'(u_{i+1/2}^-) + f'(u_{i+1/2}^+) \big]/2$, where $u_{i+1/2}^{\pm}$ are reconstructions of the current linearization point at $(x, t) = (x_{i+1/2}, t_n + \delta t)$.

In our numerical experiments, we use the following expression for the semi-Lagrangian truncation error operator in \eqref{eq:Psi}:
\begin{align} \label{eq:T-direct}
{\cal T}_{\rm direct}^{n}
=
-h^{p+1}
{\cal D}_1
\diag 
\big( 
g_{p+1}
(
\bm{\varepsilon}
)
\big) 
{\cal D}_p^\top.
\end{align}
This expression is based on a heuristic truncation error analysis for the semi-Lagrangian method applied to \eqref{eq:cons-lin}; further details can be found in Supplementary Materials Section \ref{SMsec:SL-extra}. 
%
%
Here, $\bm{\varepsilon} = \big( \varepsilon_{1+1/2}, \ldots \varepsilon_{n_x + 1/2} \big) \in \mathbb{R}^{n_x}$, with $\varepsilon_{i+1/2}$ the mesh-normalized difference between the $i$th departure point (over the coarse-grid time-step $t_n \to t_n + m \delta t$) and its east-neighboring mesh point; see the fine-grid example in \cref{fig:SL-FV-conservation}.
The function $g_{p+1}$ is a degree $p+1$ polynomial (see \eqref{eq:SL-g-poly-def}) and is applied component-wise to $\bm{\varepsilon}$.
As previously, $({\cal D}_1 \bb{e})_i = (\bar{e}_i - \bar{e}_{i-1})/h$, and ${\cal D}_p$, for $p$ odd, is a 1st-order accurate FD discretization of a $p$th derivative on a stencil with a 1-point bias to the left.
Finally, we remark that \eqref{eq:T-direct} has a structure similar to the truncation error operator we used in \cite{DeSterck_etal_2023_SL}, for FD discretizations of non-conservative, variable-wave-speed problems, which was of the form $-h^{p+1}
\diag 
\big( 
g_{p+1}
(
\bm{\varepsilon}
)
\big) 
{\cal D}_{p+1}$.

\section{Numerical results}
\label{sec:num-res}

In this section, we numerically test the solver described in \cref{alg:richardson}, utilizing the linearization procedure from \cref{sec:linearization} and the linearized MGRIT solver from \cref{sec:mgrit}.
Test problems and numerical details are discussed in \cref{sec:model-probs}, with numerical tests then given in \cref{sec:num-tests}, followed by a discussion of speed-up potential in \cref{sec:speed-up}.

\subsection{Test problems and numerical setup}
\label{sec:model-probs}

We consider two conservation laws of type \eqref{eq:cons-law}, namely:
\begin{align} 
\label{eq:burgers} \tag{B}
\frac{\partial u}{\partial t} + \frac{\partial f(u)}{\partial x} &= 0,
\quad
f(u) = \frac{u^2}{2},
\\
\label{eq:buck} \tag{BL}
\frac{\partial u}{\partial t} + \frac{\partial f(u)}{\partial x}  &= 0,
\quad
f(u) = \frac{4 u^2}{4 u^2 + (1 - u)^2},
\end{align}
with periodic boundary conditions in space.
Equation \eqref{eq:burgers} is the Burgers equation and serves as the simplest nonlinear hyperbolic PDE, and \eqref{eq:buck} is the Buckley--Leverett equation and can have much richer solution structure than Burgers equation owing to its non-convex flux $f$---see \cite[Section 16.1.1]{LeVeque_2004} for detailed discussion on the solution structure of \eqref{eq:buck}.
For both PDEs, we use the square wave initial condition $u_0(x) = 1$ for $x \in (-0.5, 0)$ and $u_0(x) = 0$ for $x \in (-1, 1) \setminus (-0.5, 0)$.
We choose this initial data because it gives rise to both rarefaction and shock waves.
Additional numerical tests using a smooth initial condition can be found in Supplementary Materials Section \ref{SMsec:num-res-additional-u0}, where solver convergence is qualitatively similar to that for the square wave initial condition.
We solve \eqref{eq:burgers} on the time domain $t \in [0, 4]$, and \eqref{eq:buck} on $t \in [0, 2]$; solution plots are shown in \cref{fig:test-prob}.
Note that for \eqref{eq:buck}, when the LLF flux \eqref{eq:LFF} with  \eqref{eq:LFF-local} is used, reconstructions are always limited to be within the physical range $u \in [0,1]$: Any $u_{i+1/2}^{\pm} > 1$ are mapped to $u_{i+1/2}^{\pm} = 1$, and any $u_{i+1/2}^{\pm} < 0$ are mapped to $u_{i+1/2}^{\pm} = 0$; see Supplementary Materials Remark \ref{SMrem:BL-limiting} for further discussion.

\newcommand{\fd}{./figures/}
\newcommand{\vs}{1}
\newcommand{\hs}{1}
\newcommand\fs{0.35}

\begin{figure}[t!]
\centerline{
\includegraphics[scale=\fs]{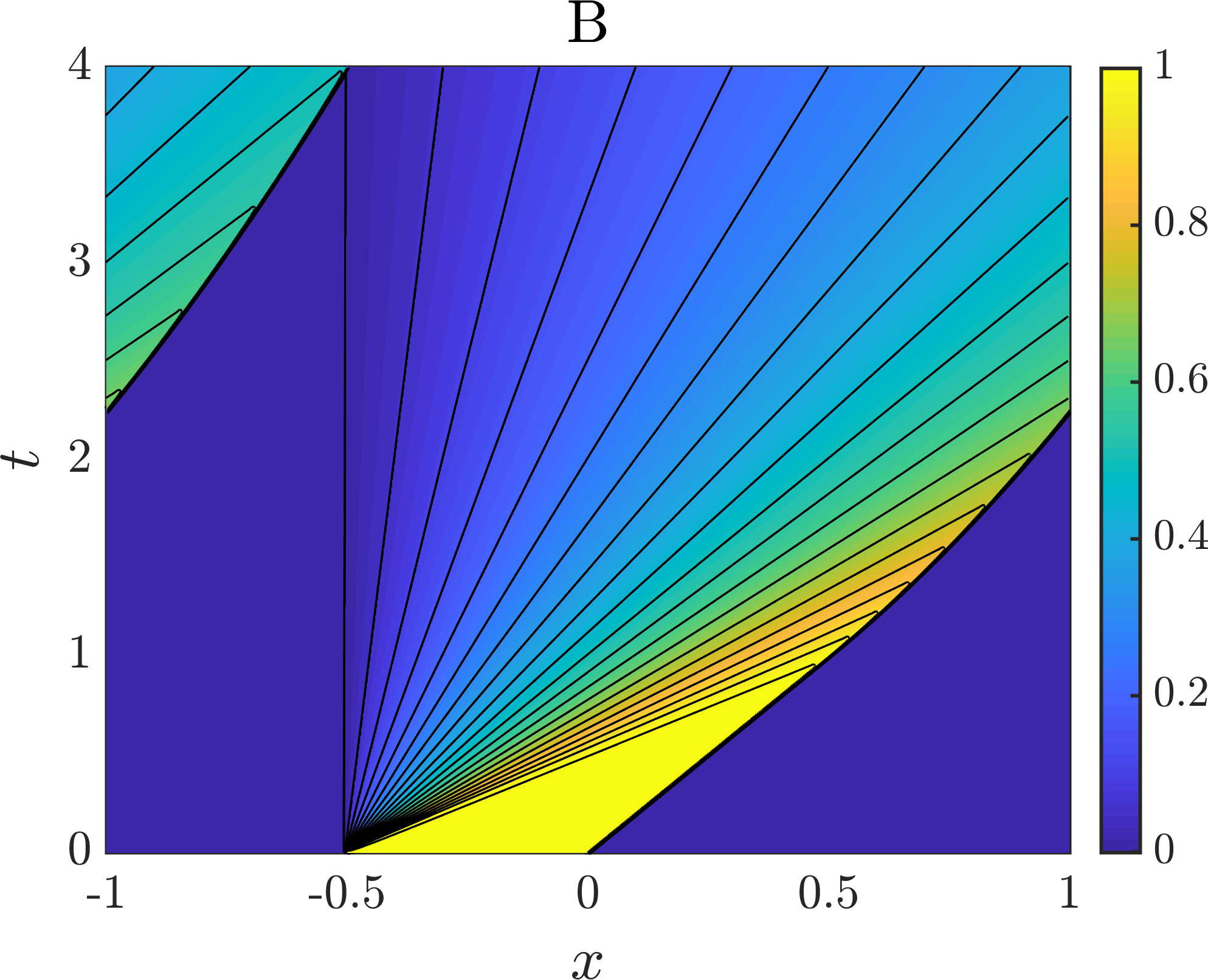}
\hspace{\hs ex}
\includegraphics[scale=\fs]{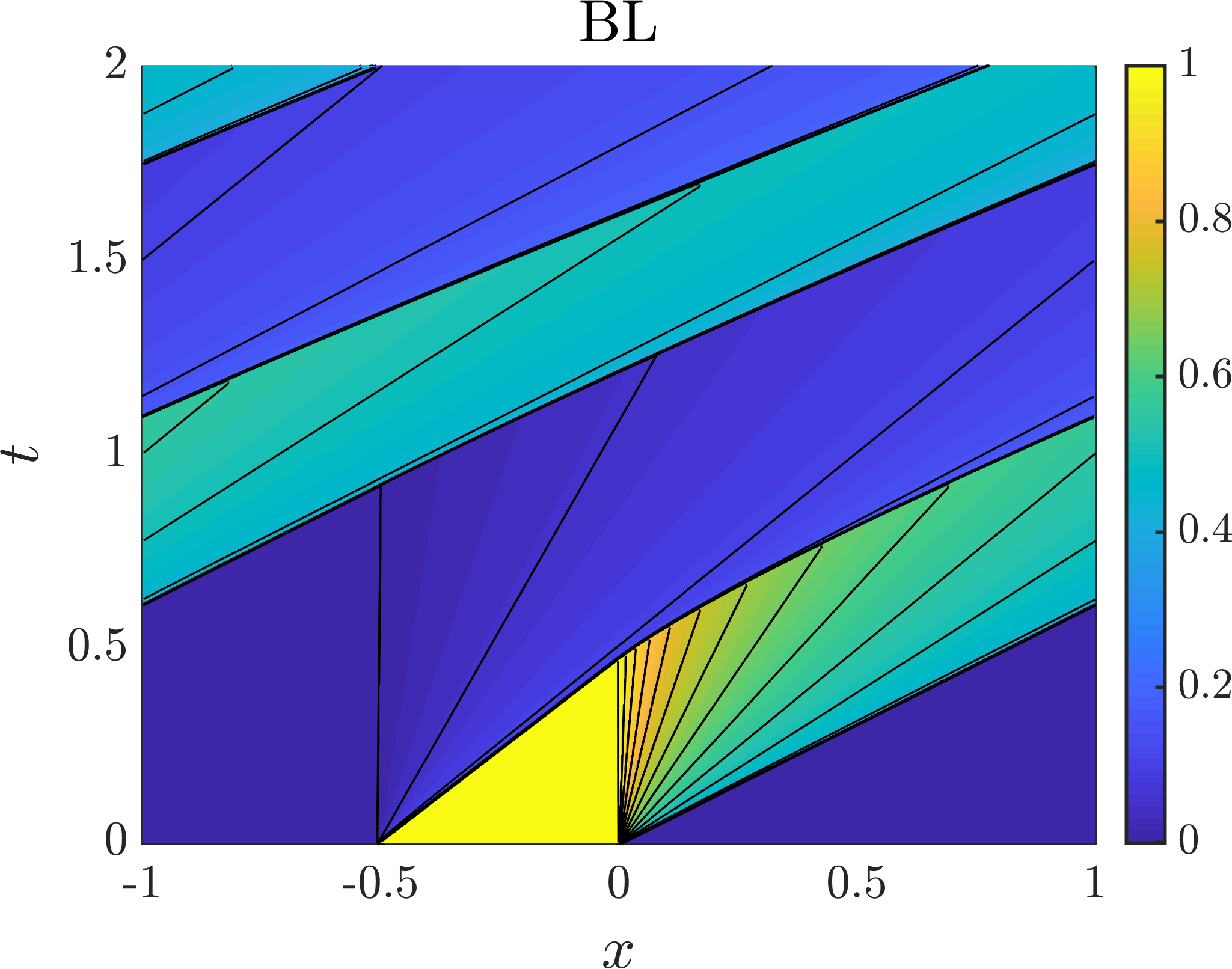}
}
\vspace{\vs ex}
\centerline{
\includegraphics[scale=\fs]{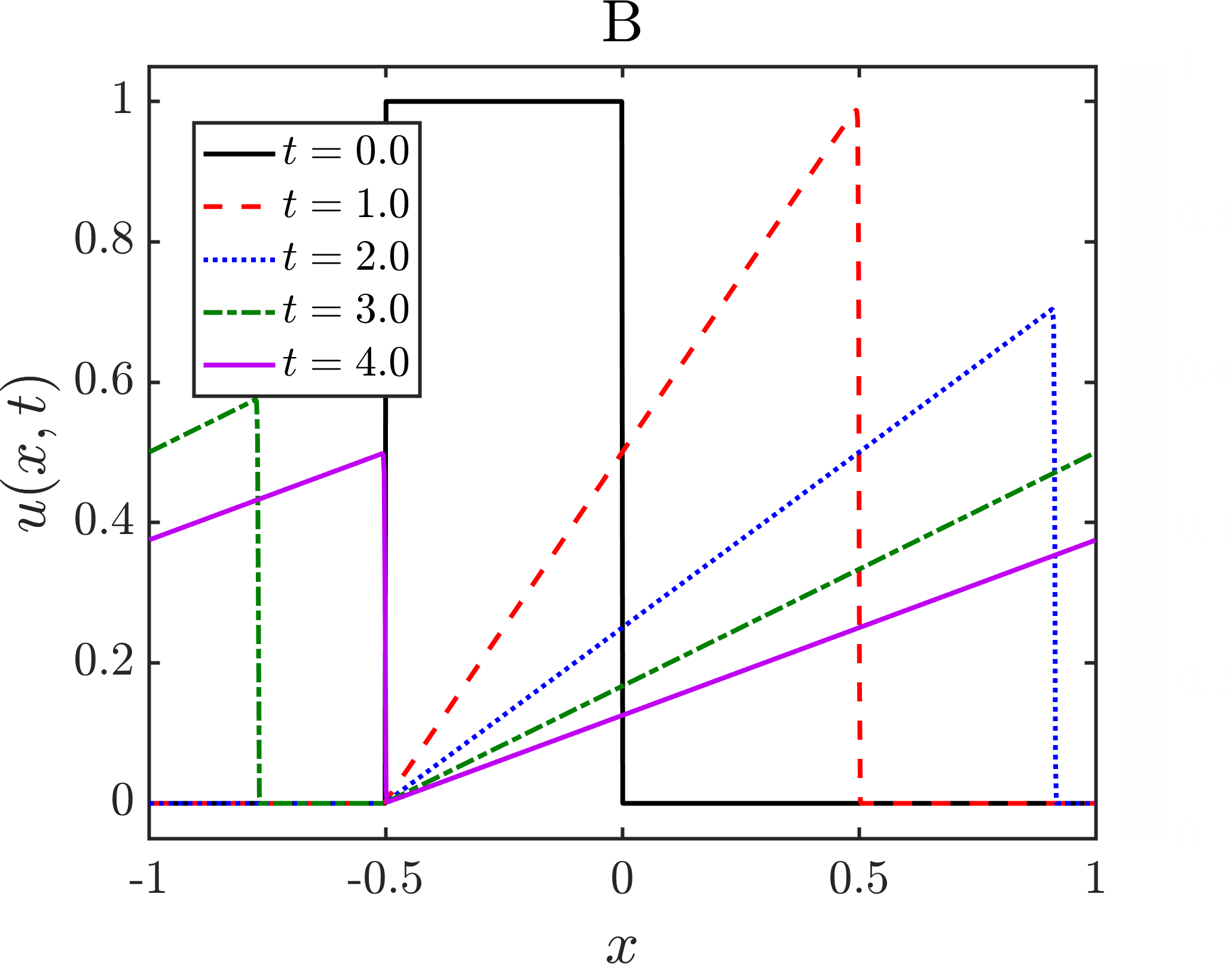}
\hspace{\hs ex}
\includegraphics[scale=\fs]{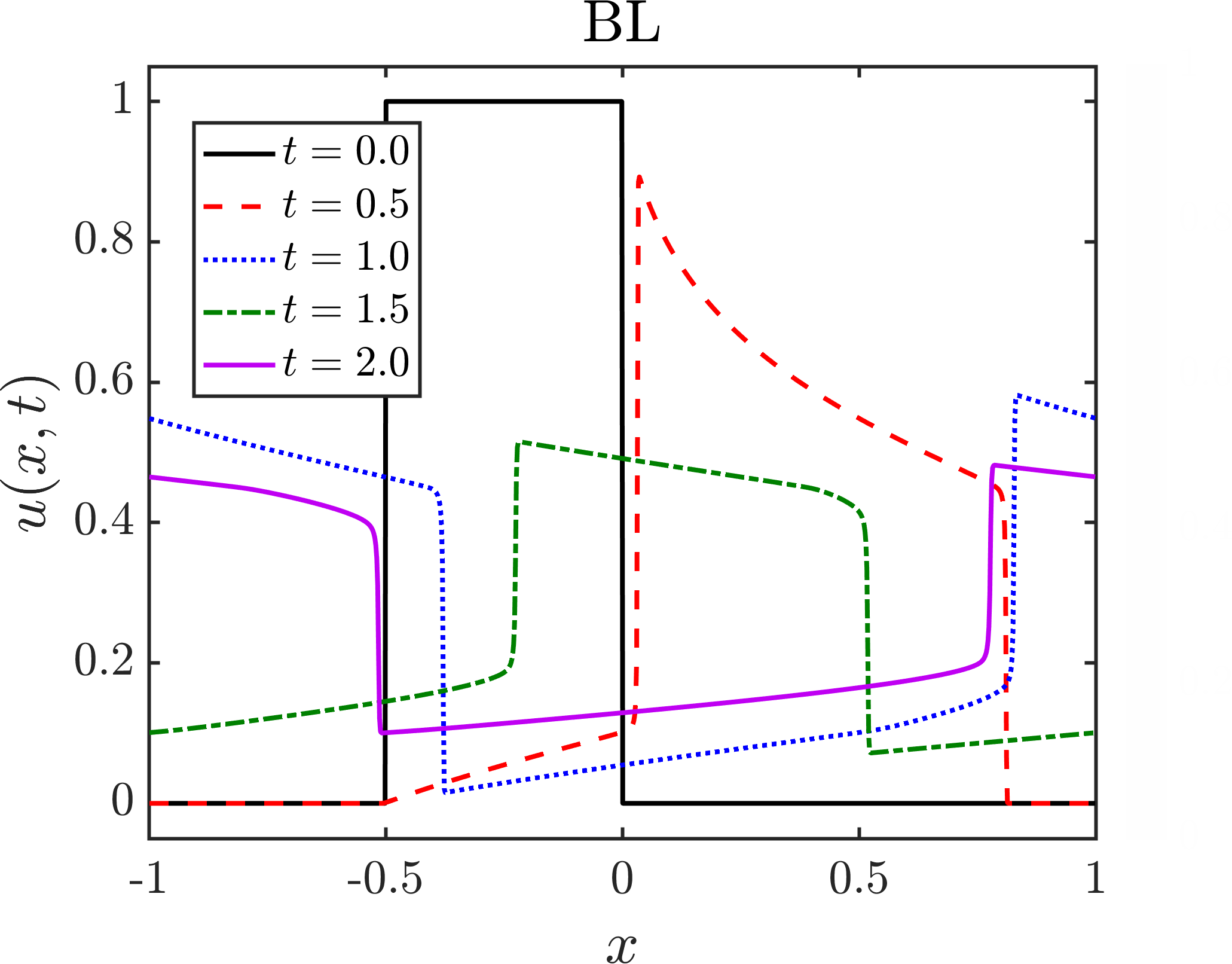}
}
\caption{Solution plots for numerical test problems: Burgers equation \eqref{eq:burgers} (left column) and Buckley--Leverett equation \eqref{eq:buck} (right column).
Plots in the top row are space-time contours, with the black lines being 20 contour levels evenly spaced between 0 and 1.
Plots in the bottom row are solution cross-sections at the times indicated in the legends.
For the Burgers equation, the initial discontinuity at $x = -0.5$ propagates as a rarefaction wave with speed $1$ while the one at $x = 0$ propagates as a shock wave with speed $1/2$. At $t = 1$ these two waves merge and propagate at a reduced speed thereafter (see Supplementary Material Section \ref{SMsec:over-solving} for further details).
For the Buckley--Leverett equation, both discontinuities propagate as compound waves: They contain both a shock and an attached rarefaction.
\label{fig:test-prob}
}
\end{figure}

We now discuss some numerical details.
Each problem is solved on a sequence of meshes to understand the scalability of the solver with respect to problem size.
For 1st-order discretizations, we consider meshes with $n_x = 64,$ 128, 256, 512, 1024, 2048, 4096 spatial FV cells. For \eqref{eq:burgers}, the corresponding temporal meshes have $n_t = 161,$ 321, 641, 1281, 2561, 5121, 10241 points, respectively, and for \eqref{eq:buck}, the corresponding temporal meshes have $n_t = 188,$ 375, 748, 1494, 2986, 5971, 11941 points, respectively.
For 3rd-order discretizations we consider meshes with $n_x = 64,$ 128, 256, 512, 1024, 2048 spatial FV cells.
The initial nonlinear iterate on a given space-time mesh, that is, for a fixed $n_x$ and $n_t$, is obtained by linear interpolation of the solution to the same problem on a mesh with $n_x/2$ FV cells in space and $n_t/2$ points in time.

All tests use nonlinear F-relaxation with CF-splitting factor $m = 8$ (Line \ref{ln:nonlin-relax} in \cref{alg:richardson}).
In all tests, a single MGRIT iteration is used to approximately solve each linearized problem (Line \ref{ln:approx-solve} in \cref{alg:richardson}), with the initial MGRIT iterate set equal to the corresponding right-hand side vector of the linearized system, i.e., the current nonlinear residual. 
We perform a single F-relaxation for the MGRIT pre-relaxation using a coarsening factor of $m = 8$.
This MGRIT initialization strategy is equivalent, although computationally cheaper, to using a zero initial MGRIT iterate followed by pre CF-relaxation on the first MGRIT iteration on the fine level.\footnote{For $\bm{r}_k \approx \bm{0}$, the linearization error is small, such that the linearized problem is a highly accurate model for the nonlinear residual equation \eqref{eq:non-lin-sys}. 
Thus, for sufficiently small $\bm{r}_k$, our relaxation and initial iterate strategy is effectively the same as if no nonlinear relaxation were applied and MGRIT was applied to the linearized problem with FCF-relaxation. Note that FCF-relaxation is typically used in the context of MGRIT for linear hyperbolic problems, with F-relaxation sometimes not sufficient for obtaining a convergent solver \cite{DeSterck_etal_2023_MOL}.}
For simplicity, we consider only two-level MGRIT; however, the interested reader can find multilevel results for first-order discretizations in Supplementary Materials Section \ref{SMsec:num-res-additional-multilevel}, where the multilevel algorithm proceeds analogously to those in our previous works using modified semi-Lagrangian coarse operators \cite{KrzysikThesis2021,DeSterck_etal_2023_SL,DeSterck_etal_2023_MOL}.

The approximate truncation error correction at coarse-grid time steps, i.e., the matrix inversion in \eqref{eq:Psi}, is applied exactly via LU factorization. 
In our previous work for linear problems \cite{DeSterck_etal_2023_MOL,DeSterck_etal_2023_SL}, and for the linear results in Supplementary Materials Section \ref{SMsec:MGRIT-linear-standard}, these linear systems were approximately solved using a small number of GMRES iterations. 
Numerical tests (not shown here) indicate that GMRES is possibly less effective for the current linearized problems, and we hope to address this in future work.
%

\subsection{Results}
\label{sec:num-tests}

\renewcommand{\fd}{./figures/}
\renewcommand{\hs}{2}
\renewcommand{\vs}{1}
\renewcommand\fs{0.335}
\begin{figure}[t!]
\centerline{
\includegraphics[scale=\fs]{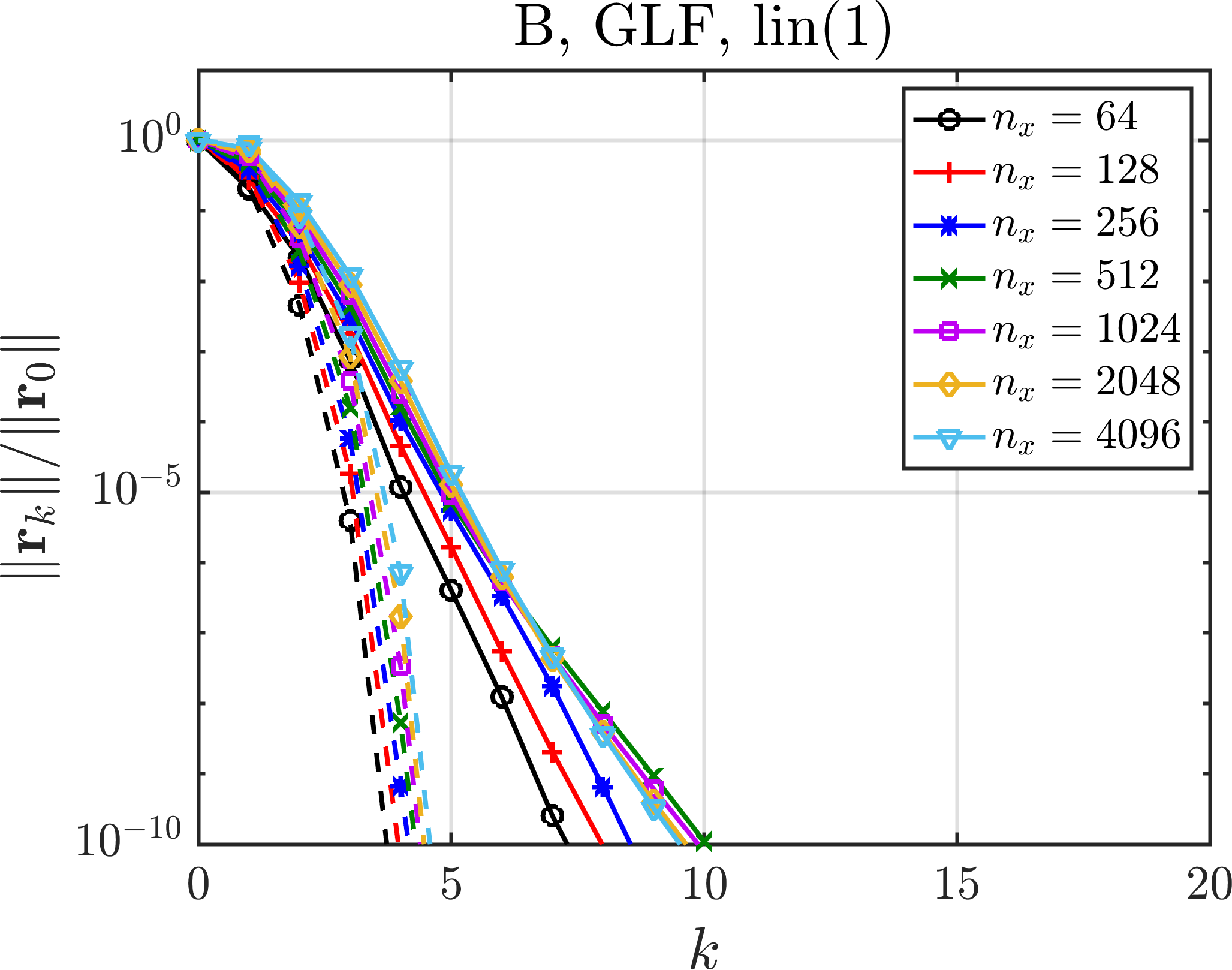}
\hspace{\hs ex}
\includegraphics[scale=\fs]{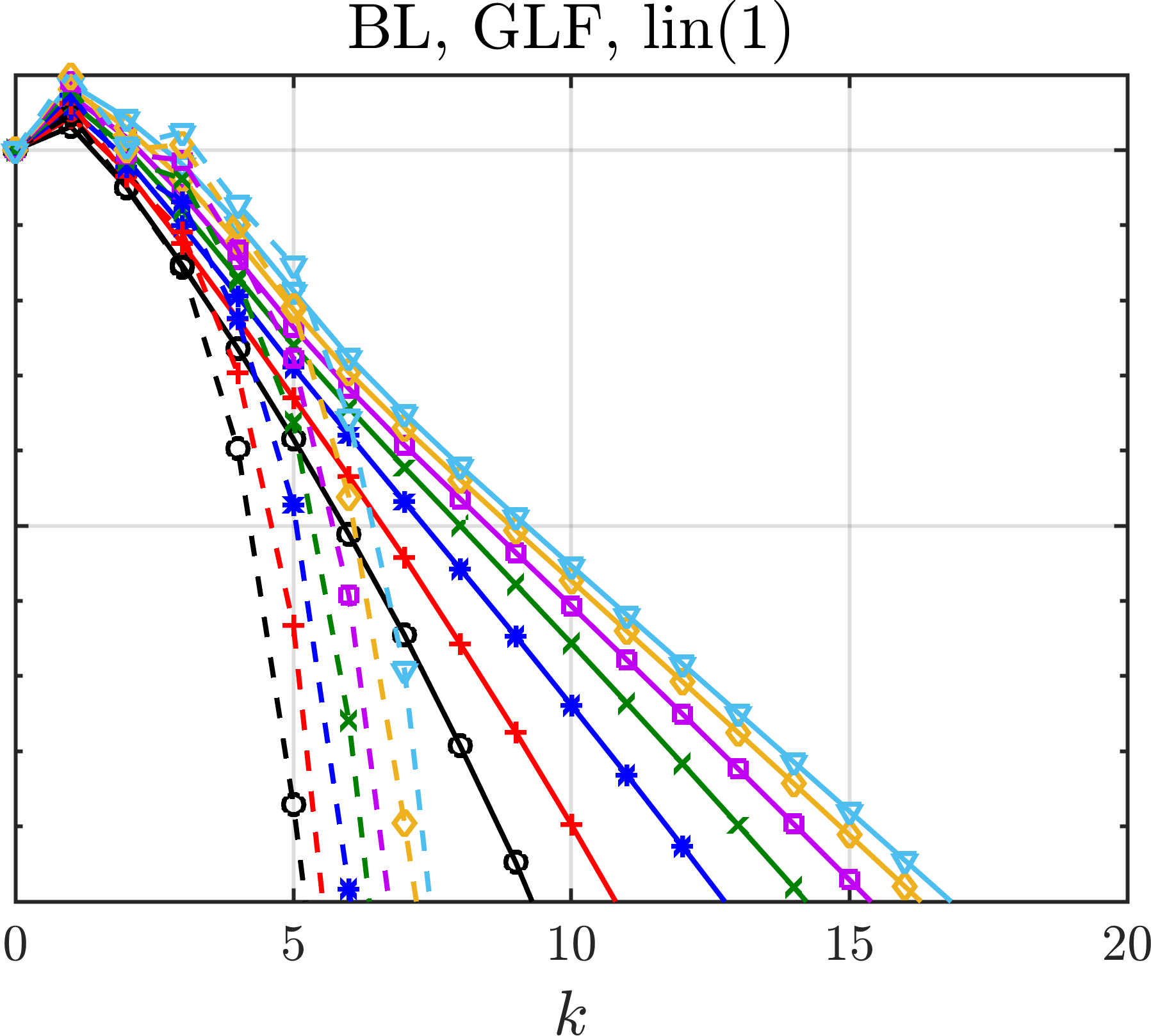}
}
\vspace{\vs ex}
\centerline{
\includegraphics[scale=\fs]{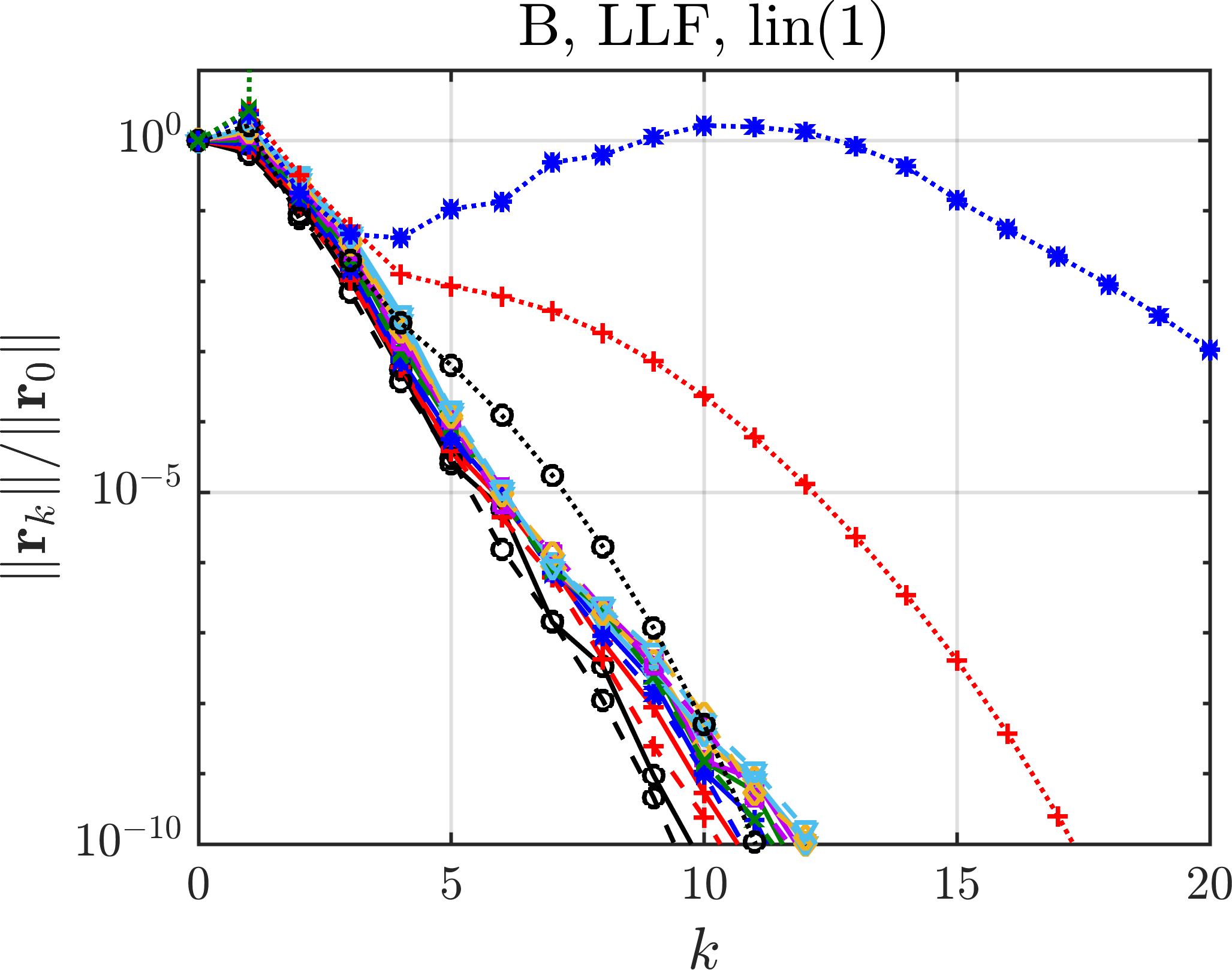}
\hspace{\hs ex}
\includegraphics[scale=\fs]{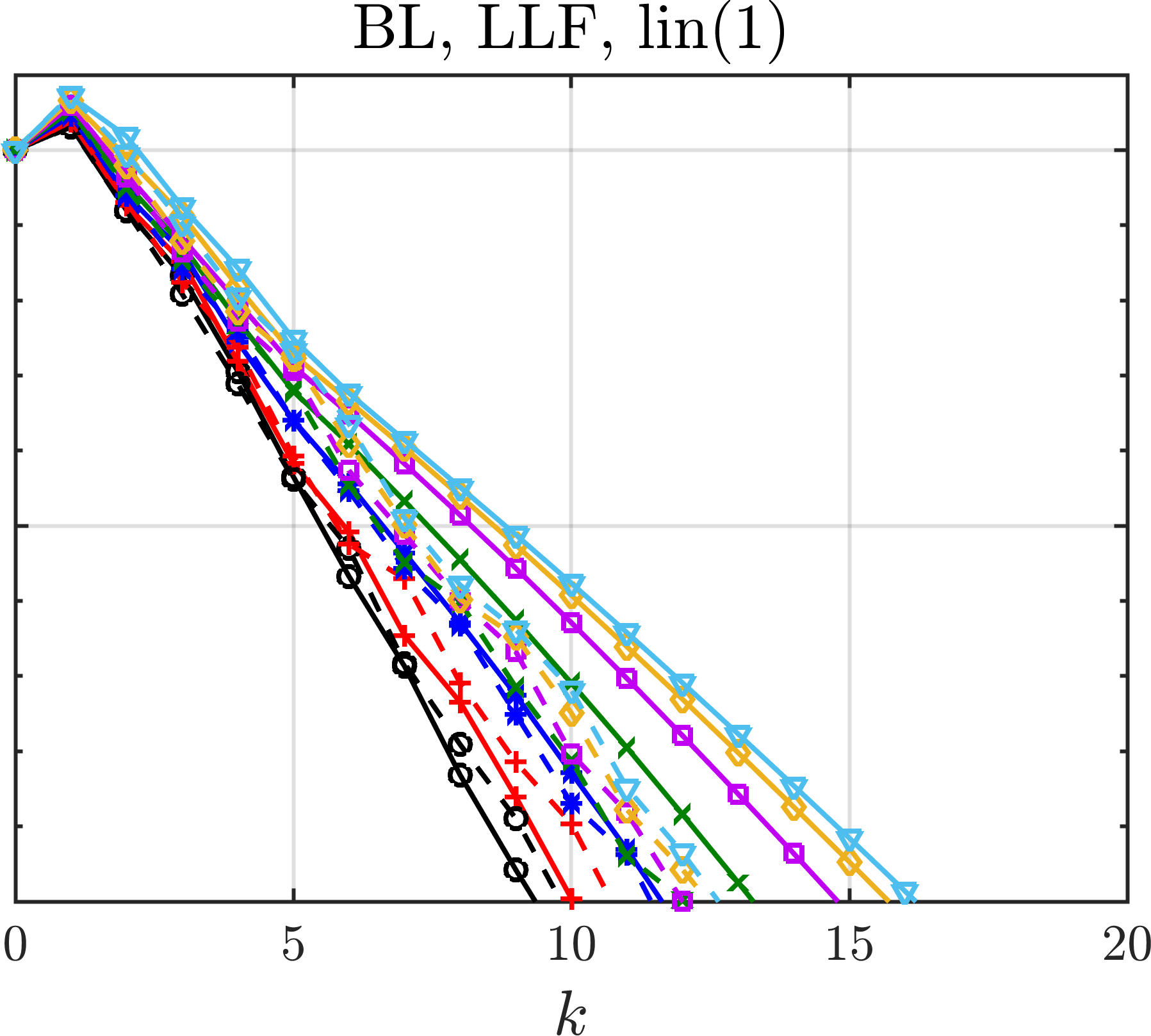}
}
\caption{Residual convergence histories for the nonlinear solver \cref{alg:richardson} applied to 1st-order accurate discretizations.
Solid lines correspond to linear systems being approximately solved by one MGRIT iteration, and broken lines correspond to linear systems being solved exactly. 
Left column: Burgers \eqref{eq:burgers}. Right column: Buckley--Leverett \eqref{eq:buck}.
The GLF numerical flux is used in the top row, and the LLF numerical flux in the bottom row.
All solves use nonlinear F-relaxation with $m = 8$ (Line \ref{ln:nonlin-relax} in \cref{alg:richardson}).
Dotted lines in the bottom left plot show residual histories for $n_x = 64, 128, 256, 1024$ corresponding to using the MGRIT coarse-grid operator \eqref{eq:Psi} without a truncation error correction, that is, ${\cal T}_{\rm ideal}^{n} = {\cal T}_{\rm direct}^{n} = 0$; note that the $k=2$ data point for $n_x = 1024$ is not visible since $\Vert \bm{r}_2 \Vert / \Vert \bm{r}_0 \Vert \approx 10^{76}$.
\label{fig:num-res-inexact-1st}
}
\end{figure}

\renewcommand{\fd}{./figures/}
\renewcommand{\hs}{2}
\renewcommand{\vs}{1}
\renewcommand\fs{0.335}
\begin{figure}[t!]
\centerline{
\includegraphics[scale=\fs]{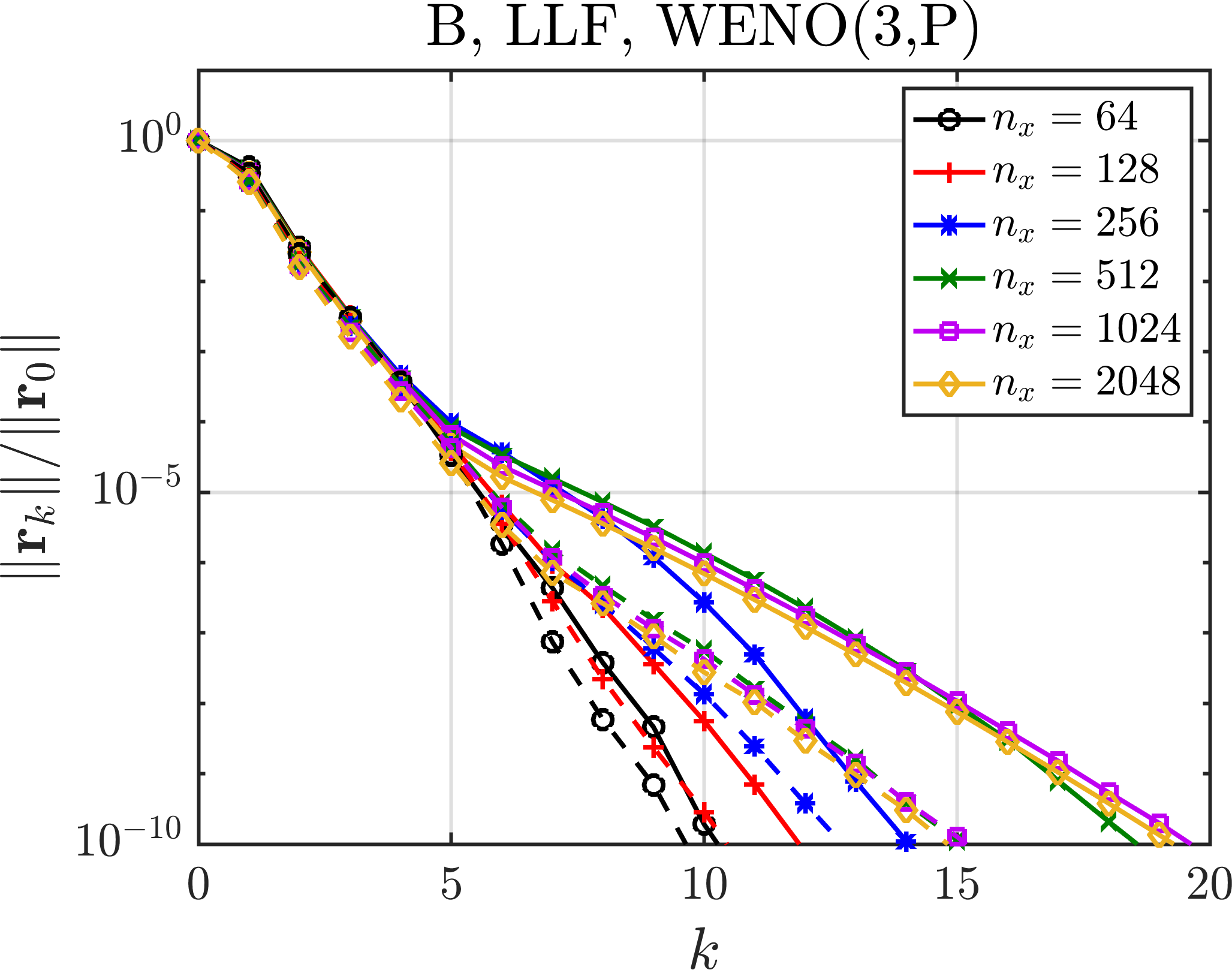}
\hspace{\hs ex}
\includegraphics[scale=\fs]{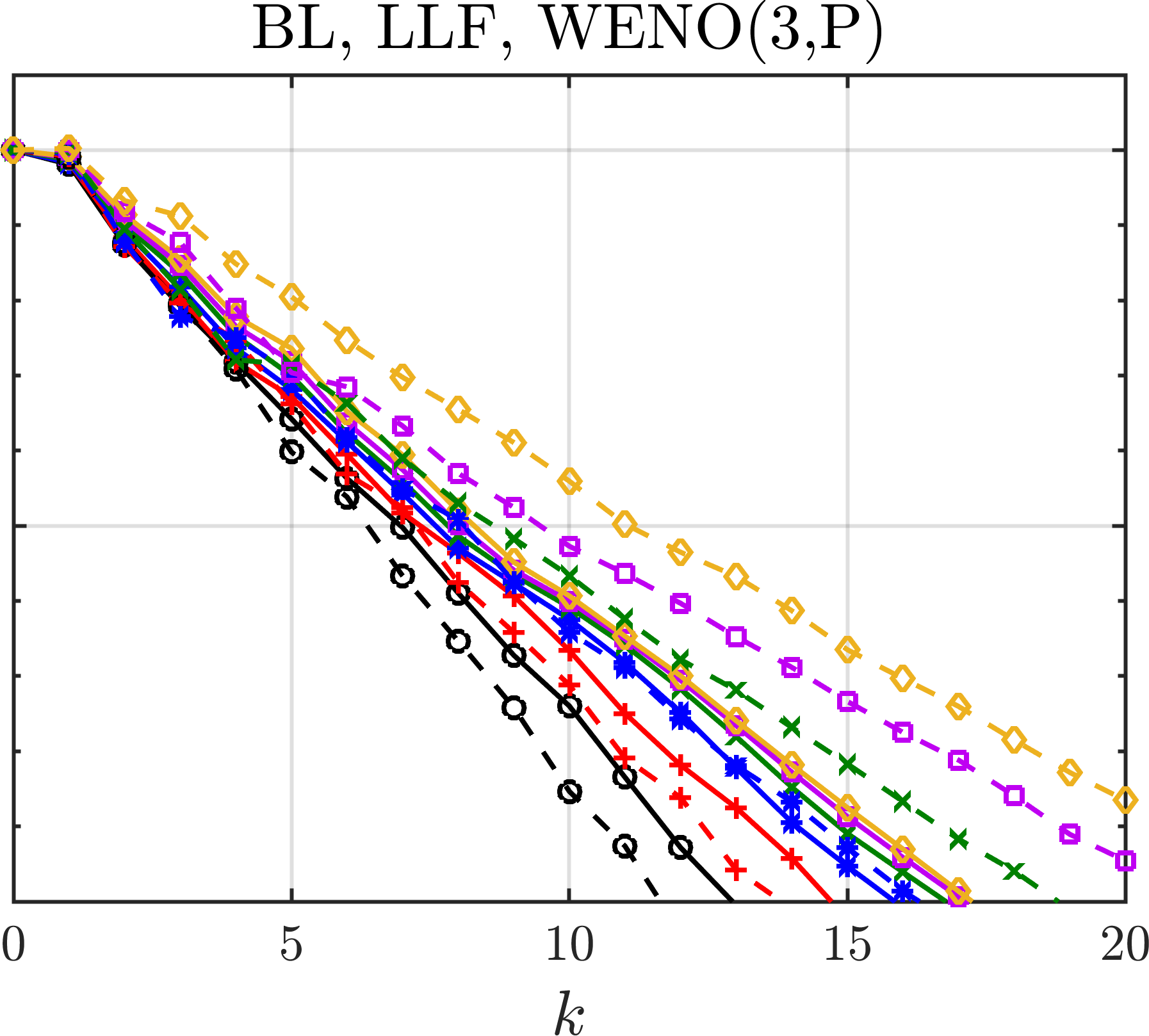}
}
\vspace{\vs ex}
\centerline{
\includegraphics[scale=\fs]{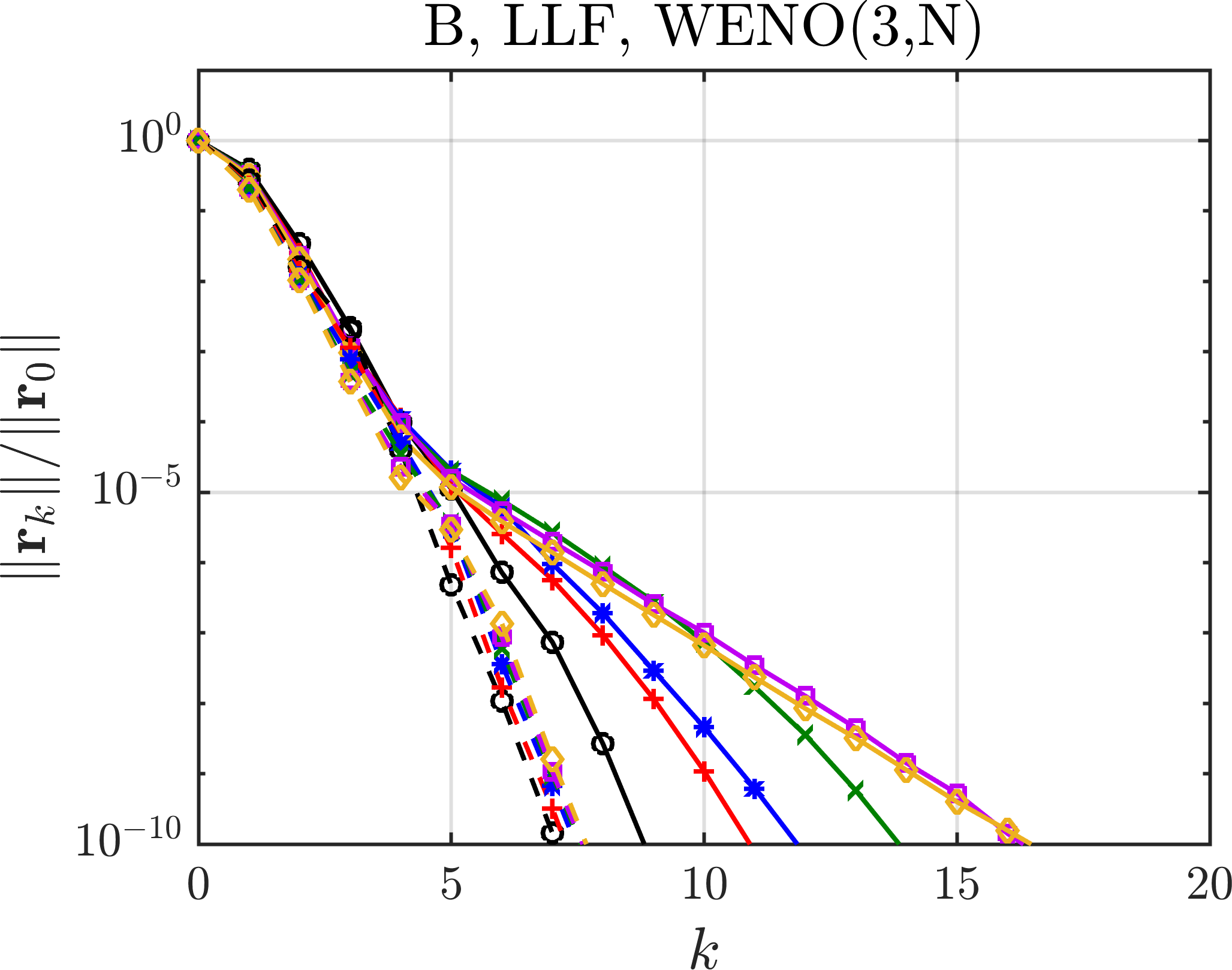}
\hspace{\hs ex}
\includegraphics[scale=\fs]{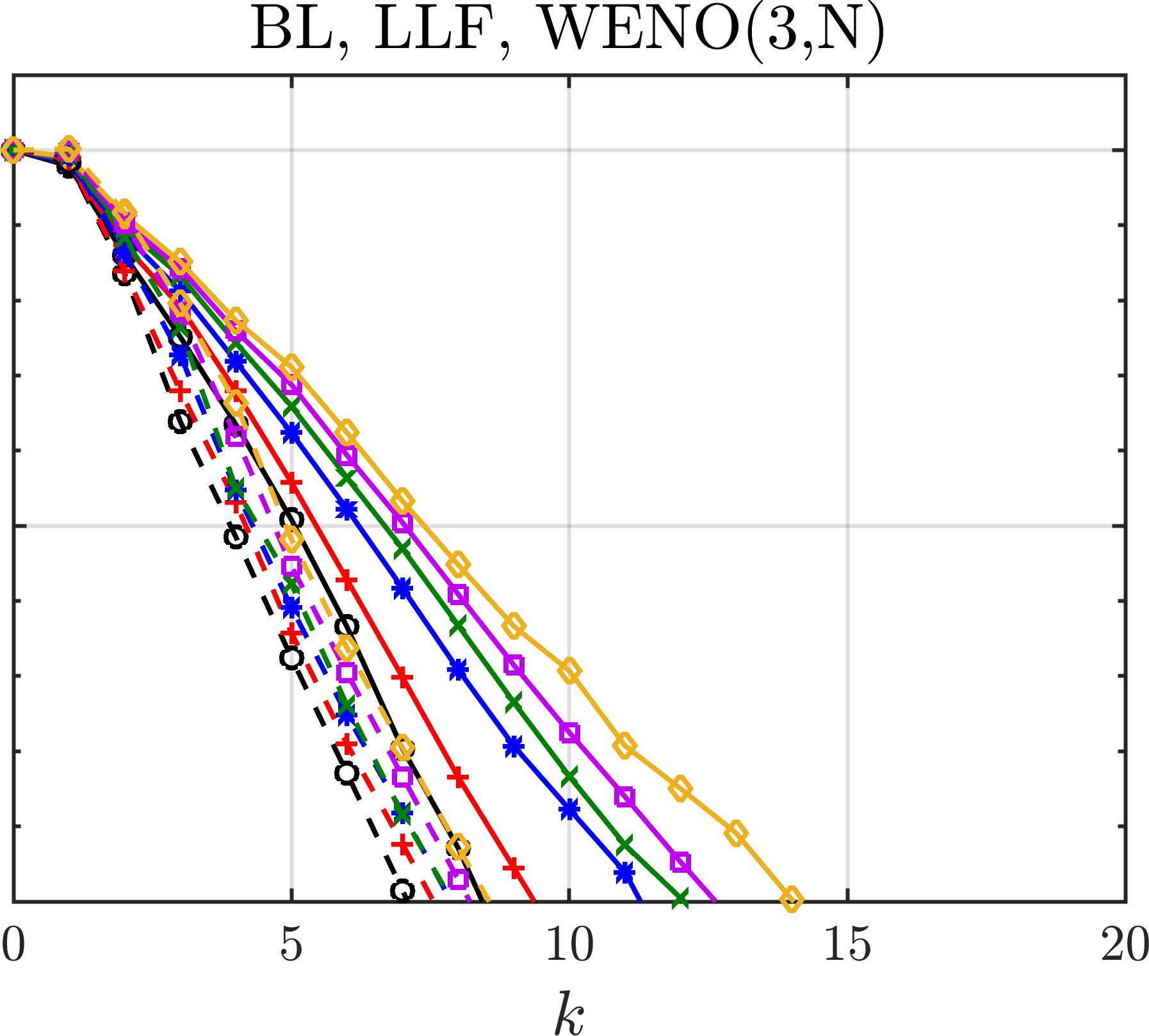}
}
\caption{Residual convergence histories for the nonlinear solver \cref{alg:richardson} applied to 3rd-order accurate discretizations using LLF numerical fluxes.
Solid lines correspond to linear systems being approximately solved by one MGRIT iteration, and broken lines correspond to linear systems being solved exactly. 
Left column: Burgers \eqref{eq:burgers}. Right column: Buckley--Leverett \eqref{eq:buck}.
All solves use nonlinear F-relaxation with $m = 8$ (Line \ref{ln:nonlin-relax} in \cref{alg:richardson}).
Tow row: Reconstructions are linearized with \eqref{eq:weighted-reconstruct-grad-zero}, corresponding to Picard linearization of the WENO weights.
Bottom row: Reconstructions are linearized with \eqref{eq:weighted-reconstruct-grad-FD}, corresponding to (approximate) Newton linearization of the WENO weights.
\label{fig:num-res-inexact-3rd-LLF}
}
\end{figure}

If \cref{alg:richardson} is to be effective when approximate solves of the linearized systems are used, it must converge adequately in the best-case setting that direct solves are used for the linearized systems (that is, when Line \ref{ln:direct-solve} is executed rather than Line \ref{ln:approx-solve}).
Numerical tests confirming the efficacy of \cref{alg:richardson} in the special case of direct linear solves can be found in Supplementary Materials Section \ref{SMsec:num-res-additional}.

We now consider plots of residual convergence histories for the nonlinear solver \cref{alg:richardson}.
In all cases the two-norm of the space-time residual is shown relative to its initial value.
The solver is iterated until: 20 iterations are performed or the relative residual falls below $10^{-10}$.

Solid lines in \cref{fig:num-res-inexact-1st} show nonlinear residual histories for 1st-order accurate discretizations.
For reference, broken lines in these plots show residuals when the linearized problems are solved directly.
For problems using the GLF flux (top row), the convergence rate reduces to linear when approximate solves are used; superlinear convergence could be maintained by applying more than a single MGRIT iteration, as with inexact Newton methods \cite{Eisenstat_Walker_1996}. 
Quite remarkably, for the Burgers LLF problem (bottom left), we are able to maintain the same convergence rate as when direct solves are used. 
For the non-convex Buckley--Leverett LLF problem (bottom right), we see some deterioration.
In any event, for all problems using inexact linear solves, iteration counts asymptotically tend to a constant as the mesh is refined.
Included in the bottom left plot are dotted lines showing residual histories corresponding to using the MGRIT coarse-grid operator \eqref{eq:Psi} without a truncation correction. Clearly, the truncation error correction in \eqref{eq:Psi} is a key component of our solution methodology. 

\Cref{fig:num-res-inexact-3rd-LLF} shows results for 3rd-order discretizations using LLF numerical flux functions. Note results for GLF numerical fluxes are omitted since they are qualitatively similar.
As previously, solid lines correspond to approximate parallel-in-time linear solves, and broken lines to direct linear solves.
For the Burgers problems (left column), convergence of the nonlinear solver deteriorates when approximate MGRIT solves are used instead of direct solves.
However, the number of iterations appears to be asymptoting to a constant as the mesh is refined.
For the Buckley--Leverett problems (right column), iteration counts also seem to be roughly constant when approximate MGRIT solves are introduced.
Quite remarkably, the iteration counts even improve for the largest problems in  cases of Picard linearization (top right), compared to direct linear solves. We speculate that the single MGRIT iteration here may be acting somewhat akin to a line search, in the sense that it is not accurately resolving components of the linearized error that poorly approximate the true nonlinear error.

\begin{remark}[Over-solving] \label{rem:os}
Solving to residual tolerances as tight as those in \cref{fig:num-res-inexact-1st,fig:num-res-inexact-3rd-LLF} is important for demonstrating stability of the algorithm\footnote{Typically, when applied to advection-dominated problems, multigrid-in-time methods are divergent on a set of error modes when used with naive direct coarse-grid discretizations. If these error modes are not present or only weakly present in the initial iterate, the method will appear to converge on initial iterations, but will diverge on later iterations due to exponential growth of these problematic error modes \cite{DeSterck_etal_2025_LFA,DeSterck_etal_2023_MOL,Gander_Lunet_2020}.} and understanding the mesh-dependence of the asymptotic convergence rate.
However, such tight tolerances likely do not result in more accurate approximations to the underlying PDE solution.
For shocked scalar problems, the discretizations we consider in this work do not converge in the $L^{\infty}$-norm, and their convergence rate in the discrete $L^{1}$-norm is, at best, ${\cal O}(h)$ for strictly convex fluxes, and more generally is ${\cal O}(\sqrt{h})$  \cite{Harabetian1988,Tang_Teng_1997,Teng_Zhang_1997}.
Thus, the practical value of iterating until the algebraic residual is point-wise small globally is questionable.
Supplementary Materials Section \ref{SMsec:over-solving} considers a Burgers problem with $n_x = 1024$, showing that the approximation after just two nonlinear iterations is as good an approximation to the true PDE solution as the sequential time-stepping solution is.
\end{remark}

\subsection{Speed-up potential}
\label{sec:speed-up}

We again stress that our chief goal in this paper is developing and verifying a new parallel-in-time strategy for difficult nonlinear problems, and that this alone represents significant work.
With this is mind, we now briefly discuss the potential of our method to provide speed-up over sequential time-stepping when it is implemented in parallel. 
In our previous work for linear problems \cite{DeSterck_etal_2023_MOL,DeSterck_etal_2023_SL}, when solving to tight residual tolerances, we achieved speed-up factors on the order of two to 12 times, depending on the problem. 
%
%
We anticipate our method developed here for nonlinear problems is capable of generating similar or slightly smaller speed-ups provided that the truncation error correction can be approximated with a cheap iterative method (i.e., not LU factorization).
Our reasoning is as follows.
The nonlinear convergence rates here are comparable to the linear convergence rates in \cite{DeSterck_etal_2023_MOL,DeSterck_etal_2023_SL}.
We note also that a comparable amount of work per iteration is done on the fine grid in each approach. Here we apply nonlinear F-relaxation, a nonlinear residual evaluation at C-points, and an F-relaxation on the linearized problem, while in \cite{DeSterck_etal_2023_MOL,DeSterck_etal_2023_SL} FCF-relaxation was used.

An extra cost in the current setting is that the linearized problem changes at every outer nonlinear iteration, and, so, the terms in the truncation error correction \eqref{eq:T-ideal} get updated at every iteration. Similarly, the coarse-grid characteristics are updated every iteration. We suspect that the updates for both quantities could be frozen after a few iterations, but we have not tested this.
In addition, our current approach has significant memory demands relative to our previous approaches for linear problems \cite{DeSterck_etal_2023_MOL,DeSterck_etal_2023_SL}. This is because the solution, dissipation coefficients, wave-speeds and WENO weights (if used) must be stored for all points in space-time, including for intermediate Runge-Kutta stages.
Future work will develop a parallel implementation of our solver.
%

\section{Conclusion and future outlook}
\label{sec:conclusion}

We have developed an iterative solver for space-time discretizations of scalar nonlinear hyperbolic conservation laws in one spatial dimension.
The solver uses a global linearization and solves at each iteration a space-time discretized linear conservation law to obtain an error correction.
The efficacy of the solver is demonstrated by its ability to solve high-order, WENO-based discretizations in a small number of iterations.
This includes problems with non-differentiable discretizations, problems with  shocks and rarefaction waves, and even PDEs with non-convex fluxes such as the Buckley--Leverett equation.

The nonlinear solver can be made parallel-in-time by replacing direct solves of the linearized problems with inexact parallel-in-time solves.
We consider using a single MGRIT iteration for this purpose, although other linear parallel-in-time methods could be considered.
The MGRIT solution approach generalizes those we have developed previously for non-conservative linear advection equations.
In many test problems the convergence of the nonlinear solver when paired with inner MGRIT solves is fast, and has mesh independent convergence rates.
For example, for certain Burgers problems with interacting shock and rarefaction waves, the residual norm is decreased by 10 orders of magnitude with just 10 MGRIT iterations.

This work leaves open many avenues for future research. An immediate next step is to develop a parallel implementation; extensions to problems in multiple spatial dimensions are also important, and are likely to be based on the multi-dimensional MGRIT methodology developed in \cite{DeSterck_etal_2023_SL}.
In \cite{Krzysik_etal_2025_systems} we discuss the extension of this work to one-dimensional hyperbolic systems, including the acoustic equations and the Euler equations of gas dynamics.

\section*{Acknowledgments}

Critical feedback from anonymous referees is gratefully acknowledged.

\bibliographystyle{siamplain}
\bibliography{nonlinear-scalar-1d-bib}


\pagebreak

\setcounter{section}{0}
\setcounter{equation}{0}
\setcounter{figure}{0}
\setcounter{table}{0}
\setcounter{page}{1}
\makeatletter
\renewcommand{\thesection}{SM\arabic{section}}
\renewcommand{\theequation}{SM\arabic{equation}}
\renewcommand{\thefigure}{SM\arabic{figure}}
\renewcommand{\thetable}{SM\arabic{table}}
\renewcommand{\thepage}{SM\arabic{page}}

\thispagestyle{plain} 

\headers{SUPPLEMENTARY MATERIALS: PinT for nonlinear conservation laws}{H. De Sterck, R. D. Falgout, O. A. Krzysik, J. B. Schroder}

\begin{center}
    \textbf{\normalsize\MakeUppercase{
    Supplementary Materials:
    Parallel-in-time solution of scalar nonlinear conservation laws}} 
    \vspace{6ex}
\end{center}


These Supplementary Materials are structured as follows.
\Cref{SMsec:reconstruction-overview} provides an introduction to the topic of polynomial reconstruction.
\Cref{SMsec:num-res-additional,SMsec:num-res-additional-u0,SMsec:num-res-additional-multilevel} present further numerical tests.
\Cref{SMsec:SL-extra} develops error estimates for the FV semi-Lagrangian method described in Section \ref{sec:SL}.
\Cref{SMsec:ideal-err-est} develops error estimates for method-of-lines discretizations applied to a simplified linear conservation law.
\Cref{SMsec:MGRIT-linear-standard} presents an MGRIT solver for the linear conservation law \eqref{eq:cons-lin} when it is discretized (on the fine grid) with standard linear method-of-lines discretizations.
\Cref{SMsec:ideal-err-est-linearized} develops an expression for the ideal truncation error in the coarse-grid operator \eqref{eq:Psi} where the fine-grid problem is a linearized method-of-lined discretization.
\Cref{SMsec:over-solving} considers the issue of over-solving the discretized problem for a Burgers problem.

\section{Overview of polynomial reconstruction}
\label{SMsec:reconstruction-overview}

In this section, we give a brief overview of polynomial reconstruction.
Polynomial reconstruction is a key ingredient of the method-of-lines discretization described in Section \ref{sec:discretization}, and of the semi-Lagrangian discretization discussed in Section \ref{sec:SL}.
Moreover, the error estimates we develop for these discretizations (see \cref{SMsec:ideal-err-est,SMsec:SL-extra}) require some understanding of this reconstruction process. 
More detailed descriptions on polynomial reconstruction can be found in the excellent works by Shu \cite{Shu_1998,Shu2009}.

Recall that the $j$th FV cell is defined by ${\cal I}_j = [x_{j - 1/2}, x_{j + 1/2}]$.
Now consider the following reconstruction problem: \textit{Given the cell averages $( \ldots, \bar{u}_{-1}, \bar{u}_0, \bar{u}_1, \ldots )$ of the function $u(x)$ in cells $( \ldots, {\cal I}_{-1}, {\cal I}_0, {\cal I}_1, \ldots )$, can we approximately reconstruct the function $u(x)$ for $x \in {\cal I}_0$?}
The answer is yes. Specifically, we seek a polynomial function $q(x) \approx u(x)$.
Since only the cell averages of $u(x)$ are known, these are used to select $q(x)$.

To this end, consider the following stencil of $k \geq 1$ contiguous cells
\begin{align} \label{eq:SL-FV-stencil-S}
S^{\ell} = \{ {\cal I}_{- \ell}, \ldots, {\cal I}_r \},
\end{align}
in which the left-most cell is ${\cal I}_{- \ell}$, and the right-most cell is ${\cal I}_r$. 
In this notation, $\ell, r \geq 0$, and we have that $\ell + r + 1 = k$. 
We say that $S^{\ell}$ has a \textit{left shift} of $\ell$.
The number of cells $k$ determines the order of accuracy of the reconstruction. For example, a 1st-order reconstruction uses the stencil $S^{0} = \{ {\cal I}_0 \}$, i.e., $\ell = 0$, $r = 0$, and a (centered) 3rd-order reconstruction uses the stencil $S^{1} = \{ {\cal I}_{-1}, {\cal I}_0, {\cal I}_1\}$, i.e., $\ell = r = 1$.
Note the cells of $S^{\ell}$ in \eqref{eq:SL-FV-stencil-S} cover the interval $x \in [x_{-\ell-1/2}, x_{r + 1/2}]$.

Now, we seek a polynomial $q_{k-1}(x) \approx u(x)$, $x \in [x_{-\ell-1/2}, x_{r + 1/2}]$, of degree at most $k-1$ such that it recovers the cell averages of the function $u(x)$ over the $k$ cells in the stencil $S$:
\begin{align} \label{eq:qk-implicit-def}
\frac{1}{h} \int_{x_{j-1/2}}^{x_{j+1/2}} q_{k-1}(x) \d x
=
\frac{1}{h} \int_{x_{j-1/2}}^{x_{j+1/2}} u(x) \d x
=:
\bar{u}_j,
\quad
\textrm{for}
\quad j = -\ell, \ldots, r.
\end{align}
Clearly, $q_{k-1}(x) \approx u(x)$.

Next, it is instructive to consider the \textit{primitive} of $u(x)$, which we denote as $U(x)$:
\begin{align}
U(x) = \int_{- \infty}^{x} u(\zeta) \d \zeta,
\quad
\textrm{and}
\quad
U'(x) = u(x).
\end{align}
Note that the lower limit on the integral here is arbitrary, and has no real significance.
It follows that cell averages of $u$ can be written in terms of differences of its primitive evaluated at cell interfaces:
\begin{align} \label{eq:primitive-differences}
\bar{u}_j
:=
\frac{1}{h} \int_{x_{j-1/2}}^{x_{j+1/2}} u(x) \d x
=
\frac{1}{h}
\left(
U(x_{j + 1/2}) - U(x_{j - 1/2})
\right),
\quad
\textrm{for}
\quad
j = -\ell, \ldots, r.
\end{align}

Now let us consider a second polynomial $Q_{k}(x)$, $x \in [x_{-\ell-1/2}, x_{r + 1/2}]$, of degree at most $k$ that is the unique polynomial interpolating the primitive $U(x)$ at the $k+1$ interface points in $S^{\ell}$:
\begin{align} \label{eq:primitive-interp-conds}
Q_k(x_{j + 1/2}) = U(x_{j + 1/2}), 
\quad
\textrm{for}
\quad
j = -\ell-1, \ldots, r.
\end{align}
Since $Q_k(x)$ interpolates the primitive $U(x)$ at cell interfaces, it follows from \eqref{eq:primitive-differences} that differences of it can be used to give cell averages of $u$:
\begin{align} 
\bar{u}_j
&:=
\frac{1}{h}
\left(
U(x_{j + 1/2}) - U(x_{j - 1/2})
\right)
=
\frac{1}{h}
\left(
Q_k(x_{j + 1/2}) - Q_k(x_{j - 1/2})
\right),
\\
\label{eq:primitive-differences-II}
&=
\frac{1}{h} \int_{x_{j-1/2}}^{x_{j+1/2}} Q_k'(x) \d x,
\end{align}
where $Q_k'(x)$ is the derivative of $Q_k$.
Setting \eqref{eq:primitive-differences-II} equal to \eqref{eq:qk-implicit-def} we identify that the original polynomial $q_{k-1}(x) \approx u(x)$ being sought is nothing but the derivative of $Q_k(x)$: 
\begin{align} \label{eq:interp-poly-derivative}
q_{k-1}(x) \equiv Q_k'(x).
\end{align}
With this knowledge in hand, using the tools of polynomial interpolation, it is possible to explicitly construct the polynomial $q_{k-1}$, although we do not provide further details about this here; see \cite{Shu_1998,Shu2009}.

By standard interpolation theory, we know that $Q_k(x)$ approximates $U(x)$ to order $k+1$, i.e., $Q_k(x) = U(x) + {\cal O}(h^{k+1})$ for $x$ in the interpolation stencil. 
In greater detail, we have the following classic error estimate (see, e.g., \cite[Theorem 3.1.1]{Davis_1975}). 
\begin{lemma}[Error estimate for polynomial interpolation]
\label{SMlem:standard-poly-interp-est}
Let $Q_k(x)$ be the unique polynomial of degree at most $k$ that interpolates the function $U(x)$ at the $k+1$ nodes $\{ x_{j+1/2} \}_{j = -\ell-1}^{r}$, $r = k - \ell -1$.
Suppose that $U$ is at least $k+1$ times differentiable over the interpolation interval $x \in [x_{-\ell - 1/2}, x_{r+1/2}]$.
Then,
\begin{align} \label{eq:primitive-interp-error}
U(x) - Q_k(x)
=
\frac{1}{(k+1)!}
\prod \limits_{j = -\ell-1}^{r} \big(x - x_{j + 1/2} \big) U^{(k+1)} ( \varphi(x) ),
\quad 
x \in [x_{-\ell - 1/2}, x_{r+1/2}],
\end{align}
with $\varphi(x) \in (x_{-\ell - 1/2}, x_{r+1/2})$ some (unknown) point in the interior of the stencil, depending on $x$, $U$, and the interpolation nodes.
\end{lemma}
%

\section{Additional numerical results: Direct solves of linearized systems}
\label{SMsec:num-res-additional}

%
\renewcommand{\fd}{./figures/}
\renewcommand{\hs}{2}
\renewcommand{\vs}{1}
\renewcommand\fs{0.335}
\begin{figure}[b!]
\centerline{
\includegraphics[scale=\fs]{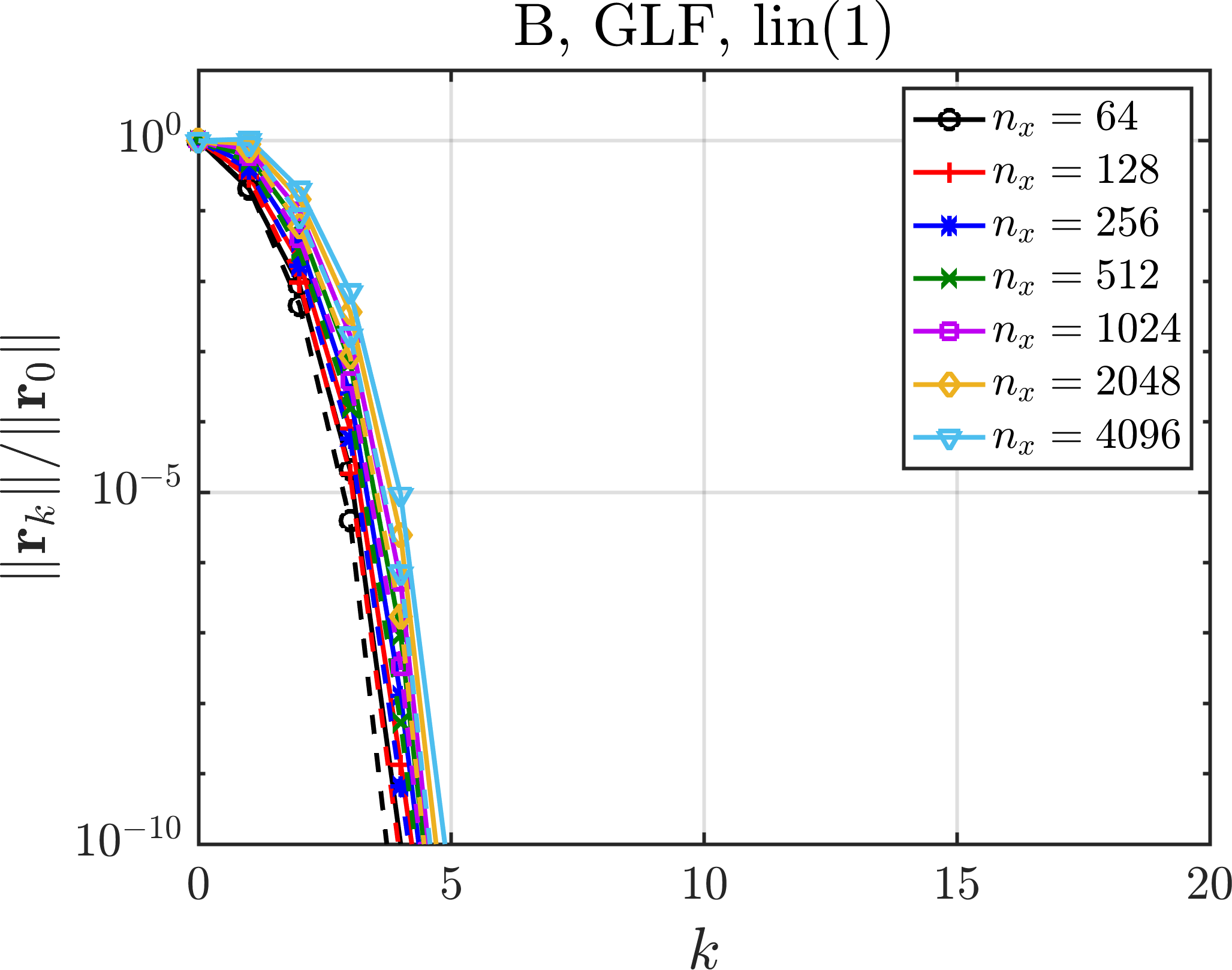}
\hspace{\hs ex}
\includegraphics[scale=\fs]{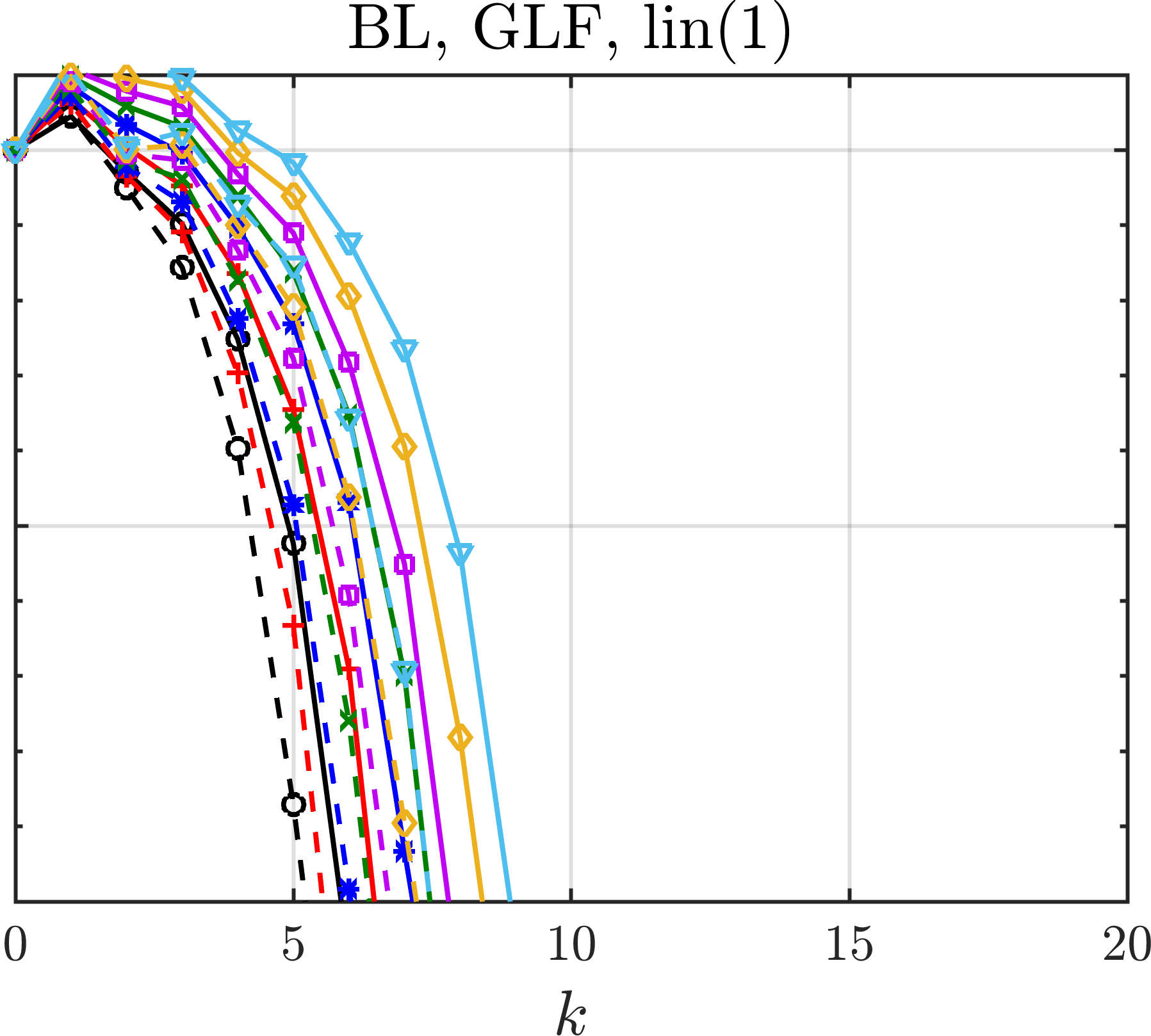}
}
\vspace{\vs ex}
\centerline{
\includegraphics[scale=\fs]{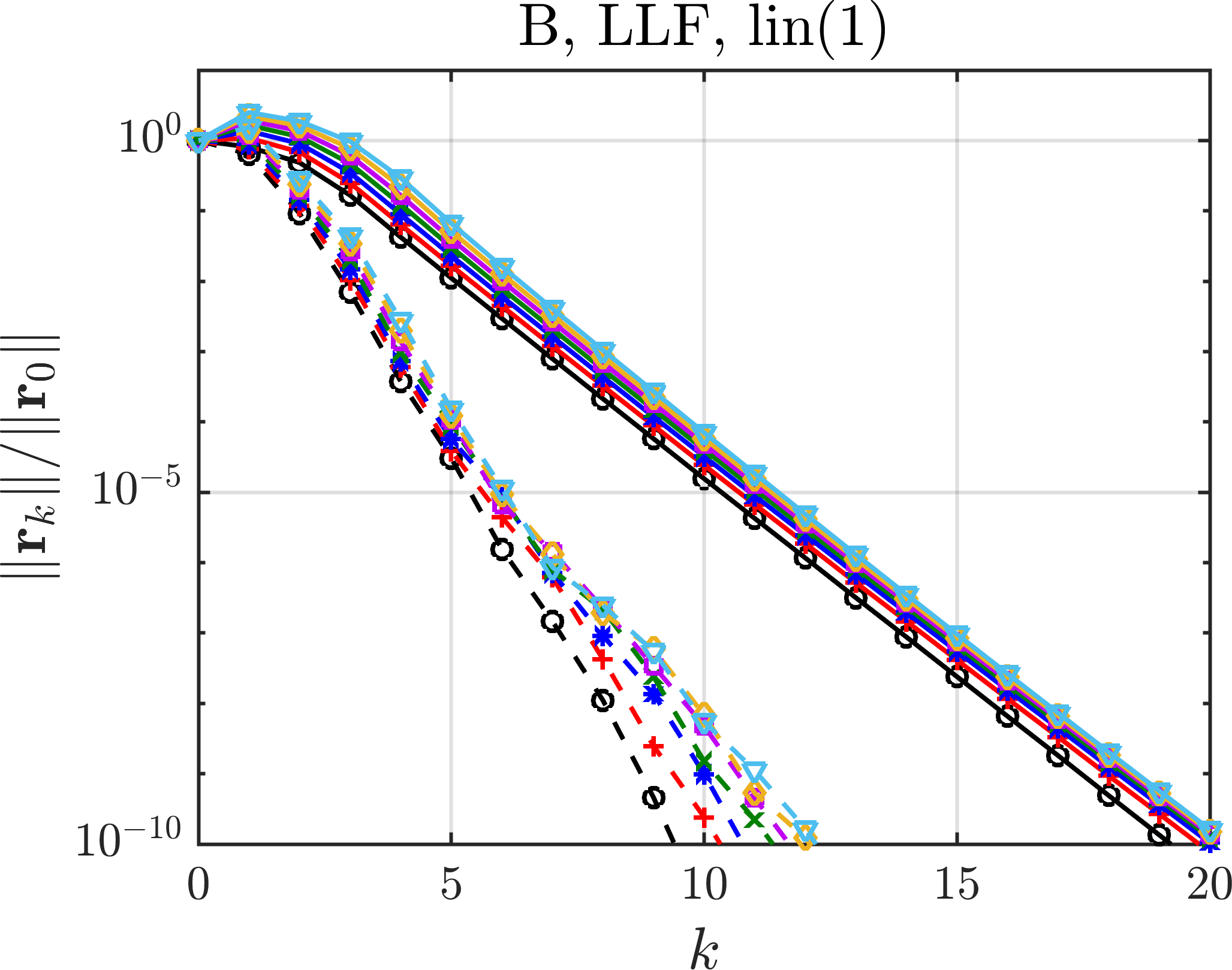}
\hspace{\hs ex}
\includegraphics[scale=\fs]{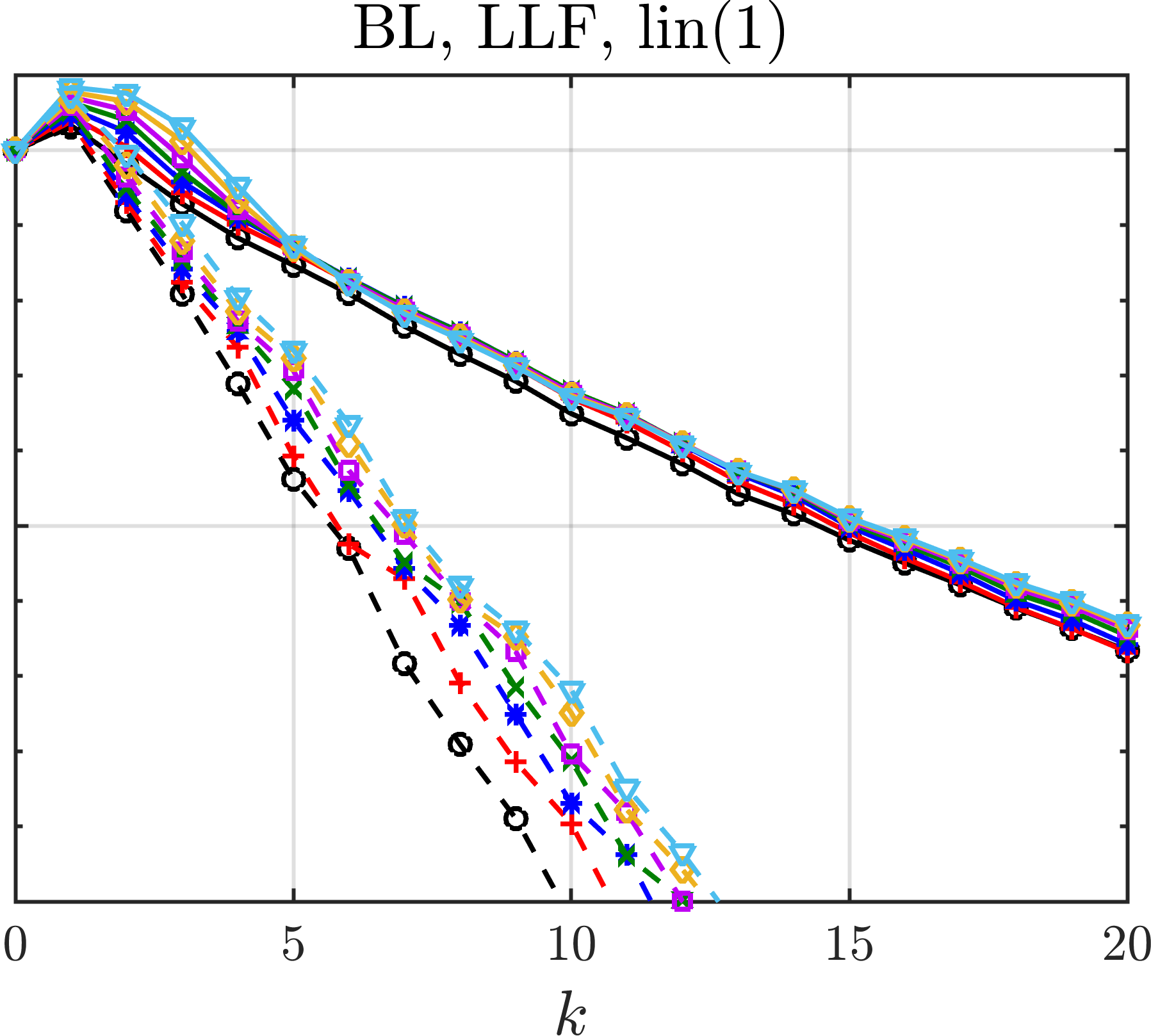}
}
\caption{Residual convergence histories for the nonlinear solver Algorithm \ref{alg:richardson} when the linearized problems are solved directly.
Left column: Burgers \eqref{eq:burgers}. Right column: Buckley--Leverett \eqref{eq:buck}.
The GLF numerical flux is used in the top row, and the LLF numerical flux in the bottom row.
Solid lines correspond to no nonlinear relaxation, and broken lines to nonlinear F-relaxation using a CF-splitting factor $m = 8$.
All discretizations are 1st-order accurate.
\label{SMfig:num-res-exact-1st}
}
\end{figure}

\begin{remark}[Limiting for Buckley--Leverett]
\label{SMrem:BL-limiting}
For the Buckley--Leverett problem, when computing the LLF dissipation \eqref{eq:LFF-local}, if $w \geq 0$, and $w \leq 1$, then the optimization problem is simpler to solve than if $w$ is contained in some arbitrary interval (due to the locations of the various local extrema of $|f'(w)|$).
Furthermore, the PDE itself arises in a context where the physically relevant range for $u$ is $[0,1]$ \cite[Section 16.1.1]{LeVeque_2004}. 
For these reasons, in our numerical implementation of the LLF flux for the Buckley--Leverett problem, reconstructions are always limited such that any $u_{i+1/2}^{\pm} > 1$ are mapped to $u_{i+1/2}^{\pm} = 1$, and any $u_{i+1/2}^{\pm} < 0$ are mapped to $u_{i+1/2}^{\pm} = 0$.
While the PDE itself obeys a maximum principle, the higher-order numerical discretizations described in Section \ref{sec:discretization} do not (WENO reconstructions can produce solutions that violate global extrema, even if only by an amount of ${\cal O}(h)$), and the iterative solver described in Section \ref{sec:nonlin-scheme} can produce iterates that violate this maximum principle. 
%
\end{remark}

\renewcommand{\fd}{./figures/}
\renewcommand{\hs}{2}
\renewcommand{\vs}{1}
\renewcommand\fs{0.335}
\begin{figure}[b!]
\centerline{
\includegraphics[scale=\fs]{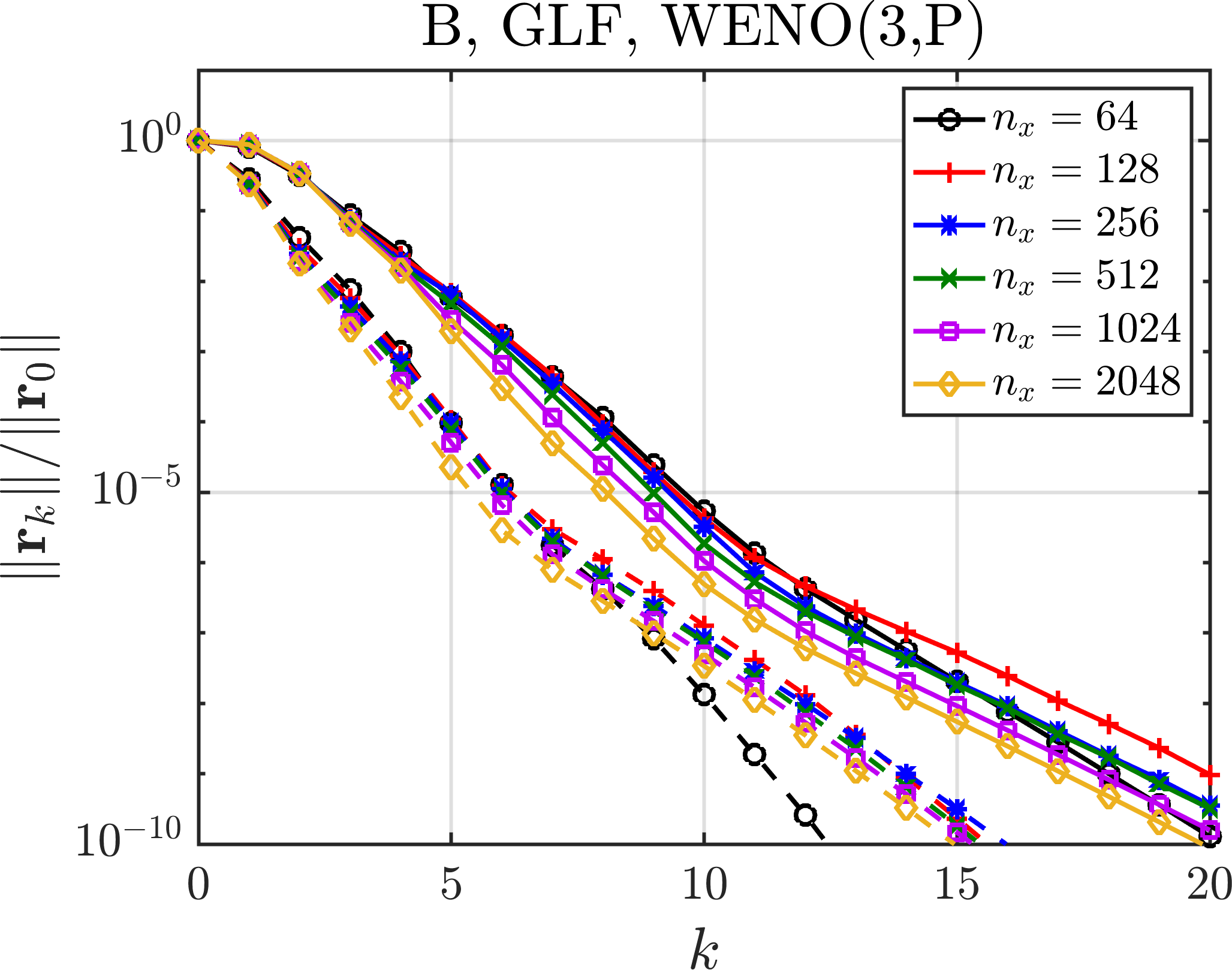}
\hspace{\hs ex}
\includegraphics[scale=\fs]{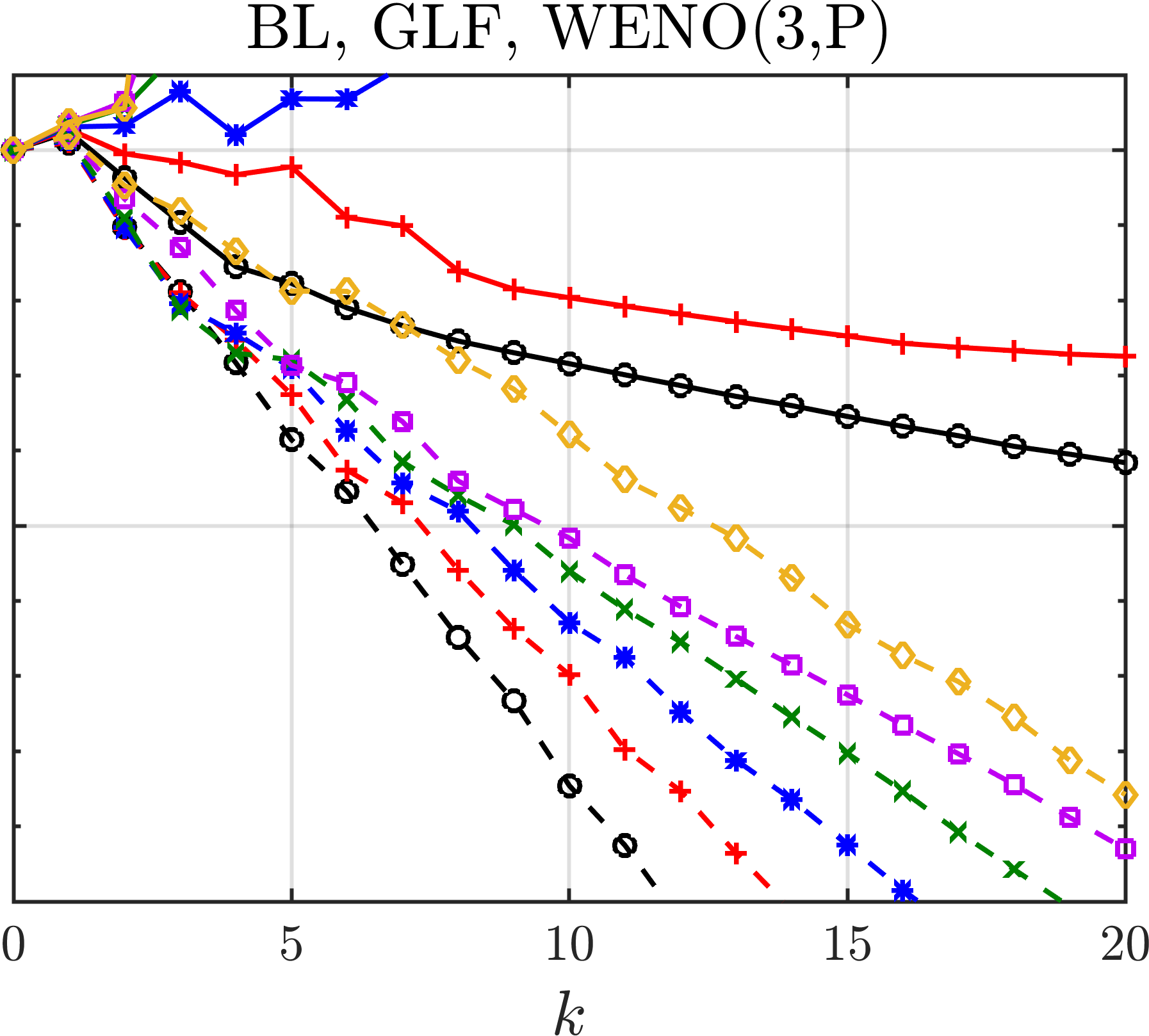}
}
\vspace{\vs ex}
\centerline{
\includegraphics[scale=\fs]{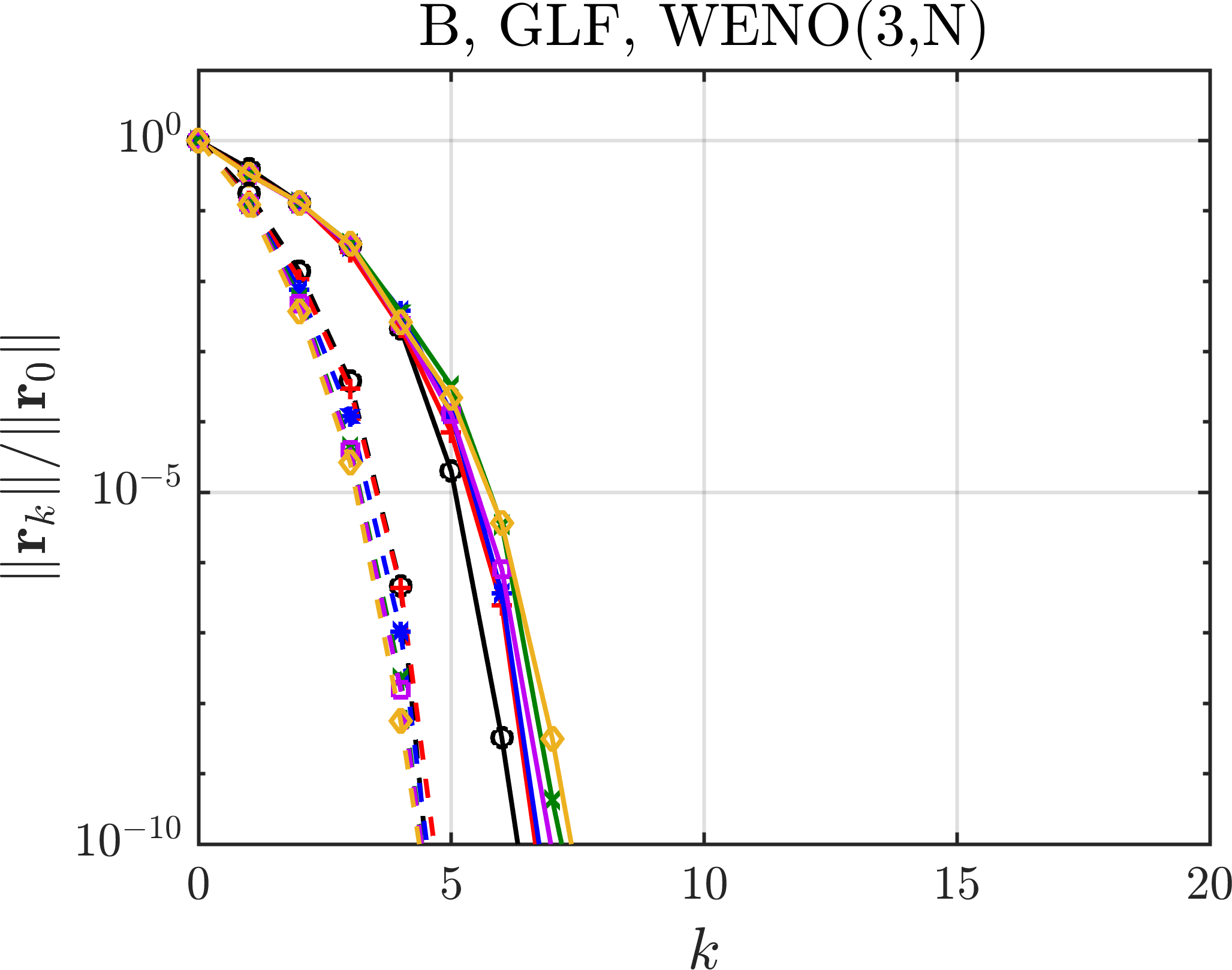}
\hspace{\hs ex}
\includegraphics[scale=\fs]{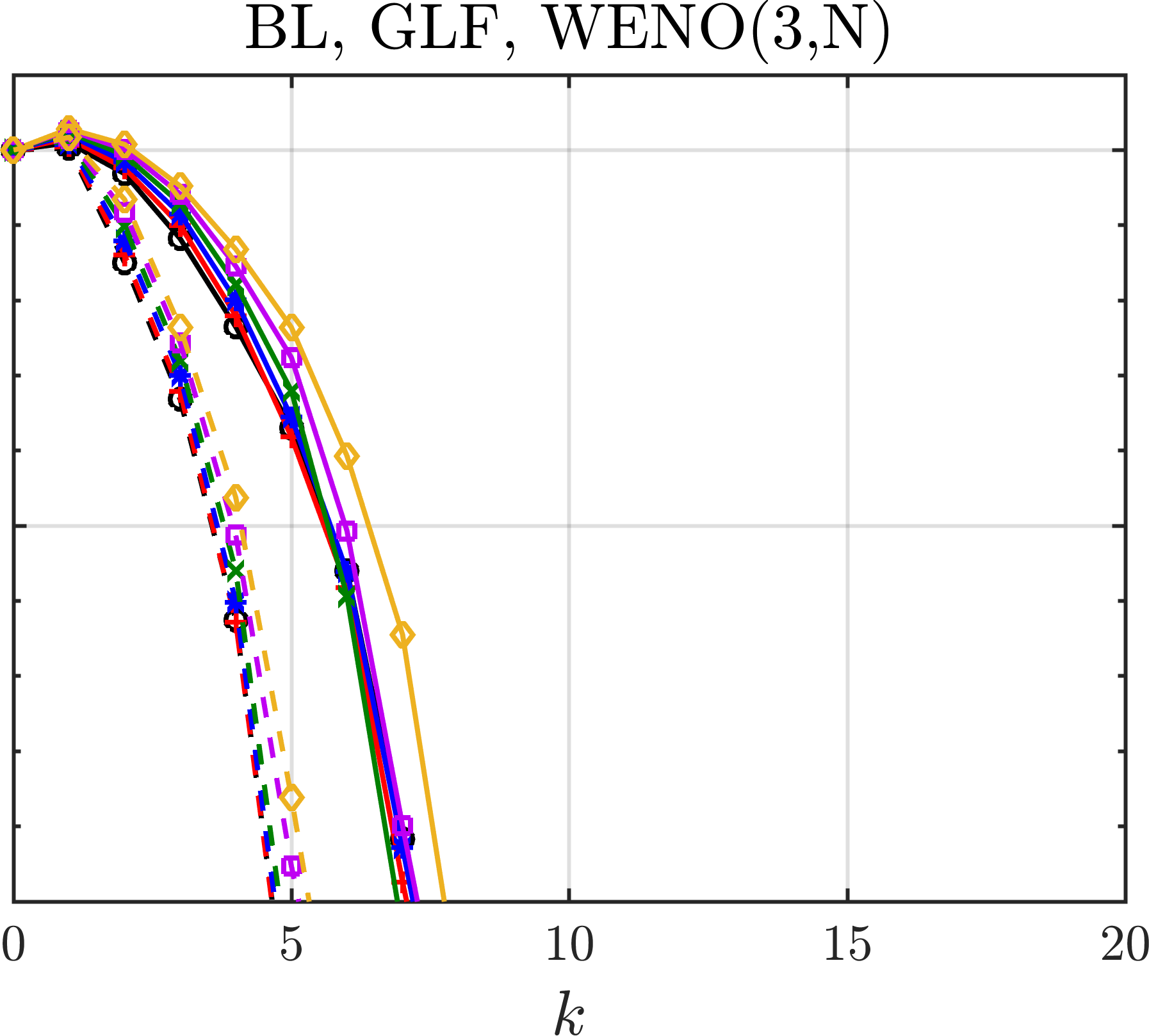}
}
\caption{Residual convergence histories for the nonlinear solver Algorithm \ref{alg:richardson} when the linearized problems are solved directly.
All discretizations are 3rd-order accurate and use a GLF numerical flux; see \cref{SMfig:num-res-exact-3rd-LLF} for LLF fluxes.
Left column: Burgers \eqref{eq:burgers}. Right column: Buckley--Leverett \eqref{eq:buck}.
Broken lines correspond to nonlinear F-relaxation with CF-splitting factor $m = 8$, and solid lines to no nonlinear relaxation.
Top row: Reconstructions are linearized with \eqref{eq:weighted-reconstruct-grad-zero}, corresponding to Picard linearization of the WENO weights.
Bottom Row: Reconstructions are linearized with \eqref{eq:weighted-reconstruct-grad-FD}, corresponding to (approximate) Newton linearization of the WENO weights.
\label{SMfig:num-res-exact-3rd-GLF}
}
\end{figure}

This section presents results for Algorithm \ref{alg:richardson} in the special case that  the linearized problems are solved directly with sequential time-stepping.
While our end goal is temporal parallelism, it is first important to understand the convergence behaviour of the nonlinear solver in the best-case setting of direct linear solves. 
Regardless of the linearized solver, the iterative solution of discretized nonlinear hyperbolic PDEs is itself a rather complicated subject, with convergence of Newton-type methods often being an issue \cite{Gottlieb_Mullen_2003,Coffey_etal_2003}.

We consider plots of residual convergence histories for the nonlinear solver Algorithm \ref{alg:richardson}.
In all cases the two-norm of the space-time residual is shown relative to its initial value.
The solver is iterated until: 20 iterations are performed or the relative residual falls below $10^{-10}$, or it increases past $10^{3}$.

\Cref{SMfig:num-res-exact-1st} considers 1st-order discretizations. The convergence rate is superlinear for problems using the GLF flux (top row), consistent with the solver coinciding with Newton's method for these problems when nonlinear relaxation (Line \ref{ln:nonlin-relax} in Algorithm \ref{alg:richardson}) is not used.
In contrast, convergence rates for problems using the LLF flux (bottom row) are linear, consistent with the flux being non-smooth.
Adding nonlinear relaxation does not impact significantly the convergence speed for the GLF examples because the solver already converges rapidly without it. However, relaxation dramatically improves convergence rates for the LLF examples, particularly for Buckley--Leverett. 
%

%
\renewcommand{\fd}{./figures/}
\renewcommand{\hs}{2}
\renewcommand{\vs}{1}
\renewcommand\fs{0.325}
\begin{figure}[t!]
\centerline{
\includegraphics[scale=\fs]{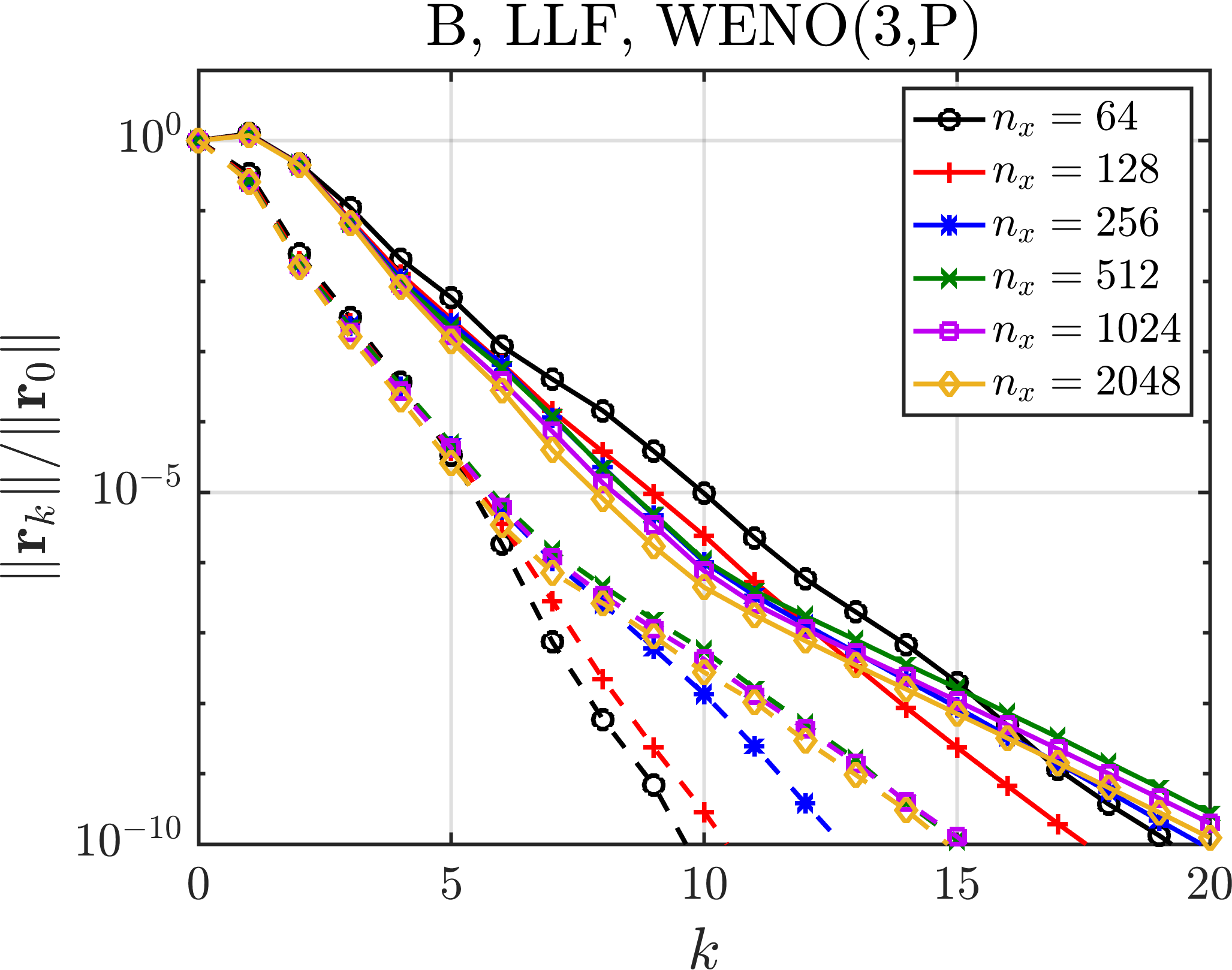}
\hspace{\hs ex}
\includegraphics[scale=\fs]{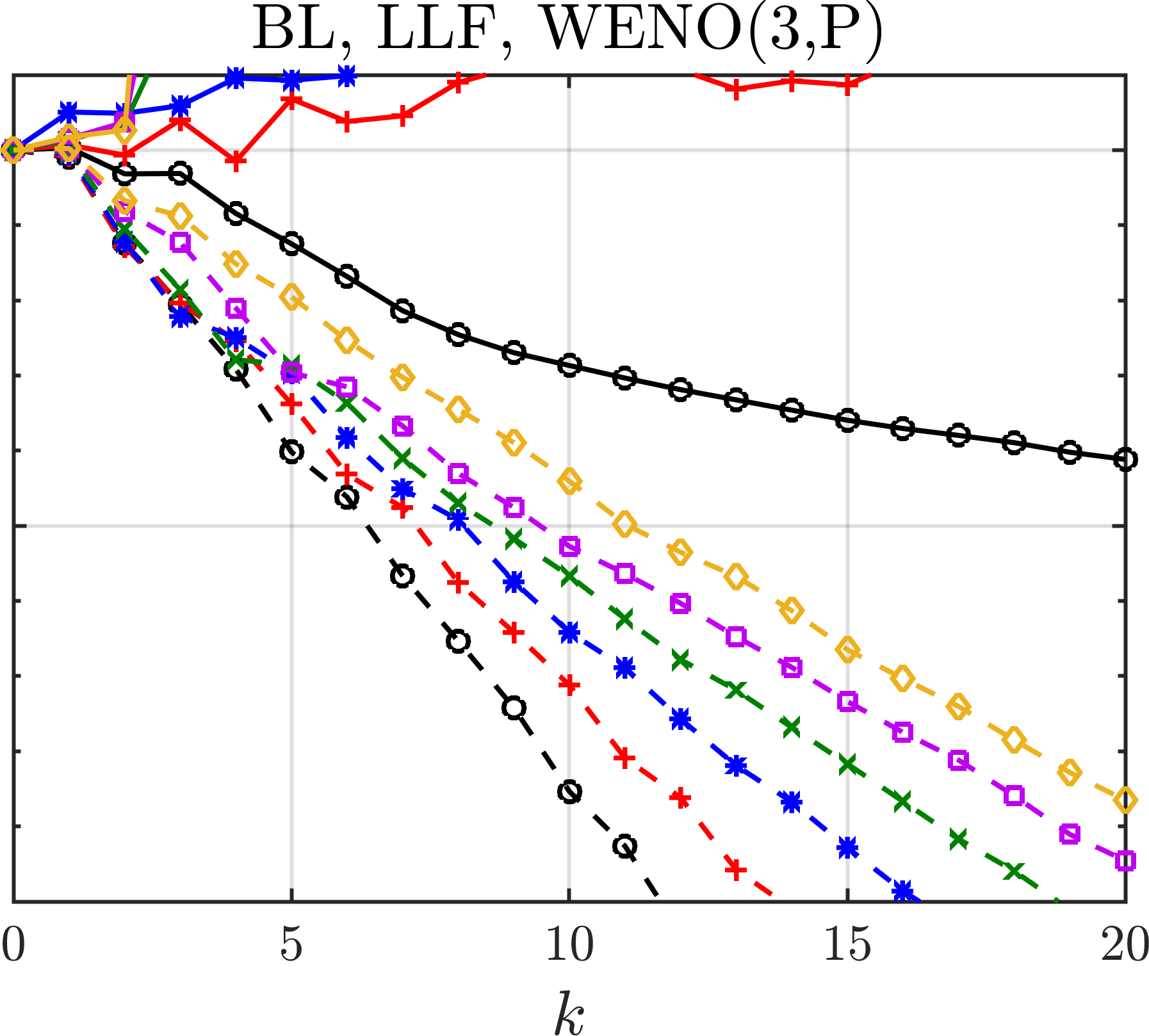}
}
\vspace{\vs ex}
\centerline{
\includegraphics[scale=\fs]{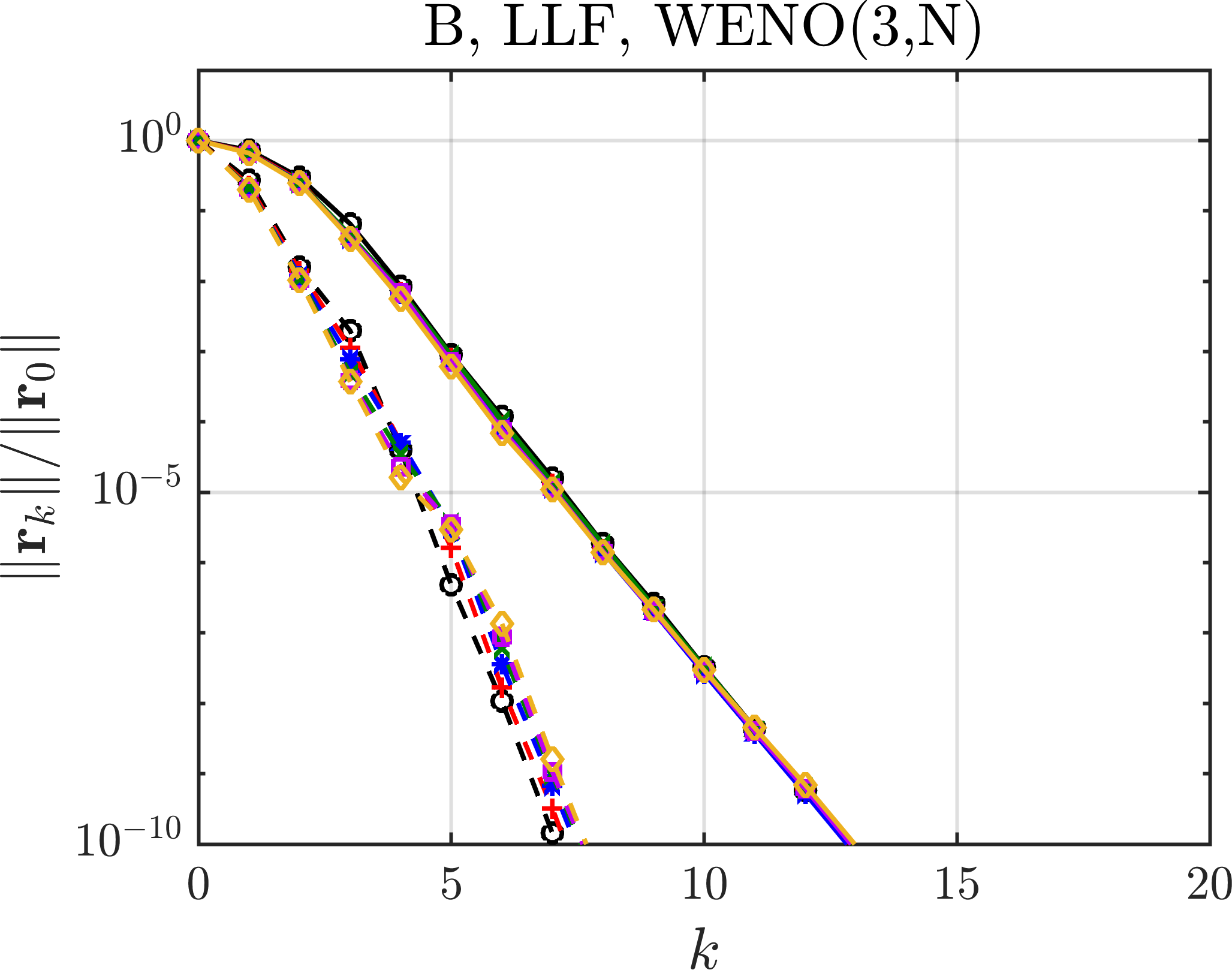}
\hspace{\hs ex}
\includegraphics[scale=\fs]{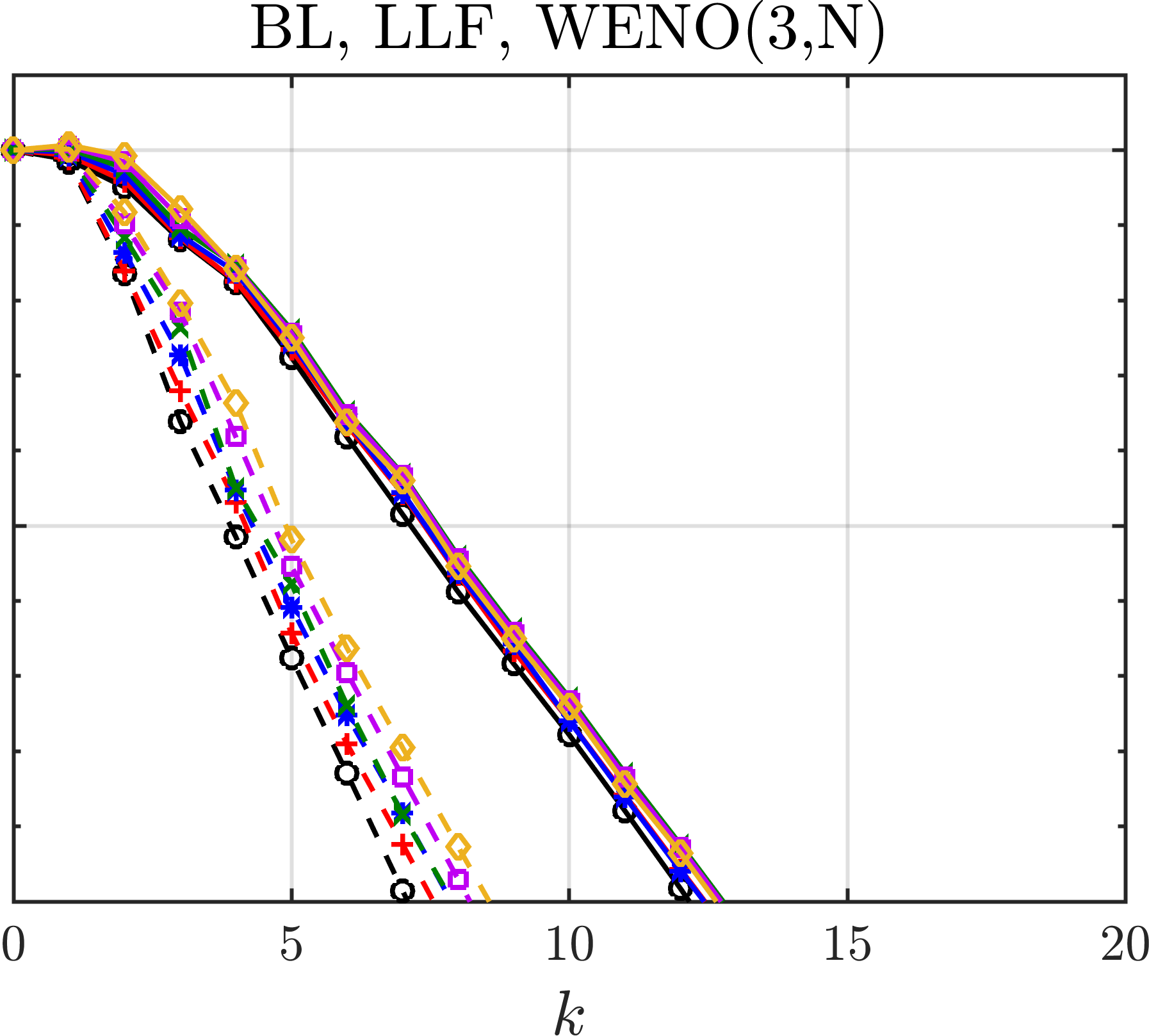}
}
\caption{Convergence histories for the same set up as in \cref{SMfig:num-res-exact-3rd-GLF} except that LLF numerical fluxes are used rather than GLF numerical fluxes.
\label{SMfig:num-res-exact-3rd-LLF}
}
\end{figure}

\Cref{SMfig:num-res-exact-3rd-LLF,SMfig:num-res-exact-3rd-GLF} consider 3rd-order discretizations.
Recall that for the 3rd-order discretizations there are solution-dependent WENO weights, with both approximate Newton \eqref{eq:weighted-reconstruct-grad-FD} and Picard linearization \eqref{eq:weighted-reconstruct-grad-zero} possible.
In all cases, the Newton linearization yields faster and more robust convergence than the Picard linearization; however, recall that the Newton linearization is more expensive than the Picard linearization.
In all cases, adding nonlinear F-relaxation improves the convergence rate, although the improvement is less pronounced than in the bottom row of \cref{SMfig:num-res-exact-1st} for 1st-order LLF discretizations.
One exception here is the Buckley--Leverett problem when Picard linearization of the WENO weights is used, where, without nonlinear relaxation the solver diverges and with it the solver converges quickly.
Note that it would be possible to avoid the divergence for this problem by employing a line search, but we want to avoid this due to the relatively high costs of residual evaluations.

Overall, we find that the nonlinear solver converges quickly, including with mesh-independent iteration counts in many cases, even for non-differentiable discretizations. 
Note that we have also tested the solver on Burgers problems in which a shock forms from smooth initial conditions (not shown here), and the solver then appears to, in general, converge faster than for the Burgers problem considered in Figure \ref{fig:test-prob}.

\section{Additional numerical tests: Different initial condition}
\label{SMsec:num-res-additional-u0}

\newcommand{\fdone}{./figures/}
\newcommand{\fdtwo}{./figures/}
\renewcommand{\vs}{2}
\renewcommand{\hs}{1}
\renewcommand\fs{0.35}

\begin{figure}[b!]
\centerline{
\includegraphics[scale=\fs]{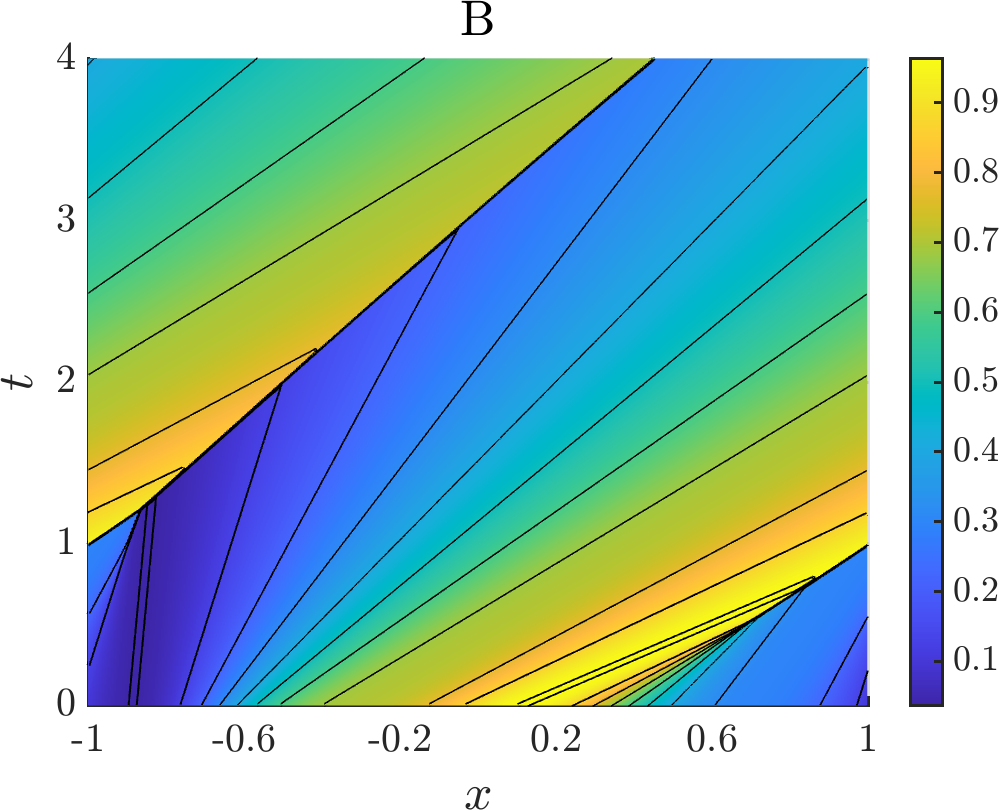}
\hspace{\hs ex}
\includegraphics[scale=\fs]{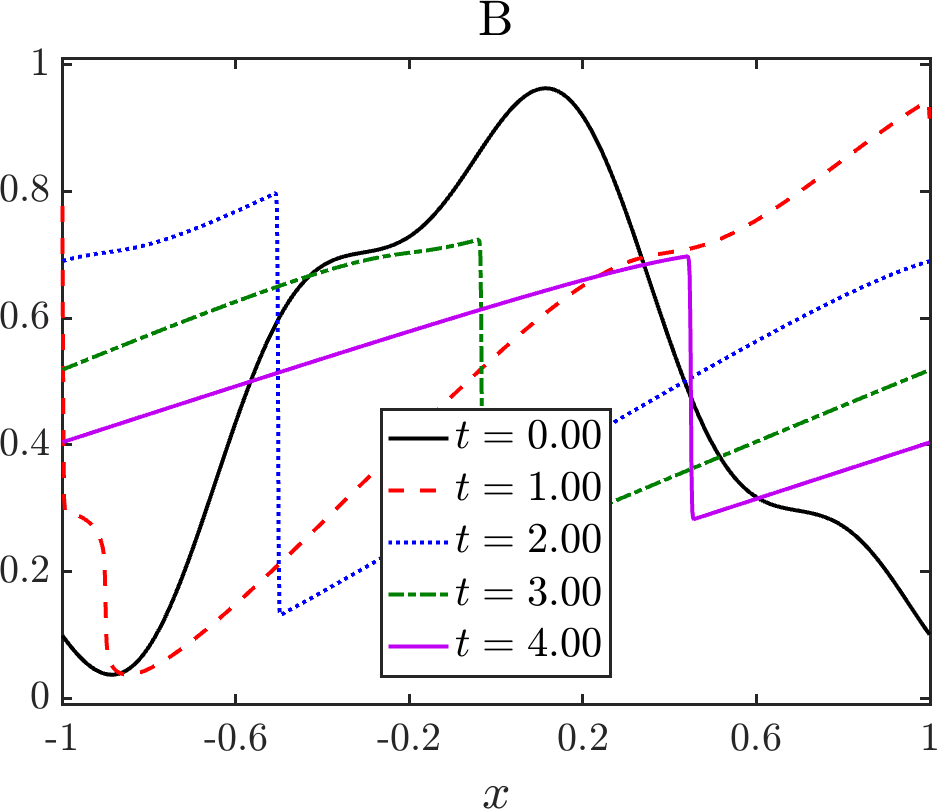}
}
\vspace{\vs ex}
\centerline{
\includegraphics[scale=\fs]{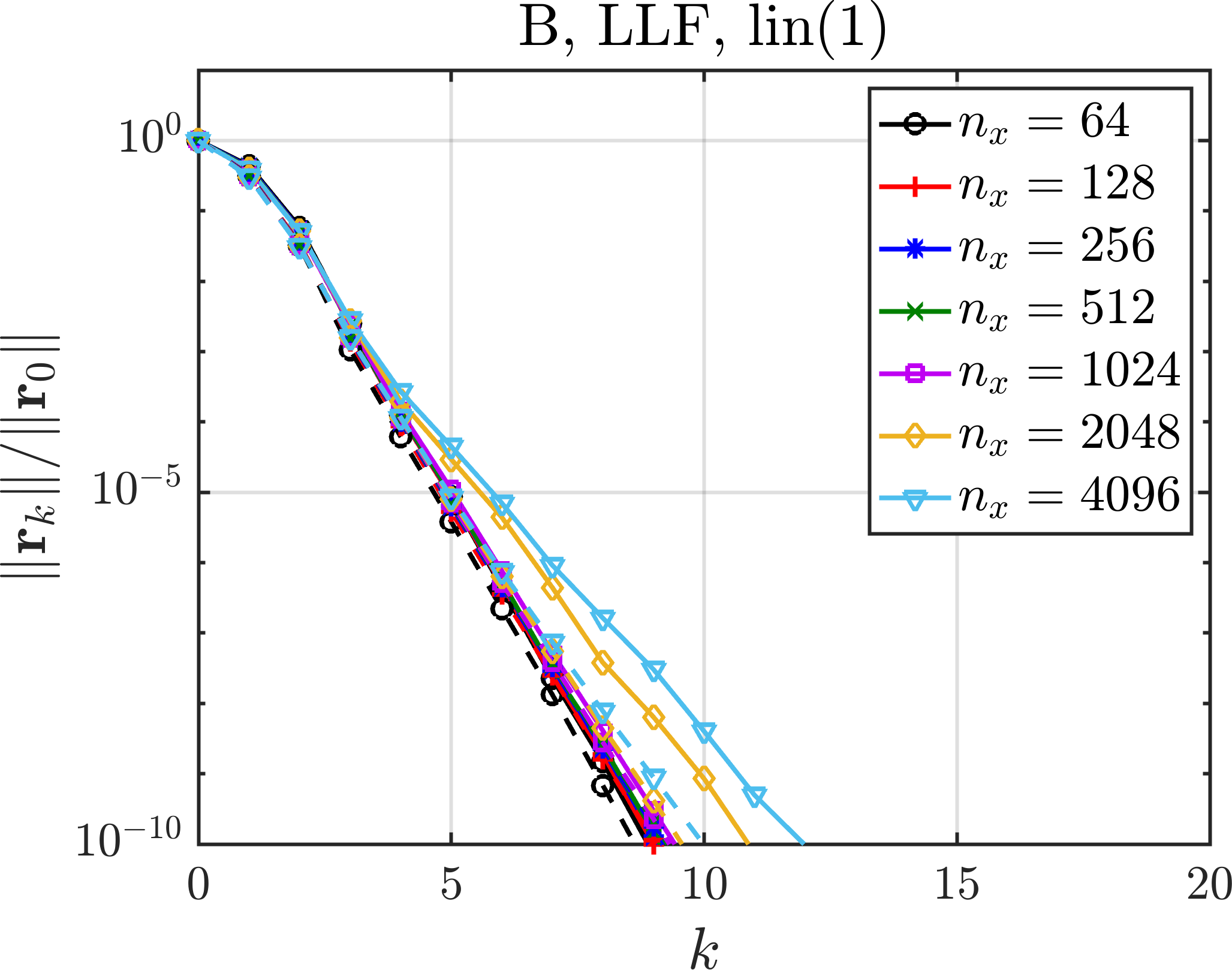}
\hspace{\hs ex}
\includegraphics[scale=\fs]{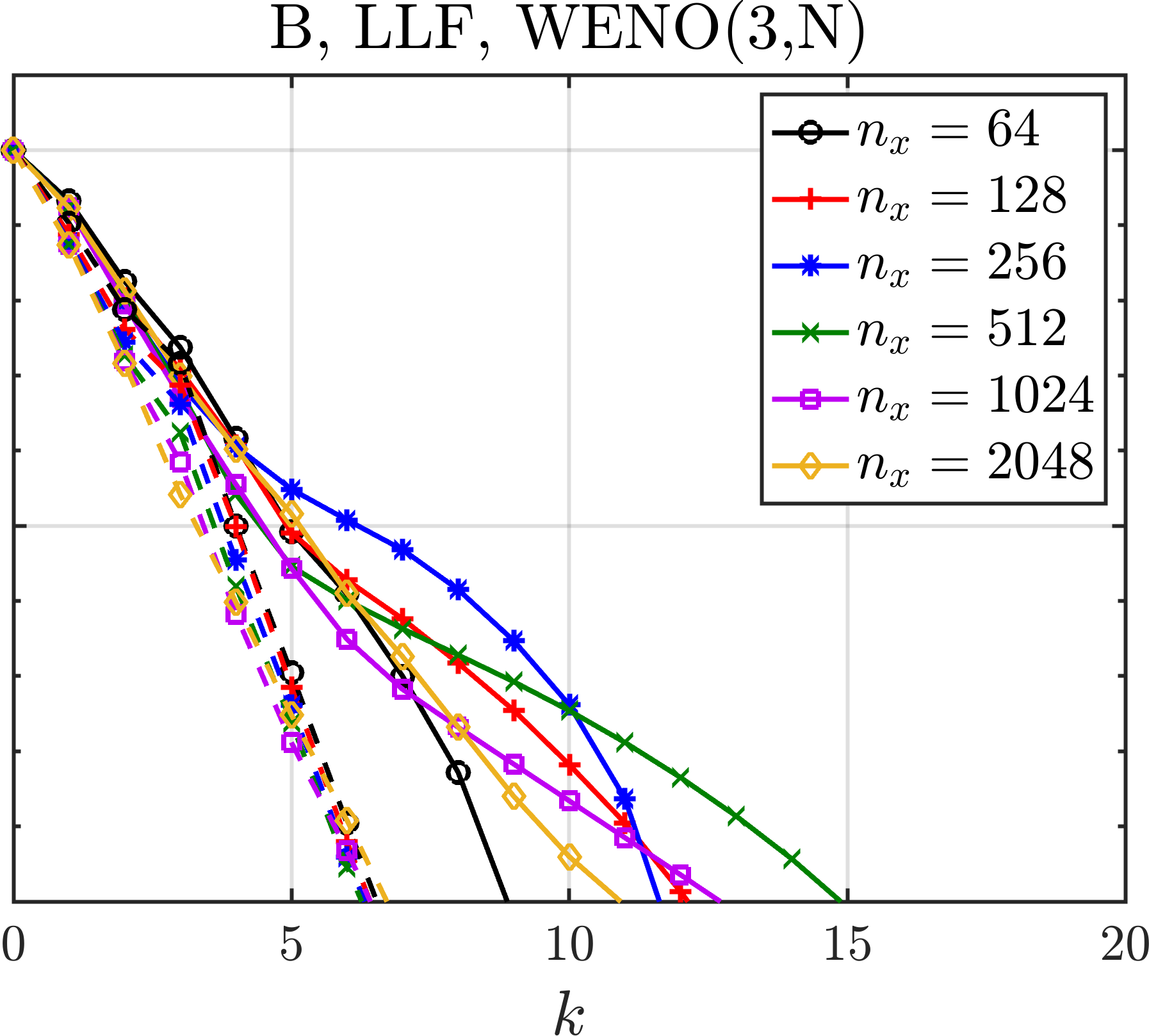}
}
\caption{
Additional numerical tests for Burgers equation \eqref{eq:burgers} using a smooth initial condition. 
Top row: Space-time contours, with the black lines being 20 contour levels evenly spaced between 0 and 1 (left); cross-sections of the solution at the times indicated in the legend (right).
Bottom row: Residual convergence histories for the nonlinear solver Algorithm \ref{alg:richardson} applied to a 1st-order discretization (left), and a 3rd-order discretization (right).
Solid lines correspond to linear systems being approximately solved by one two-level MGRIT iteration, and broken lines correspond to linear systems being solved exactly. 
\label{SMfig:num-res-extra-u0-B}
}
\end{figure}

In this section we consider additional numerical tests to complement those from Section \ref{sec:model-probs}, considering a smooth initial condition in contrast to the square wave that was used there.
Specifically, we use the initial condition ${u_0(x) = 0.5 + 0.4 \cos(\pi x) + 0.1 \sin (3 \pi x)}$.
The numerical setup for these tests is the same as described in Section \ref{sec:model-probs}.
We again consider the Burgers equation \eqref{eq:burgers} and the Buckley--Leverett equation \eqref{eq:buck}.
We solve \eqref{eq:burgers} on the time domain $t \in [0,4]$, and on a sequence of space-time meshes with resolutions $(n_x, n_t) \in \{ (64, 78), (128, 309), (256, 617), (512, 1233), (1024, 2465), (2048, 4929), \\ (4096, 9857) \}$; solution plots are shown in the top row of Figure \ref{SMfig:num-res-extra-u0-B}.
We solve \eqref{eq:buck} on the time domain $t \in [0,2]$, and on a sequence of space-time meshes with resolutions $(n_x,n_t) = \{ (64, 188),
(128,  375),
(256, 748),
(512, 1494),
(1024, 2986),
(2048, 5971), \\
(4096, 11941) \}$; solution plots are shown in the top row of Figure \ref{SMfig:num-res-extra-u0-BL}.

For each problem we consider both 1st- and 3rd-order discretizations, but consider only the LLF flux, and for the 3rd-order WENO discretization we consider only the Newton linearization (c.f., Figures \ref{fig:num-res-inexact-1st} and \ref{fig:num-res-inexact-3rd-LLF} where the GLF flux was considered and also Picard linearization of the WENO discretization).

Residual convergence histories for \eqref{eq:burgers} are shown in the bottom row of Figure \ref{SMfig:num-res-extra-u0-B}, and those for \eqref{eq:buck} are shown in the bottom row of Figure \ref{SMfig:num-res-extra-u0-BL}.
Considering the convergence histories, we see largely the same trends as occurred for the square wave initial condition considered in Section \ref{sec:model-probs}; see Figures \ref{fig:num-res-inexact-1st} and \ref{fig:num-res-inexact-3rd-LLF}.
That is, when direct solves are used for the linearized space-time problems,  the solver converges in a small number of iterations with mesh-independent convergence rates.
When the linearized space-time problems are solved inexactly using a single two-level MGRIT iteration, there seems to be some slight deterioration in convergence speed in most cases; however, in all cases, the nonlinear iteration counts appear to be bounded as the space-time mesh is refined.
Moreover, iteration counts for this initial condition are qualitatively similar to those for the square wave considered in the main paper.

\begin{figure}[t!]
\centerline{
\includegraphics[scale=\fs]{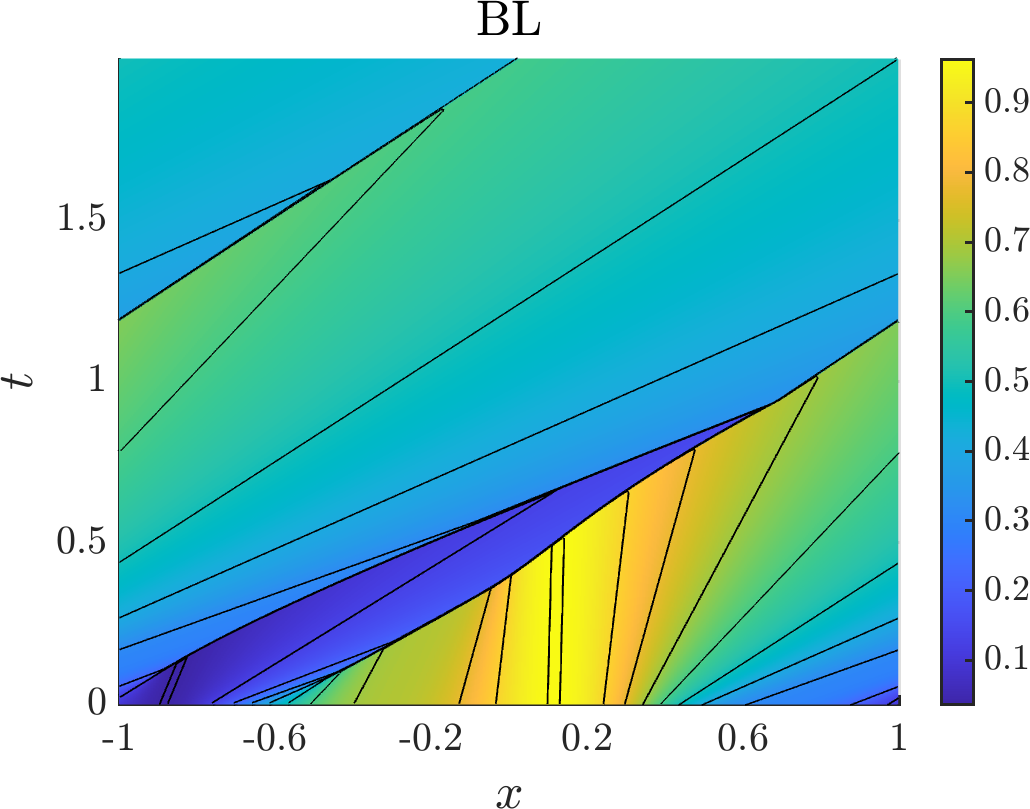}
\hspace{\hs ex}
\includegraphics[scale=0.335]{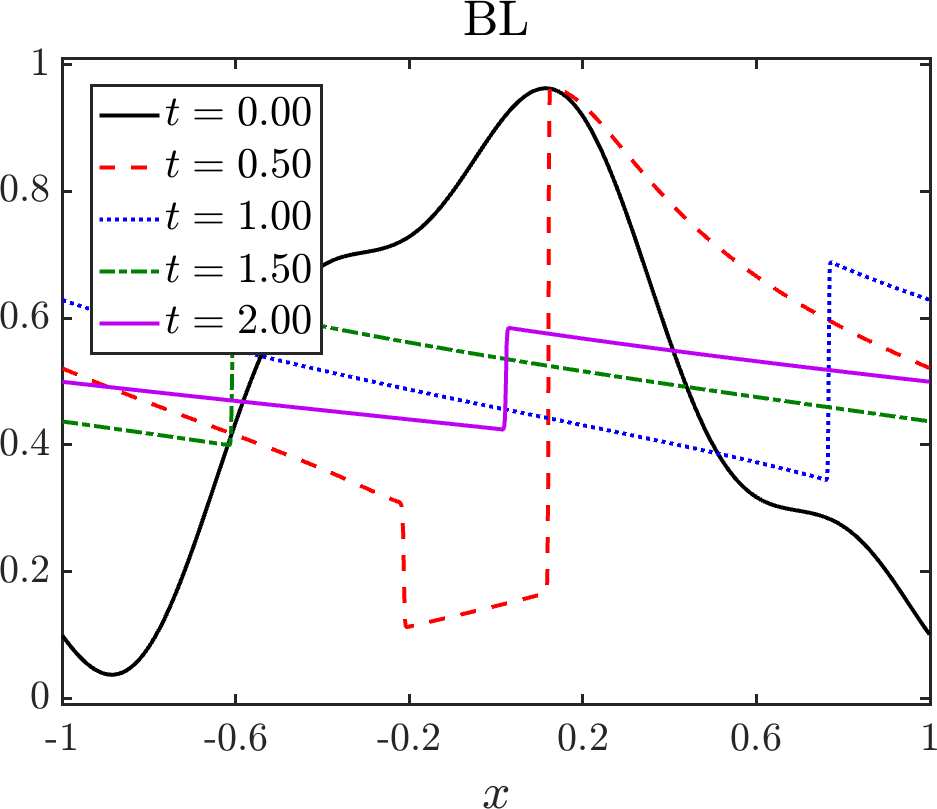}
}
\vspace{\vs ex}
\centerline{
\includegraphics[scale=\fs]{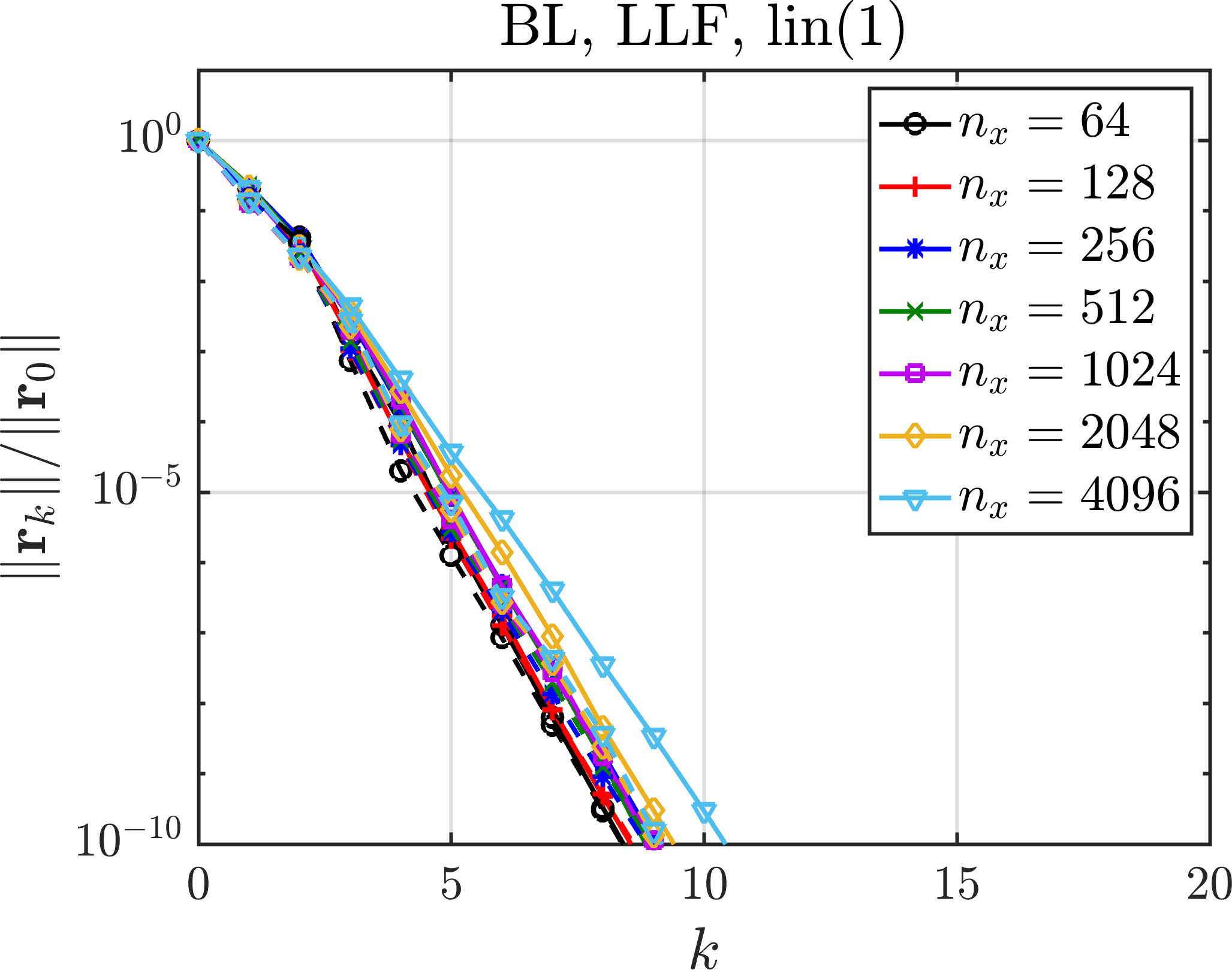}
\hspace{\hs ex}
\includegraphics[scale=\fs]{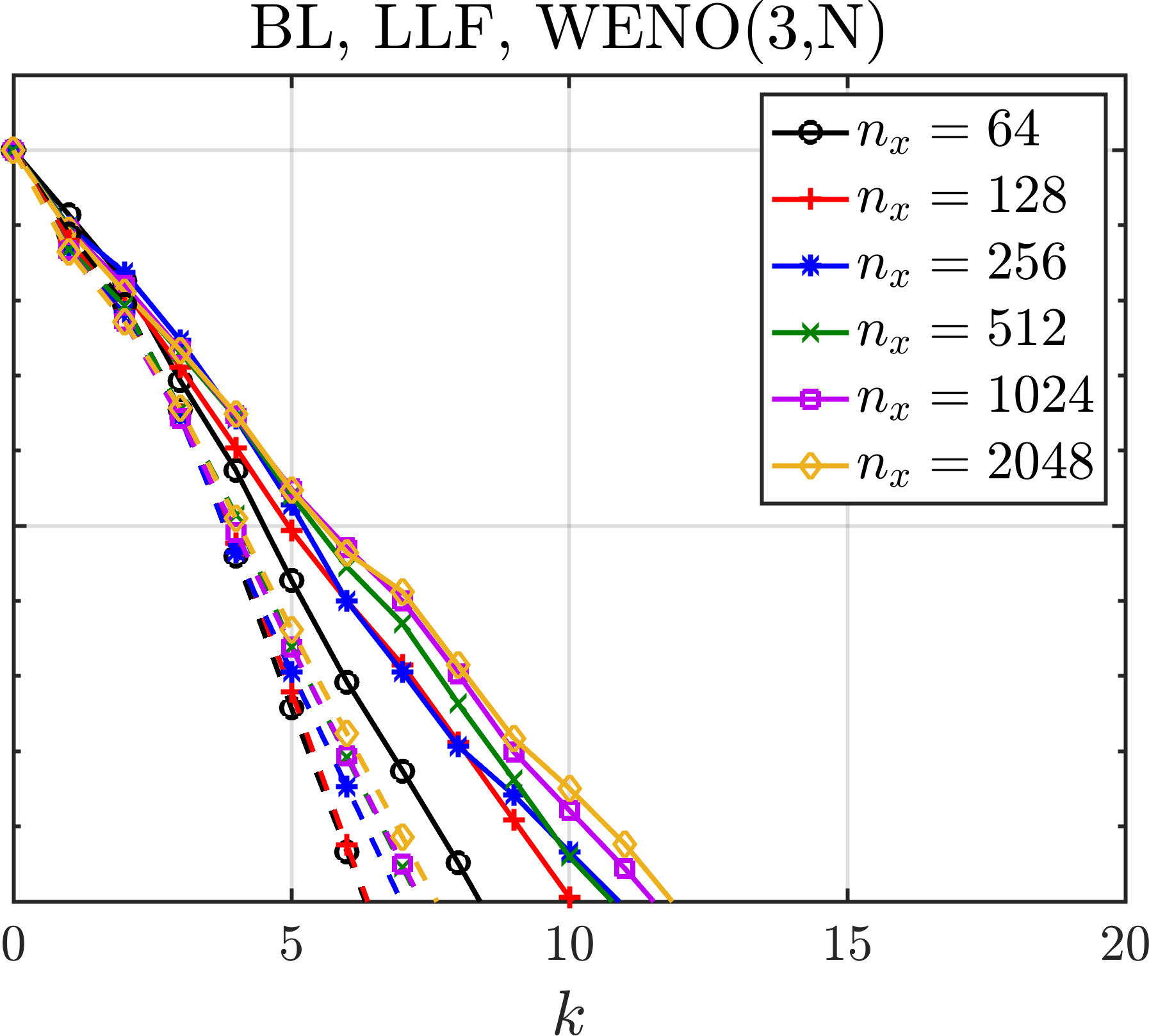}
}
\caption{
Additional numerical tests for Buckley--Leverett equation \eqref{eq:buck} using a smooth initial condition. 
Top row: Space-time contours, with the black lines being 20 contour levels evenly spaced between 0 and 1 (left); cross-sections of the solution at the times indicated in the legend (right).
Bottom row: Residual convergence histories for the nonlinear solver Algorithm \ref{alg:richardson} applied to a 1st-order discretization (left), and a 3rd-order discretization (right).
Solid lines correspond to linear systems being approximately solved by one two-level MGRIT iteration, and broken lines correspond to linear systems being solved exactly. 
\label{SMfig:num-res-extra-u0-BL}
}
\end{figure}
%

\section{Additional numerical tests: Multilevel MGRIT}
\label{SMsec:num-res-additional-multilevel}

Since there are many complicated details in the main paper, when employing MGRIT to approximately solve the linearized systems we used only a two-level method for simplicity's sake.
Here we discuss the extension of this MGRIT method to the multilevel setting, relying on recursively applying the two-level method, analogous to our multilevel extensions of modified semi-Lagrangian schemes in prior works \cite{DeSterck_etal_2023_SL,DeSterck_etal_2023_MOL,KrzysikThesis2021}.

To keep things as simple as possible, we consider multilevel extensions for first-order discretizations only, i.e., $p = k = 1$. The methodology can be extended to the third-order discretizations considered in the main paper, but it is more complicated because the truncation-error correction uses weighted stencils, as in the $k = 2$ case shown in \eqref{eq:T-ideal}.
Some of the notation from the main paper is now slightly changed since we need to account explicitly for the level index $\ell \in \mathbb{N}_0$, with $\ell = 0$ denoting the finest level.
Let us write $\Phi^{(n)}_{\ell}$ as the MGRIT time-stepping operator that steps from $t_n \to t_n + m^{\ell}  \delta t$, since on level $\ell$ the time-step size is $m^{\ell} \delta t$.
Let ${\cal S}^{(n, m^{\ell} \delta t)}$ be a first-order semi-Lagrangian discretization taking one step of size $m^{\ell} \delta t$ to step from $t_n \to t_n + m^{\ell} \delta t$.
Moreover, let $\bm{\varepsilon}_{\ell}^n \in \mathbb{R}^{n_x}$ be the vector of mesh-normalized departure points that the semi-Lagrangian method ${\cal S}^{(n, m^{\ell} \delta t)}$ uses to advance the solution from $t_n$ to $t_n + m^{\ell} \delta t$.

The level $\ell = 1$ time-stepping operator from the main paper is given by \eqref{eq:Psi}, and for first-order discretizations it simplifies as
\begin{align} \label{eq:Phi_1}
\Phi_{1}^{(n)}
&=
\Big(
I 
+ 
{\cal T}_{\mathrm{ideal}, 1}^{n}
-
{\cal T}_{\mathrm{direct}, 1}^{n}
\Big)^{-1}
{\cal S}^{(n, m\delta t)}
\\
&=
\Big(
I 
+ 
{\cal D}_1 \diag( \bm{\delta}^{n}_{1} ) {\cal D}_1^\top
\Big)^{-1}
{\cal S}^{(n, m\delta t)}, 
\quad
\bm{\delta}^{n}_{1}
=
h^2 g_2( \bm{\varepsilon}^n_1 )
+
\sum_{j = 0}^{m - 1} \bm{\beta}^{n+j},
\end{align}
in which the coefficients $\bm{\delta}^{n}_{1}$ are obtained from plugging in ${\cal T}_{\mathrm{ideal}, 1}^{n}$ from \eqref{eq:T-ideal}, and ${\cal T}_{\mathrm{direct}, 1}^{n}$ from \eqref{eq:T-direct}.
That is, $\bm{\delta}^{n}_{1}$ accounts approximately for the difference in the leading-order truncation error terms between the ``ideal'' and ``direct'' methods used to cover the interval $[t_n, t_n + m \delta t]$, with the former being $m$ linearized method-of-lines steps of size $\delta t$, and the latter being one semi-Lagrangian step of size $m \delta t$.

In the forthcoming derivation, on coarser levels $\ell > 1$, we approximate $m$ smaller semi-Lagrangian steps with a single large semi-Lagrangian step using the truncation-error-based approximation \eqref{eq:Psi}.
That is, on levels $\ell > 1$ we invoke the approximation (where the product is assumed to expand leftwards with increasing $k$)
\begin{align} 
\prod_{k = 0}^{m-1}
{\cal S}^{(n + k, m^{
\ell-1} \delta t)}
\approx
\Big(
I 
+ 
{\cal T}_{\mathrm{ideal}, \ell}^{n}
-
{\cal T}_{\mathrm{direct}, \ell}^{n}
\Big)^{-1}
{\cal S}^{(n, m^{\ell}\delta t)},
\end{align}
where now the ``ideal'' method on the interval $[t_n, t_n + m^{\ell} \delta t]$ is $m$ semi-Lagrangian steps of size $m^{\ell-1} \delta t$, and the ``direct'' method is one semi-Lagrangian step of size $m^{\ell} \delta t$.
The right hand side of \eqref{eq:T-direct} gives the general form of our approximation of the semi-Lagrangian truncation error.
Plugging this in gives, for $\ell > 1$,
\begin{align}
{\cal T}_{\mathrm{ideal}, \ell}^{n}
-
{\cal T}_{\mathrm{direct}, \ell}^{n}
&=
\sum_{k = 0}^{m-1}
-h^2
{\cal D}_1
\diag 
\big( 
g_{2}
(
\bm{\varepsilon}_{\ell-1}^{n+k}
)
\big) 
{\cal D}_1^\top
+
h^2
{\cal D}_1
\diag 
\big( 
g_{2}
(
\bm{\varepsilon}_{\ell}^{n}
)
\big)
{\cal D}_1^\top,
\\
&=
{\cal D}_1
\diag (
\bm{\Delta}_{\ell}^n
)
{\cal D}_1^\top,
\quad
\bm{\Delta}_{\ell}^n := h^2 \Big[ g_{2}
(
\bm{\varepsilon}_{\ell}^{n}
)
-
\sum_{k = 0}^{m-1}
\big( 
g_{2}
(
\bm{\varepsilon}_{\ell-1}^{n+k}
)
\big)
\Big].
\end{align}
That is, $\bm{\Delta}_{\ell}^n$ accounts approximately for the difference in the leading-order truncation error terms between using $m$ semi-Lagrangian steps of size $m^{\ell - 1} \delta t$ to cover the interval $[t_n, t_n + m^{\ell} \delta t]$ vs. one semi-Lagrangian step of size $m^{\ell} \delta t$.

To generalize the level $\ell = 1$ operator \eqref{eq:Phi_1} to the multilevel setting, we now consider the ideal operator on level $\ell = 2$ and the following sequence of approximations to it:
\begin{align}
\Phi_{{{\mathrm{ideal}, 2}}}^{(n)}
&:=
\Phi^{(n + m-1)}_1
\, \cdots \,
\Phi^{(n + 1)}_1
\Phi^{(n)}_1,
\\
&=
\prod_{k = 0}^{m-1}
\Big[
\Big(
I 
+ 
{\cal D}_1 \diag( \bm{\delta}^{n + k}_{1} ) {\cal D}_1^\top
\Big)^{-1}
{\cal S}^{(n + k, m\delta t)}
\Big],
\\
&\approx
\left[
\prod_{k = 0}^{m-1}
\Big(
I 
+ 
{\cal D}_1 \diag( \bm{\delta}^{n + k}_{1} ) {\cal D}_1^\top
\Big)^{-1}
\right]
\left[
\prod_{k = 0}^{m-1}
{\cal S}^{(n + k, m\delta t)}
\right],
\\
&\approx
\left[
\prod_{k = 0}^{m-1}
\Big(
I 
+ 
{\cal D}_1 \diag( \bm{\delta}^{n + k}_{1} ) {\cal D}_1^\top
\Big)^{-1}
\right]
\left[
\Big(
I 
+ 
{\cal D}_1 \diag( \bm{\Delta}^{n}_{2} ) {\cal D}_1^\top
\Big)^{-1}
{\cal S}^{(n, m^2\delta t)}
\right],
\\
\label{eq:Phi_2}
&\approx
\left[
I 
+ 
{\cal D}_1 \diag
\left( 
\bm{\Delta}^{n}_{2} + \sum_{k = 0}^{m-1} \bm{\delta}^{n + k}_{1} 
\right)
{\cal D}_1^\top
\right]^{-1}
{\cal S}^{(n, m^2\delta t)}
:=
\Phi^{(n, m^2 \delta t)}_2.
\end{align}
We now explain the above approximations. The first is that the $m$ semi-Lagrangian operators commute with the $m$ truncation error terms.
The next approximation is to replace $m$ semi-Lagrangian steps of size $m \delta t$ by one single semi-Lagrangian step of size $m^2 \delta t$ and to include the leading-order truncation error terms.
The final approximation has two steps: The first is to replace the $m$ products with a sum (analogous to the usual approximation one makes for MGRIT applied to a diffusion problem \cite{Falgout_etal_2014}), and the second is to pull the two matrices under the same inverse and neglect higher-order terms from their product; see \cite{DeSterck_etal_2023_SL,DeSterck_etal_2023_MOL,KrzysikThesis2021}.

Based on \eqref{eq:Phi_1} and \eqref{eq:Phi_2} we propose the following coarse-level operator on level $\ell > 0$:
\begin{align} \label{eq:Phi_ell}
\Phi^{(n, m^{\ell} \delta t)}_{\ell}
=
\Big[
I 
+ 
{\cal D}_1 \diag
( 
\bm{\sigma}^{n}_{\ell}
)
{\cal D}_1^\top
\Big]^{-1}
{\cal S}^{(n, m^{\ell}\delta t)},
\quad
\ell \in \mathbb{N},
\end{align}
where the coefficient $\bm{\sigma}^{n}_{\ell}$ is defined recursively by
\begin{align}
\bm{\sigma}^{n}_{\ell}
=
\begin{cases}
\bm{\delta}^{n}_{1}, \quad &\ell = 1,
\\[1ex]
\bm{\Delta}^{n}_{\ell}
+
\sum  \limits_{k = 0}^{m-1}
\bm{\sigma}^{n + k}_{\ell - 1}, \quad &\ell > 1.
\end{cases}
\end{align}

Now we present multilevel numerical results using the coarse-level operator \eqref{eq:Phi_ell}, extending several of the two-level results from Section \ref{sec:num-res} of the main paper.
To generate a multievel hierarchy, we recursively coarsen in time by a factor of $m = 8$ until doing so would result in a temporal coarse mesh with fewer than two points. The finest temporal resolution for the Burgers problem is $n_t = 10241$ and for the Buckley--Leverett problem it is $n_t = 11941$; in either case, the MGRIT hierarchy for each problem has 5 levels.
We use \emph{one} MGRIT V-cycle per inner iteration, analogous to how \emph{one} two-level iteration was used in the main paper.
Our MGRIT relaxation strategy is level dependent, using F-relaxation on the finest MGRIT level, as for the two-level solves used in the main text, and then FCF-relaxation on all coarser levels. 

Numerical results are shown in \cref{SMfig:num-res-multilevel}, with multilevel results shown in the top row, and the bottom row showing the two-level results from the bottom row of Figure \ref{fig:num-res-inexact-1st} from the main paper.
For the Burgers problem (left) there is no deterioration in convergence when moving from two-level to multilevel.
For the Buckley--Leverett problem there is a slight deterioration in iteration count for the finest resolution problems, where one or two more iterations are required to reach the halting tolerance.

\renewcommand{\fd}{./figures/}
\renewcommand{\hs}{2}
\renewcommand{\vs}{1}
\renewcommand\fs{0.335}
\begin{figure}[t!]
\centerline{
\includegraphics[scale=\fs]{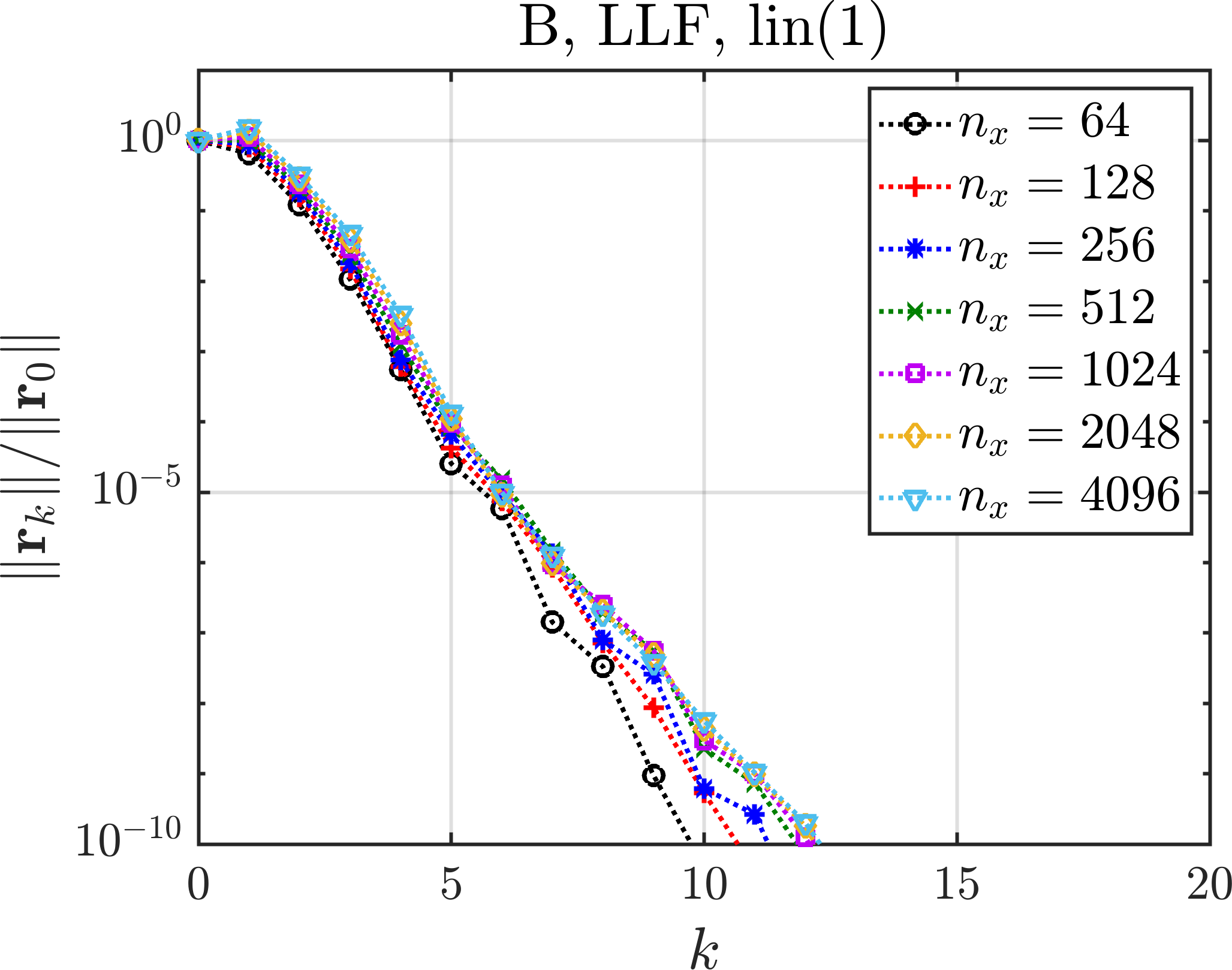}
\hspace{\hs ex}
\includegraphics[scale=\fs]{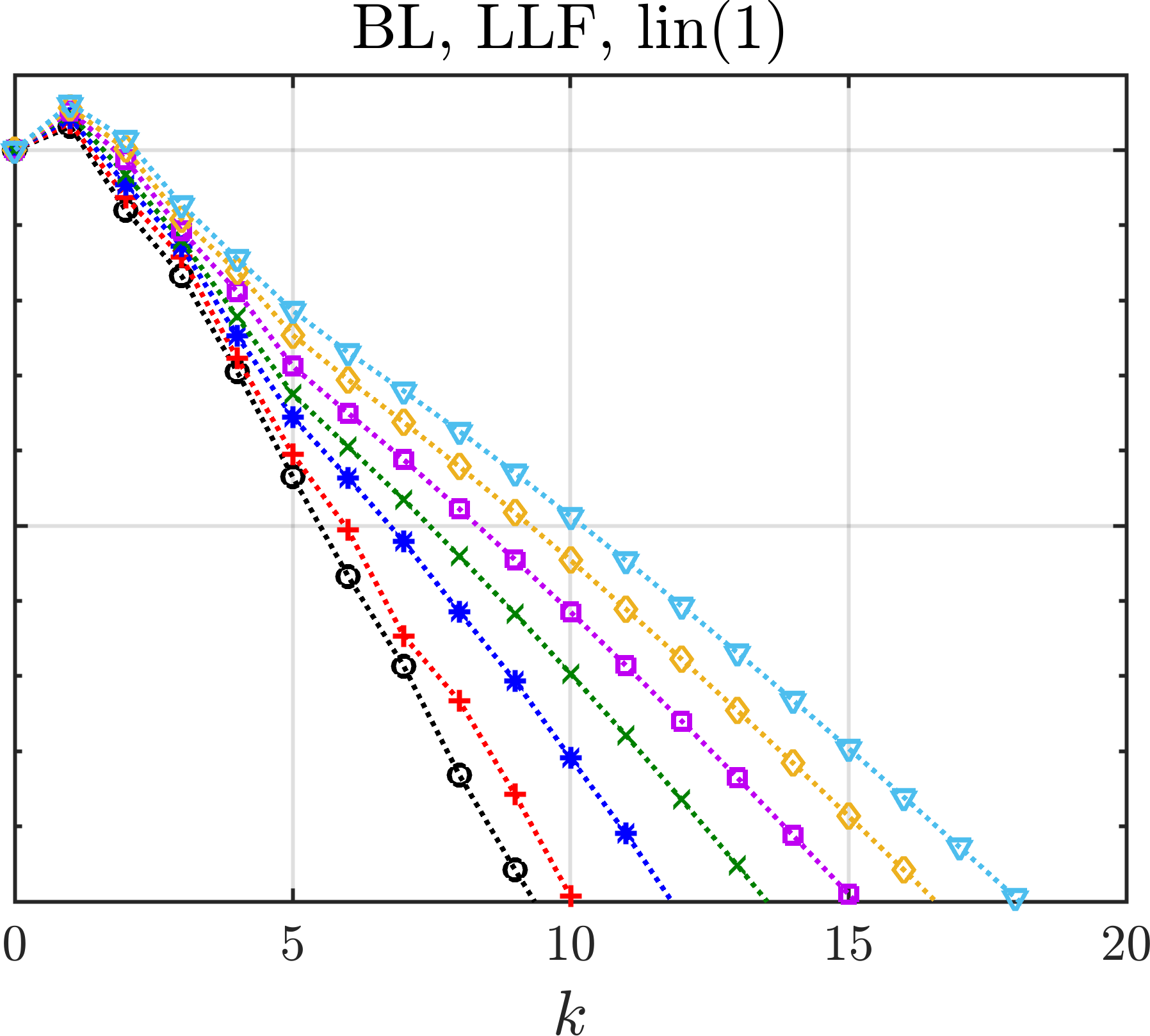}
}
\vspace{\vs ex}
\centerline{
\includegraphics[scale=\fs]{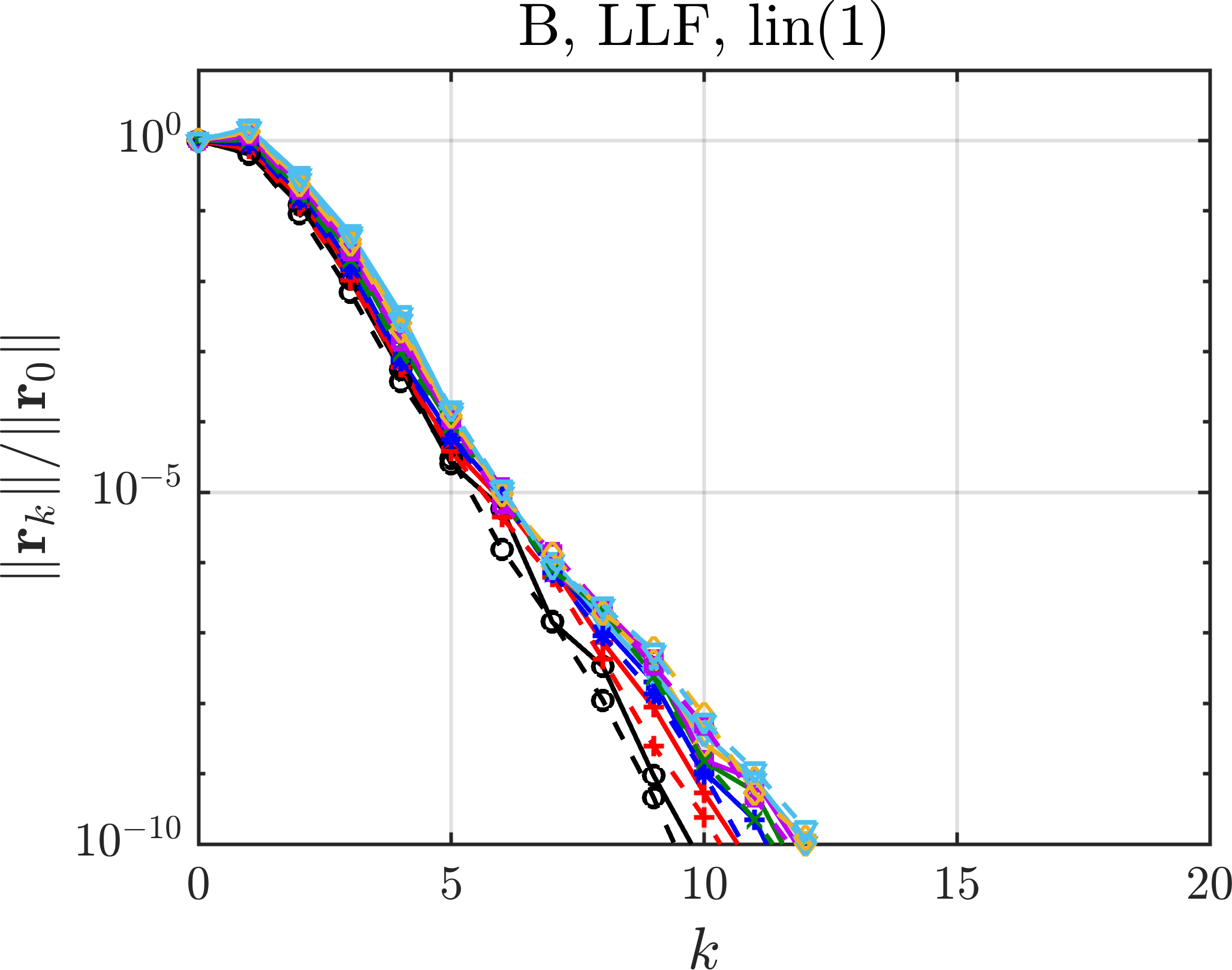}
\hspace{\hs ex}
\includegraphics[scale=\fs]{\fd BL3-LLF-lin1-F8-DIRECT-MGRIT}
}
\caption{
Residual convergence histories for the nonlinear solver Algorithm \ref{alg:richardson} applied to 1st-order accurate discretizations using the LLF flux.
Left column: Burgers \eqref{eq:burgers}. Right column: Buckley--Leverett \eqref{eq:buck}.
Top row: Linearized systems are solved approximately with one MGRIT \underline{V-cycle}.
Bottom row: Solid lines correspond to linear systems being approximately solved by one \underline{two-level} MGRIT iteration, and broken lines correspond to linear systems being solved exactly. 
\label{SMfig:num-res-multilevel}
}
\end{figure}
%

\section{Error estimate for semi-Lagrangian method and approximate truncation error}
\label{SMsec:SL-extra}

In this section, we consider further details of the conservative, FV semi-Lagrangian method outlined in Section \ref{sec:SL}.
\Cref{SMsec:SL-flux} presents relevant details for how the numerical flux is implemented, \cref{SMsec:SL-flux-trun-err} develops a truncation error estimate, and then \cref{SMsec:SL-trun-err} describes how this estimate is used to develop the expression for the associated truncation error operator ${\cal T}_{\rm direct}^{n}$ used in the coarse-grid operator \eqref{eq:Psi}.

\subsection{Further details on numerical flux}
\label{SMsec:SL-flux}

Recall from \eqref{eq:SL-FV-num-flux} the numerical flux $\wt{f}_{i+1/2} := \int_{\wt{\xi}_{i + 1/2}}^{x_{i+1/2}} R_{i + 1/2}[ \bar{\bm{e}}^n ](x) \d x$.
Implementationally, the flux is split into two integrals
\begin{align} \label{eq:SL-flux-split}
\wt{f}_{i+1/2} 
= 
\int_{\wt{\xi}_{i + 1/2}}^{\wt{x}_{i+1/2}} R_{i + 1/2}[ \bar{\bm{e}}^n ](x) \d x
+
\int_{\wt{x}_{i + 1/2}}^{x_{i+1/2}} R_{i + 1/2}[ \bar{\bm{e}}^n ](x) \d x,
\end{align}
where, recall, $\wt{x}_{i+1/2}$ is the FV cell interface immediately to the east of the departure point $\wt{\xi}_{i + 1/2}$; see Figure \ref{fig:SL-FV-conservation}.
The second integral in \eqref{eq:SL-flux-split} is evaluated exactly, noting that it is simply a sum of cell averages across the cells in the interval $x \in [\wt{x}_{i+1/2}, x_{i+1/2}]$. 
The first integral in \eqref{eq:SL-flux-split} is more complicated.
Recall that $R_{i + 1/2}[ \bar{\bm{e}}^n ](x)$ is a polynomial reconstruction of $e^n(x)$ and it is constructed by enforcing that it preserves cell averages of $e^n(x)$ across the FV cells used in its stencil.
For the remainder of this discussion, let us drop temporal subscripts since they are not important.

At this point, it is useful to be familiar with the polynomial reconstruction discussion in \cref{SMsec:reconstruction-overview}.

Recall from Section \ref{sec:SL} that we consider semi-Lagrangian methods of order $p = 2k-1$ in space, with $k = 1, 2, \ldots$.
To this end, let $R_{i + 1/2}[\bb{e}^n](x)$ be a reconstruction polynomial using a stencil consisting of $p = 2k - 1$ cells. 
Further, suppose the stencil is centered about the cell containing the departure point. 
For simplicity of notation, suppose that the cell containing the departure point, i.e., $x \in [\wt{x}_{i+1/2} - h, \wt{x}_{i+1/2}]$, is cell ${\cal I}_0$. Then, the cells in the stencil are $\{ {\cal I}_{-(k-1)}, \ldots, {\cal I}_{k-1} \}$.\footnote{In the notation of \cref{SMsec:reconstruction-overview}, the stencil has a total number of cells $k \mapsto 2k -1 = p$, and a left shift of $\ell = k - 1$.}
Define $E(x)$ as the primitive of $e(x)$, i.e., $E(x) = \int_{-\infty}^{x} e(\zeta) \d \zeta$ and $E'(x) = e(x)$.
Then, based on \cref{SMsec:reconstruction-overview}, we know that $R_{i + 1/2}[\bb{e}](x)$ is a polynomial of degree at most $p - 1$, and that it is the derivative of an interpolating polynomial $Q_{p}(x)$ of degree at most $p$,
\begin{align}
R_{i + 1/2}[\bb{e}](x) \equiv Q_{p}'(x),
\end{align}
where, $Q_p(x)$ interpolates $E(x)$ at the $p+1$ cell interfaces in the stencil,
\begin{align} \label{eq:SL-flux-interpolation-conditions}
Q_{p}(x_{j \pm 1/2}) = E(x_{j \pm 1/2}), \quad j = -(k-1), \ldots, k-1.
\end{align}
%

\subsection{Error estimate}
\label{SMsec:SL-flux-trun-err}

The estimate we develop herein and its proof share close parallels with the error estimate we developed previously in \cite[Lemma 3.1]{DeSterck_etal_2023_SL} for a \textit{non-conservative, FD semi-Lagrangian method}; however, the estimate itself and its proof are fundamentally different due to the fact that we are now considering a semi-Lagrangian method that is both conservative and based on finite volumes.
The full estimate is split across \cref{SMlem:SL-trunc-est,SMlem:SL-flux-est}.
Note that the proof of \cref{SMlem:SL-flux-est} in particular requires to have read \cref{SMsec:SL-flux}.

\begin{lemma}[Error estimate for semi-Lagrangian method] \label{SMlem:SL-trunc-est}
Let ${\cal S}^{t_n, \delta t}_{p, r} \in \mathbb{R}^{n_x \times n_x}$ be the FV semi-Lagrangian time-stepping operator defined in \eqref{eq:SL-FV-scheme} that approximately advances the solution of \eqref{eq:cons-lin} from $t_n \to t_n + \delta t$. 
Let ${\cal S}^{t_n, \delta t}_{p, \infty}$ denote the aforementioned semi-Lagrangian discretization that locates departure points exactly at time $t_n$.
Define $\bb{e}(t) \in \mathbb{R}^{n_x}$ as the vector composed of cell averages of the exact solution of the PDE \eqref{eq:cons-lin} at time $t$.
Then, the local truncation error of ${\cal S}^{t_n, \delta t}_{p, \infty}$ can be expressed as
\begin{align} \label{eq:SL-FV-trunc-est}
\tau^n_i 
&:= 
\Big(
\bb{e}(t_{n+1}) - {\cal S}^{t_n, \delta t}_{p, \infty} \bb{e}(t_n) \Big)_i
=
\frac{\wt{E}_{i+1/2}^n - \wt{E}_{i-1/2}^n}{h},
\end{align}
where $\wt{E}_{i+1/2}^n$ is the local error of the 
numerical flux $\wt{f}_{i+1/2}$ (see \eqref{eq:SL-FV-num-flux}):
\begin{align} \label{eq:SL-FV-trunc-E-def}
\wt{E}_{i+1/2}^n 
&:=
\int_{\wt{\xi}_{i + 1/2}}^{\wt{x}_{i + 1/2}}
R_{i + 1/2}[\bb{e}(t_n)](x) - e(x, t_n) \d x.
\end{align}
\end{lemma}

\begin{proof}
To aid with readability, let us drop all $t_n$ and $\delta t$ superscripts.
Using the definition of the numerical flux \eqref{eq:SL-FV-num-flux}, plug the exact PDE solution into the numerical scheme \eqref{eq:SL-FV-scheme} to get the truncation error
\begin{align} 
\tau^n_i
&=
\bar{e}_i(t_{n+1}) 
-
\bar{e}_i(t_{n})
+ 
\frac{
\int_{\wt{\xi}_{i + 1/2}}^{x_{i + 1/2}}
R_{i + 1/2}[\bb{e}(t_n)](x) \d x
-
\int_{\wt{\xi}_{i - 1/2}}^{x_{i - 1/2}}
R_{i - 1/2}[\bb{e}(t_n)](x) \d x
}
{h}.
\end{align}
In the above equation, replace $\bar{e}_i(t_{n})$ with its definition, and add and subtract $e(x, t_n) / h$ integrated over $x \in [\wt{\xi}_{i - 1/2}, x_{i-1/2}]$ and over $x \in [x_{i + 1/2}, \wt{\xi}_{i + 1/2}]$ to get (note the use of $\int_{a}^{b} = - \int_{b}^{a}$)
\begin{align} 
\begin{split}
&\tau^n_i
=
\left[
\bar{e}_i(t_{n+1}) 
-
\frac{
\int_{\wt{\xi}_{i - 1/2}}^{x_{i - 1/2}} e(x, t_{n}) \d x
+
\int_{x_{i - 1/2}}^{x_{i + 1/2}} e(x, t_{n}) \d x
+
\int_{x_{i + 1/2}}^{\wt{\xi}_{i + 1/2}} e(x, t_{n}) \d x 
}
{h}
\right]+
\\
&
\frac{
\Big(
\int_{\wt{\xi}_{i + 1/2}}^{x_{i + 1/2}}
R_{i + 1/2}[\bb{e}(t_n)](x)  -  e(x, t_n) \d x
\Big)
-
\Big(
\int_{\wt{\xi}_{i - 1/2}}^{x_{i - 1/2}}
R_{i - 1/2}[\bb{e}(t_n)](x) - e(x, t_n) \d x 
\Big)
}
{h}.
\end{split}
\end{align}
The term inside the closed parentheses vanishes due to the local conservation property \eqref{eq:cons-lin-SL-cons-local} satisfied by the exact PDE solution.
Now split the first of the two remaining integrals over $x \in [\wt{\xi}_{i+1/2}, \wt{x}_{i+1/2}]$ and $x \in [\wt{x}_{i+1/2}, x_{i+1/2}]$; next, recall from discussion in \cref{SMsec:SL-flux} that the integral over $[\wt{x}_{i+1/2}, x_{i+1/2}]$ will be zero (it is an integral over an integer number of FV cells).
Carrying out the analogous steps for the second integral and then invoking the definition of the error $\wt{E}_{i+1/2}^n$ from \eqref{eq:SL-FV-trunc-E-def} gives the claim.
\end{proof}

\begin{lemma}[Error estimate for semi-Lagrangian numerical flux]
\label{SMlem:SL-flux-est}
Suppose the assumptions of \cref{SMlem:SL-trunc-est} hold, and that the numerical flux is evaluated as described in \cref{SMsec:SL-flux}.
Further suppose that the solution $e(x,t)$ of \eqref{eq:cons-lin} is sufficiently smooth, and that the wave-speed function in \eqref{eq:cons-lin} varies only in time, $\alpha(x, t) \equiv \alpha(t)$.
Then, the error \eqref{eq:SL-FV-trunc-E-def} of the numerical flux satisfies
\begin{align} \label{eq:SL-flux-est}
\wt{E}_{i+1/2}^n = - h^{p+1} g_{p+1}(\varepsilon_{i+1/2}) \frac{\partial^{p} e}{\partial x^p} \bigg|_{(x_{i+1/2}, t_{n+1})} + {\cal O}(h^{p+2}),
\end{align}
where $g_{p+1}$ is the degree $p+1$ polynomial 
\begin{align} \label{eq:SL-g-poly-def}
g(z) := \frac{1}{(p+1)} \prod \limits_{j = -k}^{k-1} ( z + j ).
\end{align}

\end{lemma}

\begin{proof}
For notational simplicity let us momentarily drop temporal dependence.
To begin, recall from \cref{SMsec:SL-flux} that $R_{i+1/2}[\bb{e}](x) \equiv Q_{p}'(x)$, with $Q_p(x)$ a polynomial interpolating $E(x)$, the primitive of $e(x)$, i.e., $e(x) = E'(x)$, at the $p+1$ cell interfaces in an interpolation stencil centered about the departure point $\wh{\xi}_{i+1/2}$.
Thus, by the fundamental theorem of calculus, the flux error \eqref{eq:SL-FV-trunc-E-def} can be expressed as
\begin{align} 
\wt{E}_{i+1/2} 
&:=
\int_{\wt{\xi}_{i + 1/2}}^{\wt{x}_{i + 1/2}}
R_{i + 1/2}[\bb{e}](x) - e(x) \d x,
\\
&=
\big[
Q_p(\wt{\xi}_{i + 1/2})
-
E(\wt{\xi}_{i + 1/2})
\big]
-
\big[
Q_p(\wt{x}_{i + 1/2})
-
E(\wt{x}_{i + 1/2})
\big],
\\
\label{eq:SL-E-prim-diff}
&=
Q_p(\wt{\xi}_{i + 1/2})
-
E(\wt{\xi}_{i + 1/2}),
\end{align}
with the last equation holding since $\wt{x}_{i + 1/2}$ is a point in the interpolation stencil.

Now to estimate \eqref{eq:SL-E-prim-diff} we use standard polynomial interpolation theory in the sense of \cref{SMlem:standard-poly-interp-est}.
To this end, we first restate the result of \cref{SMlem:standard-poly-interp-est} as it applies to the current context, i.e., with $U \mapsto E$, $k \mapsto 2k-1 = p$ and $\ell = k-1$.
Specifically, for any $x \in [\wt{x}_{i + 1/2} - k h,  \wt{x}_{i + 1/2} + (k-1) h]$, we have 
\begin{align} \label{eq:primitive-interp-error-E}
E(x) - Q_{p}(x)
=
\frac{1}{(p+1)!}
\prod \limits_{j = -k}^{k-1} \big[x - (\wt{x}_{i + 1/2} + jh) \big] E^{(p+1)} ( \varphi(x) ),
\end{align}
for $\varphi(x) \in [\wt{x}_{i + 1/2} - k h,  \wt{x}_{i + 1/2} + (k-1) h]$ an unknown point.
Now evaluate this error estimate at $x = \wt{x}_{i+1/2} - h \varepsilon_{i+1/2}$. 
First, let us reintroduce temporal dependence since it is now important. 
Noting that $\wt{x}_{i+1/2} - h \varepsilon_{i+1/2} - (\wt{x}_{i + 1/2} + jh) = -h( \varepsilon_{i+1/2} + j )$ and that $p$ is odd, we find from \eqref{eq:SL-E-prim-diff}
\begin{align} 
\wt{E}_{i+1/2}^n 
&=
-\frac{h^{p+1}}{(p+1)!}
\prod \limits_{j = -k}^{k-1} ( \varepsilon_{i+1/2} + j  ) \frac{\partial^{p+1} E}{\partial x^{p+1}}\bigg|_{(\varphi( \wt{\xi}_{i+1/2} ), t_n)},
\\
\label{eq:primitive-interp-error-E2}
&=
-h^{p+1} g_{p+1}(\varepsilon_{i+1/2}) \frac{\partial^{p} e}{\partial x^{p}} \bigg|_{(\varphi( \wt{\xi}_{i+1/2} ), t_n)},
\end{align}
with $g_{p+1}$ defined in \eqref{eq:SL-g-poly-def}.

Next, note that since $\varphi( \wt{\xi}_{i+1/2} ) \in [\wt{x}_{i + 1/2} - k h,  \wt{x}_{i + 1/2} + (k-1) h]$, we have $\varphi( \wt{\xi}_{i+1/2} ) = \wt{\xi}_{i+1/2} + {\cal O}(h)$ because $\wt{\xi}_{i+1/2} \in [\wt{x}_{i + 1/2} - h,  \wt{x}_{i + 1/2}]$. 
Therefore, by Taylor expansion, $\frac{\partial^{p} e}{\partial x^{p}} \big|_{(\varphi( \wt{\xi}_{i+1/2} ), t_n)} = \frac{\partial^{p} e}{\partial x^{p}} \big|_{( \wt{\xi}_{i+1/2}, t_n)} + {\cal O}(h)$.  
Finally, we note that differentiating the PDE in question, $\frac{\partial }{\partial t} (e) + \frac{\partial }{\partial x} (\alpha(t) e) = 0$, $p$ times with respect to $x$ gives $\frac{\partial }{\partial t} \big( \frac{\partial^{p} e}{\partial x^{p}} \big) + \frac{\partial }{\partial x} \big( \alpha(t) \frac{\partial^{p} e}{\partial x^{p}} \big) = 0$. Thus, the derivative $\frac{\partial^{p} e}{\partial x^{p}}$ is constant along the same characteristics as the PDE in question, and, so, $ \frac{\partial^{p} e}{\partial x^{p}} \big|_{( \wt{\xi}_{i+1/2}, t_n)} =  \frac{\partial^{p} e}{\partial x^{p}} \big|_{( x_{i+1/2}, t_{n+1})}$ because $(x_{i+1/2}, t_{n+1})$ is the arrival point associated with the departure point $( \wt{\xi}_{i+1/2}, t_n)$; see Figure \ref{fig:SL-FV-conservation}.
Inserting this into \eqref{eq:primitive-interp-error-E2} concludes the proof.

\end{proof}

\subsection{Approximating truncation error}
\label{SMsec:SL-trun-err}

In this section, we develop an expression for the truncation error operator ${\cal T}_{\rm direct}^{n}$ used in \eqref{eq:Psi}, corresponding to the truncation error of a single step of the coarse-grid semi-Lagrangian method.

Combining the results of \cref{SMlem:SL-trunc-est,SMlem:SL-flux-est}, we suppose that the truncation error for a single time-step of the coarse-grid semi-Lagrangian discretization method used in \eqref{eq:Psi} satisfies
\begin{align}
\tau^n_i 
&:= 
\Big(
\bb{e}(t_{n+1}) - {\cal S}^{t_n, m \delta t}_{p, 1_*} \bb{e}(t_n) \Big)_i
\\
\label{eq:SL-trunc-approx1}
&\approx
-
h^{p+1} 
\frac{
\displaystyle{
g_{p+1}(\varepsilon_{i+1/2}) \frac{\partial^{p} e}{\partial x^p} \bigg|_{(x_{i+1/2}, t_{n+1})}
-
g_{p+1}(\varepsilon_{i-1/2}) \frac{\partial^{p} e}{\partial x^p} \bigg|_{(x_{i-1/2}, t_{n+1})}
}
}{h},
\\
\label{eq:SL-trunc-approx2}
&\approx
-h^{p+1}
\Big(
{\cal D}_1
\diag 
\big(
g_{p+1} 
( 
\bm{\varepsilon}  
) 
\big)
{\cal D}_p^\top
\bb{e}(t_{n+1})
\Big)_i.
\end{align}
Here, $\bm{\varepsilon} = \big( \varepsilon_{1+1/2}, \ldots, \varepsilon_{n_x+1/2} \big) \in \mathbb{R}^{n_x}$, and the polynomial $g_{p+1}$ from \eqref{eq:SL-g-poly-def} is applied elementwise.
As previously, ${\cal D}_{j} \in \mathbb{R}^{n_x \times n_x}$ for $j$ odd is a 1st-order accurate FD approximation to the $j$th derivative of a periodic grid function, on a stencil that has a 1-point bias to the left.

The second approximation \eqref{eq:SL-trunc-approx2} arises simply from discretizing the $p$th-order derivatives in \eqref{eq:SL-trunc-approx1}. On the other hand, \eqref{eq:SL-trunc-approx1} arises from a series of approximations that are more difficult to justify. That being said, they are ultimately justified by the fact that they lead to an expression for ${\cal T}_{\rm direct}^{n}$ that results in fast MGRIT convergence. 
First, the error estimate \cref{SMlem:SL-trunc-est} assumed departure points are located exactly; however, in practice they are located approximately.
Second, the error estimate \cref{SMlem:SL-flux-est} applies only to spatially independent wave-speed functions $\alpha$; however, in practice the wave-speed function varies in space.
Third, the higher-order terms from \cref{SMlem:SL-flux-est} are assumed not to be important; however, without further analysis it is not clear that do not result in terms of the same order as those that have been retained.

Based on \eqref{eq:SL-trunc-approx2} and the above discussion, we implement the truncation error operator as
\begin{align} \label{eq:Tcal-SL}
{\cal T}_{\rm direct}^{n}
=
-h^{p+1}
{\cal D}_1
\diag 
\big(
g_{p+1} 
( 
\bm{\varepsilon}  
) 
\big)
{\cal D}_p^\top.
\end{align}
%

\section{Error estimates for linear(ized) method-of-lines discretizations}
\label{SMsec:ideal-err-est}


In this section, we present an error analysis for a linear method-of-lines discretization for a PDE that is an approximation to the linear conservation law \eqref{eq:cons-lin}.
The outcome of this analysis is that it is used to develop an expression for the approximate truncation error operator ${\cal T}_{\rm ideal}^{n}$ in the MGRIT coarse-grid operator \eqref{eq:Psi} that is used to solve certain linearized discretizations of \eqref{eq:cons-lin} (the development of ${\cal T}_{\rm ideal}^{n}$ is given later in \cref{SMsec:ideal-err-est-linearized}).
In addition, the analysis in this section is used to develop an MGRIT coarse-grid operator for solving standard, linear method-of-lines discretizations of \eqref{eq:cons-lin} (these developments come next in \cref{SMsec:MGRIT-linear-standard}).

The remainder of this section is structured as follows. 
\Cref{SMsec:cons-lin-approx} presents a simplified model PDE and describes its discretization.
\Cref{SMsec:trunc-est-MOL-cons-lin-approx} presents a spatial-discretization-agnostic truncation error estimate for the model problem and discretization.
\Cref{SMsec:est-reconstructions} presents error estimates for some standard linear reconstructions.
%

\subsection{A simplified linear PDE and its discretization}
\label{SMsec:cons-lin-approx}

While it is ultimately discretizations of the PDE \eqref{eq:cons-lin} that we solve with MGRIT, we do not develop truncation error estimates for this equation directly, but instead for a semi-autonomous approximation to it.\footnote{In essence, two key components in our analysis are computing the exact solution of a linear ODE system of the form $\d \bar{\bm{e}}(t) / \d t = L \bar{\bm{e}}(t)$, and expressing in closed-form the local truncation error of its Runge-Kutta approximation.
Both of these tasks are straightforward only when $L$ is time independent.}
Aspects of the analysis also apply to standard, linear method-of-lines discretizations of the PDE \eqref{eq:cons-lin}, enabling the development of an MGRIT solver for such problems; this scenario is considered in \cref{SMsec:MGRIT-linear-standard}.

To this end, let us introduce the short-hand ${\alpha^n(x) \equiv \alpha(x, t_n + \delta t \vartheta)}$, $n \in \mathbb{N}_0$, $\vartheta \in [0, 1]$, and define $e^n(t)$ as the solution of the following PDE holding on the local time interval $t \in (t_n, t_{n+1}]$, 
\begin{align}
\label{eq:cons-lin-space-vary} 
\frac{\partial e^n}{\partial t} + \frac{\partial }{\partial x} \big( \alpha^n(x) e^n \big) 
&= 0,
\quad (x, t) \in (-1, 1) \times (t_n, t_{n+1}],
\end{align}
where $\quad e^n(x, t_n) = e^{n-1}(x, t_n)$.
The PDE \eqref{eq:cons-lin-space-vary} arises from considering \eqref{eq:cons-lin} over the local time interval $t \in (t_n, t_{n+1}]$ and freezing the time-dependence of the wave-speed on this interval at the point $t_n + \delta t \vartheta$.\footnote{Our analysis is agnostic to the choice of $\vartheta \in [0, 1]$; however, when the results of this analysis are used heuristically to create an MGRIT coarse-grid operator, we find the choice of $\vartheta$ is extremely important in some cases. For example, in the MGRIT solution of the linearized systems arising from a 3rd-order WENO discretization of the Buckley--Leverett test problems from the main paper.}
We now outline the discretization for this problem, which tracks closely with the discretization outlined in Section \ref{sec:discretization}; however, they are some key differences. 
%
%

Integrating \eqref{eq:cons-lin-space-vary} over the discretized spatial domain gives the exact relation for $t \in (t_n, t_{n+1}]$
\begin{align}
\label{eq:cons-lin-space-vary-int} 
\frac{\d \bar{e}^n_i(t)}{\d t} = 
-\frac{f^n(x_{i+1/2}, e(x_{i+1/2})) - f^n(x_{i-1/2}, e(x_{i-1/2}))}{h},
\quad i \in \{1, \ldots, n_x \},
\end{align}
with flux $f^n(x, e) := \alpha^n(x) e^n(x)$ (omitting $t$ dependence in $e^n(x)$ for readability).

Next, the physical fluxes in \eqref{eq:cons-lin-space-vary-int} are approximated with numerical fluxes: \\ $f^n(x_{i+1/2}, e(x_{i+1/2})) \approx \wh{f}^n_{i+1/2}$. This results in the semi-discretization for $
t \in (t_n, t_{n+1}]$
\begin{align} \label{eq:cons-lin-space-vary-semi-disc} 
\frac{\d \bar{e}_i^{n}}{\d t} 
\approx
- 
\frac{\wh{f}^n_{i+1/2} - \wh{f}^n_{i-1/2}}{h}
=:
\big( \wh{L}^{n} \bar{\bm{e}}^{n} \big)_i, 
\quad i \in \{1, \ldots, n_x \},
\end{align}
where $\bb{e}^n = \big( \bar{e}_1^n, \ldots, \bar{e}_{n_x}^n \big)^\top \in \mathbb{R}^{n_x}$.
The numerical flux $\wh{f}^n_{i+1/2}$ introduced here is of the LF type for a linear problem given by \eqref{eq:LFF-lin}.
For the time being, we do not want to make any assumptions about the numerical flux, including how $e^n(x_{i+1/2})$ is reconstructed.\footnote{Later we consider several different reconstruction techniques including standard polynomial reconstruction and Picard-linearized WENO reconstructions. Therefore our goal is to develop a general error estimate that is agnostic to specifics of the reconstruction.}
That being said, the analysis requires us to make the following assumption.
\begin{assumption}[Error for general linear spatial discretization]
\label{ass:E-est}
Assume that the numerical flux function $\wh{f}^n_{i+1/2} = \wh{f}^n_{i+1/2}(\bb{e}^n) \approx \alpha^n(x_{i+1/2}) e^n(x_{i+1/2})$ introduced in \eqref{eq:cons-lin-space-vary-semi-disc} is linear in $\bb{e}^n$.
Further, assume that the numerical flux satisfies
\begin{align}
\wh{f}^n_{i+1/2}(\bb{e}^n) = \alpha^n(x_{i+1/2}) e^n(x_{i+1/2}) + \wh{E}^n_{i+1/2} (\bb{e}^n),
\end{align}
with the error $\wh{E}^n_{i+1/2} = \wh{E}^n_{i+1/2}(\bb{e}^n)$ also a linear function of $\bb{e}^n$.

Defining the linear difference operator $\wh{{\cal E}}^{n} \colon \mathbb{R}^{n_x} \to \mathbb{R}^{n_x}$ such that $\big(\wh{{\cal E}}^{n} \bb{e}^n \big)_i = \frac{\wh{E}^n_{i+1/2} - \wh{E}^n_{i-1/2}}{h}$, from the above relation, it follows that the right-hand side of \eqref{eq:cons-lin-space-vary-int} can be expressed as a linear function of $\bb{e}^n$ as follows:
\begin{align}
-
\frac{\alpha^n(x_{i+1/2}) e^n(x_{i+1/2}) - \alpha^n(x_{i-1/2}) e^n(x_{i-1/2})}{h} 
=
\big[ \big( \wh{L}^n + \wh{{\cal E}}^{n}\big) \bb{e}^n \big]_i,
\end{align}
with $\wh{L}^n$ the spatial discretization in \eqref{eq:cons-lin-space-vary-semi-disc}.
\end{assumption}
Once specific details of the reconstruction are specified, estimates can be made for the error term $\wh{E}_{i+1/2}$ (see, e.g., \cref{SMlem:num-flux-est}).
Since the flux is assumed linear in $\bb{e}^n$, it follows that the spatial discretization operator $\wh{L}^{n}$ from \eqref{eq:cons-lin-space-vary-semi-disc} is also assumed to be linear.

Next, the ODEs \eqref{eq:cons-lin-space-vary-semi-disc} are approximately advanced forwards in time with an ERK method resulting in
\begin{align} \label{eq:lin-cons-one-step}
\bar{\bm{e}}^{n+1} = \Phi^{n} \bar{\bm{e}}^{n},
\end{align}
with $\Phi^{n}$ the time-stepping operator, and $\bb{e}^{n} \approx \bb{e}(t_{n})$.
Since the ODEs \eqref{eq:cons-lin-space-vary-semi-disc} are linear and autonomous, the time-stepping operator can be written as
\begin{align} \label{eq:RK-stab-expansion}
\Phi^{n} = R \big( \delta t \wh{L}^{n} \big) = \exp \big( \delta t \wh{L}^{n} \big) + \wh{e}_{\rm RK} \big( \delta t \wh{L}^{n} \big)^{q+1} + {\cal O} \big( \big( \delta t \wh{L}^{n} \big)^{q+2} \big),
\end{align}
in which $R$ is the \textit{stability function} of the ERK method, and $q+1$ is the order of its local truncation error.
Here, we define $\wh{e}_{\rm RK}$ as the error constant of the method. For the forward Euler method \eqref{eq:ERK1}, $\wh{e}_{\rm RK} = -1/2$, and for the 3rd-order ERK method \eqref{eq:ERK3}, $\wh{e}_{\rm RK} = -1/4!$; see \cite{DeSterck_etal_2023_MOL} for further discussion.

\subsection{Truncation error estimate}
\label{SMsec:trunc-est-MOL-cons-lin-approx}

We now develop a truncation error estimate for the previously described discretization of the integral conservation law \eqref{eq:cons-lin-space-vary-int}.
First \cref{SMlem:exact-sol-int-con} presents an exact solution to \eqref{eq:cons-lin-space-vary-int}, and then \cref{SMlem:trunc-est-int-cons-law} gives a truncation error estimate for the discretization of it.


\begin{lemma}[Exact solution of integral form of conservation law]
\label{SMlem:exact-sol-int-con}
Assume that \cref{ass:E-est} holds, and let $\bb{e}^n(t)$ be the exact solution of the local conservation law \eqref{eq:cons-lin-space-vary-int} at time $t \in [t_n, t_{n+1}]$.
Then,
\begin{align} \label{eq:cons-lin-adv-spatially-vary-sol}
\bb{e}^{n}(t) = \exp \big[ (t - t_n) \big( \wh{L}^n +  {\cal E}^n \big) \big] \bb{e}^{n} (t_n), \quad t \in [t_n, t_n + \delta t],
\end{align}
where the linear operator $\wh{L}^n + \wh{{\cal E}}^{n}$ is as defined in \cref{ass:E-est}.
\end{lemma}

\begin{proof}
To see that \eqref{eq:cons-lin-adv-spatially-vary-sol} indeed solves \eqref{eq:cons-lin-space-vary-int}, consider the $i$th component of \eqref{eq:cons-lin-adv-spatially-vary-sol}, differentiating both sides with respect to $t$ and then invoking the expression for $\wh{L}^n + \wh{{\cal E}}^{n}$ from \cref{ass:E-est}:
\begin{align}
\frac{\d \bar{e}_i^n}{\d t} 
&= 
\Big[ \big( \wh{L}^n + \wh{{\cal E}}^{n} \big) \exp \big[ (t - t_n) \big( \wh{L}^n + {\cal E}^n \big) \big] \bb{e}^{n} (t_n) \Big]_i,
\\
&=
\big[ \big( \wh{L}^n + \wh{{\cal E}}^{n} \big) \bb{e}^n \big]_i,
\\
&=
-
\frac{\alpha^n(x_{i+1/2}) e^n(x_{i+1/2}) - \alpha^n(x_{i-1/2}) e^n(x_{i-1/2})}{h}.
\end{align}
\end{proof}

We are now ready to present an error estimate for the fully discrete scheme \eqref{eq:lin-cons-one-step}. 
The steps in this proof track closely with those used in the proof of \cite[Lemma 3]{DeSterck_etal_2023_MOL}, in which we developed a truncation error estimates for non-conservative, FD, Runge-Kutta discretizations of constant wave-speed linear advection problems.

\begin{lemma}[Truncation error estimate for $\Phi^{n}$]
\label{SMlem:trunc-est-int-cons-law}
Assume that \cref{ass:E-est} holds, and let $\bb{e}^n(t)$ be the exact solution of the local conservation law \eqref{eq:cons-lin-space-vary-int} at time $t \in [t_n, t_{n+1}]$.
Assume this function is sufficiently smooth.
Then, the local truncation error of the one-step discretization \eqref{eq:lin-cons-one-step} at time $t_{n+1}$ can be written as
%
%
\begin{align} 
\label{eq:trunc-est-cons-lin-def}
\bm{\tau}^{n}
&:=
\bb{e}^n(t_{n+1})
- 
\Phi^{n}
\bb{e}^n(t_n),
\\
\label{eq:trunc-est-cons-lin}
&=
\Big( - \delta t \wh{{\cal E}}^{n} + \wh{e}_{\rm RK} \big( \delta t \wh{L}^{n} \big)^{q+1} \Big)
\bb{e}^{n}(t_{n+1}) 
+ 
{\cal O} \Big( \big( \delta t \wh{L}^{n} \big)^{q+2}, \big(\delta t \wh{{\cal E}}^{n} \big)^2 \Big),
\end{align}
in which $\wh{L}^{n}$ is the spatial discretization introduced in \eqref{eq:cons-lin-space-vary-semi-disc}, and $\wh{{\cal E}}^{n}$, as defined in \cref{ass:E-est}, is the associated spatial discretization error.
\end{lemma}

\begin{proof}
Using \eqref{eq:cons-lin-adv-spatially-vary-sol}, the solution to \eqref{eq:cons-lin-space-vary-int} at time $t = t_{n}$ may be written as
\begin{align} \label{eq:trunc-est-proof-aux0}
\bb{e}^n(t_n) 
=
\exp \big( - \delta t ( \wh{L}^{n} + \wh{{\cal E}}^n ) \big) \bb{e}^n(t_{n+1}).
\end{align}
Plugging this into definition \eqref{eq:trunc-est-cons-lin-def} gives
\begin{align} \label{eq:trunc-est-proof-aux}
\bm{\tau}^{n}
=
\Big[
I
- 
\Phi^{n}
\exp \big( - \delta t ( \wh{L}^{n} + \wh{{\cal E}}^n ) \big)
\Big]
\bb{e}^n(t_{n+1}).
\end{align}
Next, plug in the series expansion for $\Phi^{n}$ given in \eqref{eq:RK-stab-expansion} and simplify to give (abusing notation by dropping the vector $\bb{e}^n(t_{n+1})$)
\begin{align}
\begin{split}
& I
- 
\Phi^{n}
\exp \big( - \delta t ( \wh{L}^{n} + \wh{{\cal E}}^n ) \big)
\\
&=
I
- 
\Big[ 
\exp \big( \delta t \wh{L}^{n} \big) + \wh{e}_{\rm RK} \big( \delta t \wh{L}^{n} \big)^{q+1} + {\cal O}\Big( \big( \delta t \wh{L}^{n} \big)^{q+2} \Big)
\Big] 
\exp \big( - \delta t ( \wh{L}^{n} + \wh{{\cal E}}^n ) \big),
\end{split}
\\
&=
\Big( \delta t \wh{{\cal E}}^{n} - \wh{e}_{\rm RK} \big( \delta t \wh{L}^{n} \big)^{q+1} \Big)
+ 
{\cal O} \Big( \big( \delta t \wh{L}^{n} \big)^{q+2}, \big( \delta t \wh{{\cal E}}^{n} \big)^2 \Big).
\end{align}
Plugging this final expressing into \eqref{eq:trunc-est-proof-aux} concludes the proof.
\end{proof}

From \cref{SMlem:trunc-est-int-cons-law} we want to develop a truncation error estimate for the ideal coarse-grid operator $\Psi_{\rm ideal}^{n, m \delta t} := \prod_{j = 0}^{m-1} \Phi^{n+j}$, i.e., $m$ fine-grid steps across the coarse-grid interval $t \in [t_n, t_{n+m}]$.
This can be developed rigorously using Taylor-series-type argument, but this is not necessary for our purposes since this analysis is heuristic anyway. 
Considering the form of \eqref{eq:trunc-est-cons-lin} we suppose that the ideal coarse-grid operator satisfies the approximation
\begin{align} \label{eq:trunc-est-int-cons-ideal} 
\begin{split}
&\bb{e}^{n+m-1}(t_{n+m})
- 
\Psi_{\rm ideal}^{n, m \delta t} 
\bb{e}^n(t_n)
\approx
\\
&\hspace{4ex}
\Bigg[
\sum_{j = 0}^{m-1}
\Big( \delta t \wh{{\cal E}}^{n+j} - \wh{e}_{\rm RK} \big( \delta t \wh{L}^{n+j} \big)^{q+1} \Big)
\Bigg]
\bb{e}^{n+m-1}(t_{n+m}).
\end{split}
\end{align}

\subsection{Error estimates for reconstructions}
\label{SMsec:est-reconstructions}

We conclude \cref{SMsec:ideal-err-est} with the following lemma: An error estimate for the left-shifted polynomial reconstructions that are described Sections \ref{sec:reconstruction} and \ref{SMsec:reconstruction-overview}.
\begin{lemma}[Error estimate for shifted linear reconstruction using $k$ cells]
\label{SMlem:est-lin-rec}
Consider the polynomial reconstructions $\big( \wt{R}^{\ell} \bb{u} \big)_i$ and $\big( {R}^{\ell} \bb{u} \big)_i$ of the function $u(x)$ at the left- and right-hand interfaces of cell ${\cal I}_i$, respectively, as given in \eqref{eq:reconstuct-lin}.
Recall that, for fixed \textit{left-shift} $\ell \in \{0, \ldots, k - 1\}$, these reconstructions use the $k$ cell averages of $u$ in the stencil $S^{\ell}_i := \{  {\cal I}_{i - \ell}, \ldots , {\cal I}_{ i - \ell + ( k - 1 ) } \}$.
Suppose that $u(x)$ is $k+1$ times differentiable over the cells in $S^{\ell}_i$. 
Then, these reconstructions satisfy
\begin{align} \label{eq:lin-rec-est-LHS}
\big( \wt{R}^{\ell} \bb{u} \big)_i
&=
u( x_{i - 1/2} ) 
- 
\frac{(-h)^k \wt{\zeta}^{\ell}}{(k+1)!}
\frac{\d^k u}{\d x^k} \bigg|_{\xi_{i,\ell}(x_{i - 1/2})}, 
\quad
\wt{\zeta}^{\ell} = (-1)^{\ell} \ell ! (k - \ell) !,
\\
\label{eq:lin-rec-est-RHS}
\big( R^{\ell} \bb{u} \big)_i
&=
u( x_{i + 1/2} ) 
- 
\frac{(-h)^k \zeta^{\ell} }{(k+1)!} 
\frac{\d^k u}{\d x^k} \bigg|_{\xi_{i,\ell}(x_{i + 1/2})},
\quad
\zeta^{\ell} = (-1)^{\ell+1} (\ell+1)! (k - \ell - 1)!,
\end{align}
with $\xi_{i,\ell}(x_{i \pm 1/2})$ a pair of (unknown) points in the interior of the stencil, $\xi_{i,\ell}(x_{i \pm 1/2}) \in (x_{i - \ell - 1/2}, x_{i - \ell + (k - 1) + 1/2})$.

\end{lemma}

\begin{proof}
Because the reconstruction polynomial is the derivative of an interpolation polynomial (see \cref{SMsec:reconstruction-overview}), this proof more-or-less follows the same steps as one used to derive the error for finite-difference approximations to a first derivative (these too are derivatives of interpolation polynomials); see e.g., \cite[Section 8.1]{Hesthaven_2017} or \cite[Lemma 2]{DeSterck_etal_2023_MOL}.

Recall that for fixed $\ell$ that
$\big( \wt{R}^{\ell} \bb{u} \big)_i =q_{i}^{\ell}(x_{i-1/2})$ and $ \big( R^{\ell} \bb{u} \big)_i = q_{i}^{\ell}(x_{i+1/2})$, where $q_i^{\ell}(x) = \d Q_i^{\ell} / \d x$ is a degree at most $k-1$ polynomial recovering the cell averages of $e(x)$ over $k \geq 1$ cells $\{ {\cal I}_{i - \ell}, \ldots, {\cal I}_{i - \ell + k -1} \}$.
Moreover, $Q_i^{\ell}$ is the unique polynomial of degree at most $k$ that interpolates $U(x) := \int^x u(\xi) \d \xi$ at the $k+1$ interfaces of the cells in the above stencil: $Q_i^{\ell}(x_{j+1/2}) = U(x_{j+1/2})$ at points $\{ x_{j+1/2} \}_{j = i - \ell - 1}^{i - \ell - 1 - k}$.

For notational simplicity, let us drop the $i$ subscripts. 
Because $u$ is sufficiently smooth, we can invoke \cref{SMlem:standard-poly-interp-est} to see that, for any $x \in [x_{i - \ell - 1/2}, x_{i - \ell + k - 1/2}]$
\begin{align} \label{eq:linear-reconstruction-prim-est}
U(x) - Q^{\ell}(x) = \frac{K(x)}{(k+1)!}  \frac{\d^{k+1} U}{\d x^{k+1}} \bigg|_{\xi_{\ell}(x)},
\quad
K(x) := \prod \limits_{j = -\ell - 1}^{-\ell - 1 + k} (x - x_{j+1/2}),
\end{align}
with $\xi_{\ell}(x)$ some (unknown) point in the interior of the interpolation stencil, $\xi_{\ell}(x) \in (x_{i - \ell - 1/2}, x_{i - \ell + k - 1/2})$.
Since $\d U / \d x = u$ and we desire an estimate for $u$, we differentiate \eqref{eq:linear-reconstruction-prim-est} to get
%
\begin{align} \label{eq:linear-reconstruction-est-aux}
u(x) - q^{\ell}(x) 
= 
\frac{1}{(k+1)!} \frac{\d K}{\d x} \bigg|_{x} \frac{\d^{k} u}{\d x^{k}} \bigg|_{\xi_{\ell}(x)} 
+ 
\frac{K(x)}{(k+1)!} \frac{\d^{k+1} u}{\d x^{k+1}} \bigg|_{\xi_{\ell}(x)} \frac{\d \xi_{\ell}}{\d x} \bigg|_{x}.
\end{align}
To arrive at the estimates \eqref{eq:lin-rec-est-LHS} and \eqref{eq:lin-rec-est-RHS}, all that is required is to evaluate \eqref{eq:linear-reconstruction-est-aux} at $x = x_{-1/2}$ and $x = x_{1/2}$, respectively.
To this end, notice from \eqref{eq:linear-reconstruction-prim-est} that $K(x_{\pm 1/2}) = 0$, and, thus, the second added on the right-hand side of \eqref{eq:linear-reconstruction-est-aux} vanishes when the estimate is evaluated at $x_{\pm 1/2}$.
Next, from \eqref{eq:linear-reconstruction-prim-est} we have that
\begin{align}
\frac{\d K}{\d x} = \sum \limits_{s = -\ell - 1}^{-\ell - 1 + k} 
\hspace{2ex}
\prod \limits_{\substack{j = -\ell - 1 \\ \neq s}}^{-\ell - 1 + k} (x - x_{j+1/2}).
\end{align}
It is not difficult to show that $\d K (x_{-1/2}) / \d x = (-h)^{k} \wt{\zeta}^{\ell}$, and $\d K  (x_{1/2}) / \d x = (-h)^{k} {\zeta}_{\ell}$.
%
%
Evaluating \eqref{eq:linear-reconstruction-est-aux} at $x_{\pm 1/2}$, and plugging in the above values concludes the proof.
\end{proof}
%


\section{MGRIT solution of linear conservation laws using standard, linear method-of-lines discretization}
\label{SMsec:MGRIT-linear-standard}

In this section, we consider the MGRIT solution of the linear conservation law \eqref{eq:cons-lin} when it is discretized with standard, linear method-of-lines  discretizations, and the wave-speed $\alpha(x,t)$ is some given function, readily computable for any $x, t$.
We are \underline{not} considering the scenario described in the main paper in which the PDE \eqref{eq:cons-lin} is discretized with a linearized method-of-lines discretization, and $\alpha(x, t)$ is the linearized wave-speed of \eqref{eq:cons-law}.

The motivation for this section is twofold. 
First, it generalizes our previous work from \cite{DeSterck_etal_2023_MOL}, which used MGRIT to solve method-of-lines discretizations of constant wave-speed advection problems, because we now solve variable wave-speed conservation laws. 
Second, and most relevant to the main paper, our development in the next section, \cref{SMsec:ideal-err-est-linearized}, of the truncation error operator ${\cal T}_{\rm ideal}^{n}$ used in \eqref{eq:Psi} for the MGRIT solution of linearized method-of-lines discretizations is heavily inspired by the development of the ${\cal T}_{\rm ideal}^{n}$ in this section for standard, linear method-of-lines discretizations.

The remainder of this section is organized as follows.
\Cref{SMsec:MGRIT-linear-standard-disc} gives details of the fine-grid discretization and presents error estimates required to approximate the truncation error.
\Cref{SMsec:Psi-MGRIT-linear-standard} presents the coarse-grid operator, then numerical MGRIT results follow in \cref{SMsec:num-res-MGRIT-linear-standard}.
%

\subsection{Fine-grid discretization and required error estimates}
\label{SMsec:MGRIT-linear-standard-disc}

Suppose that the PDE \eqref{eq:cons-lin} is discretized (on the fine grid) with the method-of-lines discretizations outlined in Section \ref{sec:discretization} with the following specific details.

\begin{itemize}

\item The LF numerical flux \eqref{eq:LFF-lin} is used wherein the wave-speed reconstructions $\alpha_{i+1/2}^{\pm}$ are replaced with exact evaluations of the given wave-speed function $\alpha(x, t)$. Dropping temporal dependence for readability, the numerical flux is
\begin{align} 
\label{eq:LFF-lin-exact-eval}
\wh{f}_{i+1/2}
&=
\frac{1}{2}
\left[
\alpha(x_{i+1/2}) \big( e_{i+1/2}^{-} + e_{i+1/2}^{+} \big)
+
\nu_{i+1/2} \big( e_{i+1/2}^{-} - e_{i+1/2}^{+} \big)
\right],
\\
\label{eq:LFF-lin-exact-eval-err-form}
&=
\alpha(x_{i+1/2}) e(x_{i+1/2}) + \wh{E}_{i+1/2},
\end{align}
where, recall from \cref{ass:E-est}, $\wh{E}_{i+1/2}$ is the error in the flux.

\item The reconstructions $e_{i+1/2}^{\pm}$ in the numerical flux \eqref{eq:LFF-lin-exact-eval} are computed with polynomial reconstruction on central stencils using $2k-1$ cell averages (specifics of polynomial reconstructions can be recalled from Sections \ref{sec:reconstruction} and \ref{SMsec:reconstruction-overview}). 
Thus, for sufficiently smooth $e(x)$, the reconstructions are of order $p = 2k-1$, $k = 1, 2, \ldots$. 

\item If using GLF, the dissipation coefficient in \eqref{eq:LFF-lin-exact-eval} is $\nu^{\rm (global)} = \max_{x, t} |\alpha(x, t)|$, with the range of $x,t$ taken as the whole space-time domain, and if using LLF $\nu_{i+1/2}^{n} = |\alpha(x_{i+1/2}, t_n)|$.

\end{itemize} 

With details of the discretization specified, we now estimate the error term in \eqref{eq:LFF-lin-exact-eval-err-form} such that the analysis from \cref{SMsec:ideal-err-est} can be applied to estimate the truncation error of the ideal coarse-grid operator.
To this end, consider the following corollary of \cref{SMlem:est-lin-rec} to express the reconstructions used in \eqref{eq:LFF-lin-exact-eval}.
\begin{corollary}[Error estimate for central, linear reconstructions using $2k-1$ cells]
\label{SMcor:est-rec-optimal}
Let $e_{i, \pm 1/2}$ be polynomial reconstructions of the function $e(x_{i \pm 1/2})$ using central stencils containing $2k - 1$ cells, $S(i) = \{ {\cal I}_{i-(k-1)}, \ldots, {\cal I}_{i+(k-1)} \}$.
Suppose that $e(x)$ is $2k$ times differentiable over the cells in this stencil.
Then, 
\begin{align} \label{eq:lin-rec-big-est-LHS}
e_{i, - 1/2}
&=
e ( x_{i - 1/2} ) 
- 
(-1)^k h^{2k-1}
\frac{ k ! (k-1)!  }{(2k)!}
\frac{\d^{2k-1} e}{\d x^{2k-1}} \bigg|_{\xi_{i}(x_{i - 1/2})}, 
\\
\label{eq:lin-rec-big-est-RHS}
e_{i, + 1/2}
&=
e ( x_{i + 1/2} ) 
- 
(-1)^{k+1} h^{2k-1}
\frac{ k ! (k-1)!  }{(2k)!}
\frac{\d^{2k-1} e}{\d x^{2k-1}} \bigg|_{\xi_{i}(x_{i + 1/2})},
\end{align}
with $\xi_{i}(x_{i \pm 1/2})$ a pair of (unknown) points in the interior of the stencil, $\xi_{i}(x_{i \pm 1/2}) \in \big( x_{i - (k-1) - 1/2}, x_{i + (k-1) + 1/2} \big)$.
\end{corollary} 

\begin{proof}
%
%
The central stencil considered here corresponds to the stencil from \cref{SMlem:est-lin-rec} with a number cells $k \mapsto 2k-1$, and left shift $\ell = k-1$.
Plugging these parameters into the results of \cref{SMlem:est-lin-rec} and simplifying gives the claim.
\end{proof}

Now we estimate the error term of the numerical flux, as it appears in \eqref{eq:LFF-lin-exact-eval-err-form}.
\begin{lemma}[Error estimate for numerical flux using central linear reconstructions]
\label{SMlem:num-flux-est}
Consider the numerical flux \eqref{eq:LFF-lin-exact-eval}, and suppose that the polynomial reconstructions $e_{i+1/2}^{\pm}$ use central stencils containing $2k - 1$ cells.
Suppose that $e(x)$ is sufficiently smooth over the stencils in question.
Then, the error of the numerical flux, as defined in \eqref{eq:LFF-lin-exact-eval-err-form}, satisfies
\begin{align}
\wh{E}_{i+1/2}
=
\frac{\nu_{i + 1/2}}{2}
\big(
e_{i+1/2}^{-} - e_{i+1/2}^{+}
\big)
+
{\cal O}(h^{2k}),
\end{align}
with
\begin{align} \label{eq:e-diff-est}
e_{i+1/2}^{-} - e_{i+1/2}^{+}
=
(-1)^k h^{2k-1} 
\frac{2 (k-1)! k! }{(2k)!} 
\frac{\d^{2k-1} e}{\d x^{2k-1}} \bigg|_{x_{i + 1/2}}
+
{\cal O}(h^{2k}).
\end{align}

\end{lemma}

\begin{proof}
%
%
Using the notation of \cref{SMcor:est-rec-optimal}, the reconstructions considered in the lemma are $e_{i+1/2}^{-} = e_{i, +1/2}$, and $e_{i+1/2}^{+} = e_{i+1, -1/2}$. First consider the sum $e_{i+1/2}^{-} + e_{i+1/2}^{+}$ in the flux \eqref{eq:LFF-lin-exact-eval}; plugging in the estimates from \cref{SMcor:est-rec-optimal} we have
\begin{align} \label{eq:flux-est-rec-sum}
\begin{split}
&e_{i+1/2}^{-} + e_{i+1/2}^{+}
=
2 e(x_{i+1/2})
\\
&\quad 
+
(-1)^{k+1} h^{2k-1}
\frac{ k ! (k-1)!  }{(2k)!}
\left[
\frac{\d^{2k-1} e}{\d x^{2k-1}} \bigg|_{\xi_{i+1}(x_{i + 1/2})}
-
\frac{\d^{2k-1} e}{\d x^{2k-1}} \bigg|_{\xi_{i}(x_{i + 1/2})}
\right],
\end{split}
\end{align}
recalling that the points $\xi_{i}(x_{i + 1/2}) $ and $\xi_{i+1}(x_{i + 1/2})$ are not known.
However, by assumption of smoothness of $e$, from the mean value theorem we have that
\begin{align}
\frac{\d^{2k-1} e}{\d x^{2k-1}} \bigg|_{\xi_{i+1}(x_{i + 1/2})}
-
\frac{\d^{2k-1} e}{\d x^{2k-1}} \bigg|_{\xi_{i}(x_{i + 1/2})}
=
\frac{\d^{2k} e}{\d x^{2k}} \bigg|_c
\big[
\xi_{i+1}(x_{i + 1/2}) - \xi_{i}(x_{i + 1/2})
\big]
\end{align}
for some $\min \big( \xi_{i+1}(x_{i + 1/2}), \xi_{i}(x_{i + 1/2}) \big) < c < \max\big( \xi_{i+1}(x_{i + 1/2}), \xi_{i}(x_{i + 1/2}) \big)$.
Next, note that $\xi_{i+1}(x_{i + 1/2}) - \xi_{i}(x_{i + 1/2}) = {\cal O}(h)$ because both of these points are contained within an interval of length ${\cal O}(h)$.
Plugging into \eqref{eq:flux-est-rec-sum} gives $e_{i+1/2}^{-} + e_{i+1/2}^{+}
=
2 e(x_{i+1/2}) + {\cal O}(h^{2k})$.

Now consider the difference $e_{i+1/2}^{-} - e_{i+1/2}^{+}$ in the numerical flux \eqref{eq:LFF-lin-exact-eval}. 
Using the estimates from \cref{SMcor:est-rec-optimal} we find
\begin{align} 
e_{i+1/2}^{-} - e_{i+1/2}^{+}
&=
(-1)^{k} h^{2k-1}
\frac{ k ! (k-1)!  }{(2k)!}
\left[
\frac{\d^{2k-1} e}{\d x^{2k-1}} \bigg|_{\xi_{i+1}(x_{i + 1/2})}
+
\frac{\d^{2k-1} e}{\d x^{2k-1}} \bigg|_{\xi_{i}(x_{i + 1/2})}
\right],
\\
\label{eq:flux-est-rec-diff}
&=
(-1)^{k} h^{2k-1}
\frac{ 2 (k-1)!  k !  }{(2k)!}
\left[ 
\frac{\d^{2k-1} e}{\d x^{2k-1}} \bigg|_{x_{i + 1/2}}
+
{\cal O}(h)
\right],
\end{align}
with the second expression following by Taylor expansion, again, due to the fact that $\xi_{i+1}(x_{i + 1/2}) + {\cal O}(h) = \xi_{i}(x_{i + 1/2}) + {\cal O}(h) = x_{i + 1/2}$.
Plugging \eqref{eq:flux-est-rec-diff} and the earlier result of $e_{i+1/2}^{-} + e_{i+1/2}^{+}
=
2 e(x_{i+1/2}) + {\cal O}(h^{2k})$ into \eqref{eq:LFF-lin-exact-eval} concludes the proof.
\end{proof}

\subsection{Modified semi-Lagrangian coarse-grid operator}
\label{SMsec:Psi-MGRIT-linear-standard}

For a a two-level MGRIT solve, we use a coarse-grid operator that has the same structure as that in \eqref{eq:Psi}, repeated here for convenience:
\begin{align} \label{eq:Psi-copy}
\Psi^{n}
=
\Big[
I 
+ 
{\cal T}_{\rm ideal}^{n}
-
{\cal T}_{\rm direct}^{n}
\Big]^{-1}
 {\cal S}^{n, m\delta t}_{p, 1_*}
 \approx
  \Psi_{\rm ideal}^{n}.
\end{align}
The approximate, direct truncation operator ${\cal T}_{\rm direct}^{n}$ is given by \eqref{eq:Tcal-SL}---more generally, see \cref{SMsec:SL-trun-err} for details on how this approximation arises.
We now discuss our choice for the approximate the ideal truncation error operator ${\cal T}_{\rm ideal}^{n}$.
Other details on the coarse-grid operator \eqref{eq:Psi-copy} can be found in Section \ref{sec:Psi-overview}.



Recall \cref{eq:trunc-est-int-cons-ideal}, which is an approximate error estimate for a method-of-lines discretization for an approximation to the PDE at hand \eqref{eq:cons-lin} in which the wave-speed function $\alpha(x, t)$ is frozen somewhere inside each of the time intervals $t \in [t_{n+j}, t_{n+1+j}]$, $j \in \{0, \ldots, m-1 \}$.
Based on this error estimate, which applies approximately to the problem at hand, we choose ${\cal T}_{\rm ideal}^{n}$ to take the form
\begin{align} \label{eq:T-ideal-approx}
{\cal T}_{\rm ideal}^{n}
=
\textrm{approximation}
\bigg(
\sum_{j = 0}^{m-1}
\Big( \delta t \wh{{\cal E}}^{n+j} - \wh{e}_{\rm RK} \big( \delta t \wh{L}^{n+j} \big)^{q+1} \Big)
\bigg).
\end{align}
The ``approximation'' here represents the fact that the terms in the sum need to be approximated, and that the choice for doing so is not unique.
Recall that $\wh{L}^{n+j} \in \mathbb{R}^{n_x \times n_x}$ is the spatial discretization at some time $t = t_{n+j} + \delta t \vartheta$, and $\wh{{\cal E}}^{n+j} \in \mathbb{R}^{n_x \times n_x}$ is the associated spatial discretization error at this time; see \cref{ass:E-est}.
The question at hand is how to approximate $\wh{L}^{n+j}$ and $\wh{{\cal E}}^{n+j}$.
It is perhaps simplest to start with $\wh{{\cal E}}^{n+j}$.

\textbf{Approximating the spatial discretization error.} The exact form of $\wh{{\cal E}}^{n+j}$ depends on the numerical flux.
Recall from \cref{SMlem:num-flux-est} of \cref{SMsec:MGRIT-linear-standard-disc} that the numerical flux at hand satisfies the estimate (dropping temporal superscripts for readability)
\begin{align}
\wh{f}_{i+1/2}
=
\alpha(x_{i+1/2}) e(x_{i+1/2})
+
\nu_{i + 1/2}
(-1)^k h^{2k-1} 
\frac{k! (k-1)!}{(2k)!} 
\frac{\d^{2k-1} e}{\d x^{2k-1}} \bigg|_{x_{i + 1/2}}
+
{\cal O}(h^{2k}).
\end{align}
Thus, based on the definition of $\wh{{\cal E}}$ given in \cref{ass:E-est}, we find that it satisfies
\begin{align}
\big( \wh{{\cal E}} \bb{e} \big)_i
&=
-\frac{\alpha(x_{i+1/2}) e(x_{i+1/2}) - \alpha(x_{i-1/2}) e(x_{i-1/2})}{h}
+
\frac{\wh{f}_{i+1/2} - \wh{f}_{i-1/2}}{h},
\\
&\approx
(-1)^{k+1} h^{2k-2} 
\frac{k! (k-1)!}{(2k)!} 
\left[
\nu_{i + 1/2}
\frac{\d^{2k-1} e}{\d x^{2k-1}} \bigg|_{x_{i + 1/2}}
-
\nu_{i - 1/2}
\frac{\d^{2k-1} e}{\d x^{2k-1}} \bigg|_{x_{i - 1/2}}
\right].
\end{align}
Whether or not the terms enumerated here accurately capture the leading-order terms in the $\big( \wh{{\cal E}} \bb{e} \big)_i$ depends on the smoothness of the numerical flux.\footnote{That is, without further analysis, it is not clear whether the difference of the ${\cal O}(h^{2k})$ terms present in both $\wh{f}_{i+1/2}$ and $\wh{f}_{i-1/2}$ leads to terms which are of the same size as those enumerated here, or whether they are one order smaller.}
In any event, we find this expression is sufficient for our purposes. 
We then implement numerically an approximation to this formula derived from approximating the degree $2k-1 = p$ derivatives with finite differences.
Specifically, we use  
\begin{align} \label{eq:E-approx-optimal-linear}
\sum_{j = 0}^{m-1}
\wh{{\cal E}}^{n+j} 
\approx
(-1)^{k+1} h^{p} 
\frac{k! (k-1)!}{(2k)!} 
{\cal D}_1
\diag 
\bigg(
\sum \limits_{j=0}^{m-1}
\bm{\nu}^{n + j}  
\bigg)
{\cal D}_{p}^\top,
\end{align}
where $\bm{\nu}^{n+j} = \big( \nu_{1+1/2}^{n+j}, \nu_{2+1/2}^{n+j}, \ldots, \nu_{n_x+1/2}^{n+j} \big) \in \mathbb{R}^{n_x}$ is the vector of dissipation coefficients used in the LF numerical flux \eqref{eq:LFF-lin} at time $t_{n + j} + \delta t \vartheta$.
Recall from \cref{SMsec:cons-lin-approx} that to develop a truncation error estimate for the linear problem, its time dependence is frozen somewhere in the interval $[t_n, t_n + \delta t]$, and the parameter $\vartheta \in [0, 1]$ controls the specific freezing point.

\textbf{Approximating powers of the spatial discretization.}
It is less obvious how to approximate $\big( \wh{L}^{n+j} \big)^{q+1}$, the term in \eqref{eq:T-ideal-approx} corresponding to the ($q+1$)st power of the spatial discretization matrix. 
We do not want to compute this exactly, since the implementation of the discretization does not even require computing $\wh{L}^{n+j}$ as a matrix, so, explicitly forming it and then raising it to the ($q+1$)st power seems too expensive.
Thus, we seek to cheaply approximate it in a way that is consistent with the approximation that we have for the spatial error matrix in \eqref{eq:E-approx-optimal-linear}.

To this end, let us recall what in fact that spatial discretization approximates.
Ignoring temporal dependence for notational simplicity, we have that 
\begin{align}
\big( \wh{L} \bb{e} \big)_i
&=
-\frac{\alpha(x_{i+1/2}) e(x_{i+1/2}) - \alpha(x_{i-1/2}) e(x_{i-1/2})}{h}
+
\rm{h.o.t}.,
\\
&=
-\frac{\partial}{\partial x}
\Big(
\alpha(x)
e(x)
\Big)
\bigg|_{x_i}
+
\rm{h.o.t}.
\end{align}
That is, even though $\wh{L}$ is not designed to approximate a derivative of $\alpha e$ (as it would if we were using a FD method rather than an FV method), to lowest order it in fact does.
Thus, for the $q+1$ powers of $\wh{L}$ we must have something like:
\begin{align}
\Big( (\wh{L})^{q+1} \bb{e} \Big)_i
=
(-1)^{q+1}
\underbrace{
\frac{\partial}{\partial x}
\Big(
\alpha(x)
\Big(
\cdots
\frac{\partial}{\partial x}
\Big(
\alpha(x)
}_{q \textrm{ terms}}
\Big(
\frac{\partial}{\partial x}
\Big(
\alpha(x)
e(x)
\Big)
\Big)
\Big)
\cdots
\Big)
\Big)
\bigg|_{x_i}
+
\rm{h.o.t}.
\end{align}
One could now approximate this by replacing the derivatives with FD approximations, this is still too expensive for our purposes, however.
Instead, we seek an approximation that uses only the largest derivative of $e$. 
This choice is loosely motivated by our previous results in \cite{DeSterck_etal_2023_SL}, in which we found it was only necessary (in terms of MGRIT convergence) to crudely approximate the truncation error by ensuring that the dominant dissipation term is captured.
By the product rule, the term with the largest derivative of $e$ is$  $
\begin{align}
\Big( \wh{L}^{q+1} \bb{e} \Big)_i
\approx
(-1)^{q+1}
\frac{\partial}{\partial x}
\Big( \alpha^{q+1}
\frac{\partial^{q} e}{\partial x^{q}}
\Big)
\bigg|_{x_i}
=
(-1)^{q+1}
\Big(
{\cal D}_1
\diag ( \bm{\alpha} )^{q+1}
{\cal D}_q^\top
\bb{e}
\Big)_i
+
\rm{h.o.t}.
\end{align}
In this final expression, $\bm{\alpha} = \big( \alpha_{1 + 1/2}, \ldots, \alpha_{n_x + 1/2} \big)^\top \in \mathbb{R}^{n_x}$ is a vector comprised of the wave-speed function evaluated at cell interfaces and powers of it are computed element-wise.
Thus, for the second term in \eqref{eq:T-ideal-approx} we use the approximation
\begin{align} \label{eq:L-powers-approx}
\sum \limits_{j = 0}^{m-1}
\big( \wh{L}^{n+j} \big)^{q+1}
\approx
(-1)^{q+1}
{\cal D}_1
\diag \bigg( \sum \limits_{j = 0}^{m-1} \big( \bm{\alpha}^{n+j} \big)^{q+1} \bigg)
{\cal D}_q^\top,
\end{align}
where $\bm{\alpha}^{n+j} \in \mathbb{R}^{n_x}$ is the wave-speed function evaluated at cell interfaces at time $t_{n+j} + \delta t \vartheta$.

\textbf{Piecing it all together.}
Plugging \eqref{eq:E-approx-optimal-linear} and \eqref{eq:L-powers-approx} into \eqref{eq:T-ideal-approx} results in the form of ${\cal T}_{\rm ideal}^{n}$ that we implement numerically:
\begin{align} \label{eq:T-ideal-linear}
\begin{split}
{\cal T}_{\rm ideal}^{n}
&=
{\cal D}_1
\bigg[
h^{p}  \delta t 
(-1)^{k+1} 
\frac{k! (k-1)!}{(2k)!} 
\diag 
\bigg(
\sum \limits_{j=0}^{m-1}
\bm{\nu}^{n + j}  
\bigg)
{\cal D}_{p}^\top
\\
&\hspace{12ex} 
+
(-\delta t)^{q+1}
\wh{e}_{\rm RK}
\diag \bigg( \sum \limits_{j = 0}^{m-1} \big( \bm{\alpha}^{n+j} \big)^{q+1} \bigg)
{\cal D}_q^\top
\bigg].
\end{split}
\end{align}

\subsection{Numerical results}
\label{SMsec:num-res-MGRIT-linear-standard}

\renewcommand{\fd}{./figures/}
\renewcommand{\vs}{2}
\renewcommand{\hs}{4}
\renewcommand\fs{0.28}

\begin{figure}[h!]
\centerline{
\includegraphics[scale=\fs]{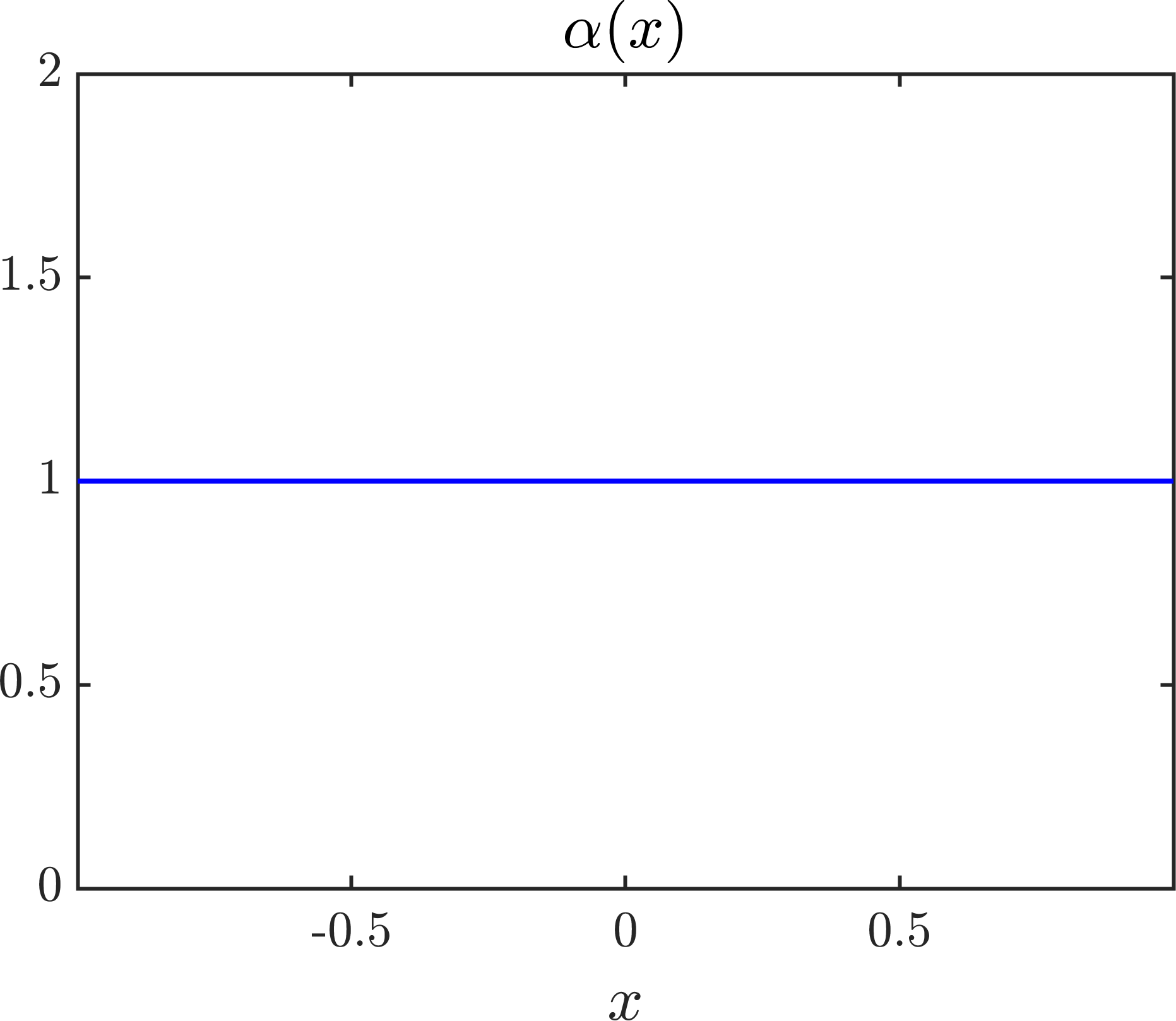}
\hspace{\hs ex}
\includegraphics[scale=\fs]{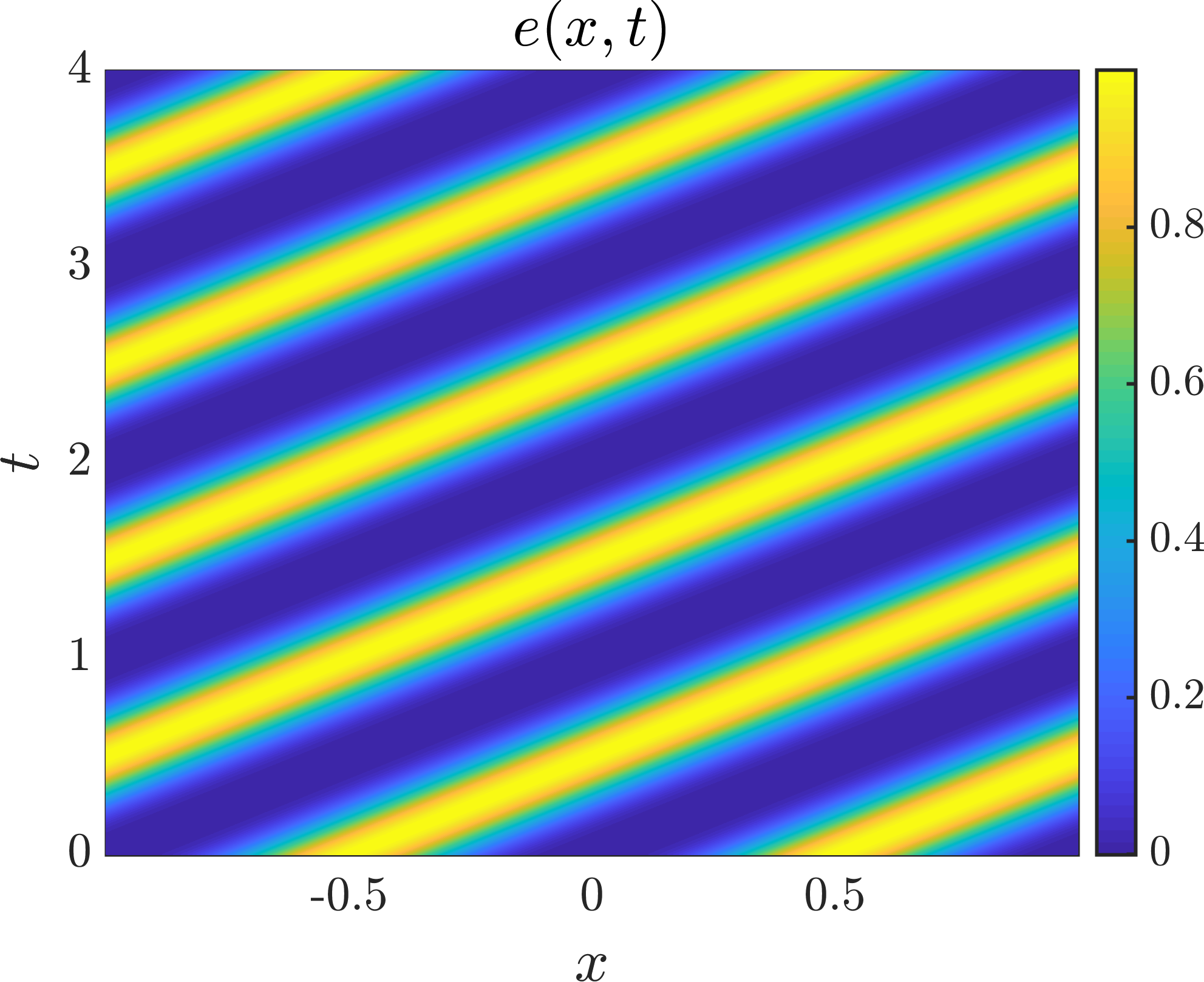}
}
\vspace{\vs ex}
\centerline{
\includegraphics[scale=\fs]{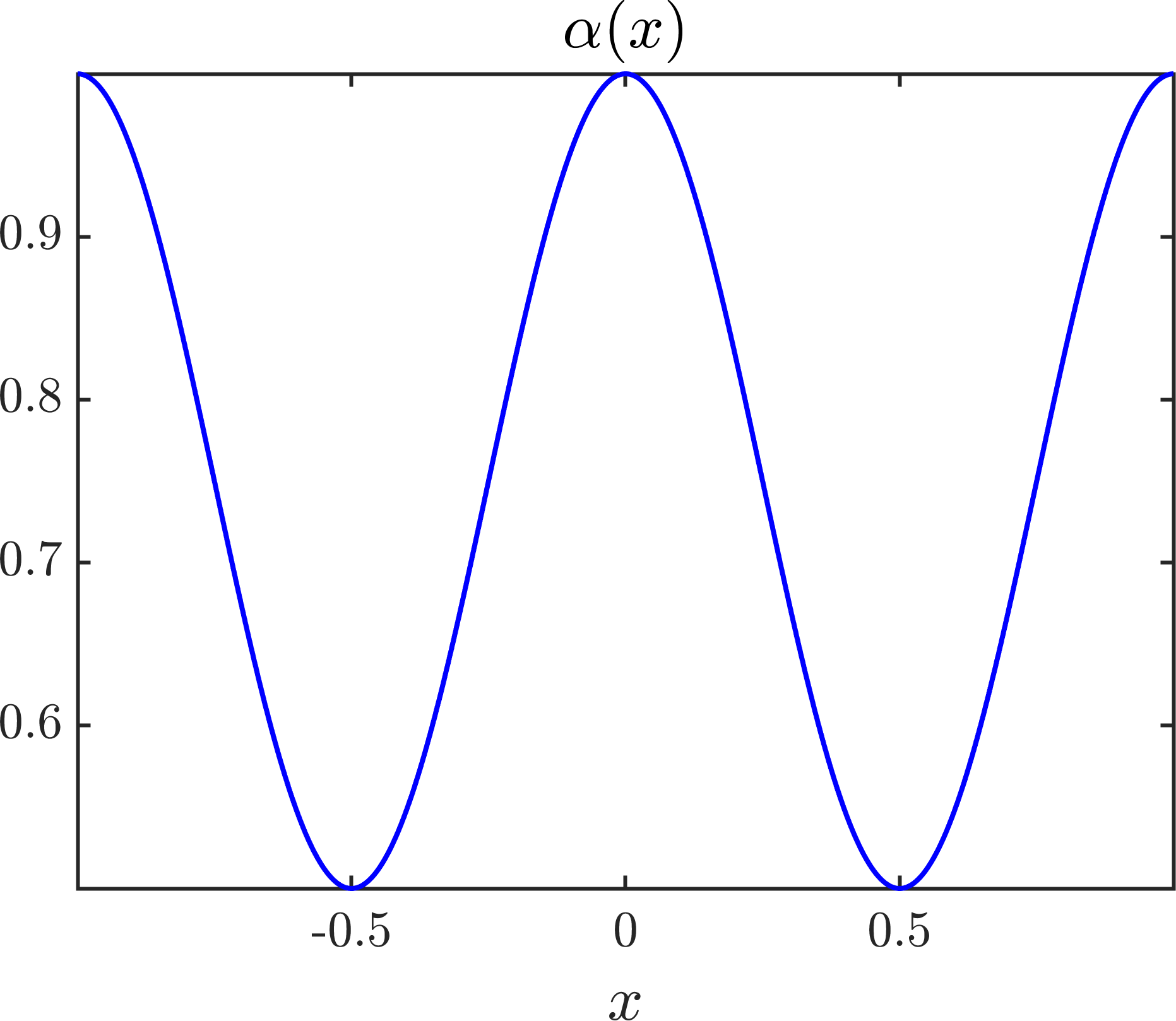}
\hspace{\hs ex}
\includegraphics[scale=\fs]{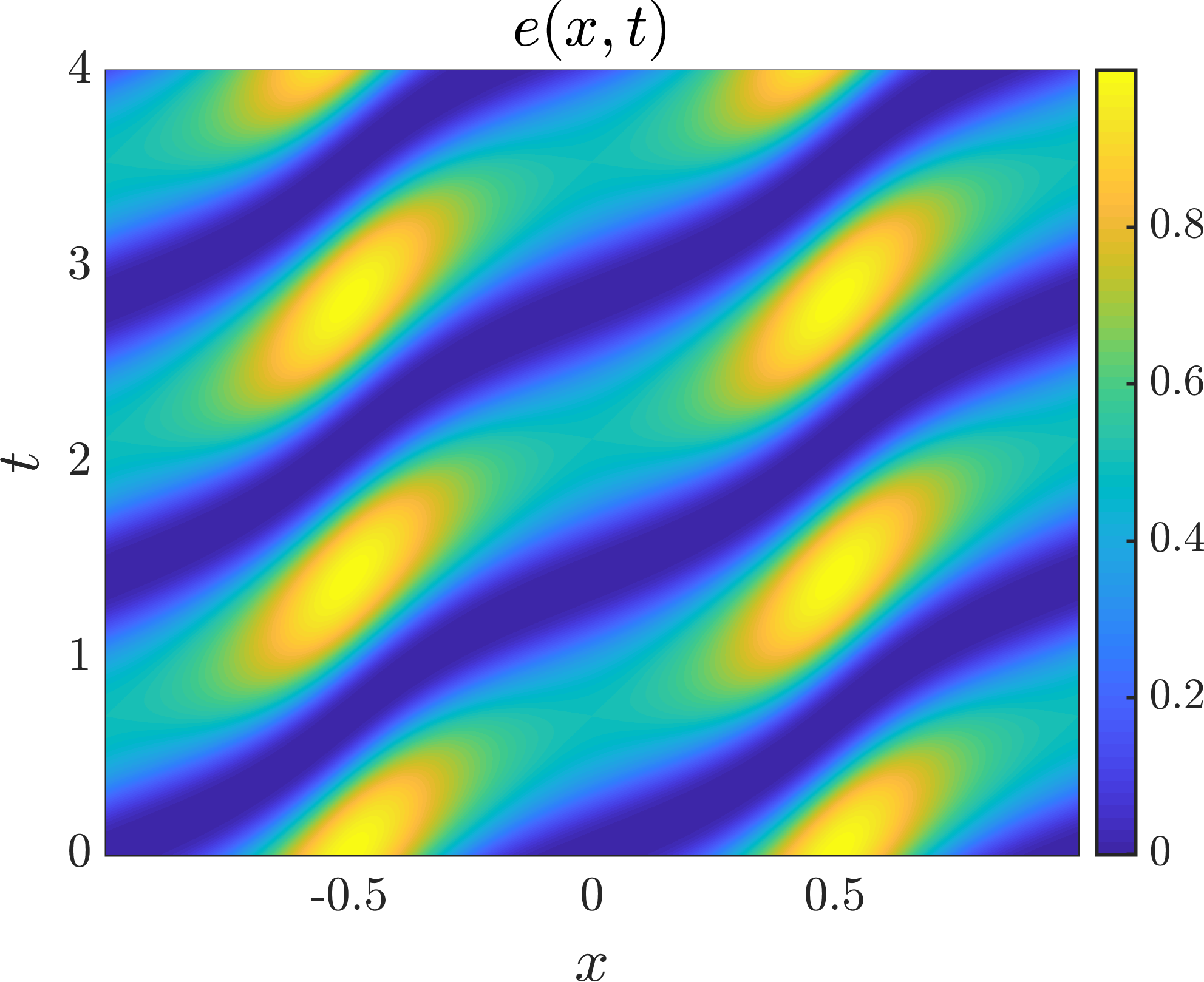}
}
\vspace{\vs ex}
\centerline{
\includegraphics[scale=\fs]{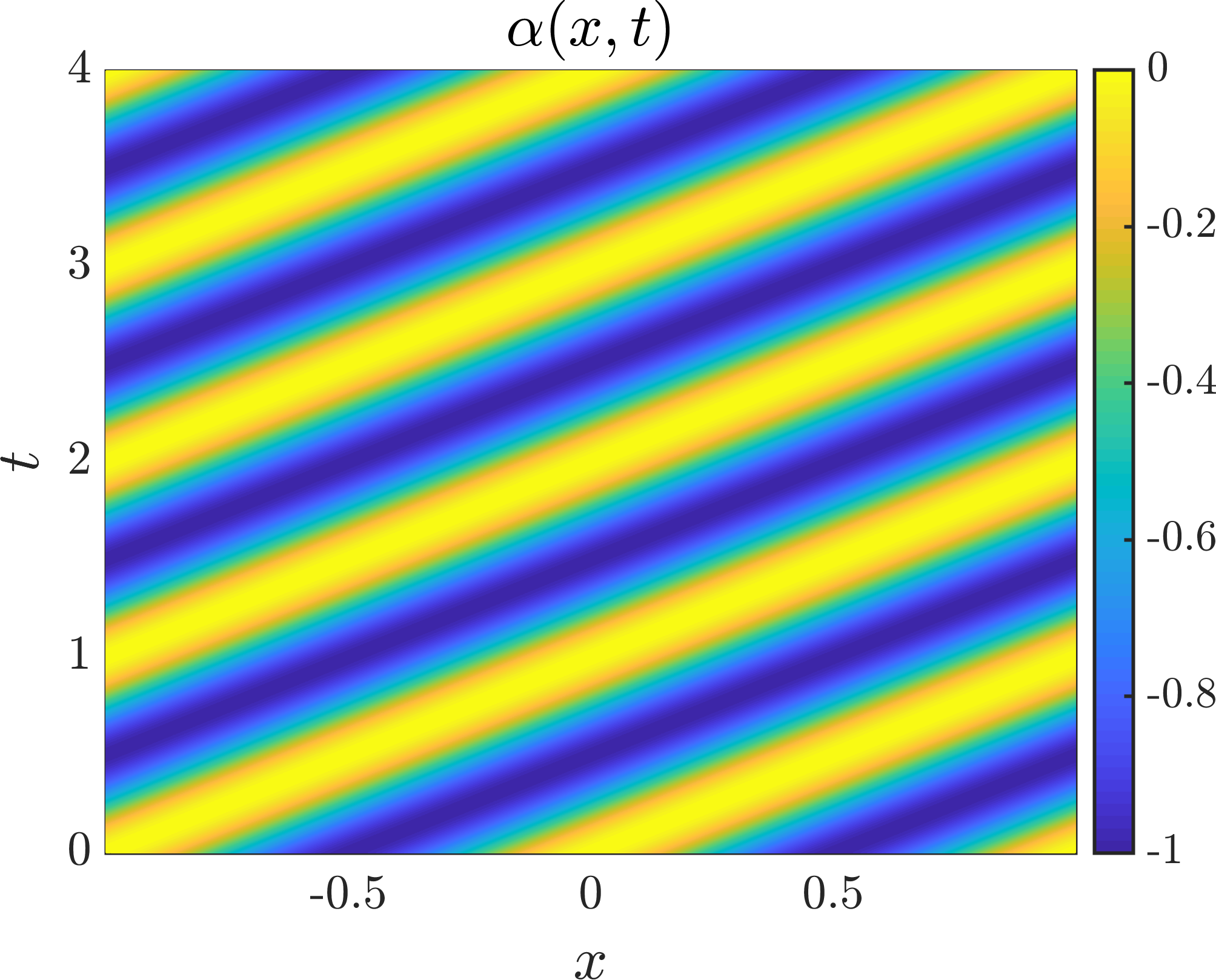}
\hspace{\hs ex}
\includegraphics[scale=\fs]{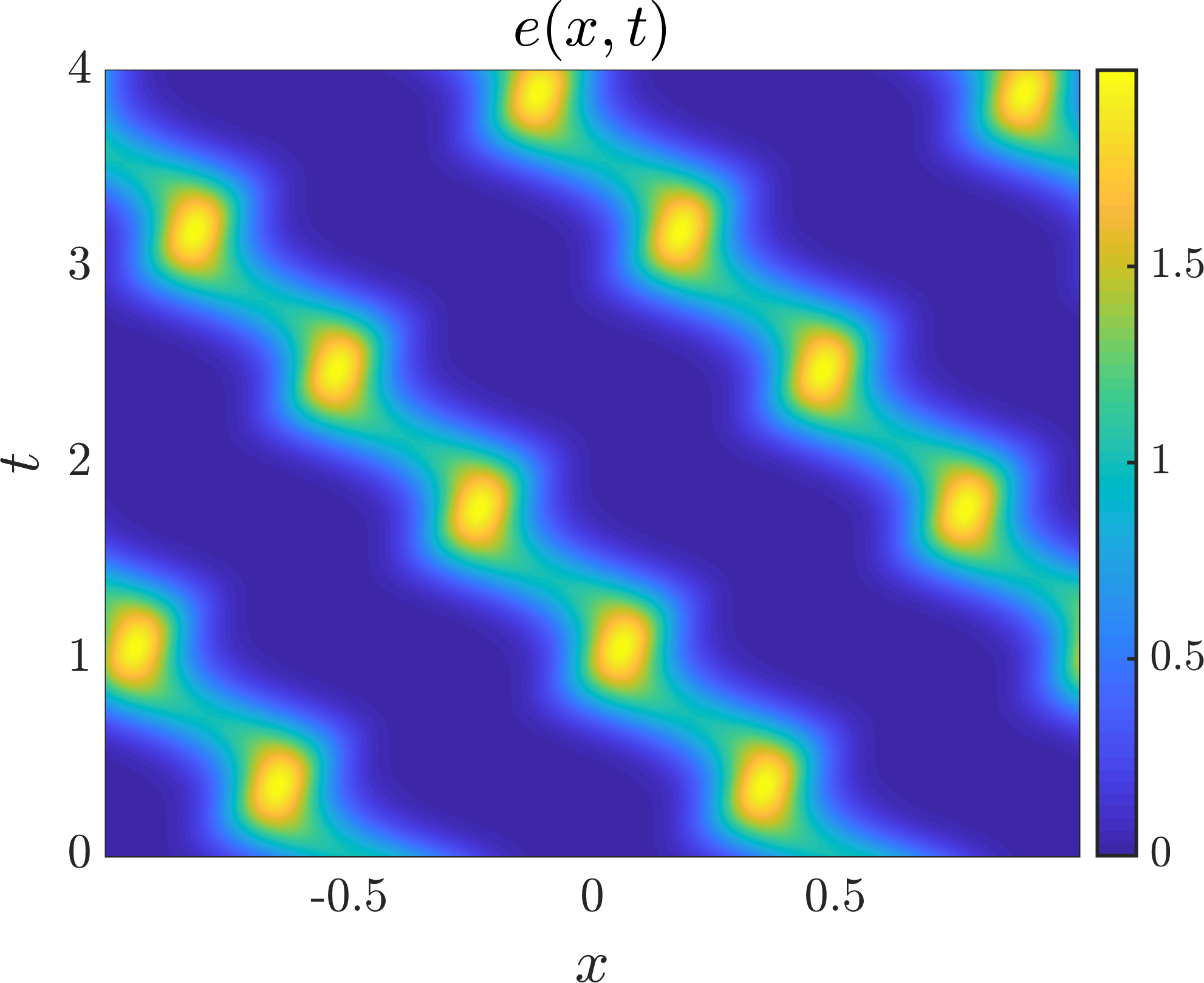}
}
\vspace{\vs ex}
\centerline{
\includegraphics[scale=\fs]{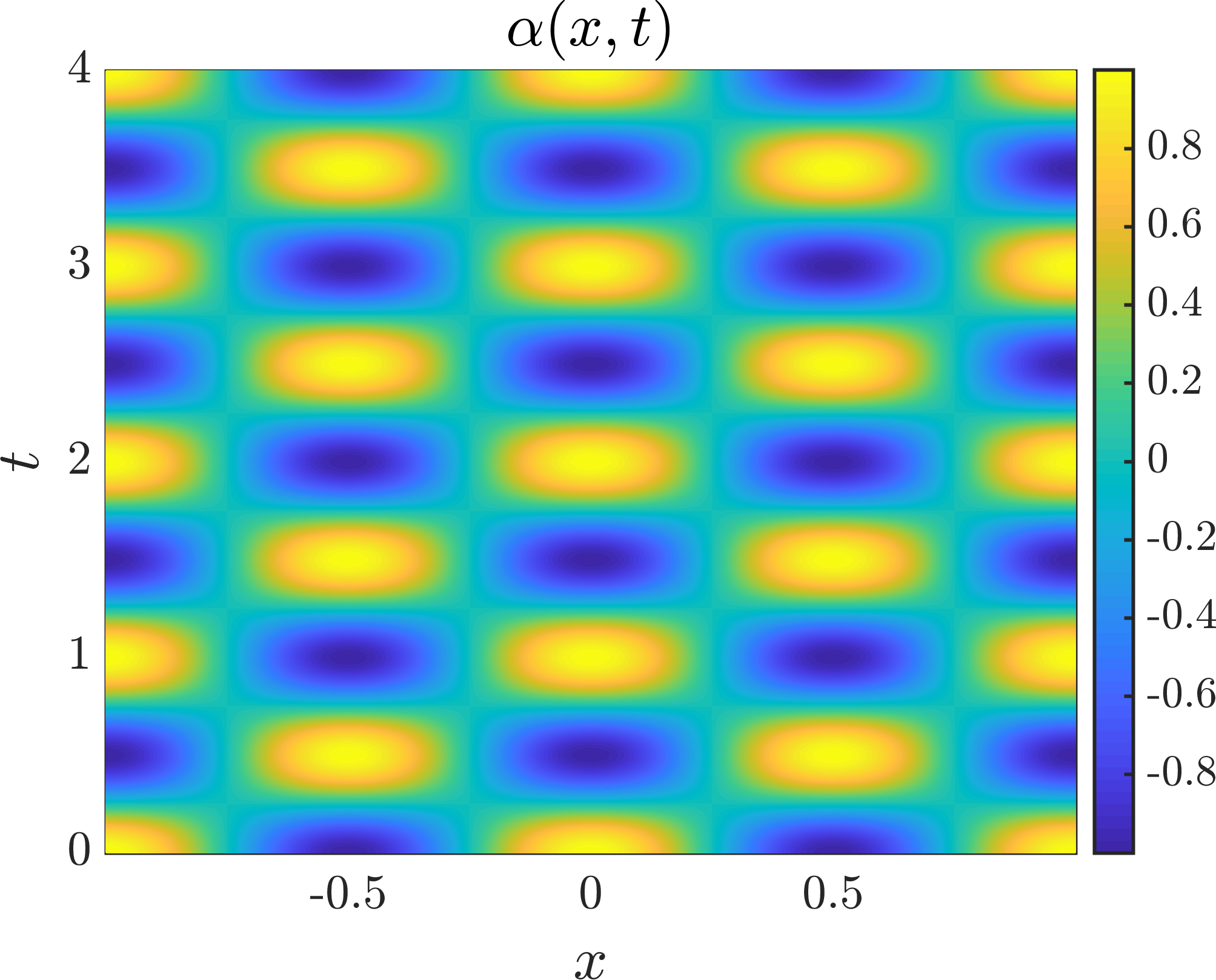}
\hspace{\hs ex}
\includegraphics[scale=\fs]{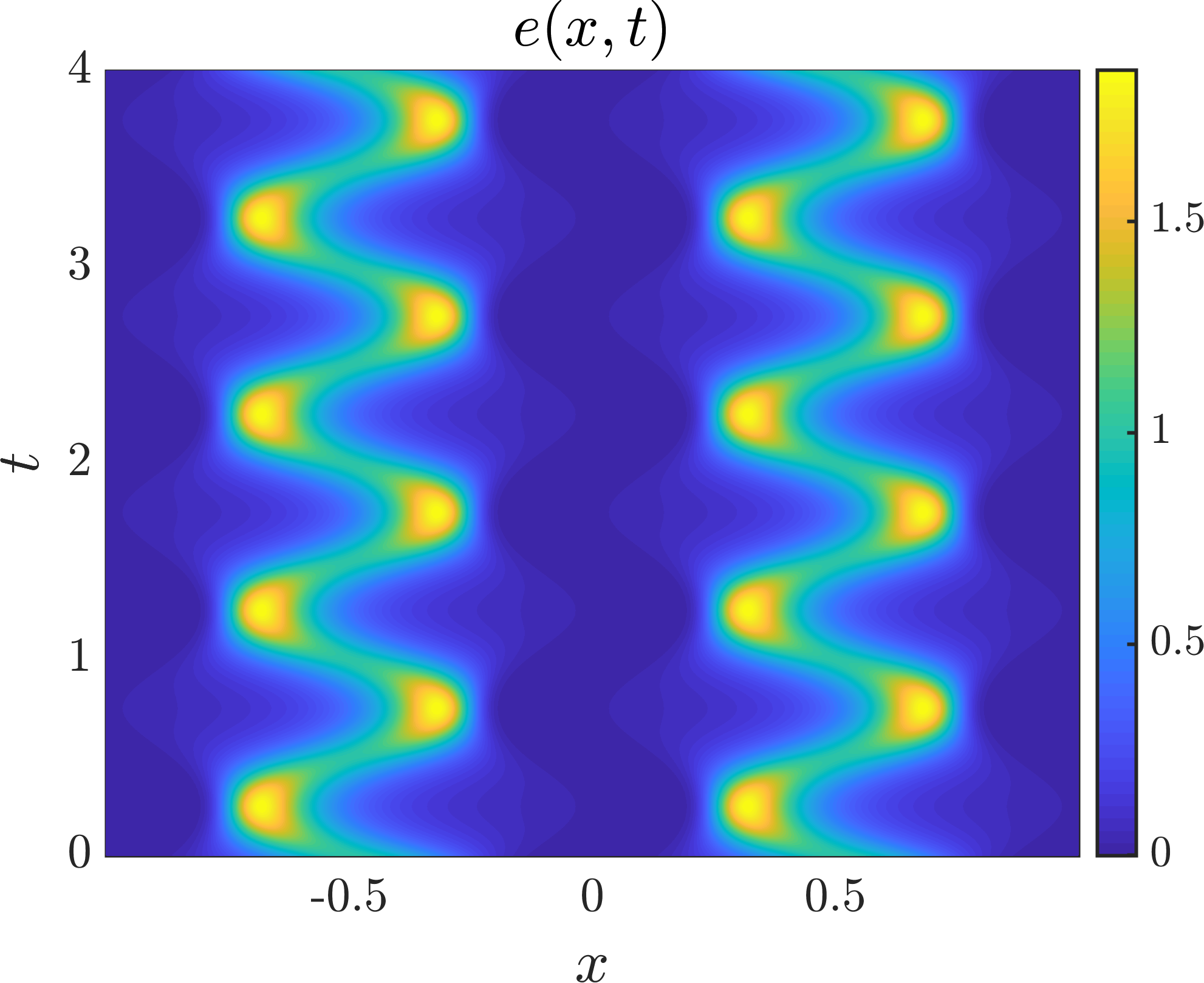}
}
\caption{Numerical test problems for MGRIT applied to solve the linear conservation law \eqref{eq:cons-lin} discretized with standard linear method-of-lines discretization (\cref{SMsec:MGRIT-linear-standard}).
Left column: The wave-speed function $\alpha(x, t)$. Right column: Space-time contour of the corresponding solution $e(x, t)$ of the PDE \eqref{eq:cons-lin}.
From top to bottom, the wave-speed functions are: 
$\alpha(x) = 1$, 
$\alpha(x) = \frac{1}{2} \big( 1 + \cos^2(\pi x) \big)$, 
$\alpha(x, t) = - \sin^2 \big( \pi (x - t) \big)$, and
$\alpha(x, t) = \cos(2 \pi x) \cos(2 \pi t)$, respectively. 
\label{SMfig:MGRIT-lin-test-prob}
}
\end{figure}

In this section, we show results for MGRIT used to solve discretizations of the PDE \eqref{eq:cons-lin} for four different wave-speed functions $\alpha$. The different wave-speed functions are shown in \cref{SMfig:MGRIT-lin-test-prob}, along with contours of the corresponding PDE solution in space-time. 
In all cases, the PDE initial condition is $e(x, 0) = \sin^4 (\pi x)$, the time domain is $t \in [0, 4]$, and the (fine-grid) time-step is $\delta t = 0.85 h$.

\renewcommand{\fd}{./figures//}
\renewcommand{\vs}{2}
\renewcommand{\hs}{4}
\renewcommand\fs{0.28}

\begin{figure}[h!]
\centerline{
\includegraphics[scale=\fs]{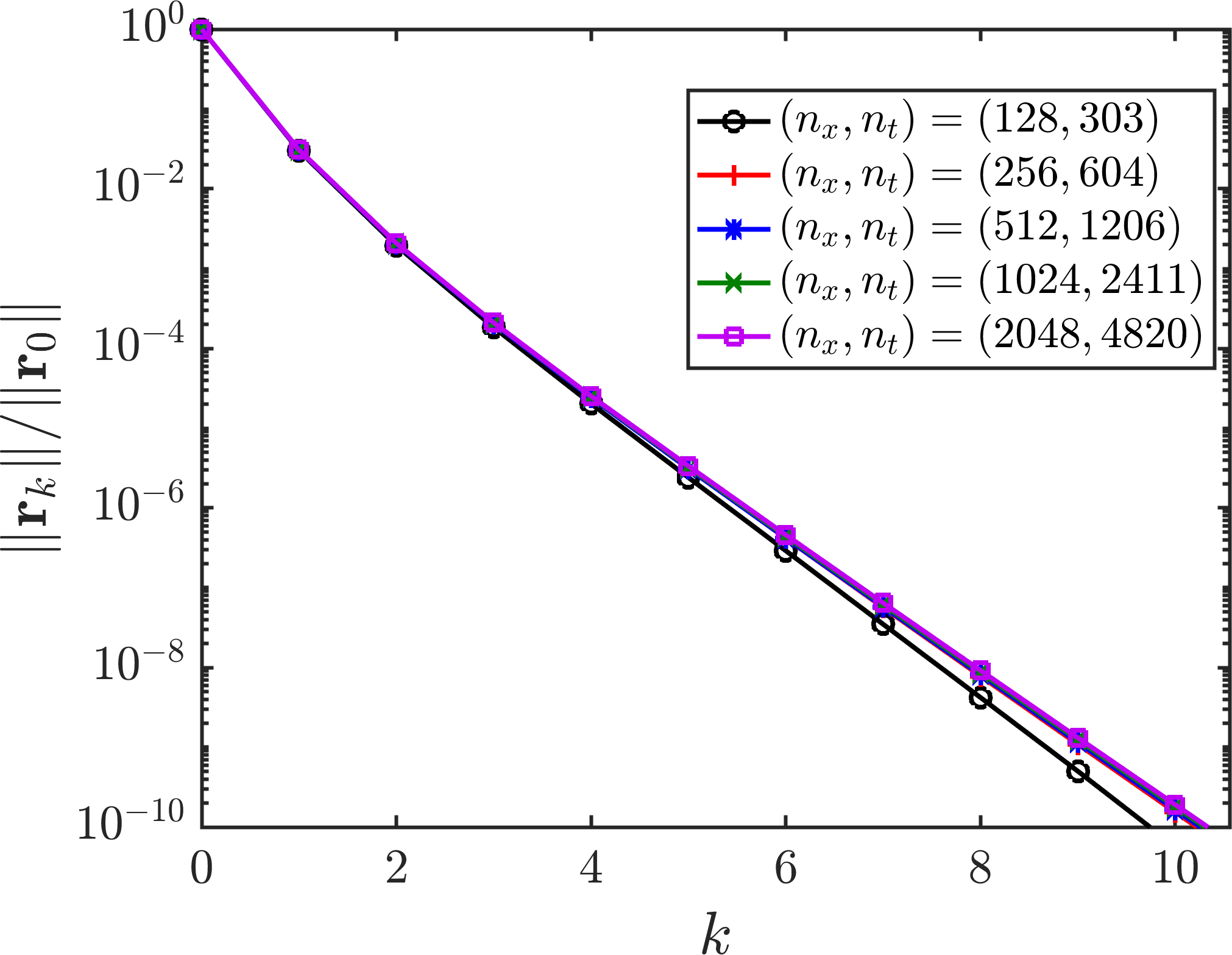}
\hspace{\hs ex}
\includegraphics[scale=\fs]{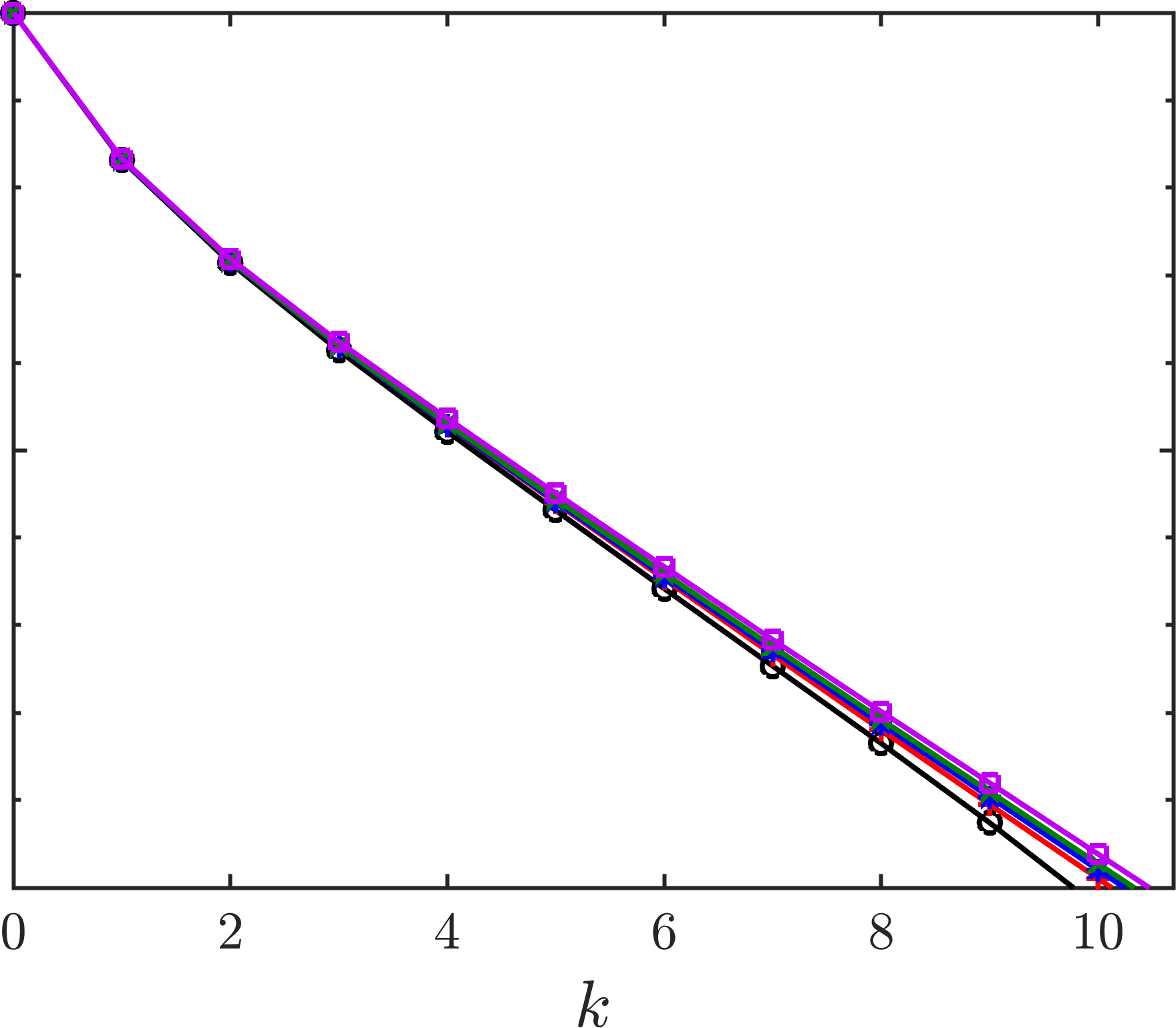}
}
\vspace{\vs ex}
\centerline{
\includegraphics[scale=\fs]{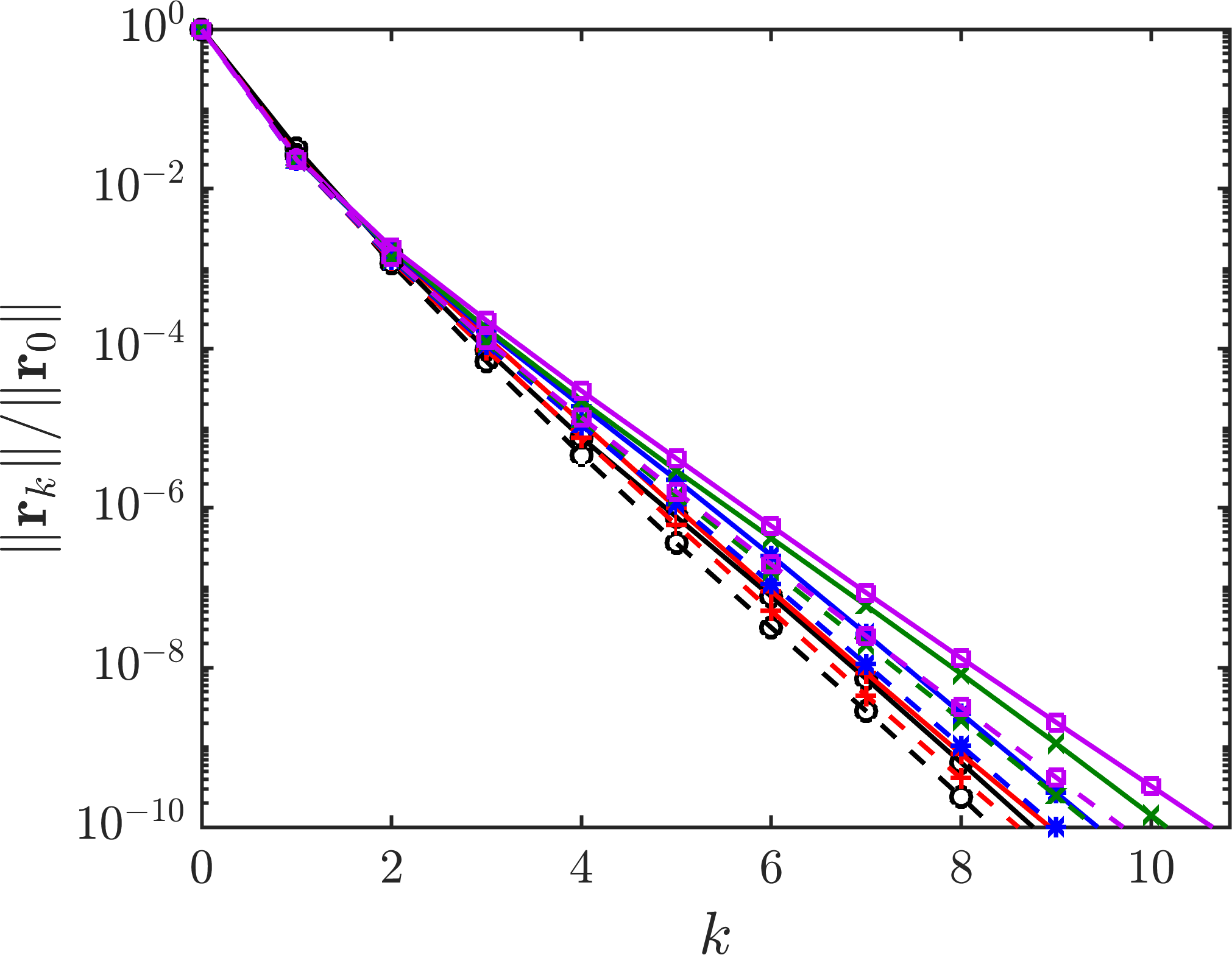}
\hspace{\hs ex}
\includegraphics[scale=\fs]{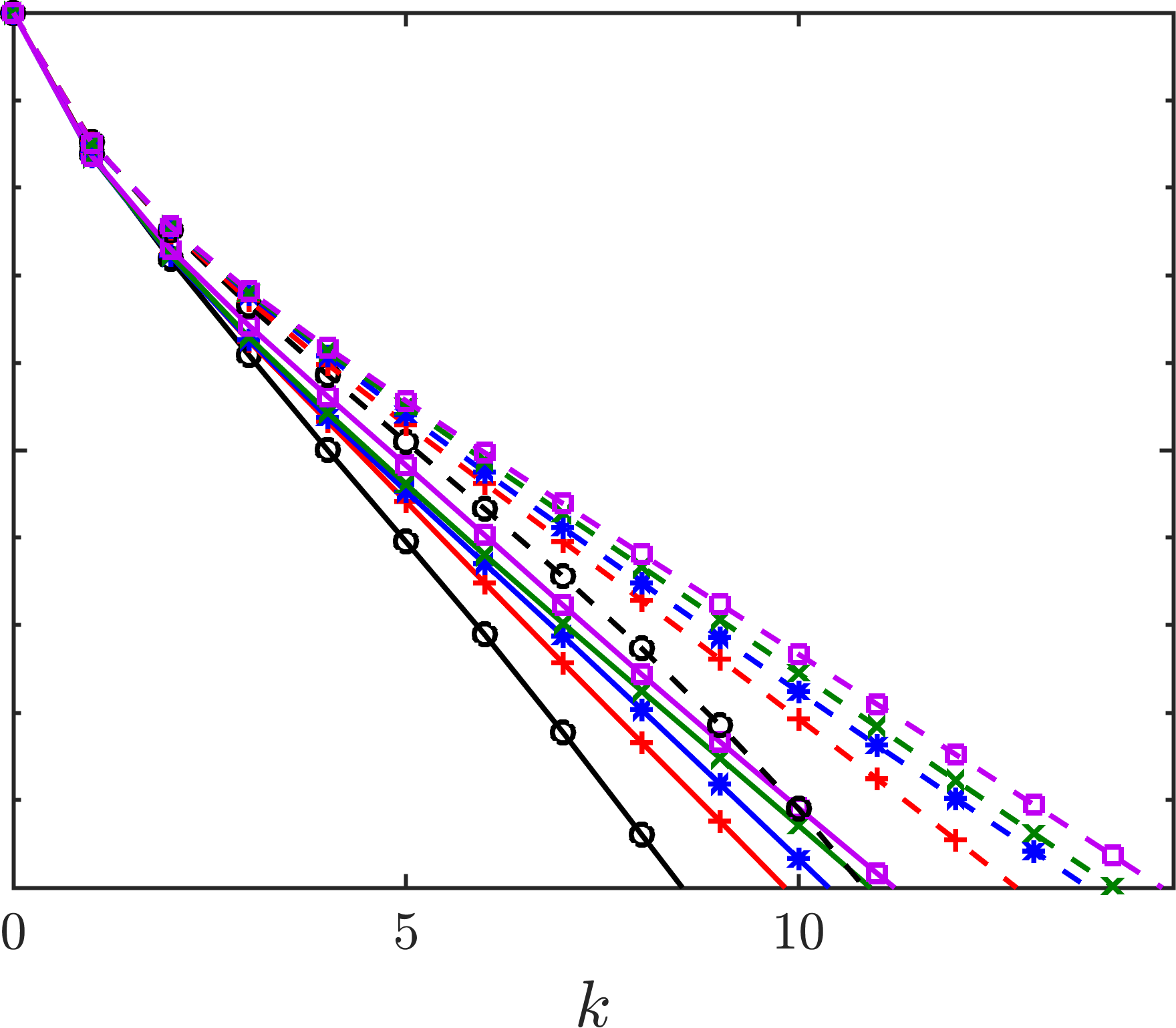}
}
\vspace{\vs ex}
\centerline{
\includegraphics[scale=\fs]{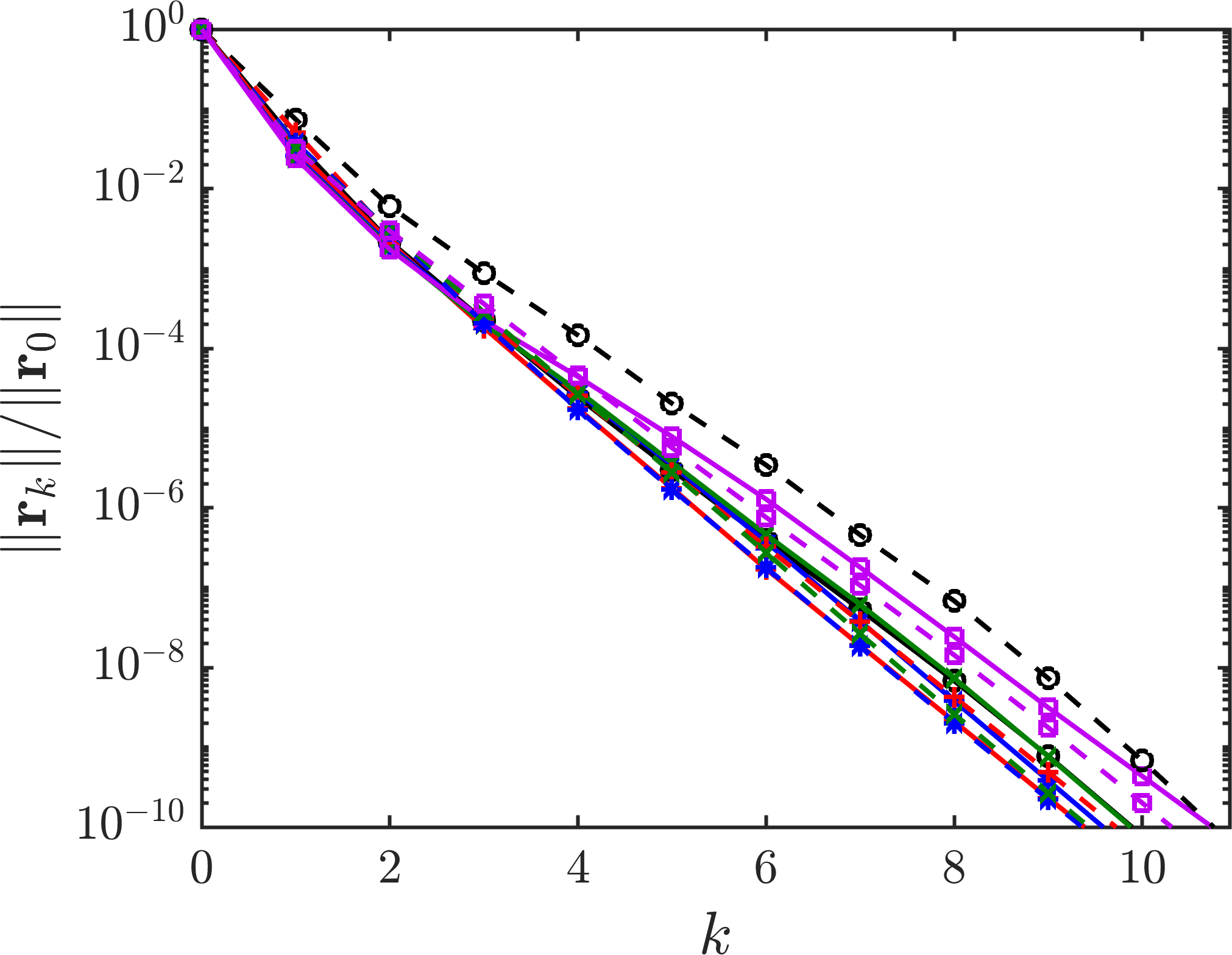}
\hspace{\hs ex}
\includegraphics[scale=\fs]{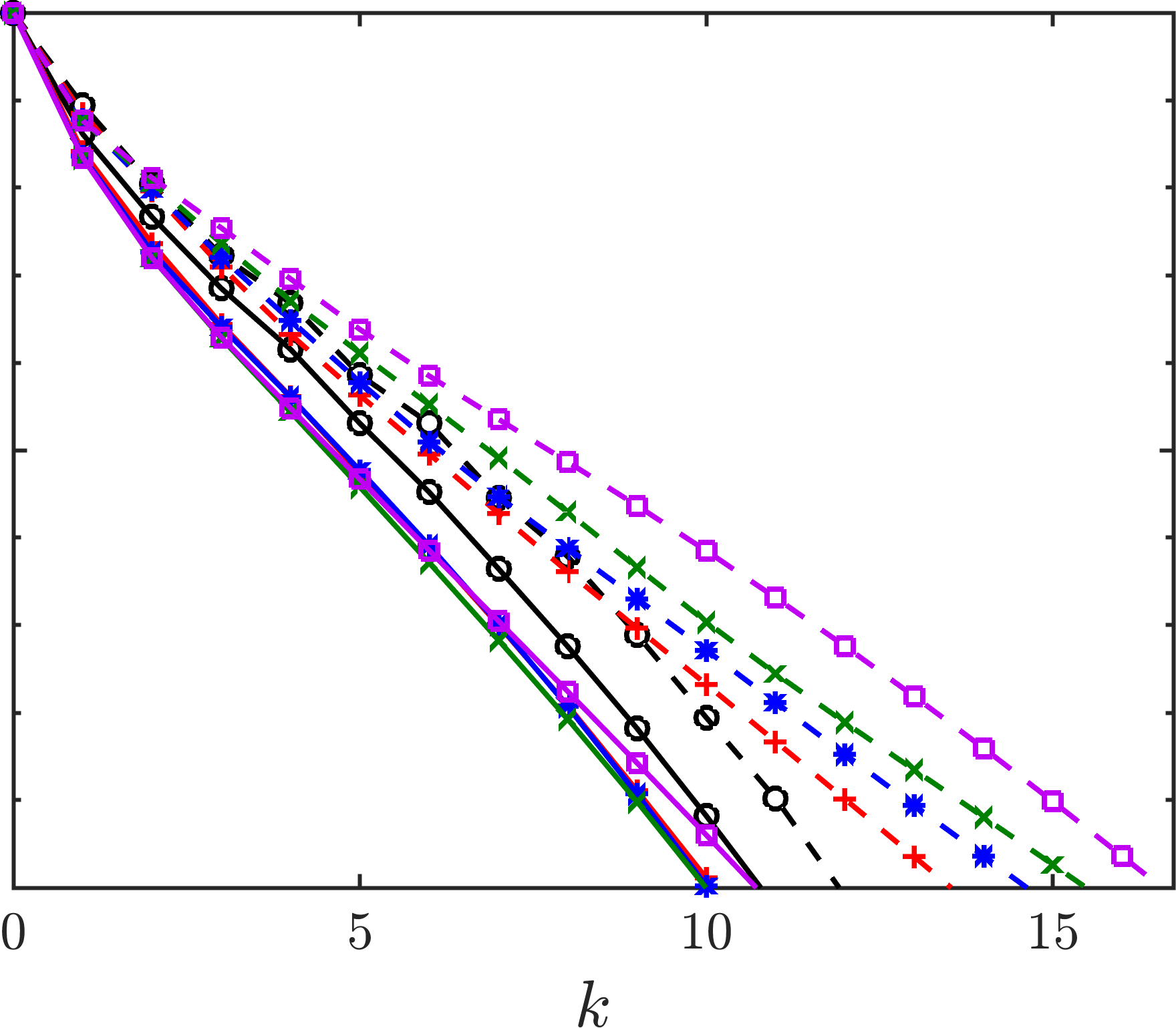}
}
\vspace{\vs ex}
\centerline{
\includegraphics[scale=\fs]{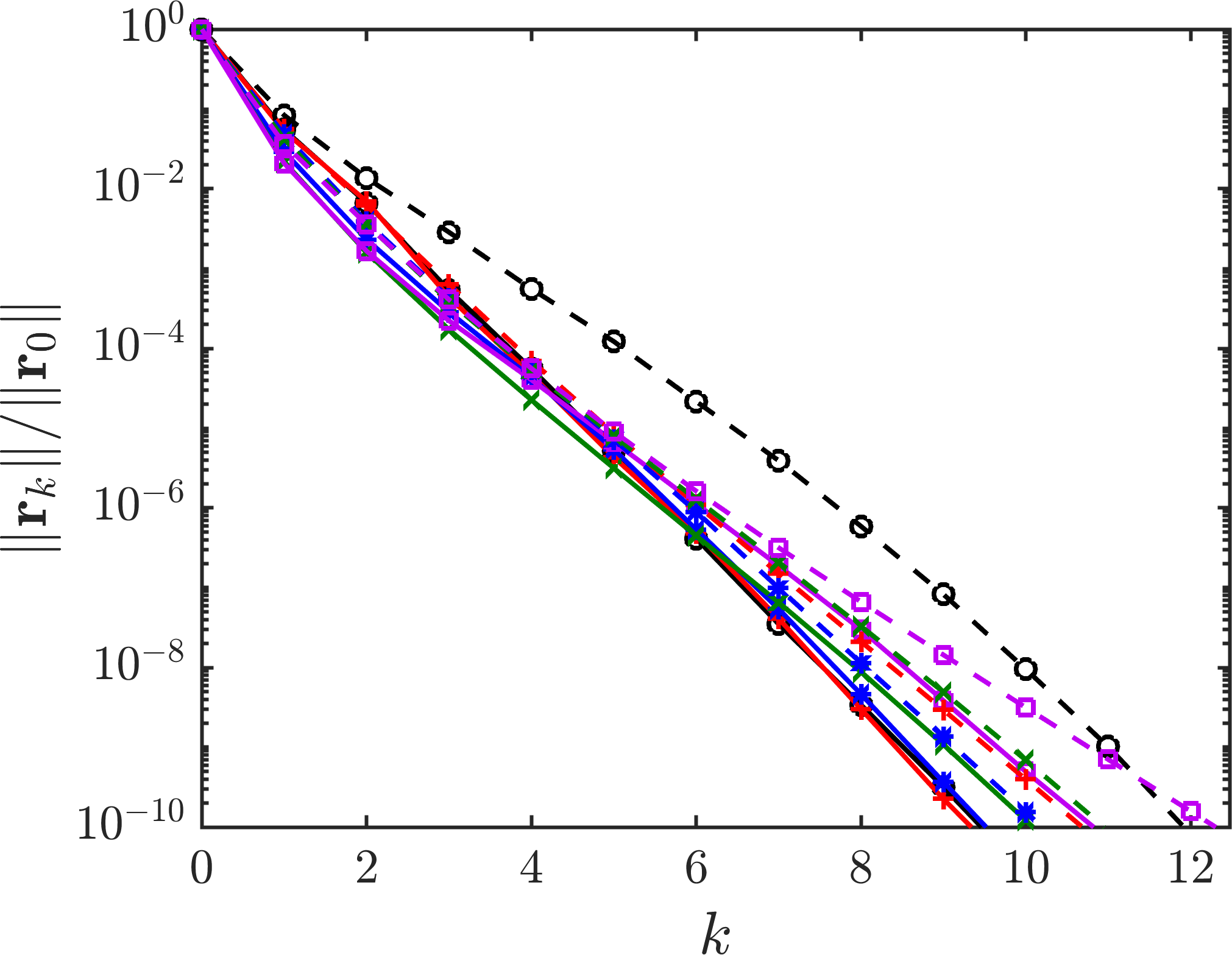}
\hspace{\hs ex}
\includegraphics[scale=\fs]{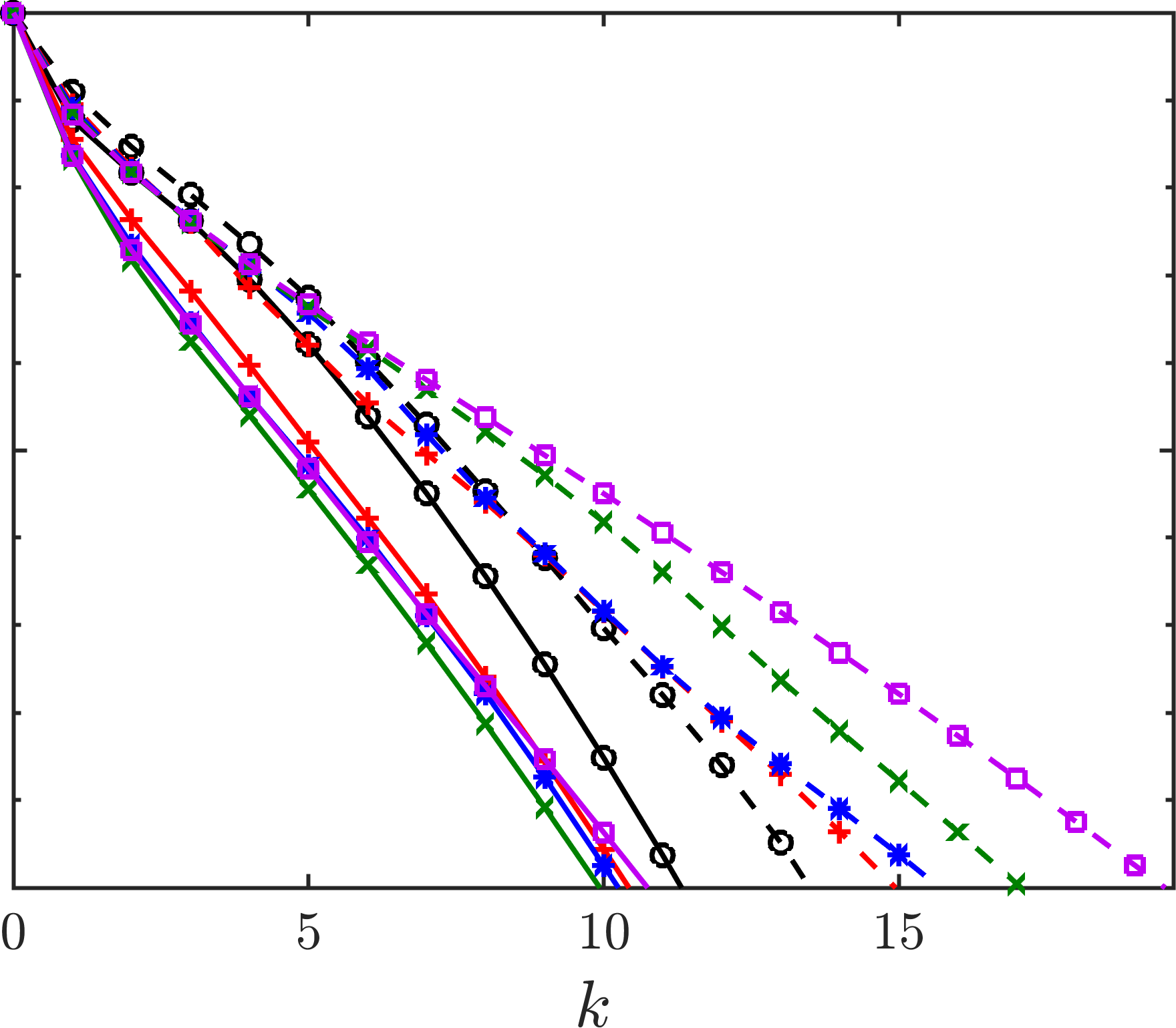}
}
\caption{Residual histories for MGRIT V-cycles applied to solve the linear conservation law \eqref{eq:cons-lin} discretized with standard linear method-of-lines discretization (\cref{SMsec:MGRIT-linear-standard}).
Left column: 1st-order accurate discretizations. Right column: 3rd-order accurate discretizations.
Each of the four rows corresponds with the test problem shown in the corresponding row of \cref{SMfig:MGRIT-lin-test-prob}.
Solid lines correspond to GLF numerical fluxes, and dashed lines correspond to  LLF numerical fluxes. 
Note that for the constant wave-speed problem (top row), the GLF and LLF numerical fluxes are identical (see discussion in \cref{SMsec:MGRIT-linear-standard-disc}).
\label{SMfig:MGRIT-lin-num-res}
}
\end{figure}

\renewcommand{\fd}{./figures/}
\renewcommand{\vs}{2}
\renewcommand{\hs}{4}
\renewcommand\fs{0.28}

\begin{figure}[h!]
\centerline{
\includegraphics[scale=\fs]{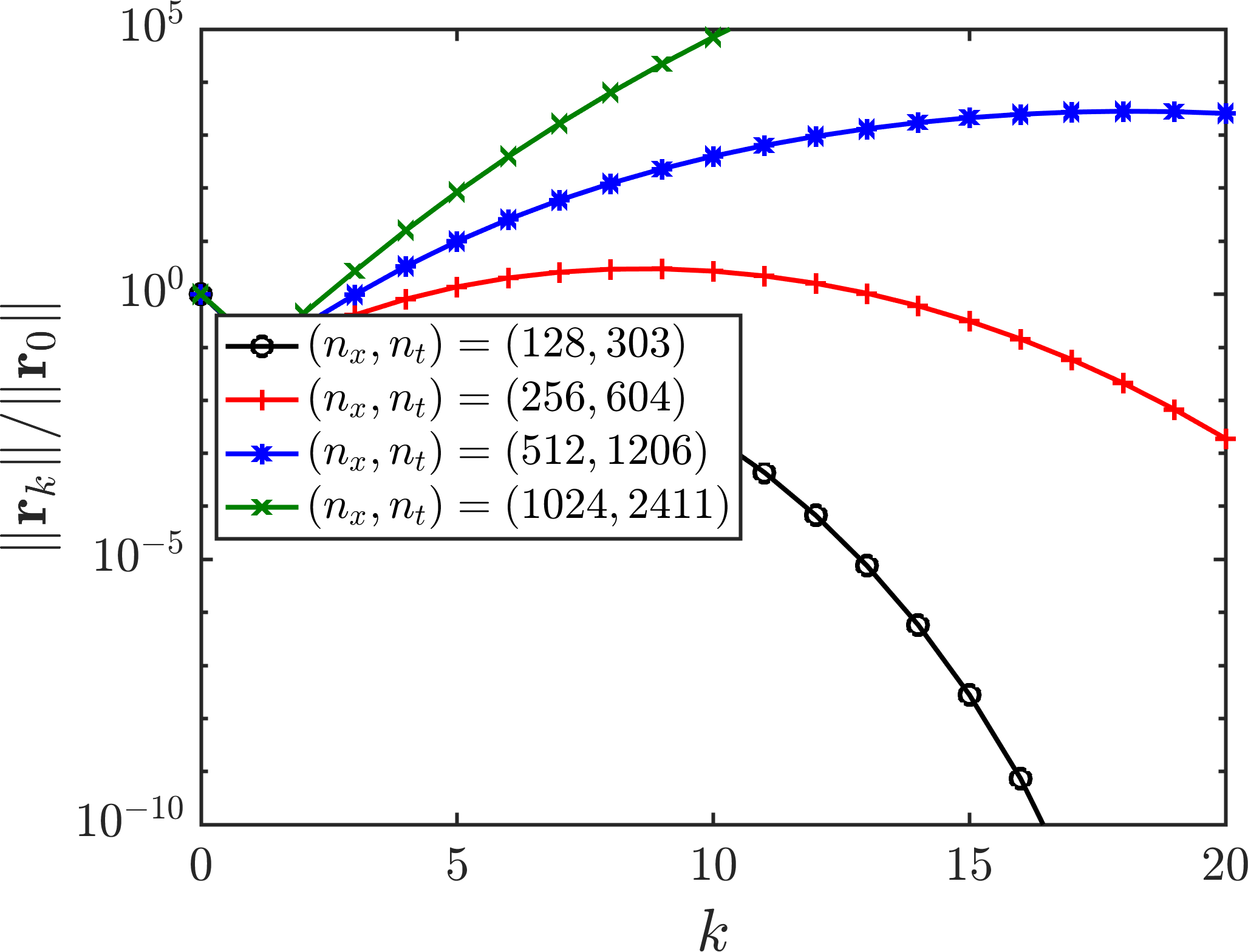}
\hspace{\hs ex}
\includegraphics[scale=\fs]{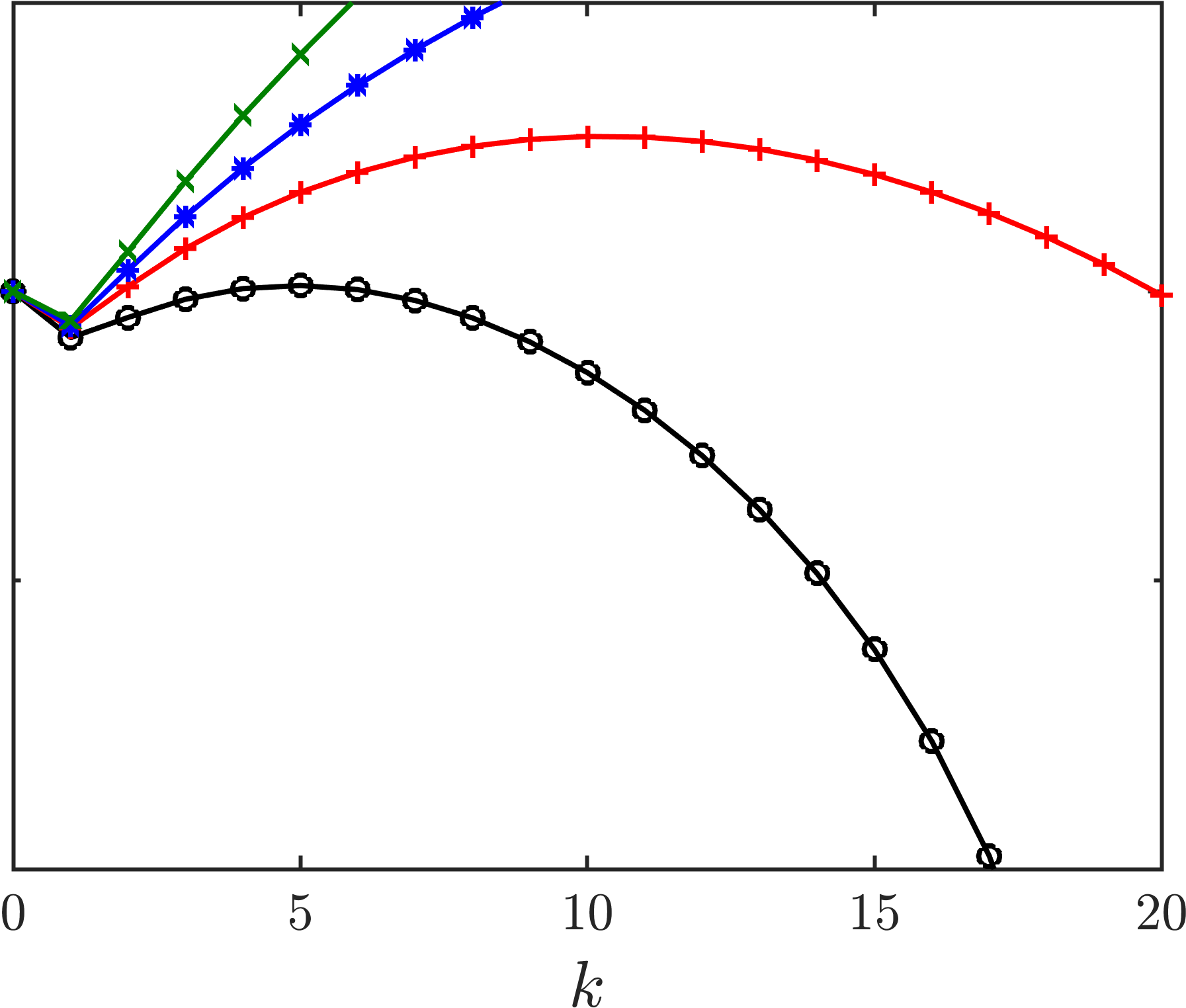}
}
\vspace{\vs ex}
\centerline{
\includegraphics[scale=\fs]{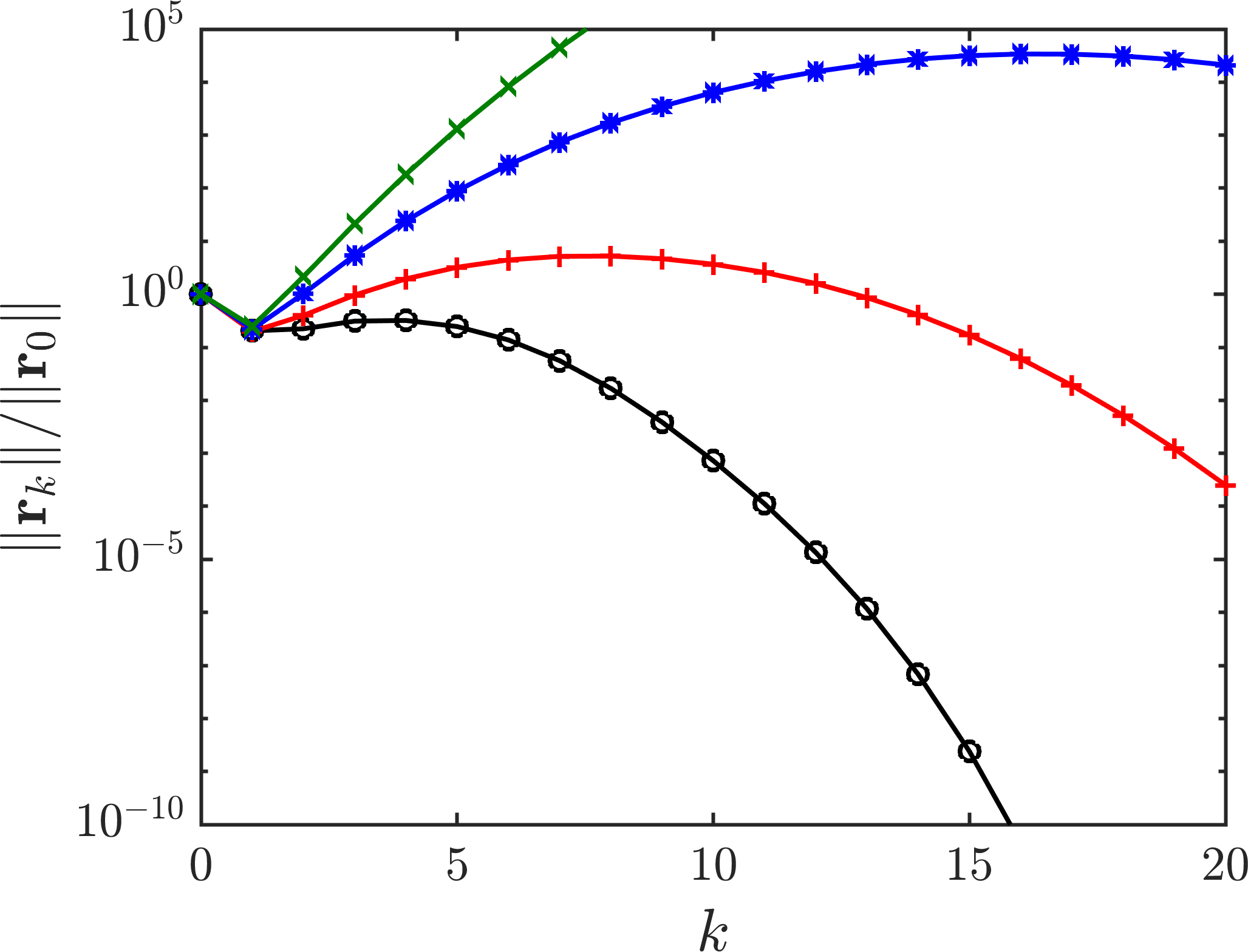}
\hspace{\hs ex}
\includegraphics[scale=\fs]{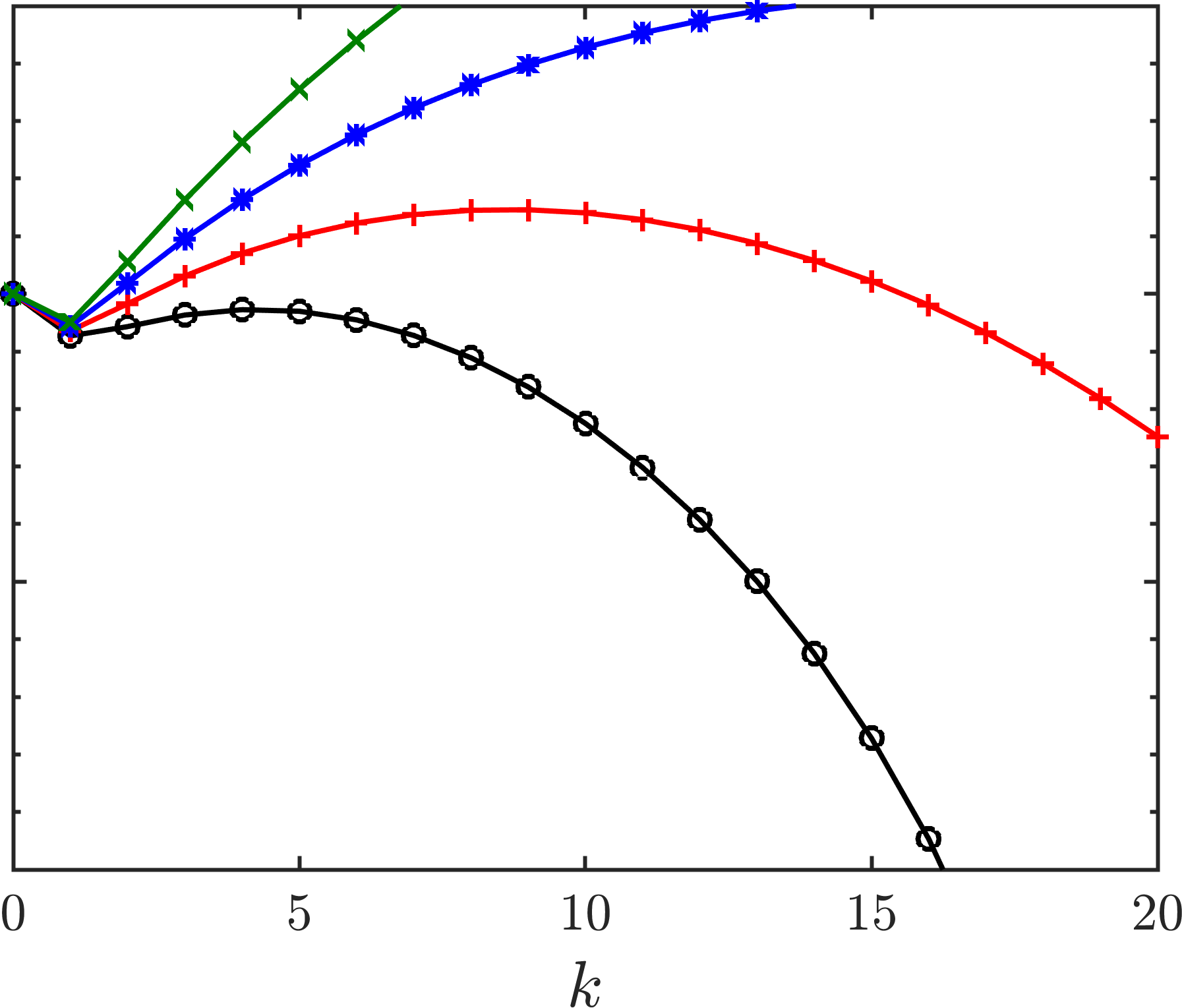}
}
\caption{Residual histories for MGRIT V-cycles applied to solve the linear conservation law \eqref{eq:cons-lin} discretized with standard linear method-of-lines discretization (\cref{SMsec:MGRIT-linear-standard}).
The test problems here are the same as the GLF problems considered in the first two rows of \cref{SMfig:MGRIT-lin-num-res}, except that the coarse-grid operator does not include a truncation error correction.
Note: All vertical axes are on the same scale ranging from $10^{-10} \to 10^{5}$, while those in \cref{SMfig:MGRIT-lin-num-res} range from $10^{-10} \to 1$.
\label{SMfig:MGRIT-lin-num-res-NO-CORRECT}
}
\end{figure}

MGRIT residual history plots are then shown in \cref{SMfig:MGRIT-lin-num-res}.
These results are for multilevel V-cycles, using a coarsening factor of $m = 8$ on all levels, and coarsening is continued until fewer than two points in time would result. 
FCF-relaxation is used on all levels.
Coarse-grid operators on deeper levels in the hierarchy have the same structure as the operator \eqref{eq:Psi-copy} on level two, with the truncation errors being computed recursively (see discussion in \cite[Section 3.4]{DeSterck_etal_2023_SL}).
All coarse-grid linear systems are approximately solved with GMRES, which is iterated until residual is reduced in norm by by 0.01 or until 10 iterations are reached.
The parameter $\vartheta$ that specifies exactly where inside each time interval the wave-speed and dissipation coefficients is frozen is set to zero.
The initial MGRIT iterate is a vector with uniformly random entries, and MGRIT is iterated until the 2-norm of the space-time residual is reduced by at least 10 orders of magnitude. 

Overall, MGRIT convergence is fast, and, for the most part, iteration counts are roughly constant independent of problem size. 
Interestingly, for the 1st-order discretizations (left column), there is not much difference in MGRIT convergence between the GLF and LLF numerical fluxes. 
For the 3rd-order discretizations (right column), MGRIT convergence appears slightly slower for the LLF flux than the GLF flux. 

For comparative purposes, we include in \cref{SMfig:MGRIT-lin-num-res-NO-CORRECT} MGRIT residual histories for a subset of the same problems from \cref{SMfig:MGRIT-lin-num-res}, except that no approximate truncation error correction has been used on coarse levels. 
That is, the coarse-grid operator used in the tests in \cref{SMfig:MGRIT-lin-num-res-NO-CORRECT} is \eqref{eq:Psi-copy} with ${\cal T}_{\rm ideal}^{n} = {\cal T}_{\rm direct}^{n} = 0$, or simply the rediscretized semi-Lagrangian method, $\Psi^{n} = {\cal S}^{n, m\delta t}_{p, 1_*}$.
Evidently, the approximate truncation error correction substantially improves the convergence rate of MGRIT, and appears necessary for the solver to converge in a number of iterations $k \ll n_t / (2m)$---the number of iterations at which FCF-relaxation sequentially propagates the initial condition across the entire time domain \cite{Falgout_etal_2014}.
%

\section{Approximating truncation error for linearized method-of-lines discretization}
\label{SMsec:ideal-err-est-linearized}

In this section, we describe our approach for approximating the truncation error operator ${\cal T}_{\rm ideal}^{n}$ in the coarse-grid operator \eqref{eq:Psi}, in the context where the fine-grid operator is a linearized method-of-lines discretization.
Recall that the particular details of the linearized problem are given in Section \ref{sec:linearization}.
This situation is distinct from that in \cref{SMsec:MGRIT-linear-standard}, which developed an approximation of ${\cal T}_{\rm ideal}^{n}$ in the context of the fine-grid operator being a standard, linear method-of-lines discretization.
That being said, the heuristics we use in this section here are heavily inspired by the work in \cref{SMsec:MGRIT-linear-standard} given the excellent results in \cref{SMfig:MGRIT-lin-num-res}; we strongly recommend reading \cref{SMsec:MGRIT-linear-standard} before the current section.

As in \cref{SMsec:MGRIT-linear-standard}, we suppose that the approximate truncation error operator takes the form
\begin{align} \label{eq:T-ideal-approx-copy}
{\cal T}_{\rm ideal}^{n}
=
\textrm{approximation}
\bigg(
\sum_{j = 0}^{m-1}
\Big( \delta t \wh{{\cal E}}^{n+j} - \wh{e}_{\rm RK} \big( \delta t \wh{L}^{n+j} \big)^{p+1} \Big)
\bigg).
\end{align}
Note that the truncation error operator in \cref{SMsec:MGRIT-linear-standard} was written generally for a $q$th-order ERK method, but here we have fixed $q = p = 2k-1$ since we only consider discretizations of the nonlinear PDE \eqref{eq:cons-law} with equal spatial and temporal orders of accuracy.

The rest of this section is organized as follows. \Cref{SMsec:ideal-err-est-linearized-1} presents an approximation for spatial discretization errors $\wh{{\cal E}}^{n+j}$ in \eqref{eq:T-ideal-approx-copy}. 
\Cref{SMsec:ideal-err-est-linearized-2} presents an approximation for the powers of the spatial discretization $\big( \wh{L}^{n+j} \big)^{p+1}$  appearing in \eqref{eq:T-ideal-approx-copy}. 
The final approximation for \eqref{eq:T-ideal-approx-copy} is presented in \cref{SMsec:ideal-err-est-linearized-3}.

\subsection{Approximating the spatial discretization error}
\label{SMsec:ideal-err-est-linearized-1}

Estimating rigorously the error term $\wh{E}^{n+j}$ in the present setting is quite complicated. 
So, we instead pursue a heuristic approach based the result in \cref{SMlem:num-flux-est}. Recall that \cref{SMlem:num-flux-est} estimates the error of the LF numerical flux \eqref{eq:LFF-lin} when: (i) reconstructions of the wave-speed in the flux are replaced with exact evaluations of the wave-speed, and (ii) the linear reconstructions $e_{i+1/2}^\pm$ use standard weights.
This differs from the present setting, in which the LF numerical flux \eqref{eq:LFF-lin} uses: (i) reconstructions of the wave-speed that are based on the current linearization point (see Section \ref{sec:linearization-II}), and (ii) the linear reconstructions $e_{i+1/2}^\pm$ use non-standard weights (see Section \ref{sec:linearization-III}).
Nonetheless, we suppose that the dominant error term in the numerical flux takes the same form as it does in the case of the standard linear discretization described above, and as is analyzed in \cref{SMlem:num-flux-est} of  \cref{SMsec:MGRIT-linear-standard}. 
That is, we suppose (dropping temporal superscripts for readability):\footnote{It is possible to analyze the error $\wh{E}_{i+1/2}$ more rigorously than this; however, we do not present such details here since they do not appear necessary or helpful for our purposes.}
\begin{align} \label{eq:E-linearized-approx-0}
\wh{E}_{i+1/2} \approx \frac{\nu_{i+1/2}}{2} \big( e_{i+1/2}^- - e_{i+1/2}^+ \big).
\end{align}

Thus, our first step for developing an expression for \eqref{eq:T-ideal-approx-copy} is to estimate the difference $e_{i+1/2}^- - e_{i+1/2}^+$ in \eqref{eq:E-linearized-approx-0}. 
Recall from Section \ref{sec:linearization-III} that the exact form of the linearized reconstructions $e_{i+1/2}^\pm$ depends on the order of the reconstructions $u_{i+1/2}^\pm$ in the outer nonlinear iteration and on the linearization scheme being used.
Specifically, if 1st-order reconstructions are used then $ e_{i+1/2}^\pm$ are just standard linear reconstructions.
For higher-order reconstructions $u_{i+1/2}^\pm$, the reconstructions are nonlinear since they depend on $\bb{u}$ (supposing WENO weights are used, that is). 
The linearized reconstructions $e_{i+1/2}^\pm$ are then either: (i) approximate gradients of $u_{i+1/2}^\pm$ if the approximate Newton linearization \eqref{eq:weighted-reconstruct-grad-FD} is used; or (ii) weighted reconstructions using the same weights as were used to compute $u_{i+1/2}^\pm$ if the Picard linearization \eqref{eq:weighted-reconstruct-grad-zero} is used.

We only provide an estimate for $e_{i+1/2}^- - e_{i+1/2}^+$ under the assumption that a Picard linearization has been used because it is much simpler than the approximate Newton case. We can also cover the case of 1st-order reconstructions in this case, supposing the ``WENO'' weights for $k = 1$ are simply $w^{0} = \wt{w}^{0} = 1$.
However, note that in our numerical tests, we use the same formula for ${\cal T}_{\rm ideal}^{n}$ regardless of whether in fact Picard or Newton linearization of the WENO weights is being used.

Recalling \eqref{eq:weighted-reconstruct-grad-zero}, we assume the linearized reconstructions are given by
\begin{align} \label{eq:weighted-reconstruct-grad-zero-copy}
e_{i+1/2}^-
=
\sum \limits_{\ell =0}^{k-1} 
b_i^{\ell} ( \bb{u} ) 
\big( R^{\ell} \bb{e} \big)_i,
\quad
e_{i+1/2}^+
=
\sum \limits_{\ell =0}^{k-1} 
\wt{b}_{i+1}^{\ell} ( \bb{u} ) 
\big( \wt{R}^{\ell} \bb{e} \big)_{i+1}.
\end{align}
Our starting point for estimating $e_{i+1/2}^- - e_{i+1/2}^+$ is the following corollary of \cref{SMlem:est-lin-rec}, which considered estimates for the standard polynomial reconstructions $\big( \wt{R}^{\ell} \bb{e} \big)_{i}, \big( {R}^{\ell} \bb{e} \big)_{i}$.
\begin{corollary}[Error estimate for weighted reconstructions using $2k-1$ cells]
\label{SMcor:est-rec-weight}
Consider weighted reconstructions of $e(x_{i \pm 1/2})$ that combine polynomial reconstructions of $e(x_{i \pm 1/2})$ on the $k$ stencils $\{ {\cal I}_{i - \ell}, \ldots, {\cal I}_{i - \ell + (k-1)} \}_{\ell = 0}^{k-1}$ using weights $\big\{ \wt{b}_{i}^{\ell} \big\}_{\ell = 0}^{k-1}$, $\big\{ b_{i}^{\ell} \big\}_{\ell = 0}^{k-1}$.
Suppose the weights are consistent, $\sum_{\ell} b_i^{\ell} = \sum_{\ell} \wt{b}_i^{\ell} = 1$, and that $e(x)$ is sufficiently smooth over all cells in the stencil.
Then,
\begin{align}
\sum \limits_{\ell = 0}^{k-1} \wt{b}_i^{\ell} \big( \wt{R}^{\ell} \bb{e} \big)_i
=
e(x_{i - 1/2}) 
- 
\frac{(-h)^k}{(k+1)!} \sum \limits_{\ell = 0}^{k-1} 
\left[
\wt{b}_{i}^{\ell} \, \wt{\zeta}^{\ell} \, \frac{\d^k e}{\d x^k} \bigg|_{\xi_{i,\ell}(x_{i-1/2})}
\right], 
\end{align}
and
\begin{align}
\sum \limits_{\ell = 0}^{k-1} {b}_i^{\ell} \big( {R}^{\ell} \bb{e} \big)_i
=
e(x_{i + 1/2}) 
- 
\frac{(-h)^k}{(k+1)!} \sum \limits_{\ell = 0}^{k-1} 
\left[
b_{i}^{\ell} \, \zeta^{\ell} \, \frac{\d^k e}{\d x^k} \bigg|_{\xi_{i,\ell}(x_{i+1/2})}
\right].
\end{align}
\end{corollary}

\begin{proof}
The result follows by substituting the error estimates of $\big( \wt{R}^{\ell} \bb{e} \big)_i$ and $\big( {R}^{\ell} \bb{e} \big)_i$ given in \cref{SMlem:est-lin-rec} and then using the consistency of the weights.
\end{proof}

Plugging in the results from this corollary, we have
\begin{align} 
e_{i+1/2}^- - e_{i+1/2}^+
&=
\frac{(-h)^k}{(k+1)!} \sum \limits_{\ell = 0}^{k-1} 
\left[
\wt{b}_{i+1}^{\ell} \, \wt{\zeta}^{\ell} \, \frac{\d^k e}{\d x^k} \bigg|_{\xi_{i+1,\ell}(x_{i+1/2})}
-
b_{i}^{\ell} \, \zeta^{\ell} \, \frac{\d^k e}{\d x^k} \bigg|_{\xi_{i,\ell}(x_{i+1/2})}
\right],
\\
\label{eq:e-diff-est-linearization}
&\approx
\frac{(-h)^k}{(k+1)!} \sum \limits_{\ell = 0}^{k-1} 
\Big[
\wt{b}_{i+1}^{\ell} \, \wt{\zeta}^{\ell} 
-
b_{i}^{\ell} \, \zeta^{\ell} 
\Big]
\frac{\d^k e}{\d x^k} \bigg|_{\xi_{\ell}},
\end{align}
where $\{ \xi_{\ell} \}$ are unknown points. 
The approximation made here is that, for fixed $\ell$, we can replace the two unknown points $\xi_{i,\ell}(x_{i+1/2})$ and $\xi_{i+1,\ell}(x_{i+1/2})$ by single unknown point $\xi_{\ell}$.
This is true to at least lowest order because $\xi_{i,\ell}(x_{i+1/2})$ and $\xi_{i+1,\ell}(x_{i+1/2})$ are at most a distance of ${\cal O}(h)$ apart (see \cref{SMlem:est-lin-rec} for why this is true).

Recalling that ultimately we want to discretize $\wh{E}_{i+1/2} \approx \tfrac{\nu_{i+1/2}}{2} \big( e_{i+1/2}^- - e_{i+1/2}^+ \big)$ from \eqref{eq:E-linearized-approx-0}, our next approximation is to replace the unknown points $\{ \xi_{\ell} \}$ in \eqref{eq:e-diff-est-linearization} with a set of equidistant points such that we can discretize with finite differences the derivatives involved. 
For $k = 1$, we simply set $\xi_{\ell} = x_{i+1/2}$. Recalling that the reconstruction weights are solution independent for $k = 1$ with $b_i = \wt{b}_i = 1$, plugging into \eqref{eq:E-linearized-approx-0} gives
\begin{align} \label{eq:Ewh-k=1}
\wh{E}_{i+1/2} 
\approx 
\frac{\nu_{i+1/2}}{2} \big( e_{i+1/2}^- - e_{i+1/2}^+ \big)
\approx
-h \frac{\nu_{i+1/2}}{2}
\frac{\d e}{\d x} \bigg|_{x_{i+1/2}}
\quad
{\rm for}
\quad
k = 1. 
\end{align}

Now consider $k > 1$; in principle this could be done for arbitrary $k$, but since we only present numerical results for $k = 2$, let us just consider that case.
This case requires slightly more care because we have no reason to expect that $\xi_0$ and $\xi_1$ are a distance of $h$ apart.
Introducing in \eqref{eq:e-diff-est-linearization} the ansatz that $\xi_0 = \xi_1 - h$ requires us to make the following adjustment:
\begin{align} \label{eq:e-diff-est-linearization-k=2}
e_{i+1/2}^- - e_{i+1/2}^+
&\approx
\beta 
\frac{(-h)^k}{(k+1)!} \bigg|_{k = 2} \sum \limits_{\ell = 0}^{1} 
\Big[
\wt{b}_{i+1}^{\ell} \, \wt{\zeta}^{\ell} 
-
b_{i}^{\ell} \, \zeta^{\ell} 
\Big]
\frac{\d^2 e}{\d x^2} \bigg|_{\xi_0 - \ell h}.
\end{align}
That is, we introduce the unknown coefficient $\beta$, which can be understood by appealing to the mean value theorem.
%
%
To determine $\beta$, we enforce that when the optimal linear weights are used in \eqref{eq:e-diff-est-linearization-k=2}, i.e., $b^{\ell} = d^{\ell}, \wt{b}^{\ell} = \wt{d}^{\ell}$, that the result is the same as \eqref{eq:e-diff-est} when $k = 2$. Recall that \eqref{eq:e-diff-est} estimates $e_{i+1/2}^- - e_{i+1/2}^+$ when $e_{i+1/2}^\pm$ are weighted reconstructions using optimal linear weights.
Setting the resulting two expressions equal after rearrangement gives
\begin{align}
\frac{\d^3 e}{\d x^3} \bigg|_{x_{i + 1/2}}
=
-\beta
\frac{1}{h} \frac{4}{3} 
\left[
\frac{\d^2 e}{\d x^2} \bigg|_{\xi_0}
-
\frac{\d^2 e}{\d x^2} \bigg|_{\xi_0 - h}
\right]
+
\textrm{h.o.t.}
\end{align}
Evidently, this is satisfied with $\beta = -\frac{3}{4}$ and $\xi_0 = x_{i+1}$.
Plugging the results into \eqref{eq:e-diff-est-linearization-k=2} and then \eqref{eq:E-linearized-approx-0} we arrive at:
\begin{align} \label{eq:Ewh-k=2}
\wh{E}_{i+1/2}
\approx
-\frac{3}{4}
\frac{h^2}{6} 
\frac{\nu_{i+1/2}}{2}
\bigg[
\Big(
2 \wt{b}_{i+1}^{0} 
+
b_{i}^{0} 
\Big)
\frac{\d^2 e}{\d x^2} \bigg|_{x_{i+1}}
-
\Big(
\wt{b}_{i+1}^{1}
+
2 b_{i}^{1} 
\Big)
\frac{\d^2 e}{\d x^2} \bigg|_{x_{i}}
\bigg]
\quad {\rm for} \quad k = 2.
\end{align}

Recall from \cref{ass:E-est} that $\big( \wh{{\cal E}} \bb{e} \big)_i = \big(\wh{E}_{i+1/2} - \wh{E}_{i-1/2} \big) / h$.
Thus, from \eqref{eq:Ewh-k=1} and \eqref{eq:Ewh-k=2} we have the following approximations
\begin{align} \label{eq:Ewh_sum}
\sum \limits_{j = 0}^{m-1} \wh{{\cal E}}^{n+j}
=
\begin{cases}
&\displaystyle{\frac{h}{2}}
{\cal D}_1 \diag 
\bigg( 
\sum \limits_{j = 0}^{m-1} \bm{\nu}^{n+j} 
\bigg) 
{\cal D}_1^\top, 
\quad k = 1,\\
\begin{split}
&\frac{3}{4} \frac{h^2}{12} 
{\cal D}_1
\bigg[
\diag
\bigg( 
\sum \limits_{j = 0}^{m-1} \bm{\nu}^{n+j}  \odot \bm{\gamma}^{1,n+j} 
\bigg) 
{\cal D}_2^{1}
\\
&\hspace{8ex}
-
\diag
\bigg( 
\sum \limits_{j = 0}^{m-1} \bm{\nu}^{n+j}  \odot \bm{\gamma}^{0,n+j} 
\bigg) 
{\cal D}_2^{0}
\bigg],
\end{split}
\quad k = 2.
\end{cases}
\end{align} 
Here, $({\cal D}_1 \bb{e})_i = (\bar{e}_i - \bar{e}_{i-1})/h$, $({\cal D}_2^{0} \bb{e})_{i} = (\bar{e}_i - 2\bar{e}_{i+1} + \bar{e}_{i+2})/h^2$, and $({\cal D}_2^{1} \bb{e})_{i} = (\bar{e}_{i-1} - 2 \bar{e}_{i} + \bar{e}_{i+1})/h^2$. 
The vector $\bm{\nu}^{n+j} = \big( \nu_{1+1/2}^{n+j}, \nu_{2+1/2}^{n+j}, \ldots, \nu_{n_x+1/2}^{n+j} \big) \in \mathbb{R}^{n_x}$ is the vector of dissipation coefficients used in the LF numerical flux \eqref{eq:LFF-lin} at time $t_{n + j} + \delta t \vartheta$. Recall from \cref{SMsec:cons-lin-approx} that to develop a truncation error estimate for the linearized problem, its time dependence on $[t_n, t_n + \delta t]$ is frozen at the point $t_n + \delta t \vartheta$, with $\vartheta \in [0, 1]$. 
Also, we have defined the shorthands
\begin{align} \label{eq:gamma-sh-def}
\gamma_i^{0, n} := 2 \wt{b}_{i+1}^{0, n} 
+
b_{i}^{0, n},
\quad
\gamma_i^{1, n} := \wt{b}_{i+1}^{1, n}
+
2 b_{i}^{1, n},
\end{align}
where, e.g., $b_i^{1, n}$ means the reconstruction coefficient $b_i^{1}$ from time $t_{n} + \delta t \vartheta$.
Further, ``$\odot$'' denotes element-wise product between vectors.
Notice that the expressions in \eqref{eq:Ewh_sum} are, by design, equivalent those in \eqref{eq:E-approx-optimal-linear} when the general weights $\{ \wt{b}^{\ell} \}$, $\{ {b}^{\ell} \}$ are replaced with the optimal linear weights $\{ \wt{d}^{\ell} \}$, $\{ {d}^{\ell} \}$.

\subsection{Approximating powers of the spatial discretization}
\label{SMsec:ideal-err-est-linearized-2}
Recalling the form of the approximate truncation error operator 
${\cal T}_{\rm ideal}^{n}$ in \eqref{eq:T-ideal-approx-copy}, we need also to estimate powers of the spatial discretization matrix.
To approximate these we use a similar strategy to that used previously in \cref{SMsec:Psi-MGRIT-linear-standard} in the context of standard linear spatial discretizations; see \eqref{eq:L-powers-approx}.
Specifically, for the $k = 1$ we now use an expression that is identical to that in \eqref{eq:L-powers-approx}:
\begin{align} \label{eq:L-powers-approx-k=1}
\sum \limits_{j = 0}^{m-1}
\big( \wh{L}^{n+j} \big)^{2}
\approx
{\cal D}_1
\diag \bigg( \sum \limits_{j = 0}^{m-1} \big( \bm{\alpha}^{n+j} \big)^{2} \bigg)
{\cal D}_1^\top
\quad
\textrm{for}
\quad 
k = 1,
\end{align}
where $\bm{\alpha}^{n+j} = \big( \alpha_{1 + 1/2}^{n+j}, \ldots, \alpha_{n_x + 1/2}^{n+j} \big)^\top \in \mathbb{R}^{n_x}$ is the vector of interfacial wave-speeds at time $t_{n+j} + \delta t \vartheta$, and powers of it are computed element-wise.

However, for $k = 2$ we do something slightly different than \eqref{eq:L-powers-approx}.
Namely, rather than using the FD discretization ${\cal D}_3$ that appears in \eqref{eq:L-powers-approx}, we use a FD discretization of the 3rd-derivative that makes use of the WENO weights, analogous to what we did in \eqref{eq:Ewh-k=2}. 
This results in
\begin{align} \label{eq:L-powers-approx-k=2}
\begin{split}
\sum \limits_{j = 0}^{m-1}
\big( \wh{L}^{n+j} \big)^{4}
&\approx
-\frac{3}{4} \frac{1}{h}
{\cal D}_1
\bigg[
\diag
\bigg( 
\sum \limits_{j = 0}^{m-1} \big( \bm{\alpha}^{n+j} \big)^4 \odot \bm{\gamma}^{1,n+j} 
\bigg) 
{\cal D}_2^{1}
\\
&\hspace{8ex}
-
\diag
\bigg( 
\sum \limits_{j = 0}^{m-1} \big( \bm{\alpha}^{n+j} \big)^4  \odot \bm{\gamma}^{0,n+j} 
\bigg) 
{\cal D}_2^{0}
\bigg]
\quad
\textrm{for}
\quad 
k = 2,
\end{split}
\end{align}
with elements of $\bm{\gamma}^{*,n+j}$ defined in \eqref{eq:gamma-sh-def}.

\subsection{Piecing it all together}
\label{SMsec:ideal-err-est-linearized-3}
Now we produce our approximation for ${\cal T}_{\rm ideal}^{n}$ based on  \eqref{eq:T-ideal-approx-copy}.
This requires the spatial discretization estimates from \eqref{eq:Ewh_sum}, and powers of the spatial discretization given in \eqref{eq:L-powers-approx-k=1}  and \eqref{eq:L-powers-approx-k=2}.
Piecing these together results in the following:
\begin{align}
{\cal T}_{\rm ideal}^{n}
=
\begin{cases}
{\cal D}_1 \diag 
\bigg( 
\sum \limits_{j = 0}^{m-1} 
\bm{\beta}^{n+j} 
\bigg) 
{\cal D}_1^\top, 
\quad k = 1,
\\[3ex]
{\cal D}_1
\bigg[
\diag
\bigg( 
\sum \limits_{j = 0}^{m-1} \bm{\beta}^{0,n+j} 
\bigg) 
{\cal D}_2^{0}
-
\diag
\bigg( 
\sum \limits_{j = 0}^{m-1} \bm{\beta}^{1,n+j} 
\bigg) 
{\cal D}_2^{1}
\bigg],
\quad k = 2.
\end{cases}
\end{align} 
The vectors $\bm{\beta}^{n}, \bm{\beta}^{0, n}, \bm{\beta}^{1, n} \in \mathbb{R}^{n_x}$ are defined element-wise by
\begin{align}
\beta_{i}^{n}
&:=
\left[
\wh{e}_{\rm RK} \Big( \delta t \alpha^{n}_{i+1/2} \Big)^2
+
\frac{h \delta t}{2} \nu^{n}_{i+1/2}
\right],
\\
\beta_{i}^{*, n}
&:=
\frac{3}{4} \left[
-
\frac{\wh{e}_{\rm RK}}{h} \Big( \delta t \alpha^{n}_{i+1/2} \Big)^4
+
\frac{h^2 \delta t }{12} \nu^{n}_{i+1/2} 
\right]
\gamma_i^{*, n}. 
\end{align}
Recall from \cref{SMsec:cons-lin-approx} that the ERK constants are $\wh{e}_{\rm RK} = -1/2$ (for the forward Euler method \eqref{eq:ERK1}, as used when $k = 1$), and $\wh{e}_{\rm RK} = -1/4!$ (for the 3rd-order ERK method \eqref{eq:ERK3}, as used when $k = 2$).
We remark that, by design, when the general weights $\{ \wt{b}^{\ell} \}$, $\{ {b}^{\ell} \}$ are replaced with the optimal linear weights $\{ \wt{d}^{\ell} \}$, $\{ {d}^{\ell} \}$ the truncation error operators here reduce to those used in the case of a standard linear discretization given in \eqref{eq:T-ideal-linear} (modulo the fact that the wave-speeds and dissipation coefficients here are  evaluated in terms of the linearization point).

\section{Over-solving: Discretization and algebraic error for a Burgers problem}
\label{SMsec:over-solving}


In this section, we expand upon the contents of \cref{rem:os}.
We consider the algebraic error of the iterates produced by  Algorithm \ref{alg:richardson} for a Burgers problem from the main paper, and we relate these to the discretization error.
To measure discretization error we require the exact solution of the Burgers problem.
Recall that in the main paper we work on the periodic spatial domain $x \in (-1,1)$; however, it is simpler to compute and express the solution on the non-periodic, unbounded domain $x \geq -1$.
Further recall from Section \ref{sec:model-probs} that the initial condition used for Burgers equation \eqref{eq:burgers} is
\begin{align}
u_0(x) 
= 
\begin{cases}
0, 
\quad
& -1 \leq x \leq -\tfrac{1}{2},
\\
1, 
\quad
&-\tfrac{1}{2} < x < 0,
\\
0,
\quad
& \hphantom{-} 0 \leq x.
\end{cases}
\end{align}
Then, for times $t > 0$, the associated exact solution of \eqref{eq:burgers} can be computed by using the method-of-characteristics, and the fact that the shock speed at a given point is the average of the solution on either side of it.
Carrying out these calculations, we find the exact solution is given by
\begin{align} \label{eq:burgers-exact-sol}
u_{\textrm{exact}}(x, t) = 
\begin{cases}
u_{\textrm{pre-merge}}(x, t), 
\quad 
&0 < t < 1, 
\\
u_{\textrm{post-merge}}(x, t), 
\quad 
&1 \leq t,
\end{cases}
\end{align}
where $u_{\textrm{pre-merge}}$ and $u_{\textrm{post-merge}}$ represent the solution before and after the rarefaction wave has run into the shock wave.
These functions are given by
\begin{align}
u_{\textrm{pre-merge}}(x) 
= 
\begin{cases}
0, 
\quad 
&-1 \leq x \leq -\tfrac{1}{2}, \\
\tfrac{x+ 1/2}{t}, 
\quad 
&-\tfrac{1}{2} \leq x \leq x_{\textrm{front}}(t), \\
1, 
\quad 
&x_{\textrm{front}}(t) \leq x \leq x_{\textrm{shock}}(t), 
\\
0, 
\quad 
&x_{\textrm{shock}} < x,
\end{cases}
\end{align}
and
\begin{align}
u_{\textrm{post-merge}}(x) 
= 
\begin{cases}
0, 
\quad 
&-1 \leq x \leq -\tfrac{1}{2}, \\
\tfrac{x+1/2}{t}, 
\quad 
&-\tfrac{1}{2} \leq x \leq x_{\textrm{shock}}(t), \\
0,
\quad   
& x_{\textrm{shock}}(t) < x,
\end{cases}
\end{align}
respectively. Here $x_{\textrm{shock}}$ and $x_{\textrm{front}}$ are the locations of the shock and front of the rarefaction wave, respectively, and are given by
\begin{align}
x_{\textrm{shock}}(t)
=
\begin{cases}
\tfrac{t}{2}, 
\quad
&0 < t \leq 1, 
\\
\sqrt{t} - \tfrac{1}{2}, 
\quad
&1 \leq t, 
\end{cases},
\quad
x_{\textrm{front}}(t)
=
\begin{cases}
t - \tfrac{1}{2},
\quad
&0 < t \leq 1, 
\\
x_{\textrm{shock}}(t),
\quad
& 1 \leq t.
\end{cases}
\end{align}
It is implementationally trivial to extend this solution to the periodic domain $x \in (-1, 1)$ for times $t < 4$ (at $t = 4$ in the periodic problem, the shock runs into the base of the rarefaction wave).

\renewcommand{\fd}{./figures/}
\renewcommand{\hs}{4}
\renewcommand{\vs}{1}
\renewcommand\fs{0.325}
\begin{figure}[t!]
\centerline{
\includegraphics[scale=\fs]{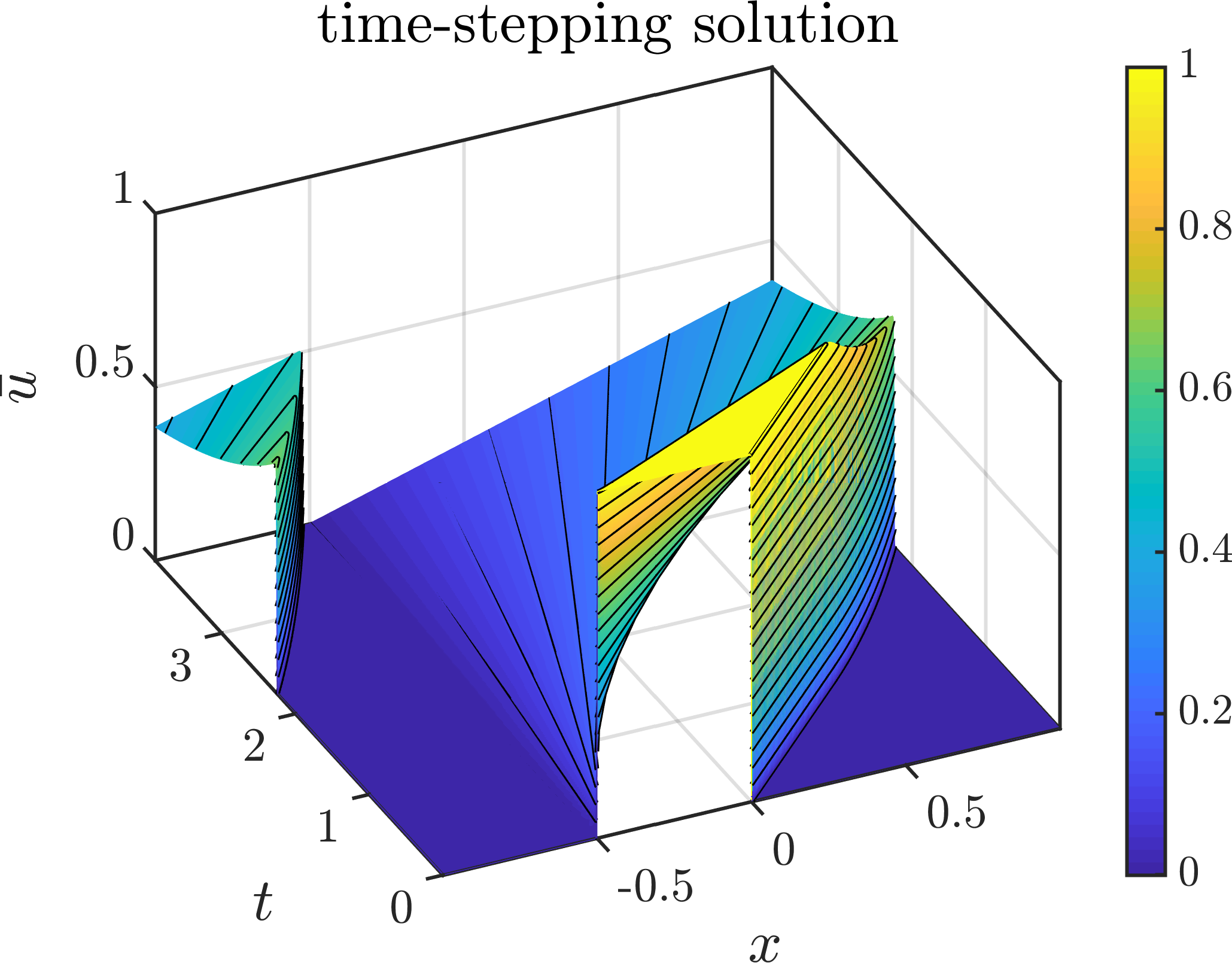}
\hspace{\hs ex}
\includegraphics[scale=\fs]{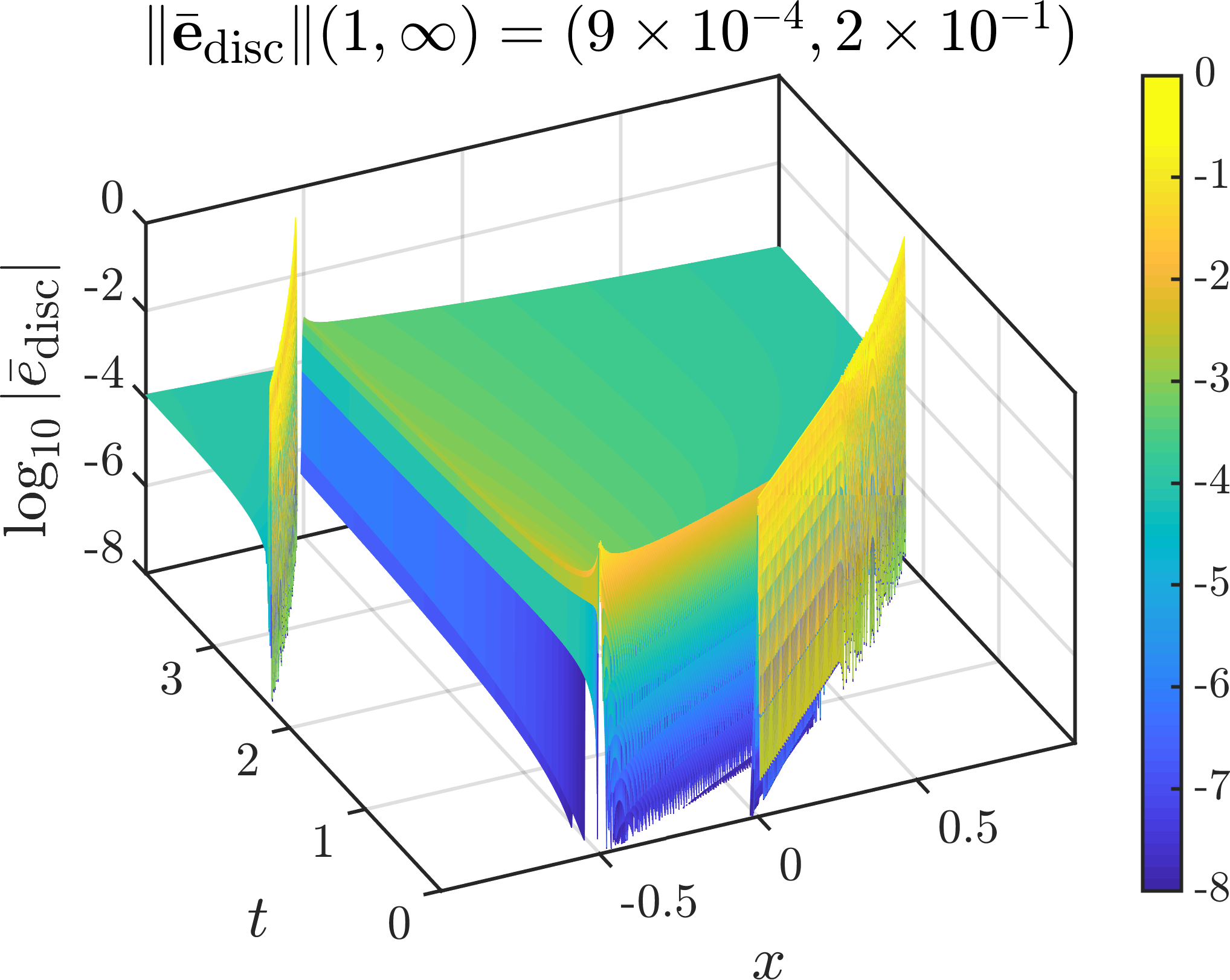}
}
\caption{Left: Discretized space-time solution of Burgers equation \eqref{eq:burgers} obtained by time-stepping on a mesh with $n_x = 1024$ FV cells, and using a LLF numerical flux, and 3rd-order WENO reconstructions.
Right: Associated discretization error in $\log_{10}$ scale as computed via the exact solution \eqref{eq:burgers-exact-sol}.
The discrete $L^1$- and $L^{\infty}$-norms of the discretization error are included in the plot's title.
To aid with visualization, cell-wise errors less than $10^{-8}$ have been masked to white.
\label{SMfig:os-exact-sol}
}
\end{figure}

We now consider a specific numerical example with $n_x = 1024$ FV cells in space. A LLF numerical flux is used with 3rd-order WENO reconstructions. 
In the left panel of \cref{SMfig:os-exact-sol} the space-time solution obtained by sequential time-stepping is shown. The right panel of \cref{SMfig:os-exact-sol} shows the associated discretization error, as computed with the exact solution \eqref{eq:burgers-exact-sol}.
More specifically, the discretization error at time $t_n$ in the $i$th FV cell is defined as $\bar{e}_{\textrm{disc},i}^n := \bar{u}_{\textrm{exact},i}(t_n) - \bar{u}_i^n$, where $\bar{u}_i^n$ is the time-stepping solution, and $\bar{u}_{\textrm{exact},i}(t_n)$ is the cell average of the exact PDE solution \eqref{eq:burgers-exact-sol}.
The discretization error is essentially zero in regions where the solution is constant. It is on the order of $0.2$ along the shock, and it is also relatively large along the front of the rarefaction wave before it merges into the shock.
It is also somewhat large along the base of the rarefaction wave (i.e., around the line $x = -0.5$).
Thus, unsurprisingly, the discretization error is greatest where the true solution lacks regularity.

Now we use our iterative parallel-in-time solver Algorithm \ref{alg:richardson} to solve the discretized problem.
The solver setting are used as in the main paper, including a single MGRIT iteration to approximately solve the linearized problem and the use of nonlinear F-relaxation. A Newton linearization is used for the WENO weights. 
At the $k$th iteration iteration, define the algebraic error as $\bar{e}_{\textrm{alg},i}^n := \bar{u}_i^n - \big( \bb{u}_{k}^n \big)_{i}$, where $\bar{u}_i^n$ is again the exact solution of the discretized problem obtained via sequential time-stepping, and $\big( \bb{u}_{k}^n \big)_i$ is the approximation at iteration $k$ of the discretized problem.
The algebraic error for iterations $k = 0, 2, 4, 6$ is shown in \cref{SMfig:os-rich}.\footnote{Note that the error is shown for the $k$th iterate after nonlinear F-relaxation has been applied to it; that is, in the notation of  Algorithm \ref{alg:richardson}, the algebraic error is actually shown for $\wt{\bm{u}}_k$ rather than $\bb{u}_k$.}

\renewcommand{\fd}{./figures/}
\renewcommand{\hs}{4}
\renewcommand{\vs}{2}
\renewcommand\fs{0.325}
\begin{figure}[t!]
\centerline{
\includegraphics[scale=\fs]{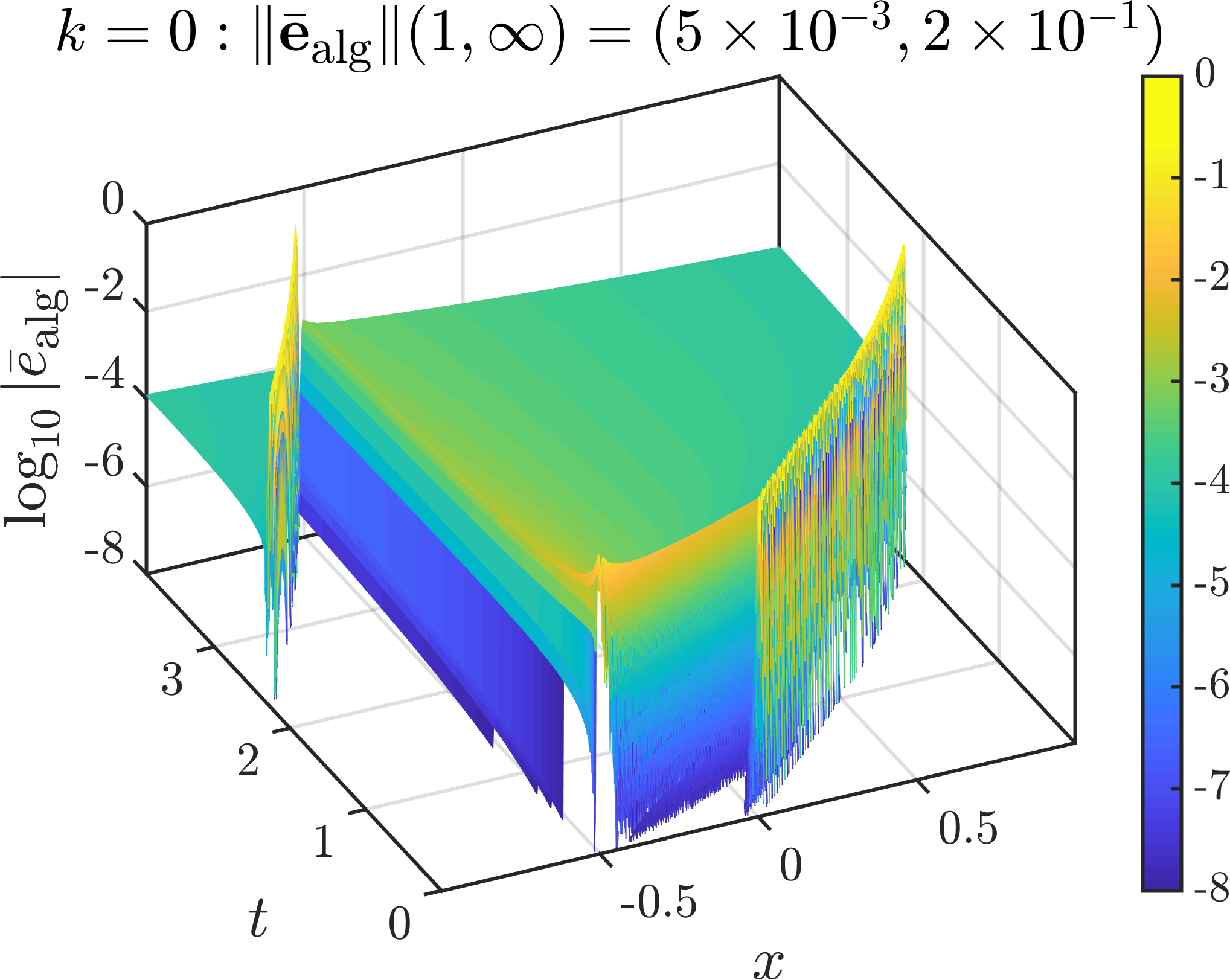}
\hspace{\hs ex}
\includegraphics[scale=\fs]{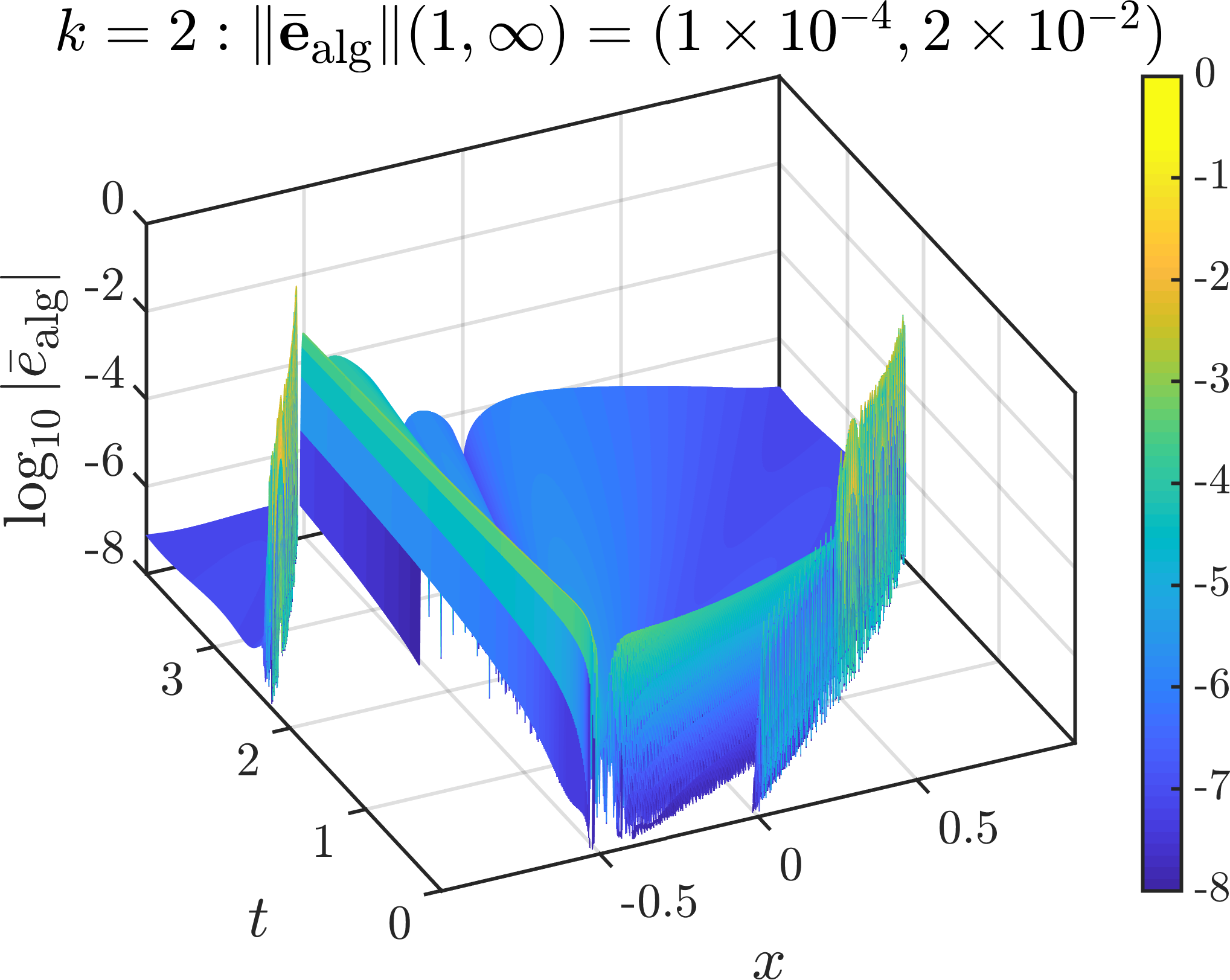}
}
\vspace{\vs ex}
\centerline{
\includegraphics[scale=\fs]{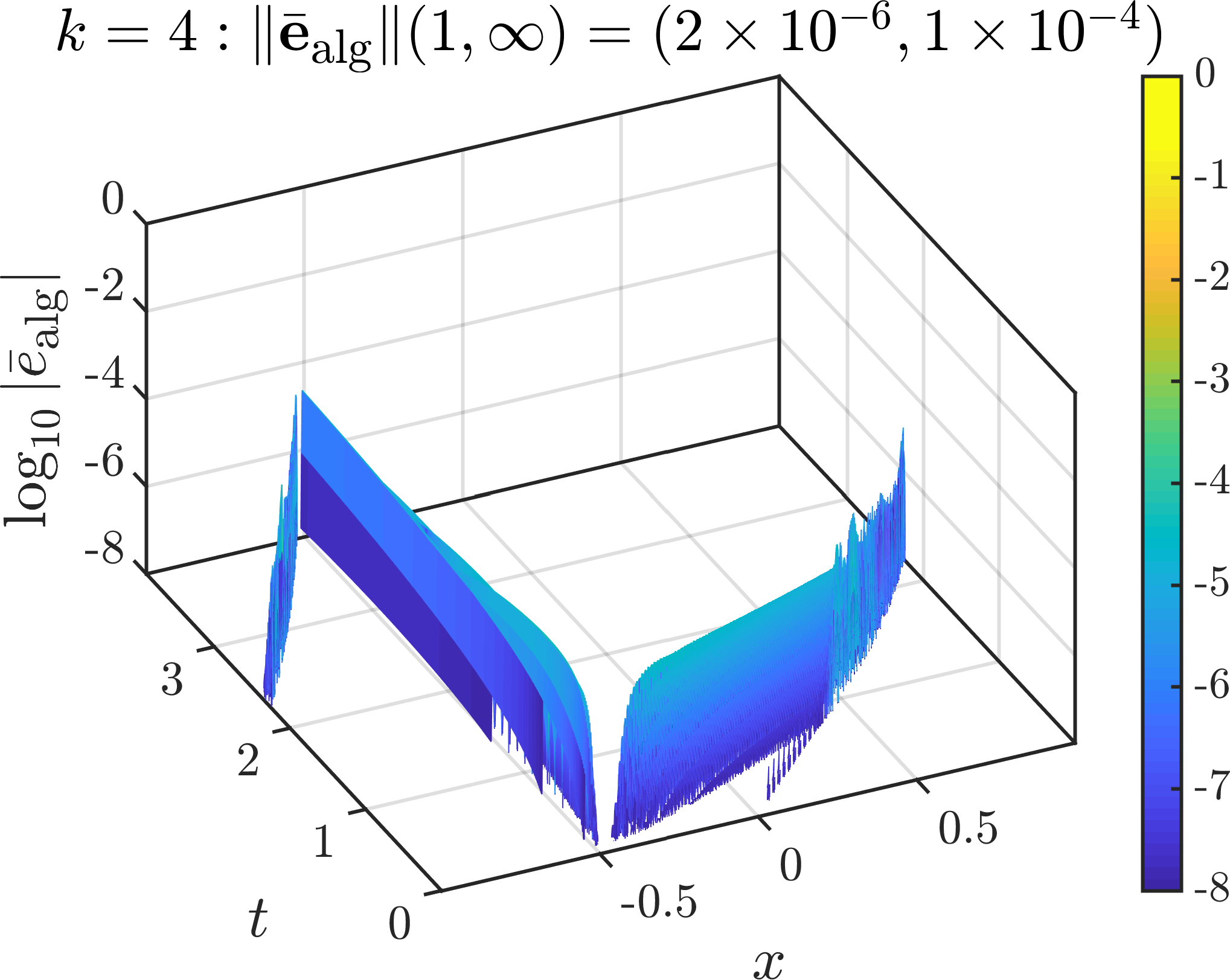}
\hspace{\hs ex}
\includegraphics[scale=\fs]{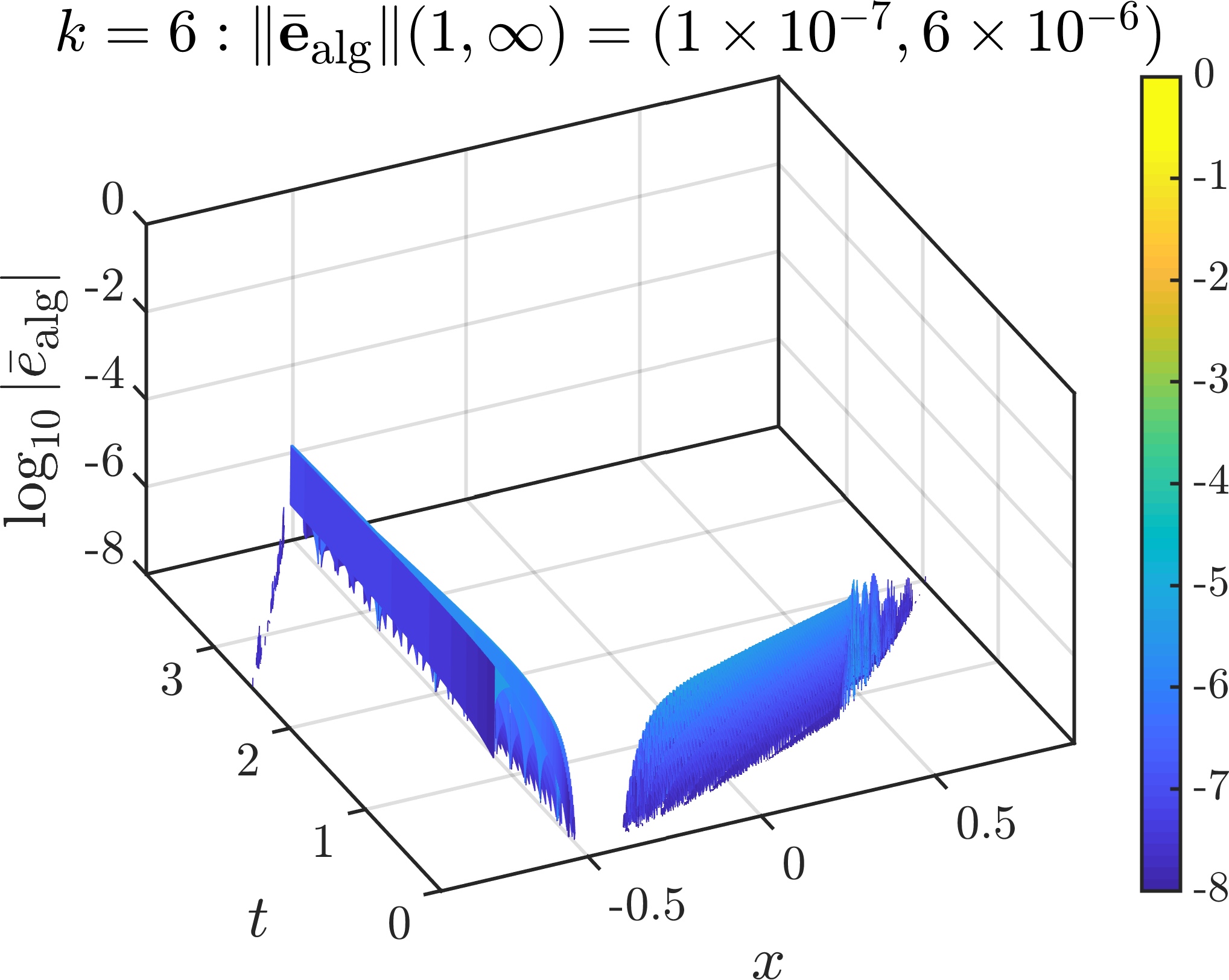}
}
\caption{Algebraic error in $\log_{10}$ scale of the iterates $\bb{u}_k$ produced by Algorithm \ref{alg:richardson} for the discretized problem from \cref{SMfig:os-exact-sol}. The error is shown for iterations $k = 0, 2, 4, 6$, as indicated in the title of each plot. Note the discrete $L^{1}$- and $L^{\infty}$-norms of each algebraic error is also shown in the title.
To aid with visualization, cell-wise errors less than $10^{-8}$ have been masked to white.
Note that the same colour scale is used for all plots.
\label{SMfig:os-rich}
}
\end{figure}

Considering the top left plot in \cref{SMfig:os-rich}, the algebraic error for the initial iterate has a very similar structure to the discretization error shown in the right of \cref{SMfig:os-exact-sol}. Recall that the initial iterate is obtained by interpolating the exact discretized solution from an $n_x = 512$ mesh. The key difference between the $n_x = 1024$ solution compared to $n_x = 512$ solution is that the distance that the shock and the base of the rarefaction wave are smeared across is halved. This is why the algebraic error is largest in these regions.

As the iteration proceeds, the algebraic error decreases rapidly, both in $L^1$- and $L^{\infty}$-norm---the fact that the $L^{\infty}$-norm goes to zero shows that we really do converge point-wise to the sequential time-stepping solution.
The algebraic error appears to be reduced much more slowly where the solution lacks regularity. We suspect this is because the linearization is least accurate in these regions where the gradient is large, or does not even exist.
Ironically, the primary effect of later iterations (e.g., $k \geq 4$) is to reduce the algebraic error in non-smooth regions, despite the fact that this is where the discretization is least accurate (see right panel of \cref{SMfig:os-exact-sol}).\footnote{This behavior of algebraic error rapidly reducing in smooth regions while more slowly decreasing in non-smooth regions is typical in our experience, even for problems with more complicated structure in smooth regions. For example, this occurs for the Buckley--Leverett test problem used in the paper (see the right of Figure \ref{fig:test-prob}), and for several Burgers problems we have tested where a shock develops from smooth initial conditions.}
In other words, these later iterations do not result in an approximation that more accurately approximates the true PDE solution.
In fact, the algebraic error at iteration $k = 2$ is already a factor of 10 smaller than the discretization error both in $L^1$- and $L^{\infty}$-norm.
Therefore, the $k = 2$ iterate is very likely just as good an approximation to the true PDE solution as the exact time-stepping solution is. 
We note, however, at iteration $k = 2$ that the algebraic error is not point-wise smaller than the discretization error everywhere in the domain. For example, there are neighbourhoods around the shock and base of the rarefaction wave where the discretization error is essentially machine zero, while the algebraic error is many orders of magnitude larger, even though the algebraic error is still relatively small in these regions, especially when compared to the discretization error in neighboring cells that are in or adjacent to the shock or the base of the rarefaction wave.


\end{document}